\documentclass[psamsfonts]{amsart}
\usepackage{amssymb,amsfonts,array}
\usepackage{amsmath}
\usepackage{mathtools}
\usepackage{tabulary}
\usepackage[all,arc]{xy}
\usepackage{enumerate}
\usepackage{easyReview}
\usepackage{mathrsfs}
\usepackage{xcolor}
\usepackage[hidelinks]{hyperref}
\hypersetup{
colorlinks=true,
linkcolor=blue,
filecolor=blue,
urlcolor=blue,
citecolor=cyan}

\usepackage{bm}
\usepackage{dsfont}
\usepackage{pgfplots}
\usepackage{tikz}
\usepackage{multirow}
\usepackage{tikz-cd}
\usetikzlibrary{cd,shapes,arrows,positioning,calc}
\usetikzlibrary{intersections,through,backgrounds}
\usepackage{tkz-base}
\usepackage{tkz-euclide}
\usepackage{geometry}
\usepackage[OT2,T1]{fontenc}
\DeclareSymbolFont{cyrletters}{OT2}{wncyr}{m}{n}
\DeclareMathSymbol{\Sha}{\mathalpha}{cyrletters}{"58}

\newtheorem{thm}{Theorem}[section]
\newtheorem{cor}[thm]{Corollary}
\newtheorem{prop}[thm]{Proposition}
\newtheorem{lemma}[thm]{Lemma}
\newtheorem{conj}[thm]{Conjecture}

\theoremstyle{definition}
\newtheorem{defn}[thm]{Definition}
\newtheorem{eg}[thm]{Example}
\newtheorem{assump}[thm]{Assumption}

\newtheorem{rmk}[thm]{Remark}

\newcommand{\AAA}{\mathbb{A}}
\newcommand{\BB}{\mathbb{B}}
\newcommand{\DD}{\mathbb{D}}

\newcommand{\GG}{\mathbb{G}}
\newcommand{\FF}{\mathbb{F}}
\newcommand{\QQ}{\mathbb{Q}}
\newcommand{\LL}{\mathbb{L}}
\newcommand{\ZZ}{\mathbb{Z}}
\newcommand{\RR}{\mathbb{R}}
\newcommand{\CC}{\mathbb{C}}

\newcommand{\MM}{\mathbb{M}}
\newcommand{\PP}{\mathbb{P}}
\newcommand{\XX}{\mathbb{X}}

\newcommand{\cK}{\mathcal{K}}
\newcommand{\cE}{\mathcal{E}}
\newcommand{\cP}{\mathcal{P}}
\newcommand{\cF}{\mathcal{F}}
\newcommand{\cO}{\mathcal{O}}
\newcommand{\cL}{\mathcal{L}}
\newcommand{\cU}{\mathcal{U}}
\newcommand{\cN}{\mathcal{N}}
\newcommand{\cI}{\mathcal{I}}

\newcommand{\cC}{\mathcal{C}}
\newcommand{\cY}{\mathcal{Y}}
\newcommand{\cZ}{\mathcal{Z}}
\newcommand{\cS}{\mathcal{S}}
\newcommand{\cA}{\mathcal{A}}
\newcommand{\cB}{\mathcal{B}}

\newcommand{\fA}{\mathfrak{A}}

\newcommand{\fu}{\mathfrak{u}}
\newcommand{\fg}{\mathfrak{g}}
\newcommand{\ft}{\mathfrak{t}}
\newcommand{\fb}{\mathfrak{b}}
\newcommand{\fp}{\mathfrak{p}}
\newcommand{\fq}{\mathfrak{q}}

\newcommand{\rC}{\mathrm{C}}
\newcommand{\rT}{\mathrm{T}}
\newcommand{\rS}{\mathrm{S}}
\newcommand{\bC}{\mathbf{C}}

\newcommand{\sT}{\mathsf{T}}
\newcommand{\sE}{\mathsf{E}}
\newcommand{\sX}{\mathsf{X}}
\newcommand{\sG}{\mathsf{G}}
\newcommand{\sZ}{\mathsf{Z}}

\newcommand{\sy}{\mathsf{y}}
\newcommand{\sx}{\mathsf{x}}
\newcommand{\szero}{\mathsf{p}}

\newcommand{\GL}{\mathrm{GL}}
\newcommand{\git}{{/\!\!/}}

\newcommand{\scA}{\mathscr{A}}
\newcommand{\scS}{\mathscr{S}}
\newcommand{\sSh}{\mathsf{Sh}}
\newcommand{\Sh}{\mathrm{Sh}}

\newcommand{\Ad}{\mathrm{Ad}}

\newcommand{\Ex}{\mathrm{Ex}}
\newcommand{\PCoh}{\mathrm{PCoh}}
\newcommand{\Coh}{\mathrm{Coh}}
\newcommand{\Shv}{\mathrm{Shv}}
\newcommand{\Rep}{\mathrm{Rep}}
\newcommand{\Hom}{\mathrm{Hom}}

\newcommand{\Aff}{\mathrm{Aff}}
\newcommand{\perf}{\mathrm{perf}}
\newcommand{\Perf}{\mathrm{Perf}}
\newcommand{\Perfd}{\mathrm{Perfd}}
\newcommand{\unip}{\mathrm{unip}}
\newcommand{\Isoc}{\mathrm{Isoc}}
\newcommand{\Loc}{\mathrm{Loc}}
\newcommand{\dom}{\mathrm{dom}}
\newcommand{\Fl}{\mathrm{Fl}}
\newcommand{\Adm}{\mathrm{Adm}}
\newcommand{\Iw}{\mathrm{Iw}}
\newcommand{\Sht}{\mathrm{Sht}}
\newcommand{\loc}{\mathrm{loc}}
\newcommand{\glob}{\mathrm{glob}}
\newcommand{\Nt}{\mathrm{Nt}}
\newcommand{\Lincat}{\mathrm{Lincat}}
\newcommand{\can}{\mathrm{can}}
\newcommand{\Ch}{\mathrm{Ch}}
\newcommand{\ch}{\mathrm{ch}}
\newcommand{\der}{\mathrm{der}}
\newcommand{\CohSpr}{\mathrm{CohSpr}}
\newcommand{\spec}{\mathrm{spec}}
\newcommand{\unr}{\mathrm{unr}}
\newcommand{\ev}{\mathrm{ev}}
\newcommand{\ad}{\mathrm{ad}}
\newcommand{\coarse}{\mathrm{coarse}}
\newcommand{\gen}{\mathrm{gen}}
\newcommand{\sgen}{\mathrm{sgen}}
\newcommand{\fingen}{\mathrm{f.g.}}
\newcommand{\Hk}{\mathrm{Hk}}
\newcommand{\Cusp}{\mathrm{Cusp}}
\newcommand{\tor}{\mathrm{tor}}
\newcommand{\id}{\mathrm{id}}
\newcommand{\hs}{\mathrm{hs}}
\newcommand{\mon}{\mathrm{mon}}
\newcommand{\Gr}{\mathrm{Gr}}
\newcommand{\Bun}{\mathrm{Bun}}
\newcommand{\HT}{\mathrm{HT}}
\newcommand{\crys}{\mathrm{crys}}
\newcommand{\BL}{\mathrm{BL}}
\newcommand{\act}{\mathrm{act}}
\newcommand{\geo}{\mathrm{geo}}
\newcommand{\ULA}{\mathrm{ULA}}
\newcommand{\sw}{\mathrm{sw}}
\newcommand{\tr}{\mathrm{tr}}
\newcommand{\Til}{\mathrm{Til}}
\newcommand{\mini}{\mathrm{min}}
\newcommand{\pFr}{\mathrm{pFr}}
\newcommand{\corr}{\mathrm{corr}}
\newcommand{\ab}{\mathrm{ab}}
\newcommand{\Fil}{\mathrm{Fil}}
\newcommand{\rs}{\mathrm{rs}}
\newcommand{\Wak}{\mathrm{Wak}}
\newcommand{\Tate}{\mathrm{Tate}}
\newcommand{\QCoh}{\mathrm{QCoh}}

\newcommand{\univ}{\mathrm{univ}}

\newcommand{\et}{{\textrm{\'et}}}
\newcommand{\fI}{\mathfrak{I}}
\newcommand{\Igs}{\mathrm{Igs}}
\newcommand{\Ig}{\mathrm{Ig}}
\newcommand{\cSht}{{\mathcal{S}ht}}
\newcommand{\cHk}{{\mathcal{H}k}}
\newcommand{\cGr}{{\mathcal{G}r}}
\newcommand{\bJ}{\mathbf{J}}

\newcommand{\cInd}{{\mathrm{c}\mbox{-}\mathrm{Ind}}}
\newcommand{\nInd}{{\mathrm{n}\mbox{-}\mathrm{Ind}}}

\newcommand{\Res}{\mathrm{Res}}
\newcommand{\Ind}{\mathrm{Ind}}

\newcommand{\colim}{\mathop{\operatorname{colim}}}
\newcommand{\Spec}{\operatorname{Spec}}
\newcommand{\Spf}{\operatorname{Spf}}
\newcommand{\Spa}{\operatorname{Spa}}
\newcommand{\Spd}{\operatorname{Spd}}

\title{On the generic part of the cohomology of Shimura varieties of abelian type}

\author{Xiangqian Yang}
\address{Beijing International Center for Mathematical Research, Peking University, China}
\email{yangxq@pku.edu.cn}

\author{Xinwen Zhu}
\address{Department of Mathematics, Stanford University, USA}
\email{zhuxw@stanford.edu}

\begin{document}
\begin{abstract}
	This article contributes to the study of the generic part of the cohomology of Shimura varieties. Under a mild restriction of the characteristic of the coefficient field, we prove a torsion vanishing result for Shimura varieties of abelian type, confirming a conjecture by Hamann--Lee. Our proofs utilize the unipotent categorical local Langlands correspondence and, in contrast to previous works, do not rely on the endoscopic classification of representations or on other results established through trace formula techniques.
\end{abstract}
%We also study the Mantovan's filtration and the Wakimoto filtration on the cohomology of Shimura varieties of Hodge type, and show that these filtrations canonically split under ``generic'' conditions. As a consequence, we improve a result by Liang Xiao and the second named author on the geometric realization of Jacquet--Langlands transform. 
\maketitle
\setcounter{tocdepth}{2}
\tableofcontents

\section{Introduction}
\subsection{Main results}
This article contributes to the study of the generic part of the cohomology of Shimura varieties.  We begin by addressing our main question, the torsion vanishing of the generic part of the cohomology of Shimura varieties, an investigation initiated in \cite{Caraiani-Scholze17,Caraiani-Scholze24}.

Let $(\sG,\sX)$ be a Shimura datum, where $\sG$ is a connected reductive group over $\QQ$ and $\sX$ is a $\sG(\RR)$-conjugacy class of cocharacters satisfying certain standard conditions. Let $\AAA_f$ denote the ring of finite ad\`eles of $\QQ$. Let $\sE\subseteq \CC$ denote the reflex field of $\sX$. For a \emph{neat} open compact subgroup $K\subseteq \sG(\AAA_f)$, there exists a Shimura variety $\sSh_K(\sG,\sX)$ defined over $\sE$. The set of $\CC$-points of $\sSh_K(\sG,\sX)$ is identified with the double quotient set $\sG(\QQ)\backslash \sX\times\sG(\AAA_f)/K.$ Fix a prime number $\ell$ and denote $\Lambda=\overline{\FF}_\ell$ or $\overline{\QQ}_\ell$. Let
$$R\Gamma(\sSh_{K}(\sG,\sX)_{\overline\sE},\Lambda)\quad\text{resp.}\quad R\Gamma_c(\sSh_{K}(\sG,\sX)_{\overline\sE},\Lambda)$$
denote the \'etale cohomology (resp. compactly supported \'etale cohomology) of $\sSh_{K}(\sG,\sX)_{\overline\sE}$ with coefficients in $\Lambda$.

Fix a prime $p\neq \ell$ unramified with respect to $(\sG,\sX,K)$, which means that $G\coloneqq \sG_{\QQ_p}$ is unramified over $\QQ_p$ and that the level $K$ splits into $K^pK_p^\hs$ where $K^p\subseteq\sG(\AAA_f^p)$ is a neat open compact subgroup and $K_p^{\hs}$ is a \emph{hyperspecial} subgroup of $G(\QQ_p)$. Let
$$H_{K_p^\hs}\coloneqq \Lambda[K_p^\hs\backslash G(\QQ_p)/K_p^\hs]$$
be the spherical Hecke algebra over $\Lambda$. After fixing a square root of $p$ in $\Lambda$, the Satake isomorphism identifies every closed point $\xi$ of $\Spec H_{K_p^\hs}$ with a conjugacy class of semisimple unramified $L$-parameters
$$\varphi_\xi\colon W_{\QQ_p}\to {}^LG(\Lambda),$$
where $W_{\QQ_p}$ is the Weil group of $\QQ_p$ and ${}^LG=\hat{G}\rtimes W_{\QQ_p}$ is the Langlands dual group of $G$ over $\Lambda$. Up to $\hat{G}(\Lambda)$-conjugation, we can assume that $\varphi_\xi$ factors through a maximal torus ${}^LT\subseteq {}^LG$.

\begin{defn}\label{def-generic-unramified}
    Let $\varphi\colon W_{\QQ_p}\to {}^LG(\Lambda)$ be a semisimple unramified $L$-parameter that factors through ${}^LT$. 
    \begin{enumerate}
        \item We say that $\varphi$ is \emph{generic} if $H^2(W_{\QQ_p},(\hat\fg/\hat\ft)_{\varphi})=0$.
        %\item We say that $\varphi$ is \emph{regular semisimple} if $H^0(W_{\QQ_p},(\hat\fg/\hat\ft)_{\varphi})$ vanishes.
        \item We say that $\varphi$ is \emph{strongly generic} if $R\Gamma(W_{\QQ_p},(\hat\fg/\hat\ft)_{\varphi})=0$.
    \end{enumerate}
    Here $(\hat\fg/\hat\ft)_{\varphi}$ is the vector space $\hat\fg/\hat\ft$ with the $W_{\QQ_p}$-action given by the composition of $\varphi$ and the adjoint action ${}^LT\to\GL(\hat\fg/\hat\ft)$.
    We say that $\xi$ is generic (resp. strongly generic) if the associated semisimple unramified $L$-parameter $\varphi_\xi$ is generic (resp. strongly generic).
\end{defn}
By Euler characteristic formula, $\varphi$ is strongly generic if and only if it is generic and $H^0(W_{\QQ_p},(\hat\fg/\hat\ft)_\varphi)$ vanishes. 

\begin{rmk}
    Let $\varphi$ be a semisimple unramified $L$-parameter. Then it is generic (resp. strongly generic) if and only if it is of weakly Langlands--Shahidi type (resp. of Langlands--Shahidi type) in the sense of \cite[Definition 6.2]{Hamann-Lee-vanishing}.
\end{rmk}

\begin{eg}\label{eg-intro-generic}
    Let $G=\GL_{n,\QQ_p}$. A semisimple unramified parameter $\varphi\colon W_{\QQ_p}\to \GL(\Lambda)$ can be written as a direct sum $\varphi=\varphi_1\oplus\cdots\oplus\varphi_n$ where $\varphi_i\colon W_{\QQ_p}\to \Lambda^\times$ are unramified characters. Then $\varphi$ is generic (resp. strongly generic) if and only if $\varphi_i/\varphi_j\neq \chi_{\mathrm{cycl}}$ (resp. $\varphi_i/\varphi_j\notin \{1 ,\chi_{\mathrm{cycl}}\}$) for any $i\neq j$, where $\chi_{\mathrm{cycl}}$ is the cyclotomic character. Therefore, in this case $\varphi$ is generic if and only if it satisfies the condition in \cite[Theorem 1.1]{Caraiani-Scholze24}.
\end{eg}

Note that $R\Gamma(\sSh_{K^pK_p^\hs}(\sG,\sX)_{\overline\sE},\Lambda)$ and $R\Gamma_c(\sSh_{K^pK_p^\hs}(\sG,\sX)_{\overline\sE},\Lambda)$ admit natural actions of $H_{K_p^\hs}$. Let
$$R\Gamma_c(\sSh_{K^pK_p^\hs}(\sG,\sX)_{\overline\sE},\Lambda)_\xi\quad\text{resp.}\quad R\Gamma(\sSh_{K^pK_p^\hs}(\sG,\sX)_{\overline\sE},\Lambda)_\xi$$
denote the localization of the cohomologies at $\xi$. Motivated by the results of \cite{Caraiani-Scholze17} and \cite{Caraiani-Scholze24}, Hamann--Lee proposed the following conjecture on the torsion vanishing of Shimura varieties in \cite[Conjecture 1.2, Remark 1.7]{Hamann-Lee-vanishing}. 

\begin{conj}\label{conj-torsion-vanishing}
    Assume that $\xi$ is generic. Then $R\Gamma(\sSh_{K^pK_p^\hs}(\sG,\sX)_{\overline\sE},\Lambda)_\xi$ (resp. $R\Gamma_c(\sSh_{K^pK_p^\hs}(\sG,\sX)_{\overline\sE},\Lambda)_\xi$) is concentrated in degrees $[d,2d]$ (resp. $[0,d]$), where $d=\dim \sSh_{K^pK_p^\hs}(\sG,\sX)$. In particular, if $\sSh_{K^pK_p^\hs}(\sG,\sX)$ is proper, then $R\Gamma(\sSh_{K^pK_p^\hs}(\sG,\sX)_{\overline\sE},\Lambda)_\xi$ is concentrated in degree $d$.
\end{conj}

Some cases of Conjecture \ref{conj-torsion-vanishing} have been proved in earlier works: 
\begin{enumerate}
    \item In the case of Harris--Taylor Shimura varieties, Boyer proved the conjecture in \cite{Boyer-torsion}. He also established cohomological bounds beyond the generic cases.
    \item When $\sSh_{K^pK_p^\hs}(\sG,\sX)$ is a unitary PEL type Shimura variety and $G=\sG_{\QQ_p}$ is essentially a product of $\GL_{n,\QQ_p}$, the conjecture follows from the works of Caraiani--Scholze \cite{Caraiani-Scholze17,Caraiani-Scholze24}, and Koshikawa \cite{Koshikawa-generic}.
    \item For Hilbert modular varieties, the conjecture follows from the work of Caraiani--Tamiozzo \cite{Caraiani-Tamiozzo}. Moreover, they also studied the cohomology amplitude in the non-generic case.
    \item In \cite{Hamann-Lee-vanishing}, Hamann--Lee proved a weaker version of the conjecture for PEL type Shimura varieties when $G=\sG_{\QQ_p}$ is essentially a product of unramified Weil restrictions of $\GL_n$, odd unitary groups, and $\mathrm{GSp}_4$, etc. See \cite[Theorem 1.8]{Hamann-Lee-vanishing} for their assumptions on $G$. In fact, they needed $\xi$ to be \emph{strongly generic} (of Langlands--Shahidi type in their terminology) in their method. 
    \item In \cite{Peng-FS-comparison-orth-unitary}, Hao Peng proved the conjecture for certain compact unitary and orthogonal Shimura varieties, still under the assumption that $\xi$ is strongly generic.
\end{enumerate}

In this paper, we prove the conjecture for Shimura varieties of abelian type. 

\begin{thm}[{Theorem \ref{thm-torsion-vanising-abelian-type}}]\label{thm-main-intro} Let $(\sG,\sX)$ be a Shimura datum of abelian type, with level $K\subset \sG(\mathbb{A}_f)$. Let $p>2$ be an unramified prime.  Assume that either $\Lambda=\overline{\QQ}_\ell$ or $\Lambda=\overline{\FF}_\ell$ with $\ell$ bigger than the Coxeter number of any simple factors of $\sG_\ad$. Then Conjecture \ref{conj-torsion-vanishing} is true. 
\end{thm}

\begin{rmk}
    We actually prove a stronger result: (a variant of) Conjecture \ref{conj-torsion-vanishing} holds for any quasi-parahoric subgroup $K_p\subseteq G(\QQ_p)$ in place of $K_p^\hs$. 
\end{rmk}

\begin{rmk}
    When $\Lambda=\overline{\FF}_\ell$, the restriction of $\ell$ in the above theorem comes from the same restriction in the unipotent categorical local Langlands correspondence as established in \cite{Tame}, which in turn comes from the same restriction in the unipotent geometric local Langlands correspondence established in \cite{BR-modular-two-affine-Hecke}. We expect that this bound can be greatly improved. For example, when $G=\GL_n$, there is in fact no restriction on $\ell$, by the work of \cite{Caraiani-Scholze24,Koshikawa-generic}. Similarly, when $G$ is of orthogonal or symplectic type, we expect the theorem to hold at least $\ell>3$.
\end{rmk}

\subsection{Overview of the proof}
Our proof of Theorem \ref{thm-main-intro} is based on the ideas of \cite{Hamann-Lee-vanishing}. But we will apply them in a different setting.
To explain this, we first review the strategy of \cite{Hamann-Lee-vanishing}. Assume that $(\sG,\sX)$ is of Hodge type. Fix a $p$-adic place $v$ of $\sE$ and denote $E=\sE_v$. Let $\mu$ be the Hodge cocharacter associated with $(\sG,\sX)$. By \cite{Scholze-on-torsion} and \cite{Caraiani-Scholze17}, there exists a Hodge--Tate period map
$$\pi_{\HT}\colon \sSh_{K^p}^{\circ}(\sG,\sX)_{E}\to \mathscr{F}\ell_{G,\mu^*},$$
where $\sSh_{K^p}^{\circ}(\sG,\sX)_{E}$ is the good reduction locus in the analytification of $\sSh_{K^p}(\sG,\sX)_E$, $\mu^*$ is the dominant Weyl translate of $\mu^{-1}$, and $\mathscr{F}\ell_{G,\mu^*}$ is the analytification of the flag variety associated to $\mu^*$. By analyzing the geometry of $\pi_{\HT}$, Hamann--Lee proved in \cite{Hamann-Lee-vanishing} that there is a filtration on $R\Gamma_c(\sSh_{K^p}(\sG,\sX)_{\overline{E}},\Lambda)[d]$ with graded pieces taking the form 
$$R\Gamma_{c-\partial}(\Ig_b,\Lambda)\otimes_{G_b(\QQ_p)}R\Gamma_c(G,b,\mu)[2\langle2\rho,\nu_b\rangle]$$
where $b\in B(G,\mu^*)$ runs through the Kottwitz set. Here $R\Gamma_{c-\partial}(\Ig_b,\Lambda)$ is defined in \cite{Caraiani-Scholze24} as the partially compactly supported cohomology of the Igusa variety $\Ig_b$, and $R\Gamma_c(G,b,\mu)$ is the compactly supported cohomology of the local Shimura variety associated to $(G,b,\mu)$. More precisely, let $\Ig_b$ denote the Igusa variety and $\Ig_b^\mini$ be the partial minimal compactification of $\Ig_b$ constructed in \cite{Mao-Hodge-well-positioned}. Then
$$R\Gamma_{c-\partial}(\Ig_b,\Lambda)\coloneqq R\Gamma(\Ig_b^\mini,(j_{\Ig_b^\mini})_!\Lambda_{\Ig_{b}}),$$
where $j_{\Ig_b^\mini}\colon \Ig_b\hookrightarrow \Ig_b^\mini$ is the natural embedding. The proof of \emph{loc. cit.} now consists of the following two steps:
\begin{enumerate}
    \item Controlling the cohomological amplitude of $R\Gamma_{c-\partial}(\Ig_b,\Lambda)$.
    \item Controlling the cohomological amplitude of the (strongly) generic part of $R\Gamma_c(G,b,\mu)$.
\end{enumerate}
The first step follows from the Artin vanishing, as the partial minimal compactifications of Igusa varieties are affine. The second step follows from the perverse $t$-exactness of the Hecke action on $D(\Bun_G,\Lambda)$ when localizing at a strongly generic $L$-parameter. Here $\Bun_G$ is the moduli stack of $G$-bundles on the Fargues--Fontaine curve, and the Hecke action is defined in \cite{FS}. The $t$-exactness was essentially proved in \cite{Hamann-Eisenstein} by studying the Eisenstein functors in the setting of \cite{FS}. However, to analyze the category $D(\Bun_G,\Lambda)_\xi$ of objects localized at $\xi$, the results in \cite{Hamann-Eisenstein} crucially rely on the comparison between the semisimple $L$-parameters defined in \cite{FS} and those constructed using trace formula methods. In particular, such comparison result is not known for $\mathrm{GSp}_{2n}$ when $n\geq 3$, which means that the torsion vanishing result for general Siegel modular varieties was not known prior to our work. 

In this paper, we adapt the ideas of \cite{Hamann-Lee-vanishing} in the setting of \cite{Tame}. One of the main results of \cite{Tame} (recalled in \S\ref{subsection-cllc}) is the construction of a fully faithful embedding
$$\LL_G^{\unip}\colon \Ind\Shv^{\unip}_\fingen(\Isoc_G,\Lambda)\hookrightarrow \Ind\Coh(\Loc^{\widehat\unip}_{{}^LG,\QQ_p}),$$
where $\Isoc_G$ is the moduli stack of isocrystals with $G$-structure, $\Shv^{\unip}_\fingen(\Isoc_G,\Lambda)$ is the category of ($\ell$-adic) sheaves on $\Isoc_G$ that are finitely generated unipotent representations when restricted to each Newton strata, and $\Loc^{\widehat\unip}_{{}^LG,\QQ_p}$ is the moduli stack of unipotent $L$-parameters. Let $\Shv^{\widehat\unip}(\Isoc_G,\Lambda)$ be the subcategory of $\Shv(\Isoc_G,\Lambda)$ consisting of the objects that are unipotent when restricted to each Newton strata. The category $\Ind\Shv^{\unip}_\fingen(\Isoc_G,\Lambda)$ is a renormalization of $\Shv^{\widehat\unip}(\Isoc_G,\Lambda)$. It follows that there is a well-defined spectral action of $\Perf(\Loc^{\widehat\unip}_{{}^LG,\QQ_p})$ on $\Shv^{\unip}_{\fingen}(\Isoc_G,\Lambda)$ and $\Shv^{\widehat\unip}(\Isoc_G,\Lambda)$. In particular, for any $\hat{G}$-representation $V$, we obtain a Hecke operator 
$$T_V\colon \Shv^{\widehat\unip}(\Isoc_G,\Lambda)\to\Shv^{\widehat\unip}(\Isoc_G,\Lambda).$$

In \cite{Tame}, an exotic $t$-structure $(\Shv(\Isoc_G,\Lambda)^{e,\leq 0},\Shv(\Isoc_G,\Lambda)^{e,\geq 0})$ on $\Shv(\Isoc_G,\Lambda)$ was defined which is parallel to the pervese $t$-structure on $D(\Bun_G,\Lambda)$. Let $\xi$ be a conjugacy class of unramified $L$-parameters. We can define the subcategories $\Shv^{\widehat\unip}(\Isoc_G,\Lambda)_\xi$ and $\Ind\Coh(\Loc^{\widehat\unip}_{{}^LG,\QQ_p})_\xi$ of objects supported at $\xi$. The following theorem is our main result concerning the generic part of the unipotent categorical local Langlands correspondence.

\begin{thm}[Proposition \ref{prop-generic-equiv}, Theorem \ref{thm-Hecke-t-exact}, and Theorem \ref{thm-char-0-t-str}]\label{thm-Hecke-t-exact-intro} Assume that $\xi$ is generic.
\begin{enumerate}
    \item  The functor $\LL_G^{\unip}$ restricts to an equivalence of categories
$$\LL_{G,\xi}^{\unip}\colon \Ind\Shv^{\unip}_\fingen(\Isoc_G,\Lambda)_\xi\cong\Ind\Coh(\Loc^{\widehat\unip}_{{}^LG,\QQ_p})_\xi.$$

    \item Let $V$ be a tilting $\hat{G}$-representation.  Then the Hecke action $T_V$ is exotic $t$-exact when restricted to the subcategory $\Shv^{\widehat\unip}(\Isoc_G,\Lambda)_\xi$. Moreover, when $\Lambda=\overline{\QQ}_\ell$ the functor $\LL^{\unip}_{G,\xi}$ is compatible with the exotic $t$-structure on $\Shv^{\widehat\unip}(\Isoc_G,\Lambda)_\xi$ and the standard $t$-structure on $\Ind\Coh(\Loc^{\widehat\unip}_{{}^LG,\QQ_p})_\xi$.    \end{enumerate}
\end{thm}

The parallel statements in the Fargues--Scholze setting were conjectured by Hansen in \cite[Conjecture 2.4.1, Conjecture 2.5.1]{Hansen-Beijing-notes} (when $\Lambda=\overline{\QQ}_\ell$). He also outlined a proof for the trivial unramified parameter, assuming the existence of the categorical local Langlands correspondence satisfying certain properties. Our proof uses a similar strategy. A crucial ingredient is provided by Proposition \ref{cor-tilting-tensor-twisted-case}, which asserts that certain coherent sheaves on the (twisted) adjoint quotient of $\hat{G}$ admit nice filtrations. We will deduce Proposition \ref{cor-tilting-tensor-twisted-case} from a result of \cite{MR-exotic-t-str} on exotic coherent sheaves on the Springer resolution. We also mention that we can deduce from the above theorem a torsion vanishing result for the cohomology of affine Deligne--Lusztig varieties. See Theorem \ref{thm-coh-vanishing-ADLV}.

To apply the above results, we need to relate the categorical local Langlands to the cohomology of Shimura varieties. This is done using the Igusa stack constructed in \cite{DHKZ-igusa}. Let $I\subseteq G(\QQ_p)$ be an Iwahori subgroup with the associated Iwahori group scheme $\cI$ over $\ZZ_p$. Let $\Sh_\mu$ denote the perfection of the special fiber of the integral model of $\sSh_{K^pI}(\sG,\sX)$ over $O_E$. By the result of \cite{DHKZ-igusa}, there is a Cartesian diagram
$$\begin{tikzcd}
    \Sh_\mu \ar[r,"\loc_p"]\ar[d,"\Nt^\glob"swap] & \Sht^\loc_{\cI,\mu} \ar[d,"\Nt"] \\
    \Igs \ar[r,"\loc_p^0"] & \Isoc_{G,\leq\mu^*}
\end{tikzcd}$$
of perfect stacks over $\overline{\FF}_p$. Here $\Igs$ is the perfect Igusa stack (of prime-to-$p$ level $K^p$), $\Sht^\loc_{\cI,\mu}$ is the moduli stack of local shtukas bounded by $\mu$, and $i_{\leq\mu^*}\colon \Isoc_{G,\leq\mu^*}\hookrightarrow \Isoc_G$ is the closed substack corresponding to the subset $B(G,\mu^*)\subseteq B(G)$. The morphism $\loc_p$ is called \emph{crystalline period map} previously introduced in \cite{XZ-cycles} and \cite{SYZ-EKOR}, which is an analogue of the Hodge--Tate period map over the special fiber.

One of the new ideas of this work is to consider an object $\omega^\can_{\Igs}\in\Shv(\Igs,\Lambda)$ defined in Definition \ref{def-omega-can}, and to define the \emph{$!$-Igusa sheaf}
$$\fI^\can\coloneqq (\loc_p^0)_\flat\omega^\can_{\Igs}\in \Shv(\Isoc_{G,\leq\mu^*},\Lambda),$$
where $(\loc_p^0)_\flat$ is the right adjoint of $(\loc_p^0)^!$. We notice that the definitions of $\omega^\can_{\Igs}$ and $\fI^\can$ are purely formal and are based on categorical nonsense. However, surprisingly the stalks of $\fI^\can$ exactly compute the partially compactly supported cohomology $R\Gamma_{c-\partial}(\Ig_b,\Lambda)$ introduced above. More precisely, we have:
\begin{prop}[Proposition \ref{prop-igusa-sheaf-stalk}]\label{prop-stalk-intro}
    For $b\in B(G,\mu^*)$, there is a natural isomorphism
    $$(i_b)^!\fI^\can\simeq R\Gamma_{c-\partial}(\Ig_b,\Lambda)$$
    in $\Shv(\Isoc_{G,b},\Lambda)\simeq \Rep(G_b(\QQ_p),\Lambda)$. Here, $i_b\colon \Isoc_{G,b}\hookrightarrow \Isoc_{G,\leq\mu^*}$ is the Newton stratum associated to $b$.
\end{prop}
This proposition is proved using the results in \cite{Mao-Hodge-well-positioned} on well-positionedness of Igusa varieties.
Together with the affineness of the partial minimal compactification of $\Ig_b$ also proved in \cite{Mao-Hodge-well-positioned} and the Artin vanishing, we show that the object $\fI^\can$ lies in the connective part of the exotic $t$-structure. This is analogue to the semi-perversity results in \cite{Caraiani-Scholze17}, \cite{Caraiani-Scholze24}, and \cite{DHKZ-igusa}. If $\sSh_{K^pI}(\sG,\sX)$ is proper, this already appears in \cite[Proposition 6.21]{Tame}. We note that we work entirely over the special fiber of the Shimura variety and only use the partial minimal compactifications of Igusa varieties. In particular, we do not need the minimal compactification of the Hodge-Tate period map (or crystalline period map) and therefore avoid the additional assumptions on the boundary of Shimura varieties as in \cite[Assumption 1.11]{Hamann-Lee-vanishing}.

%\begin{thm}[Theorem \ref{thm-igusa-sheaf-exotic}]\label{thm-igusa-conn-intro}
%The $!$-Igusa sheaf $\fI^\can$ lies $\Shv(\Isoc_G,\Lambda)^{e,\leq0}$.
%\end{thm}
%  Therefore the connectiveness of $\fI^\can$ follows from Artin vanishing. 

Next, define the \emph{unipotent spectral Igusa sheaf} as  
$$\fI^{\can,\unip}_{\spec}\coloneqq \LL_G^{\unip}(\cP^{\unip}\circ\Psi^L((i_{\leq\mu^*})_*\fI^\can))\in \Ind\Coh(\Loc^{\widehat\unip}_{{}^LG,\QQ_p}),$$
where $\Psi^L\colon\Shv(\Isoc_G,\Lambda)\hookrightarrow\Ind\Shv_\fingen(\Isoc_G,\Lambda)$ is the natural fully faithful embedding, and $\cP^{\unip}$ is the right adjoint of the natural embedding $\Ind\Shv_{\fingen}^\unip(\Isoc_G,\Lambda)\hookrightarrow \Ind\Shv_\fingen(\Isoc_G,\Lambda)$. 

Let $\widetilde{V}_\mu$ be the vector bundle on $\Loc^{\widehat\unip}_{{}^LG,\QQ_p}$ associated to the highest weight representation $V_\mu$ of $\hat{G}$ and $\CohSpr_{{}^LG}^\unip$ be the unipotent coherent Springer sheaf defined in \cite{Tame}. Note that $\widetilde{V}_\mu$ admits a tautological action of the Weil group $W_E$ and $\CohSpr^\unip_{{}^LG}$ carries a natural action of the Iwahori--Hecke algebra $H_I$ by \cite[Corollary 1.9]{Tame}. We have the following local-global compactibility result, which we believe will be useful in many other related problems.

\begin{thm}[Theorem \ref{thm-local-global-compatibility}]\label{thm-local-global-intro}
    There is an $H_{K^p}\times H_I\times W_E$-equivariant isomorphism
$$\Hom(\widetilde{V}_\mu\otimes\CohSpr^\unip_{{}^LG},\fI^{\can,\unip}_{\spec})\simeq R\Gamma_c(\sSh_{K^pI}(\sG,\sX)_{\overline{E}},\Lambda)(d/2)[d].$$
\end{thm}

In fact, we prove a more general result for not necessarily quasi-split $G=\sG_{\mathbb{Q}_p}$ (but we still require $G$ splits over an unramified extension of $\mathbb{Q}_p$). Then in the above formula, one just replaces the coherent Springer sheaf by the coherent sheaf corresponding to $\cInd_I^{G(\mathbb{Q}_p)}\Lambda$ under the categorical local Langlands correspondence.

We emphasize that the nature of the above isomorphism is that we can compute the \'etale cohomology of Shimura varieties in terms of coherent cohomology on the stack of local Langlands parameters. 
We remark that the existence of such an isomorphism follows formally from the categorical local Langlands and the base change. However, to prove that the isomorphism is compatible with Hecke actions requires more works. This is done by generalizing the $S=T$ type results in \cite{Wu-S-equal-T} to the Iwahori level. 

Combining Theorem \ref{thm-Hecke-t-exact-intro}, Proposition \ref{prop-stalk-intro}, and Theorem \ref{thm-local-global-intro}, we deduce Theorem \ref{thm-main-intro} for Shimura varieties of Hodge type. Then we deduce the abelian type case from the Hodge type case using the description of connected components of Shimura varieties in \cite{Deligne-Shimura}.

We note that we only make use of the unipotent part of the categorical local Langlands correspondence. In particular, all the reductive groups over $\QQ_p$ will split over an unramified extension of $\QQ_p$. The more general tame part of the categorical local Langlands as considered in \cite{Tame} is not needed. %In addition, the proof of  Theorem \ref{thm-main-intro} does not rely on the $W_E$-compatibility in Theorem \ref{thm-local-global-intro}, and therefore does not rely on the forthcoming work of Bando--Gleason--Lourenço--Yu on comparison of central sheaves in the mixed characteristic affine Hecke categories (see \cite[Remark 5.2.(4)]{Tame}). \Xinwen{How about this?}%\Xiangqian{The proofs here still relies on the the comparison of central sheaves used in the mixed char Bez equivalence with those defined by nearby cycles in \cite{ALWY-mixed-central}. But with the new approach which doesn't need the Bez equivalence, we can use the central sheaves in \cite{ALWY-mixed-central} instead from be beginning.}

As a final remark in the introduction, we note that although our strategy is similar to the previous works, the actual proofs are different. In particular, we do not need to use any comparison between classical local Langlands and categorical local Langlands, nor any other results established via the trace formula methods such as Jacquet--Langlands transfer, as required in all the previous works.

%\begin{rmk}
%    We note that in this article we only need the unipotent part of the categorical local Langlands correspondence established in \cite{Tame}. In particular, all the reductive groups over $\QQ_p$ will split over an unramified extension of $\QQ_p$.
    
%    On the other hand, the results in \S\ref{section-shimura} relies on the forthcoming work of Bando--Gleason--Lourenço--Yu on comparison of central sheaves in the mixed characteristic affine Hecke categories. See \cite[Remark 5.2.(4)]{Tame}. 
%\end{rmk}
%\Xinwen{After a second thought, maybe remove this remark. I originally hoped to say explicitly that we do not need of full \cite{Tame}, in particular we do not need to the tame part (which still relies on some forthcoming work). I do not want to make explicitly that even the unipotent part needs the work of Bando--Gleason--Lourenço--Yu.
%}

\subsection{Notations and conventions}
We follow \cite[\S 7]{Tame} for the notations and conventions regarding category theory. In particular, unless otherwise specified, the term ``categories‘’ refer to $(\infty,1)$-categories. By a linear category, we mean a presentable, stable $\infty$-category. All functors between linear categories are derived. If $\Lambda$ is a commutative ring, let $\Lincat_\Lambda$ be the category of presentable stable $\Lambda$-linear $\infty$-categories.

For a (derived) scheme or stack $X$, we let $\QCoh(X)$ (resp. $\Coh(X)$) denote the $\infty$-category of quasi-coherent (resp. coherent) sheaves on $X$. 
%\Xinwen{Add more here......}

If $H$ is an algebraic group over a field, let $H^\circ$ denote the neutral connected component of $H$.

If $T$ is a multiplicative group over a field $k$, let $$\XX^\bullet(T)\coloneqq\Hom(T_{\overline{k}},\GG_{m,\overline{k}}),\quad\text{resp.}\quad \XX_\bullet(T)\coloneqq\Hom(\GG_{m,\overline{k}},T_{\overline{k}})$$
denote the weight (resp. coweight) lattice of $T$.

Assume that $G$ is a connected reductive group over a field. Let $G_\der$ denote derived subgroup of $G$. Let $Z(G)$ denote center of $G$. Let $G_\ad=G/Z(G)$ denote adjoint group of $G$ and $G_\ab=G/G_\der$ denote the abelianization of $G$. Let $G_\mathrm{sc}$ denote the simply-connected cover of $G_\der$. 

\subsection{Acknowledgement}
We thank Liang Xiao for helpful discussions.
Part of the work was done when X.Y. was visiting Stanford University, and he thanks them for their hospitality. 
X.Z. is supported by NSF grant under DMS-2200940.

\section{Coherent sheaves on the twisted adjoint quotient}\label{section-Springer}
Let $G$ be a connected reductive group over an algebraically closed field $\Lambda$ equipped with a finite order automorphism $\phi$. Then one can form the disconnected reductive group $G\rtimes\langle\phi\rangle$ and consider twisted adjoint quotient $G\phi/G\to G\phi\git G$. For a point $\xi\in G\phi\git G$, let $(G\phi/G)_\xi^\wedge$ denote the formal completion of $G\phi/G$ along the fiber of $\xi$.
In this section, we study the category of coherent sheaves on $(G\phi/G)_\xi^\wedge$. 
We will construct generators of this category in Proposition \ref{prop-generators-GS}. We will show certain coherent sheaves on $(G\phi/G)_\xi^\wedge$ admit nice filtrations in Proposition \ref{cor-tilting-tensor-twisted-case}, which will be a key input in proving the $t$-exactness of Hecke actions in Theorem \ref{thm-Hecke-t-exact} and Theorem \ref{thm-char-0-t-str}.
A main tool to prove these results are the exotic coherent sheaves on the Springer resolution, as developed in \cite{Bez-coh-tilting, MR-exotic-t-str}.

The results presented in this section will be applied to the spectral side of the categorical local Langlands correspondence in the subsequent section. Consequently, the reductive group 
$G$ discussed here will correspond to the dual group $\hat{G}$ in later sections.

\subsection{Exotic and perverse coherent sheaves}\label{subsection-exotic-and-perverse-coh}
We fix some notations and conventions that will be used throughout this section.

We will fix a pinning $(B,T,e)$ of $G$, consisting of a Borel subgroup $B\subseteq G$ and a maximal torus $T\subseteq B$ of $G$. Let $U$ be the unipotent radical of $B$. Then $e: U\to \GG_a$ is a homomorphism which restricts to an isomorphism on every simple root subgroup of $U$. We let $\ft\subset\fb\subset\fg$ and $\fu\subset\fb$ be the corresponding Lie algebras.

Let $\XX^\bullet(T)^+\subseteq \XX^\bullet(T)$ be the set of dominant weights associated to the Borel subgroup $B^+$ of $G$ \emph{opposite} to $B$. Let $U\subseteq B$ be the unipotent radical. Let $W$ be the Weyl group of $G$ and $w_0\in W$ be the longest element. Let $\Phi(G,T)\subseteq \XX^\bullet(T)$ (resp. $\Phi^\vee(G,T)\subseteq \XX_\bullet(T)$) be the set of roots (resp. coroots) of $G$. Let $\Delta\subseteq \XX^\bullet(T)$ (resp. $\Delta^\vee\subseteq \XX_\bullet(T)$) be the set of simple roots (resp. coroots) of $G$ with respect to the opposite Borel $B^+$. 

We consider the following ``standard'' assumption on $G$ in this subsection. 
\begin{assump}\label{assump-standard-on-reductive}
\begin{enumerate}
    \item The characteristic of $\Lambda$ is either zero or is good with respect to every simple factor of (the adjoint group of) $G$.
    \item The derived subgroup of $G$ is simply-connected.
    %The characteristic of $\Lambda$ is either zero or does not divide the order of the fundamental group of the derived subgroup of $G$. \Xinwen{Or maybe replace this by requiring the derived subgroup of $G$ to be simply-connected.}
    \item There exists a non-degenerated $G$-invariant bilinear form on the Lie algebra $\fg$ of $G$.
\end{enumerate}
\end{assump}

\begin{rmk}\label{rmk-bilinear-form}
    If $G$ is a simple and simply-connected reductive group not of type $\mathsf{A}_n$, then (1) implies (3). On the other hand, (3) is always satisfied for $G=\GL_n$.  
\end{rmk}

We recall the theory of exotic and perverse coherent sheaves following \cite{Achar-exotic-coh} and \cite{MR-exotic-t-str}.
\subsubsection{Exotic Coherent sheaves}

Let $\cN\subseteq \fg$ be the nilpotent cone. Let $\pi\colon \fu/B\to \cN/G$ denote the Springer resolution. For $\lambda\in \XX^\bullet(T)$, let $\cO_{\fu/B}(\lambda)$ be the line bundle on $\fu/B$ defined as the pullback of the character $\lambda$ along $\fu/B\to \BB T$. For $\lambda\in \XX^\bullet(T)$, denote by $\dom(\lambda)\in \XX^\bullet(T)^+$ the dominant Weyl translation of $\lambda$.

By \cite{Bez-coh-tilting} and \cite{MR-exotic-t-str}, there is a bounded $t$-structure on $\Coh(\fu/B)$ called the \emph{exotic} $t$-structure. Let $\Ex\Coh(\fu/B)$ denote the heart of the exotic $t$-structure. By \cite[\S 3.5]{MR-exotic-t-str}, the abelian category $\Ex\Coh(\fu/B)$ has a highest weight structure, and the natural functor $D^b(\Ex\Coh(\fu/B))\to\Coh(\fu/B)$ is an equivalence of categories. The set of standard (resp. costandard) objects in $\Ex\Coh(\fu/B)$ are indexed by $\XX^\bullet(T)$. For $\lambda\in \XX^\bullet(T)$, let $\Delta_\lambda^{ex}$ (resp. $\nabla_\lambda^{ex}$) denote the corresponding standard (resp. costandard) object. By \cite[Corollary 3.4, Proposition 3.6]{MR-exotic-t-str}, we have
$$
\cO_{\fu/B}(\lambda)\simeq\left\{\begin{array}{ll} \Delta_\lambda^{ex}, &  \text{if }\lambda\in -\XX^\bullet(T)^+,\\ \nabla_\lambda^{ex}, & \text{if }\lambda\in \XX^\bullet(T)^+. \end{array}\right.
$$
%$$\Delta_\lambda^{ex}\simeq \cO_{\fu/B}(\lambda),\quad\text{if }\lambda\in -\XX^\bullet(T)^+,$$
%$$\nabla_\lambda^{ex}\simeq \cO_{\fu/B}(\lambda),\quad\text{if }\lambda\in \XX^\bullet(T)^+.$$

We will need the following result later.
\begin{prop}\label{prop-exotic-tensor-filtration}
    Suppose Assumption \ref{assump-standard-on-reductive} holds. 
    Let $V$ be a finite-dimensional representation of $G$.  
    \begin{enumerate}
        \item The functor $V\otimes(-)\colon\Coh(\fu/B)\to \Coh(\fu/B)$ is $t$-exact for the exotic $t$-structure.
        \item If $V$ admits a Weyl filtration, and $\cF\in \Ex\Coh(\fu/B)$ admits a standard filtration, then $V\otimes\cF$ admits a standard filtration.
        \item If $V$ admits a good filtration, and $\cF\in \Ex\Coh(\fu/B)$ admits a costandard filtration, then $V\otimes\cF$ admits a costandard filtration.
    \end{enumerate}
\end{prop}
\begin{proof}
    (1) is \cite[Proposition 3.13]{MR-exotic-t-str}, which, in fact, holds without assumption on the characteristic of $\Lambda$. (2) and (3) follow from \cite[Corollary 4.14]{MR-exotic-t-str}.
\end{proof}

%\Xinwen{Can the assumption the derived group simply-connected weakend by $\ell\nmid \pi_1(G_{\mathrm{der}})$?}\Xiangqian{The proof in \cite{MR-exotic-t-str} requires the derived subgroup to be simply connected because they need to use some affine braid action. Probably it works even without the assumption $\ell\nmid \pi_1(G_{\mathrm{der}})$? Again pass to $G'=G_\mathrm{sc}\times Z(G)^\circ$. Then $\fu$ for two groups are equal. Then $\fu/B'\to\fu/B$ is a $\ker(G'\to G)$-gerbe. But $\ker(G'\to G)$ is multiplicative, therefore $\Coh(\fu/B)$ is a direct summand of $\Coh(\fu/B')$?}\Xinwen{Probably, except one needs to double check whether $\fu$ does not change. (Likely this is the case if the char. is good.)  If only two or three sentences are enough to justify this generality, we can add it.}\Xiangqian{In the latter discussion on perverse coherent sheaves, is it ok to pass to $\ell\nmid\pi_1(G_\der)$? In \cite{Achar-exotic-coh}, only the case that $G_\der$ is s.c. is considered. I don't know if $\cN_{G'}\simeq \cN_G$ holds.}

\subsubsection{Perverse coherent sheaves}\label{subsubsection-perverse-coh}

In this subsection, we continue with Assumption \ref{assump-standard-on-reductive}.

By \cite{Arinkin-Bez-perverse-coh}, there is a  \emph{perverse coherent $t$-structure} $(\Coh(\cN/G)^{p\leq 0},\Coh(\cN/G)^{p\geq 0})$ 
on $\Coh(\cN/G)$. Let $\PCoh(\cN/G)$ denote the heart of the perverse $t$-structure on $\Coh(\cN/G)$. 
The perverse $t$-structure on $\Coh(\cN/G)$ is bounded and the abelian category $\PCoh(\cN/G)$ is Artinian. For $\lambda\in \XX^\bullet(T)$, define
$$A_\lambda =\pi_*\cO_{\fu/B}(\lambda) \in \Coh(\cN/G).$$
These objects are often called \emph{Andersen--Jantzen sheaves}. In particular, we have $A_0=\pi_*\cO_{\fu/B}\simeq\cO_{\cN/G}$. For a dominant weight $\lambda\in \XX^\bullet(T)^+$, denote
$$\bar\Delta_\lambda=A_{w_0\lambda}\quad\text{and}\quad\bar\nabla_\lambda=A_\lambda.$$
By {\cite[Theorem 1.4]{Achar-exotic-coh}}, the objects $\bar\Delta_\lambda$ and $\bar\nabla_\lambda$ belongs to $\PCoh(\cN/G)$. 
We recall some properties of the perverse coherent $t$-structure that will be needed in the sequel.

\begin{prop}[{\cite[Proposition 1.6]{Achar-exotic-coh}}]\label{prop-pushforward-standard}
	The functor $\pi_*\colon \Coh(\fu/B)\to \Coh(\cN/G)$ is $t$-exact and hence restricts to an exact functor $\pi_*\colon \Ex\Coh(\fu/B)\to \PCoh(\cN/G)$. For $\lambda\in \XX^\bullet(T)$, it satisfies
$$\pi_*\Delta^{ex}_\lambda\cong\bar\Delta_{\dom(\lambda)}\quad\text{and}\quad \pi_*\nabla^{ex}_\lambda\cong\bar\nabla_{\dom(\lambda)}.$$
\end{prop}

\begin{prop}\label{cor-dom-gen}
	The category $\Coh(\cN/G)$ is generated by $\{\bar{\Delta}_{\lambda}\}_{\lambda\in \XX^\bullet(T)^+}$ (resp. $\{\bar{\nabla}_{\lambda}\}_{\lambda\in \XX^\bullet(T)^+}$) under cones and retracts.
\end{prop}
\begin{proof}
    The functor $\pi_*$ is essentially surjective as
    $$\pi_*\pi^*\cF\simeq \cF\otimes\pi_*\cO_{\fu/B}\simeq \cF$$
    for $\cF\in\Coh(\cN/G)$. The collection $\{\Delta^{ex}_\lambda\}_{\lambda\in \XX^\bullet(T)}$ (resp. $\{\nabla^{ex}_\lambda\}_{\lambda\in \XX^\bullet(T)}$) generates the category $\Coh(\fu/B)$ by highest weight structure.
	By Proposition \ref{prop-pushforward-standard}, the collection $\{\bar{\Delta}_{\lambda}\}_{\lambda\in \XX^\bullet(T)^+}$ (resp. $\{\bar{\nabla}_{\lambda}\}_{\lambda\in \XX^\bullet(T)^+}$) generates $\Coh(\cN/G)$.
\end{proof}

We will need the following technical result later. 

\begin{prop}\label{prop-tilting-tensor-split-case}
	Let $V$ be a finite dimensional $G$-representation. Let $\lambda\in \XX^\bullet(T)^+$ be a dominant weight. 
    \begin{enumerate}
        \item If $V$ admits a Weyl filtration, then $V\otimes \bar\Delta_\lambda$ admits a filtration by $\bar\Delta_{\mu}$ for $\mu\in \XX^\bullet(T)^+$. 
        \item If $V$ admits a good filtration, then $V\otimes \bar\nabla_\lambda$ admits a filtration by $\bar\nabla_{\mu}$ for $\mu\in \XX^\bullet(T)^+$. 
    \end{enumerate}
\end{prop}
\begin{proof}
    We prove the second statement. The first one can be proved similarly.
    By projection formula, we have
    $$V\otimes \bar\nabla_\lambda=V\otimes \pi_*\nabla^{ex}_\lambda\simeq \pi_*(V\otimes\nabla^{ex}_\lambda).$$
    By Proposition \ref{prop-exotic-tensor-filtration}, we know that $V\otimes\nabla^{ex}_\lambda$ admits a filtration by $\nabla^{ex}_\mu$ for $\mu\in \XX^\bullet(T)$. Therefore $V\otimes \bar\nabla_\lambda$ admits a filtration by $\pi_*\nabla^{ex}_\mu$ for $\mu\in \XX^\bullet(T)$. By Proposition \ref{prop-pushforward-standard}, we have $\pi_*\nabla^{ex}_\mu\simeq \bar\nabla_{\dom(\mu)}$ and hence finish the proof.
\end{proof}

\begin{rmk}
 Let $\cU\subseteq G$ be the unipotent variety of $G$. It is well-known that under Assumption \ref{assump-standard-on-reductive}, there exists a $G$-equivariant isomorphism $\cN\cong \cU$ which restricts to a $B$-equivariant isomorphism $\mathfrak{u}\cong U$, usually known as the Springer isomorphism. (In fact the last part in Assumption \ref{assump-standard-on-reductive} is not necessary for the existence of the Springer isomorphism.)
Therefore, we can translate the previous results to the multiplicative Springer resolution $U/B\to \cU/G$ of $G$.   
\end{rmk}

\subsection{Coherent sheaves on $(G\phi/G)_\xi^\wedge$}\label{subsection-twisted-GS}
Let $\phi\colon G\to G$ be an isomorphism of finite order preserving the pinning. We make the following assumptions throughout this subsection.

\begin{assump}\label{assump-standard}
If the characteristic of $\Lambda$ is $\ell>0$, then $\ell>3$ and is good with respect to any simple factor of the adjoint group of $G$.
\end{assump}
%\Xiangqian{We are actually using Remark \ref{rmk-exotic-perverse-pass-to-reductive} to any pseudo-Levi subgroup of $G^{\phi,\circ}$. Is this assumption implies similar assumptions for $G_x$?} \Xinwen{I guess even $x\in T^{\phi}$, $G_x=G^{x\phi}$ and $(G^{\phi})^x$ may not be the same?}\Xiangqian{We need to apply results in the previous subsection to $G_x$. So I guess we need $(G_x)_\mathrm{sc}$ to satisfy Assumption \ref{assump-standard-on-reductive} and $(G_x)_\mathrm{sc}\to G_x$ has degree invertible in $\Lambda$.} \Xinwen{Right. So we end up with the originally assumption. Will double check.}\Xiangqian{Probably can use the idea of Lemma \ref{lemma-reduction-to-sc} to prove Proposition \ref{prop-tilting-tensor-split-case} when $G_\mathrm{sc}\to G$ has order not necessarily invertible in $\Lambda$.}\Xiangqian{Probably only need $\ell>3$ ($\ell >2$ if no $\mathsf{D}_4$ appears) and $\ell$ good for any simple factors.}

\subsubsection{Relative root data}\label{subsubsection-rel-root-data}

Denote $A=T/(\phi-1)T$. Then $\XX^\bullet(A)=\XX^\bullet(T)^\phi$. Let $T^\phi$ denote the $\phi$-invariant of $T$. Let $\XX^\bullet(T)_\phi$ be the $\phi$-coinvariant module. 
Then $T^\phi$ is a diagonalizable group with $\XX^\bullet(T^\phi)=\XX^\bullet(T)_\phi$. 
Following \cite{XZ-cycles}, we define the dominant cone
$$\XX^\bullet(T)_\phi^+=\{\lambda\in \XX^\bullet(T)_\phi| \langle \lambda,\sum_{i=0}^{m-1}\phi^i(\alpha^\vee)\rangle\geq 0 \text{ for every }\alpha^\vee\in\Delta^\vee\}$$
where $m$ is the order the $\phi$-action on $G$. Denote $W_0=W^\phi$. Then $W_0$ acts on $A$ and $T^\phi$. The map $\XX^\bullet(T)_\phi^+\to \XX^\bullet(T)_\phi$ induces a bijection $\XX^\bullet(T)_\phi^+\simeq \XX^\bullet(T)_\phi/W_0$.

For each $\phi$-orbit $\cO\subseteq \Phi(G,T)$, we write
$$\alpha_\cO=\sum_{\gamma\in \cO}\gamma\in \XX^\bullet(T)^\phi.$$
By \cite[Lemma 5.1]{XZ-vector}, the collection $\{\alpha_\cO\}\subseteq \XX^\bullet(A)$ form a root datum. Let $G_\phi$ be the reductive group over $\Lambda$ with maximal torus $A$ and root datum $\Phi(G_\phi,A)=\{\alpha_\cO\}$. The Weyl group of $G_\phi$ is naturally identified with $W_0$. 

By \cite[Definition 5.1.4]{XZ-vector}, there are two types of relative roots:
\begin{enumerate}
    \item[($\mathsf{A}$)]  A root in $\Phi(G_\phi,A)$ has type $\mathsf{A}$ if there is a unique $\phi$-orbit $\cO\subseteq \Phi(G,T)$ such that this root is  $\alpha_\cO$. We say that the orbit $\cO$ has type $\mathsf{A}$.
    \item[($\mathsf{BC}$)] A root in $\Phi(G_\phi,A)$ has type $\mathsf{BC}$ if there are two $\phi$-orbits $\cO^+,\cO^-\subseteq \Phi(G,T)$ such that this root is $\alpha_{\cO^+}=\alpha_{\cO^-}$, and $|\cO^-|=2|\cO^+|$. We say that the orbit $\cO^-$ has type $\mathsf{BC}^-$, and the orbit $\cO^+$ has type $\mathsf{BC}^+$.
\end{enumerate}

%On the other hand, let $G^\phi$ (resp. $B^\phi$, resp. $T^\phi$) be the $\phi$-invariant of $G$ (resp. $B$, pres. $T$). Let $G^{\phi,\circ}$ be the identity component of $G^\phi$. Let $T^{\phi,\circ}$ (resp. $B^{\phi,\circ}$) be the identity component of $T^\phi$ (resp. $B^{\phi}$). Let $\XX^\bullet(T)_\phi^\tf$ be the torsion free quotient of $\XX^\bullet(T)_{\phi}$. Then $T^{\phi,\circ}$ is a torus whose character lattice is $\XX^\bullet(T^{\phi,\circ})=\XX^\bullet(T)^\tf_\phi$. For a $\phi$-orbit $\cO\subseteq \Phi(G,T)$, let $\bar{\alpha}_\cO$ denote the image of an element $\alpha\in\cO$ in $\XX^\bullet(T)^\tf_\phi$.

\subsubsection{Main results}

Let $G\phi\git G$ be the GIT quotient of $G$ be the $\phi$-twisted conjugation action. By \cite[Proposition 4.2.3]{XZ-vector}, there is the twisted Chevalley isomorphism 
$$G\phi\git G\cong A\git W_0.$$
Define the ($\phi$-twisted) Grothendieck--Springer resolution
$$\widetilde{G}^\phi\coloneqq\{(gB,x)\in G/B\times G |x\in gB\phi(g)^{-1}\}\to G.$$
sending $(gB,x)$ to $x$. There is a natural isomorphism $\widetilde{G}^\phi/G\cong B\phi/B$ of quotient stacks. After taking the quotient by $G$, the Grothendieck--Springer resolution is identified with $\pi_\phi\colon B\phi/B\to G\phi/G$. There is a $G$-equivariant morphism $\widetilde{G}^\phi\to A$ sending $(gB,x)$ to the image of $g^{-1}x\phi(g)\in B$ in $A$. It defines a commutative diagram
\begin{equation*}%\label{eq-GS}\tag{A}
\begin{tikzcd}
	B\phi/B\ar[r]\ar[d] & T\phi/T\ar[d] \\
	G\phi/G\ar[r] & A\git W_0.
\end{tikzcd}\end{equation*}
Note that there is an isomorphism $T\phi/T\cong A\times \BB T^\phi$. Let $\xi\in (A\git W_0)(\Lambda)$ be a $\Lambda$-point, and let $(A\git W_0)^\wedge_\xi$ denote the formal completion of $A\git W_0$ at $\xi$. Let 
\[
(G\phi/G)^\wedge_\xi=(G\phi/G)\times_{(A\git W_0)}(A\git W_0)^\wedge_\xi.
\]

Taking formal completion of the above commutative diagram at a point $\xi\in (A\git W_0)(\Lambda)$, we obtain a correspondence
$$\bigsqcup_{\chi} A^\wedge_\chi\times \BB T^\phi\xleftarrow{\gamma_{\phi,\xi}}\bigsqcup_{\chi}(B\phi/B)_\chi^\wedge\xrightarrow{\pi_{\phi,\xi}} (G\phi/G)_\xi^\wedge$$
where $\chi$ runs through points in $A(\Lambda)$ that maps to $\xi$. For $\lambda\in \XX^\bullet(T)_\phi$, denote
$$\cF_\lambda\coloneqq \bigoplus_{\chi} (\iota_\chi)_*\cO_{\BB T^\phi}(\lambda),$$
where $\iota_\chi\colon \BB T^\phi\hookrightarrow A_\chi^\wedge\times \BB T^\phi$ is the closed embedding defined by the point $\chi$, and $\cO_{\BB T^\phi}(\lambda)\in \Coh(\BB T^\phi)$ is the twist of the structure sheaf $\cO_{\BB T^\phi}$ by the character $\lambda$. 

We note that the morphism $\gamma_{\phi,\xi}$ is quasi-smooth with trivial relative dualizing sheaf and therefore the $!$-pullback and the $*$-pullback along it coincide and preserve coherence. 
The following proposition is a generalization of Proposition \ref{cor-dom-gen}. 

\begin{prop}\label{prop-generators-GS}
Assume that Assumption \ref{assump-standard} holds. 
\begin{enumerate}
	\item The set of objects $\{(\pi_{\phi,\xi})_*(\gamma_{\phi,\xi})^! \cF_\lambda\}_{\lambda\in \XX^\bullet(T)_\phi^+}$ generate $\Coh((G\phi/G)^\wedge_\xi)$ under cones and retracts.
	\item The set of objects $\{(\pi_{\phi,\xi})_*(\pi_{\phi,\xi})^! \cF_\lambda\}_{\lambda\in -\XX^\bullet(T)_\phi^+}$ generate $\Coh((G\phi/G)^\wedge_\xi)$ under cones and retracts.
\end{enumerate}
\end{prop}
We also have the following generalization of Proposition \ref{prop-tilting-tensor-split-case}. 

\begin{prop}
\label{cor-tilting-tensor-twisted-case}
    Assume that Assumption \ref{assump-standard} holds. Let $V\in\Rep(G)$ be a finite dimensional $G$-representation. 
    \begin{enumerate}
        \item Assume that $V$ admits a good filtration. Let $\lambda\in \XX^\bullet(T)_\phi^+$. Then the object 
        $$V\otimes (\pi_{\phi,\xi})_*(\gamma_{\phi,\xi})^!\cF_\lambda\in \Coh((G\phi/G)^\wedge_\xi)$$ 
        is a retract of an object that admits a filtration with graded pieces given by retracts of $(\pi_{\phi,\xi})_*(\gamma_{\phi,\xi})^! \cF_\mu$ for $\mu\in \XX^\bullet(T)_\phi^+$.
        \item Assume that $V$ admits a Weyl filtration. Let $\lambda\in -\XX^\bullet(T)_\phi^+$. Then the object 
        $$V\otimes (\pi_{\phi,\xi})_*(\gamma_{\phi,\xi})^!\cF_\lambda\in \Coh((G\phi/G)^\wedge_\xi)$$ 
        is a retract of an object that admits a filtration with graded pieces given by retracts of $(\pi_{\phi,\xi})_*(\gamma_{\phi,\xi})^! \cF_\mu$ for $\mu\in -\XX^\bullet(T)_\phi^+$.
    \end{enumerate}
\end{prop}

\begin{rmk}\label{rmk-unipotent-vs-nilpotent}
 In the above two propositions, the restriction of the characteristic on $\Lambda$ as imposed in Assumption \ref{assump-standard} can sometimes be relaxed. For instance, as will be clear from the proof presented below, these results hold without any assumptions when $G=(\GL_n)^d$ with the $\phi$-action defined by cyclic permutations of the direct factors. We maintain Assumption \ref{assump-standard} to ensure our proof is uniform across all groups.
\end{rmk}

The rest of this section is devoted to the proof of the above two results.

\subsubsection{Reduction to the simply-connected case}

We will first show that it is enough to prove Proposition \ref{prop-generators-GS} and \ref{cor-tilting-tensor-twisted-case} when $G$ is almost simple and simply-connected equipped with an automorphism $\phi$ whose order is invertible in $\Lambda$.

%\begin{lemma}
%    Let $(G_i,\phi_i),\ i=1,2$ be two connected reductive groups equipped with a pinned automorphism $\phi_i$. Let $G=G_1\times G_2$, equipped with the pinned automorphism $\phi=\phi_1\times \phi_2$.
%    If Proposition \ref{prop-generators-GS} and \ref{cor-tilting-tensor-twisted-case} hold for $(G_i,\phi_i)$, then they hold for $(G, \phi)$.
%\end{lemma}
%\begin{proof}
%    The only thing to notice is that for $\xi_i\in (A_i\git (W_i)^{\phi_i})(\Lambda)$, we have $\Coh((G_1\phi_1/G_1)_{\xi_1}^\wedge)\otimes_{\Lambda} \Coh((G_2\phi_2/G_2)_{\xi_2}^\wedge)\cong \Coh((G\phi/G)_{\xi}^\wedge)$. This follows from the fact that $(G\phi/G)_{\xi}^\wedge=(G_1\phi_1/G_1)_{\xi_1}^\wedge\times (G_2\phi_2/G_2)_{\xi_2}^\wedge$ and \cite[]{Tame}
%\end{proof}

%\begin{lemma}\label{lemma-finite-flat-generation}
    %Let $f:X\to Y$ be a finite flat morphism of algebraic stacks (of finite presentation) over $\Lambda$. If the degree of $f$ is invertible in $\Lambda$, then the image of $f_*: \Coh(X)\to \Coh(Y)$ generates $\Coh(Y)$ under retracts.
%\end{lemma}
%\begin{proof}
    %We first notice that $\cO_Y\to f_*\cO_X$ admits a splitting given by the trace map $f_*\cO_X\to \cO_Y$. Now for every $\cF\in \Coh(Y)$, by the projection formula we have $f_*f^*\cF\cong \cF\otimes f_*\cO_X$. This shows that $\cF$ appears as a direct summand of an object in the image of $f_*$.
%\end{proof}

%We start with the following two observations.

\begin{lemma}\label{lemma-gen-GS-isogeny}
    Suppose $G'\to G$ is a central isogeny with compatible $\phi$-actions. Then the image of the functor 
    $$(c_G)_*\colon \Coh(G'\phi/G')\to \Coh(G\phi/G)$$
    generates the target under retracts. Let $\xi \in (A\git W_0)(\Lambda)$. 
    Then the image of the functor
    $$(c_G)_*\colon \bigoplus_{\xi'}\Coh((G'\phi/G')_{\xi'}^{\wedge})\to \Coh((G\phi/G)_{\xi}^\wedge)$$
    generates the target under retracts, where $\xi'$ runs through the finite set of points of $A'\git W_0$ that map to $\xi$.
\end{lemma}
%\Xinwen{If kernel of $G'\to G$ is a torus, is the similar statement holds when $\Coh(G'\phi/G')\to \Coh(G\phi/G)$ is replaced by $\Coh((G'\phi/G')_{\xi'}^\wedge)\to \Coh((G\phi/G)_{\xi}^\wedge)$?}\Xiangqian{It is not clear to me. Probably one can show that the statement is true for $G'\phi/G'\to G\phi/G\times_{A_0\git W_0} A'_0\git W_0$?}
\begin{proof}
    Denote $D=\ker(G'\to G)$. We can write $c_G$ as a composition
    $$G'\phi/G'\xrightarrow{a} G\phi/G'\xrightarrow{b} G\phi/G.$$
    By projection formula, we have 
    $$a_*a^*\cF\simeq \cF\otimes_{\cO(G)}\cO(G')$$
    for $\cF\in \Coh(G\phi/G)$.  There is a $G'\times G'$-equivariant direct sum decomposition
    $$\cO(G')=\bigoplus_{\zeta}\cO(G')_\zeta,$$
    where $\zeta$ runs through characters of $D$, and $\cO(G')_\zeta$ is the component where the left tranlation of $D$ acts by $\zeta$ (or equivalently, the right translation of $D$ acts by $\zeta^{-1}$). Note that $\cO(G')_0=\cO(G)$, and the above decomposition if $\cO(G)$-linear. Therefore $\cF$ is a retract of $a_*a^*\cF$. If $\cF\in \Coh((G\phi/G')_{\xi}^\wedge)$, then $a^*\cF\in \bigoplus_{\xi'}\Coh((G'\phi/G')_{\xi'}^{\wedge})$. Thus $a_*$ generates the target under retracts.
    %The functor $a_*\colon\Coh(G'\phi/G')\to \Coh(G\phi/G')$ generates the target under retracts by Lemma \ref{lemma-finite-flat-generation}, as the order of $D$ is invertible in $\Lambda$.
    The morphism $b$ fits into a Cartesian diagram
    $$\begin{tikzcd}
        G\phi/G' \ar[r,"b"]\ar[d] & G\phi/G \ar[d] \\
        \BB G' \ar[r] & \BB G.
    \end{tikzcd}$$
    Therefore for $\cF\in \Coh(G\phi/G)$, we have $b_*b^*\cF=\cF\otimes b_*\cO_{G\phi/G'}=\cF\otimes_\Lambda R\Gamma(\BB D,\Lambda)$. As the group $D$ is multiplicative, we have $R\Gamma(\BB D,\Lambda)=\Lambda$. If $\cF\in \Coh((G\phi/G)_{\xi}^\wedge)$, then $b^*\cF\in \Coh((G\phi/G')_{\xi}^\wedge)$. Hence $b_*$ is essentially surjective.
\end{proof}

\begin{lemma}\label{prop-product-adj-quotient}
    Let $G_i$ be a connected reductive groups with an automorphism $\phi_i$, for $i=1,2$. Denote $G=G_1\times G_2$ and $\phi=\phi_1\times \phi_2$. Then the natural functor
    $$\Coh(G_1\phi_1/G_1)\otimes \Coh(G_2\phi_2/G_2)\to \Coh(G\phi/G)$$
    is an equivalence of categories. Moreover, given $\xi_i\in (A_i\git W_{i,0})(\Lambda)$, and $\xi=(\xi_1,\xi_2)\in (A\git W_{0})(\Lambda)$, then the above equivalence restricts to an equivalence
    $$\Coh((G_1\phi_1/G_1)^\wedge_{\xi_1})\otimes \Coh((G_2\phi_2/G_2)^\wedge_{\xi_2})\to \Coh((G\phi/G)^\wedge_\xi).$$
\end{lemma}
%\Xinwen{The problem is that this does not directly imply the case after taking support at $\xi$.}\Xiangqian{I think $(\bC_1)_{Z_1}\otimes(\bC_2)_{Z_2}\simeq (\bC_1\otimes \bC_2)_{Z_1\times Z_2}$ always holds.} \Xinwen{Ah, good point. Then need to rearrange writing a little bit, or refer to later section?} \Xinwen{Actually need to check.}
\begin{proof}
    The functor is fully faithful by \cite[Proposition 9.31]{Tame}. We need to show essential surjectiveness. If $\pi_1(G)_\mathrm{tor}=1$, then $\Coh(G\phi/G)$ is generated by the image of $\Coh(\BB G)$ by \cite[Proposition VIII.5.12]{FS}. Then the claim is clear as $\Coh(\BB G)\simeq\Coh(\BB G_1)\otimes\Coh(\BB G_2)$. In general, denote $G_i'=(G_i)_{\mathrm{sc}}\times Z(G_i)^\circ$, where $(G_i)_{\mathrm{sc}}$ is the simply-connected cover of the derived subgroup of $G_i$. Also denote $G'=G_1'\times G_2'$. Then the diagram
    $$\begin{tikzcd}
        \Coh(G_1'\phi_1/G_1')\otimes\Coh(G_2'\phi_2/G_2') \ar[r]\ar[d,"(c_{G_1})_*\otimes(c_{G_2})_*"swap] & \Coh(G'\phi/G') \ar[d,"(c_G)_*"] \\
        \Coh(G_1\phi_1/G_1)\otimes\Coh(G_2\phi_2/G_2) \ar[r] & \Coh(G\phi/G)
    \end{tikzcd}$$
    commutes. The upper arrow is an equivalence of categories as $\pi_1(G')_\mathrm{\tor}=1$. The vertical arrows generate the targets under retracts by Lemma \ref{lemma-gen-GS-isogeny}. Therefore the lower arrow is essentially surjective. The last statement follows from Lemma \ref{lemma-tensor-supp} below.
\end{proof}

\begin{lemma}\label{lemma-reduction-to-sc}
    Suppose $G'\to G$ is a central isogeny with compatible $\phi$-actions. If Proposition \ref{prop-generators-GS} holds for $G'$, then it holds for $G$. Similarly, if Proposition \ref{cor-tilting-tensor-twisted-case} holds for $G'$, then it holds for $G$.
\end{lemma}
%\Xinwen{It seems to me the lemma holds if $\ker(G'\to G)$ is a torus. Only thing to check is that image of $\Coh((G'\phi/G')_{\xi'}^\wedge)\to \Coh((G\phi/G)_{\xi}^\wedge)$ generates the latter. Another case this might be true is $G'\to G$ is closed with cokernel a torus.}\Xiangqian{The degree of the isogeny don't need to be invertible in $\Lambda$. Lemma \ref{lemma-finite-flat-generation} still holds in this case (using the original proof by writing $\cO(G)$ as a direct summand of $\cO(G')$). Then if $G'\twoheadrightarrow G$ has a torus kernel, it can be reduced to an isogeny $G'\to G\times \mathrm{Im}(\ker(G'\to G)\to G'_\ab)^\circ$. I guess for $G'\hookrightarrow G$, one can find suitable subtorus $S\subset Z(G)^\circ$ such that $G'\times S\to G$ is an isogeny.}
    \begin{proof}
   % We have shown that the proposition holds for $G_i\phi_i/G_i$ for each $i$. Thus it also holds for $G_\mathrm{sc}\phi /G_{\mathrm{sc}}$. It is also clear for the torus $Z(G)^\circ$. Therefore the claim also holds for $G'$. 
Let $\xi$ be a point in $(A\git W_0)(\Lambda)$. For $\lambda\in \XX^\bullet(T)_\phi$, we write $\cF_{\xi,\lambda}$ instead of $\cF_{\lambda}$ in the sequel to emphasize that it is an object associated to $G$. We have similarly defined object $\cF_{\xi',\lambda'}$ for $\xi'\in A'\git W_0$ and $\lambda'\in \XX^\bullet(T')_\phi$.
    %$$\cF_{\xi',\lambda'}=\bigoplus_{\chi'} (\iota_{\chi'})_*\cO_{\BB T'^{\phi}}(\lambda')\in \Coh((T'\phi/T')^\wedge_{\xi'})$$

    Consider the commutative diagram
    $$\begin{tikzcd}
        T'\phi/T' \ar[d,"c_T"] & B'\phi/B' \ar[l,"\gamma'_\phi"swap]\ar[r,"\pi'_\phi"]\ar[d,"c_B"]\ar[ld,phantom,"\square",very near start]\ar[rd,phantom,"\square",very near start] & G'\phi/G' \ar[d,"c_G"] \\
        T\phi/T & B\phi/B \ar[l,"\gamma_\phi"swap]\ar[r,"\pi_\phi"] & G\phi/G.
    \end{tikzcd}$$
    Note that both squares are Cartesian since $B'=G'\times_GB=T'\times_TB$m $G'/B'=G/B$ and $T'/B'\simeq T/B$. %The morphism $A'\git W_0\to A\git W_0$ is finite. 
    
%     For $\xi'\in A'\git W_0$ mapping to $\xi$ and $\lambda'\in \XX^\bullet(T')_\phi$, denote
%    
%    for $\chi'$ runs through points of $A'$ mapping to $\xi'$ and 
%    $$\iota_{\chi'}\colon \BB T'^\phi\hookrightarrow (A')^\wedge_{\chi'}\times \BB T'^{\phi}\hookrightarrow (T'\phi/T')^\wedge_{\xi'}$$
%    is the closed embedding at $\chi'$. 

    Note that the natural map  $\XX^\bullet(T)_\phi\to \XX^\bullet(T')_\phi$ with finite kernel. We have
    $$(c_T)_*(\iota_{\chi'})_*\cO_{\BB T'^\phi}(\lambda')=\bigoplus_{\lambda} (\iota_\chi)_*\cO_{\BB T^\phi}(\lambda),$$
    where $\lambda$ runs through elements of $\XX^\bullet(T)_\phi$ mapping to $\lambda'$, and $\chi$ is the image of $\chi'$. Therefore for $\xi'$ that maps to $\xi$, and for $\lambda\in \XX_\bullet(T)$, the object $(c_T)_*\cF_{\xi',\lambda'}$ is a retract of finite direct sums of $\cF_{\xi,\lambda}$ for $\lambda\in \XX^\bullet(T)_\phi$ mapping to $\lambda'$. In particular, if $\lambda'\in \XX^\bullet(T')_\phi^+$ is dominant, then $(c_T)_*\cF_{\xi',\lambda'}$ is a retract of finite direct sums of $\cF_{\xi,\lambda}$ for $\lambda\in \XX^\bullet(T)_\phi^+$. Similar for $\lambda'\in -\XX^\bullet(T)_\phi^+$ anti-dominant. On the other hand, for $\xi$ (resp. $\lambda$) mapping to $\xi'$ (resp. $\lambda'$), the object $\cF_{\xi,\lambda}$ is a retract of $(c_T)_*\cF_{\xi',\lambda'}$.

    By Lemma \ref{lemma-gen-GS-isogeny}, the functor
    $$(c_G)_*\colon \bigoplus_{\xi'}\Coh((G'\phi/G')_{\xi'}^\wedge)\to \Coh((G\phi/G)^\wedge_\xi)$$
    generates the target under retracts. %where $\xi'$ runs through points of $A'\git W_0$ mapping to $\xi$. This is because both sides are identified with the subcategory of $\Coh(G\phi/G)$ (resp. $\Coh(G'\phi/G')$) of objects supported on the preimage of $\xi$. 
    If Proposition \ref{prop-generators-GS} holds for $G'$, we know that objects of the form
    $$(\pi_{\phi,\xi'}')_*(\gamma'_{\phi,\xi'})^!\cF_{\xi',\lambda'}, \quad \lambda'\in \XX^\bullet(T')_\phi^+$$
    generate $\Coh((G'\phi/G')_{\xi'}^\wedge)$ under cones and retracts. Therefore objects of the form
    $$(c_G)_*(\pi_{\phi,\xi'}')_*(\gamma'_{\phi,\xi'})^!\cF_{\xi',\lambda'}\simeq (\pi_{\phi,\xi})_*(\gamma_{\phi,\xi})^!(c_T)_*\cF_{\xi',\lambda'},\quad \lambda'\in \XX^\bullet(T)^+_\phi,\xi'\in A'\git W_0\text{ mapping to }\xi$$
    generate $\Coh((G\phi/G)_\xi^\wedge)$. By the above discussion, the objects $(c_T)_*\cF_{\xi',\lambda'}$ are retracts of finite direct sums of $\cF_{\xi,\lambda}$ for $\lambda\in \XX^\bullet(T)_\phi^+$. Thus $(\pi_{\phi,\xi})_*(\gamma_{\phi,\xi})^!\cF_{\xi,\lambda}, \lambda\in \XX^\bullet(T)^+_\phi$ generates $\Coh((G\phi/G)_\xi^\wedge)$. Similar for the anti-dominant case.

Next suppose that Proposition \ref{cor-tilting-tensor-twisted-case} holds for $G'$. We only deal with (1) as (2) can be treated similarly. Let $V$ be a $G$-representation with good filtration. Then for $\lambda'\in \XX^\bullet(T')^+_\phi$, the object
    $$V|_{G'}\otimes (\pi'_{\phi,\xi'})_*(\gamma_{\phi,\xi'}')^!\cF_{\xi',\lambda'}$$
    is a retract of an object that admits a filtration by retracts of $(\pi'_{\phi,\xi'})_*(\gamma_{\phi,\xi'}')^!\cF_{\xi',\mu'}$ for $\mu'\in \XX^\bullet(T')^+_\phi$. Therefore
    $$(c_G)_*(V|_{G'}\otimes (\pi'_{\phi,\xi'})_*(\gamma_{\phi,\xi'}')^!\cF_{\xi',\lambda'})\simeq V\otimes (\pi_{\phi,\xi})_*(\gamma_{\phi,\xi})^! (c_T)_*\cF_{\xi',\lambda'}$$
    is a retract of an object that admits a filtration by retracts of $(\pi_{\phi,\xi})_*(\gamma_{\phi,\xi})^! (c_T)_*\cF_{\xi',\mu'}$ for $\mu'\in \XX^\bullet(T')^+_\phi$. The claim for $G$ follows as $(c_T)_*\cF_{\xi',\lambda'}$ is a  retract of finite direct sums of $\cF_{\xi,\lambda}$ for $\lambda\in \XX^\bullet(T)_\phi^+$ mapping to $\lambda'$, and any $\cF_{\xi,\lambda},\lambda\in \XX^\bullet(T)_\phi^+$ appears as a direct summand of $(c_T)_*\cF_{\xi',\lambda'}$ for some $\lambda'\in \XX^\bullet(T')^+_\phi$. The claim for $G$ follows.
\end{proof}

We apply Lemma \ref{lemma-reduction-to-sc} to the central isogeny $G'=G_{\mathrm{sc}}\times Z(G)^\circ\to G$, where $G_\mathrm{sc}$ is the simply-connected cover of the derived subgroup of $G$. By Lemma \ref{lemma-reduction-to-sc}, it suffices to prove Proposition \ref{prop-generators-GS} and \ref{cor-tilting-tensor-twisted-case} for $G'$. Then by Lemma \ref{prop-product-adj-quotient}, it suffices to prove Proposition \ref{prop-generators-GS} and \ref{cor-tilting-tensor-twisted-case} for $G_\mathrm{sc}$. %\Xinwen{To justify this reduction, we somehow use $\Coh((G_1\phi_1/G_1)_{\xi_1}^\wedge)\otimes_{\Lambda} \Coh((G_2\phi_2/G_2)_{\xi_2}^\wedge)\cong \Coh((G\phi/G)_{\xi}^\wedge)$ for $G_1$ simply-connected and $G_2$ a torus. Is this clear? \cite[Corollary 9.34, Proposition 9.35]{Tame} do not literally apply. But I guess $\Coh(X\times BH)\cong \Coh(X)\otimes \Coh(BH)$ always true if $H$ is multiplicative type. Ok, I probably should add this fact in \cite{Tame}. Or maybe add a short proof here.}\Xiangqian{Maybe add a proof here. }

%    where $G_\mathrm{sc}$ is the simply connected cover of the derived subgroup of $G$. We write $(-)'$ for any schemes or morphisms as above corresponding to $G'$. Let $D$ denote the kernel of $G'\to G$. Then the order of $D$ is invertible in $\Lambda$ by Assumption \ref{assump-standard}. In fact, the order of $D$ divides the order of $Z(G_\mathrm{sc})$, which is invertible in $\Lambda$ by examining each types. 
    
Now assume that $G$ is semisimple and simply-connected. Then it splits into a product
$$G=\prod_i (G_i)^{r_i}$$
with each $G_i$ almost simple and simply connected and $r_i\geq 1$, such that $\phi$ action on $(G_i)^{r_i}$ is given by
$$\phi(x_1,\dots,x_{r_i})=(\phi_i(x_{r_i}),x_1,\dots,x_{r_i-1})$$
for some $\phi_i\colon G_i\xrightarrow{\sim }G_i$ preserving the pinning. As explained in \cite[\S 4.1]{XZ-vector}, there is an isomorphism
$$G\phi/G\simeq \prod_i G_i\phi_i/G_i,$$
where for each $i$, the morphism is induced by $G_i\hookrightarrow (G_i)^{r_i}$, $x\mapsto (x,1,\dots,1).$ 
The automorphism $\phi_i$ is trivial if $G_i$ is not of type $\mathsf{A}_n$, $\mathsf{D}_n$, or $\mathsf{E}_6$. The order of $\phi_i$ is at most $2$ if $G_i$ is of type $\mathsf{A}_n$, $\mathsf{E}_6$, or $\mathsf{D}_n (n\neq 4)$, and at most $3$ if $G_i$ is of type $\mathsf{D}_4$. In all this cases, the order of $\phi_i$ is invertible in $\Lambda$ by Assumption \ref{assump-standard}. Note that the diagram
$$\begin{tikzcd}
    \prod_{i} G_i\phi_i/G_i \ar[d]\ar[r,"\simeq"] & G\phi/G \ar[d] \\
    \prod_i\BB G_i \ar[r] & \BB G 
\end{tikzcd}$$
commutes, where the lower arrow is given by diagonal embeddings $G_i\hookrightarrow (G_i)^{r_i}$. By \cite[Theorem 1]{Mathieu-filtrations}, representations of $(G_i)^r$ that admit good filtrations (resp. Weyl filtrations) restrict to representations of $G_i$ that admit good filtrations (resp. Weyl filtrations). Therefore it suffices to prove Proposition \ref{prop-generators-GS} and Proposition \ref{cor-tilting-tensor-twisted-case} for $\prod_i G_i\phi_i/G_i$. By Lemma \ref{prop-product-adj-quotient}, we further reduce to the case that $G=G_i$ is almost simple and simply-connected.

\subsubsection{Proof of Proposition \ref{prop-generators-GS} and \ref{cor-tilting-tensor-twisted-case}}
It remains to prove Proposition \ref{prop-generators-GS} and \ref{cor-tilting-tensor-twisted-case} in the case $G$ is almost simple and simply-connected equipped with a pinned automorphism $\phi$. Note that in this case the order of the $\phi$-action is invertible in $\Lambda$ by Assumption \ref{assump-standard}.

\begin{lemma}\label{cor-phi=id}
    Let $G$ be a semisimple group such that if $\Lambda$ has characteristic $\ell$, then $\ell$ is good for any simple factors of $G_\ad$. Then Proposition \ref{prop-generators-GS} and \ref{cor-tilting-tensor-twisted-case} hold for $G$, $\phi=\id$, and $\xi=1$.
\end{lemma}
\begin{proof}
%\Xinwen{Check the writing.}
    The underlying reduced substack of $(G/G)^\wedge_1$ is identified with $\cU/G$. The morphism $U/B\to\BB T$ is smooth with trivial dualizing complex. Therefore the object $(\pi_{\phi,\xi})_*(\gamma_{\phi,\xi})^! \cF_\lambda$ is equal to the image of $A_\lambda$ (introduced in \S\ref{subsubsection-perverse-coh}) under the pushforward $\Coh(\cU/G)\to \Coh((G/G)^\wedge_1)$. If $G$ satisfies Assumption \ref{assump-standard-on-reductive}, then the claim for $G$ follows from Proposition \ref{cor-dom-gen}, Proposition \ref{prop-tilting-tensor-split-case}, and Remark \ref{rmk-unipotent-vs-nilpotent}.
    
    Let $G_\mathrm{sc}$ denote the simply-connected cover of $G$. Then
    $G_\mathrm{sc}=\prod G_i$
    is a product of almost simple and simply-connected groups $G_i$. Moreover, if $\Lambda$ has characteristic $\ell$, then $\ell$ is good for all $G_i$. 
    
    If $G_i$ is not of type $\mathsf{A}$, then by Remark \ref{rmk-bilinear-form}, we know that the claim holds for $G_i$. If $G_i=\mathrm{SL_n}$, we know that Proposition \ref{cor-dom-gen} and Proposition \ref{prop-tilting-tensor-split-case} hold for $\GL_n$. The unipotent cone of $\mathrm{SL}_n$ and $\GL_n$ are isomorphic. Pullback along $\cU_{\mathrm{SL}_n}/\mathrm{SL}_n\to \cU_{\GL_n}/\GL_n$ implies that Proposition \ref{cor-dom-gen} and Proposition \ref{prop-tilting-tensor-split-case} hold for $\mathrm{SL}_n$. Therefore the claim also holds for $\mathrm{SL}_n$.

    By Proposition \ref{prop-product-adj-quotient}, we know that the claim holds for $G_\mathrm{sc}$. As in the proof of Lemma \ref{lemma-reduction-to-sc}, the functor
    $$(c_G)_*\colon \bigoplus_{\xi}\Coh((G_\mathrm{sc}/G_\mathrm{sc})^\wedge_\xi)\to \Coh((G/G)^\wedge_1),$$
    generates the targets under retracts, where $\xi$ runs through elements of $T_\mathrm{sc}\git W$ mapping to $1\in T\git W$. Note that all the elements $\xi$ appearing are central. Therefore multiplying by $\xi$ defines an isomorphism
    $$(G_\mathrm{sc}/G_\mathrm{sc})^\wedge_1\xrightarrow{\sim} (G_\mathrm{sc}/G_\mathrm{sc})^\wedge_\xi.$$
    Therefore Proposition \ref{prop-generators-GS} and \ref{cor-tilting-tensor-twisted-case} holds for $(G_\mathrm{sc}/G_\mathrm{sc})^\wedge_\xi$ with $\xi$ mapping to $1$. Now the proof of Lemma \ref{lemma-reduction-to-sc} implies that the claim holds for $G$.
\end{proof}

%Assume that $G$ is almost simple and simply connected with an pinned automorphism $\phi\colon G\xrightarrow{\sim} G$ of order invertible in $\Lambda$.
%\Xinwen{Shall we just assume $G$ to be simply-connected in the following? Then $G^\phi$ is connected and many things are simplified. Or do you want to keep the discussion of the more general case?}\Xiangqian{But $G_\phi$ may not be simply connected, so $G_x$ may still be disconnected. The part $T^\phi/T^{\phi,\circ}$ becomes trivial but it may not simplify the proof a lot.} \Xinwen{$G_x$ is the fixed points of $G$ under the semisimple automorphism $\Ad_x\phi$, and therefore is connected by Steinberg's result.}\Xiangqian{I see.  Let's assume that $G$ is simply connected.}

%First note that the the $\phi$-fixed point $G^{\phi}$ is a connected reductive group by \cite[Theorem 8.1]{Steinberg-endomorphism}, containing $T^{\phi}$ as a maximal torus and $B^{\phi}$ as a Borel subgroup. The root system $\Phi(G^{\phi},T^{\phi})$ of $G^{\phi}$ is given by $\{\bar\alpha_{\cO}\}$. The Weyl group of $G^{\phi}$ is $W_0$. 

Fix a point $\xi\in (A\git W_0)(\Lambda)$. Let $\chi\in A(\Lambda)$ be a point lifting $\xi$. Let $x\in T(\Lambda)$ be an element lifting $\chi$. Let $G_x\coloneqq G^{x\phi}$ denote the centralizer of $x\phi$ in $G$. By \cite[Theorem 8.1]{Steinberg-endomorphism}, $G_x$ is a connected reductive group. Denote $B_x=B\cap G_x$ and $U_x=U\cap G_x$. Then $B_x$ is a Borel subgroup of $G_x$ and $B_x=T^\phi\cdot U_x$. Denote by $(W_0)_\chi$ the stabilizer of $W_0$-action on $\chi$. Then $(W_0)_\chi$ is identified with the Weyl group of $G_x$. 

%Moreover, $(W_0)^\circ_\chi$ is identified with the Weyl group of the root system $(\XX^\bullet(T^{\phi,\circ}),\bar\Phi_\chi)$. Let $x\in T(\Lambda)$ be a lift of $\chi$. There is a short exact sequence
%$$1\to T^\phi/T^{\phi,\circ}\to G_x/G_x^\circ\to (W_0)_\chi/(W_0)_\chi^\circ\to 1.$$
%Note that the order of $G_x/G_x^\circ$ is invertible in $\Lambda$: By Assumption \ref{assump-standard}, the order of $W$ is invertible in $\Lambda$. Therefore the order of $(W_0)_\chi/(W_0)_\chi^\circ$ is invertible in $\Lambda$.\Xiangqian{Need to rewrite under Assumption \ref{assump-standard}} As the order of $\phi$-action is invertible in $\Lambda$, the order of $T^\phi/T^{\phi,\circ}$, which is equal to the order of torsions in $\XX^\bullet(T)_\phi$, is invertible in $\Lambda$. \Xinwen{Assuming $G$ is simply-connected, this paragraph can be removed.}

\begin{prop}\label{prop-adjoint-quot-endoscopy}
	The map $g\mapsto gx\phi$ defines an isomorphism
	$$\cU_{G_x}^\wedge/G_x\xrightarrow{\sim} (G\phi/G)^\wedge_\xi,$$
	where $\cU_{G_x}^\wedge$ is the formal completion of $G_x$ along the unipotent cone.
\end{prop}
\begin{proof}
	This is proved in \cite[Proposition 2.44]{Tame}. We include a proof here for completeness. It suffices to show that the morphism $\cU_{G_x}^\wedge/G_x\to (G\phi/G)^\wedge_\xi$ is induces an isomorphism on $K$-points for any algebraically closed field $K$ over $\Lambda$, and induces an isomorphism on the tangent complex at each $K$-points.
	
	Let $K$ be an algebraically closed field over $\Lambda$. We first prove essential surjectivity. Let $g\phi$ be an element in  $(G\phi/G)^\wedge_\xi(K)$. By \cite[\S 5.3]{XZ-vector}, the Grothendieck--Springer resolution $B\phi/B\to G\phi/G$ is surjective. Thus we may assume that $g\phi\in B\phi$. Write $g=ut$ for $u\in U$ and $t\in T$. By \cite[Proposition 5.2.10]{XZ-vector}, after conjugation we can assume that $t=x$ and $u\in U^{t\phi}$. Therefore $u$ lies in $\cU_{G_x}(K)$. Because the order of $\phi$ is invertible in $\Lambda$, $u\cdot x\phi$ is a Jordan decomposition in the group $G\rtimes\langle\phi\rangle$.
	Assume that $u,u'\in \cU_{G_x}(K)$ with $ux\phi$ conjugate to $u'x\phi$. Thus there exists $g\in G(K)$ such that $ux\phi=gu'x\phi g^{-1}$. Because $u\cdot x\phi$ and $u'\cdot x\phi$ are Jordan decompositions, it follows that $u=gu'g^{-1}$ and $x\phi=gx\phi g^{-1}$. Therefore $g\in G_x$. Hence the morphism $(\cU_{G_x}^\wedge/G_x)(K)\xrightarrow{\sim} (G\phi/G)^\wedge_\xi(K)$ is an equivalence of groupoids.
	
	Let $u\in \cU_{G_x}(K)$ be a point. The tangent complex of $\cU_{G_x}^\wedge/G_x$ at $u$ is given by
	$$\fg^{x\phi}\xrightarrow{1-\Ad_u}\fg^{x\phi}.$$
	The tangent complex of $(G\phi/G)^\wedge_\xi$ at $ux\phi$ is given by
	$$\fg\xrightarrow{1-\Ad_{ux\phi}}\fg.$$
	Because $u\cdot x\phi$ is a Jordan decomposition, the embedding $\fg^{x\phi}\hookrightarrow\fg$ defines a quasi-isomorphism between two complexes. In fact, because $t\phi$ is semisimple, $\fg$ splits in to eigenspaces of $\Ad_{x\phi}$, and the $\Ad_u$-action preserves this decomposition. It is easy to see that only the eigenspace with eigenvalue $1$, i.e. $\fg^{x\phi}$, has non-zero cohomologies.
\end{proof}

Let $\chi\in A(\Lambda)$ be an element. Let $A^\wedge_\chi$ denote the formal completion of $A$ at $\chi$. Let $(B\phi/B)^\wedge_\chi$ denote the fiber product $(B\phi/B)\times_{A}A^\wedge_\chi$.
Let $x\in T(\Lambda)$ be a point mapping to $\chi$. We have the following result.

\begin{prop}\label{prop-Groth-endoscopy}
	The map $g\mapsto gx\phi$ defines an isomorphism 
	$$U^\wedge_x/B_x\xrightarrow{\sim} (B\phi/B)^\wedge_\chi.$$
    Here $U_x^\wedge$ is the completion of $B_x$ along the preimage of identity in $T^\phi$.
\end{prop}
\begin{proof}
	We use the similar method as Proposition \ref{prop-adjoint-quot-endoscopy}. Let $K$ be an algebraically closed field over $\Lambda$. The morphism on $K$-points is an equivalence follows from the same argument. Let $u\in U_x(K)$ be an element. The tangent complex of $U^\wedge_x/B_x$ at $u$ is computed by
	$$\fb^{x\phi}\xrightarrow{\Ad_u}\fb^{x\phi}$$
	and tangent complex of $(B\phi/B)^\wedge_x$ at $ux\phi$ is computed by
	$$\fb\xrightarrow{\Ad_{ux\phi}}\fb.$$
	The embedding $\fb^{x\phi}\hookrightarrow\fb$ induces a quasi-isomorphism between two complexes, because $u\cdot x\phi$ is a Jordan decomposition in $B\rtimes \langle\phi\rangle$.
\end{proof}

\begin{proof}[Proof of Proposition \ref{prop-generators-GS}]
    Recall that we have reduced to the case that $G$ is almost simple and simply connected. For each point $\chi\in A(\Lambda)$ mapping to $\xi$, fix a lift $t_\chi\in T(\Lambda)$ of $\chi$. By Proposition \ref{prop-adjoint-quot-endoscopy} and Proposition \ref{prop-Groth-endoscopy}, we have isomorphisms
	$$(B\phi/B)^\wedge_\chi\cong U_{t_\chi}^\wedge/B_{t_\chi}\quad\text{ and }\quad (G\phi/G)^\wedge_{\xi}\cong \cU^\wedge_{G_{t_{\chi}}}/G_{t_{\chi}}.$$
	
	Let $\chi$, $\chi'$ be two $\Lambda$-points of $A$ mapping to $\xi$. Let $w\in W_0$ be an element such that $w(\chi)=\chi'$. Then $t_\chi$ is $\phi$-conjugated to $t_{\chi'}$ be an element $\dot{w}\in N_G(T)$ with image $w$ in $W_0\subseteq W$, i.e. $\dot{w}t_\chi \phi(\dot{w})^{-1}=t_{\chi'}$. Then $g\mapsto \dot{w}g\dot{w}^{-1}$ defines an isomorphism $\dot{w}\colon G_{t_{\chi}}\xrightarrow{\sim}G_{t_{\chi'}}$. Moreover, it induces a commutative diagram 
	$$\begin{tikzcd}
		\cU^\wedge_{G_{t_\chi}}/G_{t_\chi}\ar[rd,"\simeq","\dot{w}"swap]\ar[rr,"\simeq"] & \ar[d,Rightarrow,"\dot{w}"] & (G\phi/G)_\xi^\wedge \\ 
		& \cU^\wedge_{G_{t_{\chi'}}}/G_{t_{\chi'}}\ar[ur,"\simeq"swap]
	\end{tikzcd}$$
	of isomorphisms of algebraic stacks. Assume that $w\in W_0$ has minimal length in the coset $w\cdot (W_0)_\chi$. Then the morphism $g\mapsto \dot{w}g\dot{w}^{-1}\colon G_{t_\chi}\to G_{t_{\chi'}}$ sends the Borel subgroup $B_{t_\chi}$ to the Borel subgroup $B_{t_{\chi'}}$. It follows that there is a commutative diagram
	$$\begin{tikzcd}
		A_\chi^\wedge\times \BB T^{\phi}\ar[d,"w"swap] & U^\wedge_{t_\chi}/B_{t_\chi}\ar[l]\ar[r,"\pi_\chi"]\ar[d,"\dot{w}"swap] & \cU^\wedge_{G_{t_\chi}^\circ}/G_{t_\chi}\ar[d,"\dot{w}"] \\
		A_{\chi'}^{\wedge}\times\BB T^{\phi} & U^\wedge_{t_{\chi'}}/B_{t_{\chi'}}\ar[l]\ar[r,"\pi_{\chi'}"] & \cU^\wedge_{G_{t_{\chi'}}^\circ}/G_{t_{\chi'}}
	\end{tikzcd}$$
	where $A_\chi^\wedge\times \BB T^{\phi}\xrightarrow{w}A_{\chi'}^{\wedge}\times\BB T^{\phi}$ is induced by the $W_0$-action on $T^\phi$ and $A$.
	
	Now fix a $\Lambda$-point $\chi_0\in A$ mapping to $\xi$. Denote $x=t_{\chi_0}\in T(\Lambda)$ for simplicity. For each $\chi\in A(\Lambda)$ mapping to $\xi$, there is a unique element $w_{\chi}\in W_0$ such that $w_\chi(\chi_0)=\chi$ and $w_\chi$ has minimal length in $w_\chi\cdot (W_0)_{\chi_0}$. 
	
	Denote by $\pi_x\colon U_x^\wedge/B_x\to \cU^\wedge_{G_x}/G_x$ the completion of the Grothendieck--Springer resolution for $G_x$. Denote by $\gamma_x\colon U_x^\wedge/B_x\to A_{\chi_0}^\wedge\times \BB T^\phi$ the natural projection. For $w\in W_0$, let $w\colon  A_{\chi_0}^\wedge\times \BB T^\phi\to A_{\chi_0}^\wedge\times \BB T^\phi$ denote the morphism that is identity on $ A_{\chi_0}^\wedge$, and is given by $w$-action on $\BB T^\phi$. We can rewrite the correspondence
    $$\bigsqcup_{\chi} A^\wedge_\chi\times \BB T^\phi\xleftarrow{\gamma_{\phi,\xi}}\bigsqcup_{\chi}(B\phi/B)_\chi^\wedge\xrightarrow{\pi_{\phi,\xi}} (G\phi/G)_\xi^\wedge$$
    as 
	$$\bigsqcup_{\chi} A^\wedge_{\chi_0}\times \BB T^\phi\xleftarrow{(w_\chi\circ\gamma_x)}\bigsqcup_{\chi} U_x^\wedge/B_x\xrightarrow{\pi_{x}}\cU^\wedge_{G_x}/G_x,$$
	where on the component indexed by $\chi$, the morphism $w_\chi\circ\gamma_x\colon U^\wedge_x/B^x\to A^\wedge_{\chi_0}\times \BB T^\phi$ is the composition of $\gamma_x$ with the $w_\chi$-action on $\BB T^\phi$. Pullback long $\iota_{\chi_0}\colon \BB T^\phi\hookrightarrow A^\wedge_{\chi_0}\times \BB T^\phi$, we get the correspondence
	$$\BB T^\phi\xleftarrow{(w_\chi\circ\bar\gamma_x)}\bigsqcup_{\chi\mapsto \xi} U_x/B_x\xrightarrow{\bar\pi_{x}}\cU_{G_x}^\wedge/G_x.$$
    Here $\bar\gamma_x\colon U_x/B_x\to \BB T^\phi$ and $\bar\pi_x\colon U_x/B_x\to \cU_{G_x}^\wedge/G_x$ are natural morphisms.
	It suffices to show that objects $$\{(\bar\pi_x)_*(w_\chi\circ\bar\gamma_x)^!\cO_{\BB T^\phi}(\lambda)\cong (\bar\pi_x)_*(\bar\gamma_x)^!\cO_{\BB T^\phi}(w_\chi^{-1}(\lambda))\}$$ generates $\Coh(\cU_{G_x}^\wedge/G_x)$ when $\chi$ runs through preimages of $\xi$, and $\lambda$ runs through either dominant weights, or anti-dominant weights.
	
	Let $W_0^{\chi_0}$ denote the set of elements $w$ in $W_0$ that have minimal length in the cosets $w\cdot (W_0)_{\chi_0}$. It suffices to show that $\{(\bar\pi_x)_*(\bar\gamma_x)^!\cO_{\BB T^\phi}(w^{-1}(\lambda))\}$ generates $\Coh(\cU_{G_x}^\wedge/G_x)$ when $w$ runs through $W_0^{\chi_0}$, and $\lambda$ runs through either dominant weights $\XX^\bullet(T)_\phi^+$, or anti-dominant weights $-\XX^\bullet(T)_\phi^+$. Let $\XX^\bullet(T)_\phi^{x\mbox{-}+}$ denote the dominant cone in $G_x$. Then 
	$$(W_0^{\chi_0})^{-1}\cdot \XX^\bullet(T)_\phi^+=\XX^\bullet(T)_\phi^{x\mbox{-}+}$$
	as subsets of $\XX^\bullet(T)_\phi$. Therefore the proposition follows from Lemma \ref{cor-phi=id}.	
\end{proof}

\begin{proof}[Proof of Proposition \ref{cor-tilting-tensor-twisted-case}]
    Again we may assume that $G$ is almost simple and simply-connected. We only prove (1) as (2) can be proved similarly.
    Let notations be as in the proof of Proposition \ref{prop-generators-GS}. Recall that the order of the $\phi$-action is invertible in $\Lambda$. By \cite[Proposition VIII.5.15]{FS}, the restrisction of $V$ to $G_x$ admits a good filtration. In fact, if $\Lambda$ has characteristic $0$, this is trivial. If $\Lambda$ has characteristic $\ell\geq 0$, then by spreading out, we can assume that $x$ is defined over a finite field. Then $G_x=G^{x\phi}$ is the fixed point subgroup of the conjugation action by $x\phi$, which has finite order prime to $\ell$. Thus we can apply \cite[Proposition VIII.5.15]{FS}.

    As explained in the proof of Proposition \ref{prop-generators-GS}, we know that $(\pi_{\phi,\xi})_*(\gamma_{\phi,\xi})^!\cF_{\lambda}$ has the form
    $$ \bigoplus_{w\in W_0^{\chi_0}}(\bar\pi_x)_*(\bar\gamma_x)^!\cO_{\BB T^\phi}(w^{-1}(\lambda))$$
    under the isomorphism $\Coh((G\phi/G)^\wedge_\xi)\simeq \Coh(\cU_{G_x}^\wedge/G_x)$. Note that $w^{-1}(\lambda)$ runs through the set $\XX^\bullet(T)_\phi^{x\mbox{-}+}\cap W_0\lambda$. Thus it suffices to show that for $\mu\in \XX^\bullet(T)_\phi^{x\mbox{-}+}$, the object $V\otimes (\bar\pi_x)_*(\bar\gamma_x)^!\cO_{\BB T^\phi}(\mu)$ admits a filtration by $(\bar\pi_x)_*(\bar\gamma_x)^!\cO_{\BB T^\phi}(\nu)$ for $\nu\in \XX^\bullet(T^\phi)^{x\mbox{-}+}$. This follows from Lemma \ref{cor-phi=id}.
\end{proof}

\section{Generic part of the unipotent categorical local Langlands correspondence}\label{section-generic-local}
In this section, we study the generic part of the unipotent categorical local Langlands correspondence. We will refine some results in \cite{Tame}, and also study the behavior of the $t$-structures under the categorical local Langlands correspondence. 

\subsection{Categorical local Langlands correspondence}\label{subsection-cllc}

Let $F$ be a non-archimedean local field with ring of integers $O_F$. Let $\varpi\in O_F$ be a uniformizer. Let $q$ be the cardinality of the residue field $\kappa_F=O_F/(\varpi)$. Let $k$ denote the algebraic closure of $\kappa_F$. Let $\breve{F}$ denote the completion of the maximal unramified extension of $F$. Let $\Gamma_F$ (resp. $W_F$) denote the absolute Galois group (resp. Weil group) of $F$. Let $I_F\subseteq W_F$ be the inertia subgroup and $P_F\subseteq I_F$ be the wild inertia subgroup. Fix a topological generator $\tau$ in $I_F/P_F$ and a lift of \emph{arithmetic} Frobenius $\phi$ in $W_F/P_F$. Let $\ell$ be a prime number coprime to $|\kappa_F|$. Denote $\Lambda=\overline\QQ_\ell$ or $\overline\FF_\ell$.

\subsubsection{The local Langlands category}

Let $G$ be a connected reductive group over $F$. Assume that $G$ is an unramified connected reductive group over $F$, i.e. $G$ is quasi-split and splits over $\breve{F}$. Fix a pinning $(B,T,e)$ of $G$ over $F$. Let
$$W=N_G(T)(\breve{F})/T(\breve{F}),\quad \text{resp. } \widetilde{W}=N_G(T)(\breve{F})/T(O_{\breve{F}})$$
be the Weyl group (resp. extended affine Weyl group) of $G_{\breve{F}}$. 
Let $\XX_\bullet(T)=\Hom(T_{\breve{F}},\GG_m)$ (resp, $\XX_\bullet(T)=\Hom(\GG_m,T_{\breve{F}})$) be the weight lattice (resp. coweight lattice). There is a short exact sequence
$$1\to \XX_\bullet(T)\to \widetilde{W}\to W\to 1.$$
The pinning determines a hyperspecial subgroup of $G(F)$, and hence defines a reductive model of $G$ over $O_F$. By abuse of notation, we still denote by $G$ the reductive model over $O_F$. Let $I$ denote the preimage of $B(\kappa_F)$ along the projection map $G(O_F)\to G(\kappa_F)$. Then $I$ is an Iwahori subgroup of $G(F)$, and defines an Iwahori model $\cI$ of $G$ over $O_F$. We obtain a splitting $\widetilde{W}=\XX_\bullet(T)\rtimes W$. For elements $\lambda$ in $\XX_\bullet(T)$, we denote by $t_\lambda$ the corresponding element in $\widetilde{W}$. Let $\XX_\bullet(T)^+\subseteq \XX_\bullet(T)$ be the dominant cone defined by the \emph{opposite} Borel subgroup $B^+$ of $B$. Let $2\rho\in \XX_\bullet(T)^+$ be the sum of positive coroots. 

Let $\Perf^\Aff_k$ be the category of perfect $k$-algebras. For a perfect $\kappa_F$-algebra $R$, denote by $W_{O_F}(R)$ the ring of $O_F$-Witt vectors over $R$: If $F$ has characteristic 0, then $W_{O_F}(R)=W(R)\otimes_{W(\kappa_F)}O_F$, and if $F$ has characteristic $p>0$, then $W_{O_F}(R)=R[[\varpi]]$. Denote $D_R=\Spec(W_{O_F}(R))$ and $D_R^*=\Spec(W_{O_F}(R)[1/\varpi])$. Denote by $\phi_R\colon D_R\to D_R$ (resp. $\phi_R\colon D^*_R\to D^*_R$) the endomorphism induced by the $q$-Frobenius of $R$. 

The positive loop group $\Iw=L^+\cI$ and the loop group $LG$ associated to the Iwahori group scheme $\cI$ are defined as functors from $\Perf^\Aff_k$ to the category of sets as
$$\Iw(R)=\cI(W_{O_F}(R))=\mathrm{Map}(D_R,\cI),\quad LG(R)=G(W_{O_F}(R)[1/\varpi])=\mathrm{Map}(D_R^*,G)$$
for $R\in\Perf^\Aff_k$. Then $\Iw$ is represented by a perfect affine group scheme, and $LG$ is represented by an ind-affine perfect scheme. The absolute $q$-Frobenius on $D_R$ (resp. $D_R^*$) defines the relative Frobenius automorphism $\phi\colon \Iw\xrightarrow{\sim}\Iw$ (resp. $\phi\colon LG\xrightarrow{\sim}LG$). Define the $\phi$-adjoint action $\Ad_\phi$ of $LG$ on $LG$ by $\Ad_\phi(g)(x)=gx\phi^{-1}(g).$

Define the stack of $G$-local shtukas $\Sht^\loc_\cI$ as the \'etale quotient stack $LG/\Ad_\phi \Iw$. As a functor of points, the stack $\Sht^\loc_\cI$ sends $R\in\Perf^\Aff_k$ to the groupoid of pairs $(\cE,\phi_\cE)$ where $\cE$ is a $\cI$-torsor on $D_R$ and $\phi_\cE\colon \phi_R^*\cE|_{D_R^*}\xrightarrow{\sim} \cE|_{D_R^*}$ is an isomorphism of $G$-torsors. Let $w\in \widetilde{W}$ be an element. We denote by $\Sht^\loc_{\cI,w}$ the locally closed substack of $\Sht^\loc_\cI$ where the modification $\phi_\cE$ is in position $w$.

Define the stack of $G$-isocrystals $\Isoc_G$ as the \'etale quotient stack $LG/\Ad_\phi LG$. For $R\in \Perf^\Aff_k$, the groupoid of $R$-points of $\Isoc_G$ is given by pairs $(\cE,\phi_\cE)$ where $\cE$ is a $G$-torsor on $D_R^*$ that is trivial after pullback to $D_{R'}^*$ for some \'etale cover $R\to R'$, and $\phi_\cE\colon \phi_R^*\cE\xrightarrow{\sim}\cE$ is an isomorphism of $G$-torsors. There is a natural morphism
$$\Nt\colon \Sht^\loc_\cI \to \Isoc_G$$
called the Newton morphism. We know that $\Nt$ is ind-pfp proper by \cite[Lemma 3.28]{Tame}.

The geometry of the stack $\Isoc_G$ is discussed in \cite[\S 3.2]{Tame}. There is a Newton stratification on $\Isoc_G$ indexed by the Kottwitz set $B(G)\coloneqq G(\breve{F})/(x\sim gx\phi(g)^{-1})$. There is an injective map
$$b\mapsto(\nu_b,\kappa_G(b))\colon B(G)\hookrightarrow \XX^\bullet(T)^{+,\phi}_\QQ\times\pi_1(G)_{\phi}.$$
Define the order on $B(G)$ by letting $b\leq b'$ if $\kappa_G(b)=\kappa_G(b')$ and $\nu_{b'}-\nu_b$ is a non-positive rational linear combination of positive coroots. Minimal elements in $B(G)$ are called \emph{basic}. The set of basic elements is denoted by $B(G)_{\mathrm{bsc}}$. The Kottwitz map $\kappa_G$ defines a bijection $B(G)_{\mathrm{bsc}}\xrightarrow{\sim} \pi_1(G)_{\Gamma_F}$. For $b\in B(G)$, there is an $F$-group $G_b$ defined as $\phi$-twisted centralizer of a lift of $b$ in $G(\breve{F})$. If $b$ is basic, then $G_b$ is an inner form of $G$. In general, $G_b$ is an inner form of a Levi subgroup of $G$.

The underlying topological space of $\Isoc_G$ is identified with $B(G)$. For $b\in B(G)$, let $i_b\colon \Isoc_{G,b}\hookrightarrow \Isoc_G$ denote the locally closed embedding defined by $b$. Let $i_{\leq b}\colon \Isoc_{G,\leq b}\hookrightarrow \Isoc_G$ (resp. $i_{<b}\colon \Isoc_{G,<b}\hookrightarrow \Isoc_G$) denote the closed embedding consisting of the strata indexed by $b'$ with $b'\leq b$ (resp. $b'<b$). By \cite[Theorem 3.31]{Tame}, $i_{\leq b}$ is a perfectly finite presented closed embedding, and $j_b\colon \Isoc_{G,b}\hookrightarrow \Isoc_{G,\leq b}$ is an affine open embedding. Moreover, after choosing a lift of $b$ to $G(\breve{F})$, there is a natural isomoprhism
$$\Isoc_{G,b}\cong \BB G_b(F),$$
where $\BB G_b(F)$ is pro-\'etale classifying stack of the locally profinite group $G_b(F)$.

Let 
$$\Shv(-,\Lambda)\colon \mathrm{Corr}(\mathrm{PreStk}_k)_{\mathrm{IndEproet};\mathrm{All}}\to \Lincat_\Lambda$$
denote the sheaf theory constructed in \cite[Proposition 10.97]{Tame}. If $X$ is a pfp perfect scheme over $k$, then $\Shv(X,\Lambda)$ is identified with the category of ind-constructible \'etale $\Lambda$-sheaves on $X$. If $X$ is a qcqs scheme, we may write $X=\lim_i X_i$ as an inverse limit of schemes that are pfp over $k$ and then $\Shv(X,\Lambda)=\colim_i\Shv(X_i,\Lambda)$ with transitioning functors given by $!$-pullbacks.
In general $\Shv(X,\Lambda)$ is defined by right Kan extension along $!$-pullback functors. If the prestack $X$ is sind-placid (in the sense of \cite[Definition 10.157]{Tame}), then there is a variant sheaf theory $\Ind\Shv_{\fingen}(X,\Lambda)$ of ind-finitely generated sheaves over $X$ constructed in \cite[\S 10.6]{Tame}. There is a functor $\Psi\colon \Ind\Shv_{\fingen}(X,\Lambda)\to \Shv(X,\Lambda)$ for a sind-placid stack $X$. If $X$ is ind-placid, then the functor $\Psi\colon \Shv_{\fingen}(X,\Lambda)\to \Shv(X,\Lambda)$ is fully faithful. All the perfect stacks considered in this paper will be sind-placid.

For $b\in B(G)$, let $\Rep(G_b(F),\Lambda)$ be the category of smooth $G_b(F)$-representations over $\Lambda$. The category $\Rep(G_b(F),\Lambda)$ is compactly generated with a set of compact projective generators given by compact inductions $\cInd_{K}^{G_b(F)}\Lambda$ where $K$ runs through pro-$p$ open compact subgroups of $G_b(F)$. Therefore an object $V\in \Rep(G_b(F),\Lambda)$ is admissible if and only if $V^K$ is a perfect $\Lambda$-module for any pro-$p$ open compact subgroup $K\subseteq G_b(F)$.

We summarize some properties of $\Shv(\Isoc_G,\Lambda)$ needed in the sequel. 
\begin{thm}\label{thm-local-Langlands-cat}
    \begin{enumerate}
        \item For $b\in B(G)$, there is an open and closed stratification
        $$\begin{tikzcd}[column sep=huge]
            \Shv(\Isoc_{G,<b},\Lambda)\ar[r,leftarrow,shift left=1em,"(i_{<b})^*"]\ar[r,"(i_{<b})_!\simeq (i_{<b})_*"]\ar[r,leftarrow,shift right=1em,"(i_{<b})^!"] & \Shv(\Isoc_{G,\leq b},\Lambda) \ar[r,leftarrow,shift left=1em,"(j_b)^!"]\ar[r,"(j_b)^!\simeq (j_b)^*"]\ar[r,leftarrow,shift right=1em,"(j_b)_*"] & \Shv(\Isoc_{G,b},\Lambda).
        \end{tikzcd}$$
        \item For $b\in B(G)$, after choosing a $k$-point of $\Isoc_{G,b}$, there is a natural equivalence $\Shv(\Isoc_{G,b},\Lambda)\simeq \Rep(G_b(F),\Lambda)$, where $Rep(G_b(F),\Lambda)$ is the category of $G_b(F)$-representations over $\Lambda$.
        \item The category $\Shv(\Isoc_G,\Lambda)$ is compactly generated. An object $\cF\in \Shv(\Isoc_G,\Lambda)$ is compact if and only if $(i_b)^!\cF\in \Rep(G_b(F),\Lambda)$ is zero for all but finitely many $b$ and $(i_b)^!\cF$ is compact for all $b$, or equivalently, $(i_b)^*\cF\in \Rep(G_b(F),\Lambda)$ is zero for all but finitely many $b$ and $(i_b)^*\cF$ is compact for all $b$.
        \item An object $\cF\in \Shv(\Isoc_G,\Lambda)$ is admissible if and only if $(i_b)^!\cF\in \Rep(G_b(F),\Lambda)$ is admissible for any $b$, if and only if $(i_b)^\sharp\cF$ is admissible for any $b$. Here $(i_b)^\sharp$ is the right adjoint of $(i_b)_!$.
        \item The functors $(i_b)_!,(i_b)_*$ preserve compact objects. The functors $(i_b)_*,(i_b)_\flat$ preserve admissible objects. Here $(i_b)_\flat$ is the right adjoint of $(i_b)^!$.
        \item There is a canonical duality $\DD_{\Isoc_G}^\can$ on $\Shv(\Isoc_G,\Lambda)$. There are canonical equivalences
        $$(\DD_{\Isoc_G}^\can)^\omega\circ (i_b)_*\simeq (i_b)_!\circ (\DD_{G_b(F)}^\can)^\omega[-2\langle2\rho,\nu_b\rangle](-\langle2\rho,\nu_b\rangle),$$
        and
        $$(i_b)^*\circ(\DD_{\Isoc_G}^\can)^\omega\simeq (\DD_{G_b(F)}^\can)^\omega\circ (i_b)^![-2\langle2\rho,\nu_b\rangle](-\langle2\rho,\nu_b\rangle),$$
        for any $b\in B(G)$. Here $(\DD_{G_b(F)}^\can)^\omega\colon (\Rep(G_b(F),\Lambda)^{\omega})^\mathrm{op}\xrightarrow{\sim}\Rep(G_b(F),\Lambda)^{\omega}$ is the Bernstein--Zelevinsky duality.
        \item The duality $\DD_{\Isoc_G}^\can$ induces an anti-involution $$(\DD_{\Isoc_G}^\can)^\Adm\colon (\Shv(\Isoc_G,\Lambda)^\Adm)^\mathrm{op}\xrightarrow{\sim}\Shv(\Isoc_G,\Lambda)^\Adm$$ for the subcategory of admissible objects. There are canonical equivalences
        $$(\DD_{\Isoc_G}^\can)^\Adm\circ (i_b)_*\simeq (i_b)_!\circ (\DD_{G_b(F)}^\can)^\Adm[-2\langle2\rho,\nu_b\rangle](-\langle2\rho,\nu_b\rangle),$$
        and
        $$(i_b)^\sharp\circ(\DD_{\Isoc_G}^\can)^\Adm\simeq (\DD_{G_b(F)}^\can)^\Adm\circ (i_b)^![-2\langle2\rho,\nu_b\rangle](-\langle2\rho,\nu_b\rangle),$$
        for any $b\in B(G)$. Here $(\DD_{G_b(F)}^\can)^\Adm\colon (\Rep(G_b(F),\Lambda)^\Adm)^\mathrm{op}\xrightarrow{\sim}\Rep(G_b(F),\Lambda)^\Adm$ is the smooth duality.
    \end{enumerate}
\end{thm}
\begin{proof}
    This is a summary of results in \S 3.4.1 and \S 3.4.2 of \cite{Tame}. 
\end{proof}

Let $\Rep_{\fingen}(G_b(F),\Lambda)\simeq \Shv_{\fingen}(\BB G_b(F),\Lambda)$ be the category of finitely generated representations. By \cite[Proposition 3.57]{Tame}, the natural functor $\Rep_{\fingen}(G_b(F),\Lambda)\to \Rep(G_b(F),\Lambda)$ is fully faithful, with essential image generated under cones and retracts by objects $\cInd_{K}^{G_b(F)}\Lambda$ for open compact subgroups $K\subseteq G_b(F)$. 
For our purpose, this descirption can be regarded as the definition of $\Rep_{\fingen}(G_b(F),\Lambda)$. 
In particular, if $\Lambda=\overline{\QQ}_\ell$, or $\ell$ is banal with respect to $G_b(F)$, then finitely generated objects in $\Rep(G_b(F),\Lambda)$ agree with compact objects. The following proposition summarizes some results from \cite[\S 3.4.3]{Tame}.

\begin{prop}
    The natural functor $\Shv_{\fingen}(\Isoc_G,\Lambda)\to \Shv(\Isoc_G,\Lambda)$ is fully faithful. An object $\cF\in \Shv(\Isoc_G,\Lambda)$ is finitely generated if and only if $(i_b)^*\cF=0$ for all but finitely many $b$ and $(i_b)^*\cF\in \Rep(G_b(F),\Lambda)$ is finitely generated for all $b$, or equivalently, $(i_b)^!\cF=0$ for all but finitely many $b$ and $(i_b)^!\cF\in \Rep(G_b(F),\Lambda)$ is finitely generated for all $b$. The functors $(i_b)_*$, $(i_b)_!$ preserve finitely generated objects.
\end{prop}

For our purpose, the statements in above proposition can be regarded as the definition of $\Shv_{\fingen}(\Isoc_G,\Lambda)$. Compact objects in $\Shv(\Isoc_G,\Lambda)$ are finitely generated. Therefore there is a continuous fully faithful embedding
$$\Psi^L\colon \Shv(\Isoc_G,\Lambda)\hookrightarrow \Ind\Shv_{\fingen}(\Isoc_G,\Lambda)$$
left adjoint to the natural functor $\Psi\colon\Ind\Shv_{\fingen}(\Isoc_G,\Lambda)\to  \Shv(\Isoc_G,\Lambda)$.

We will mainly work with the subcategory of unipotent sheaves on $\Isoc_G$. 

\begin{defn}[{\cite[Definition 4.116]{Tame}}]
    Let $H$ be a connected reductive group over $F$. Let $\Rep^{\widehat\unip}(H(F),\Lambda)\subseteq \Rep(H(F),\Lambda)$ be the subcategory generated under colimits by $\cInd_P^{H(F)}\pi_P$, where $P$ is a parahoric subgroup of $H(F)$, and $\pi_P$ is a unipotent Deligne--Lusztig representation of the Levi quotient $L_P$ of $P$. Objects in $\Rep^{\widehat\unip}(H(F),\Lambda)$ are called \emph{unipotent} representations.

    Let $\Rep^\unip_\fingen(H(F),\Lambda)\subseteq \Rep^{\widehat\unip}(H(F),\Lambda)$ be the idempotent complete full subcategory generated by those $\cInd_P^{H(F)}\pi_P$ with $\pi_P$ runs through unipotent Deligne--Lusztig representations.
\end{defn}

The category $\Rep_\fingen^\unip(H(F),\Lambda)$ contains $\Rep^{\widehat\unip}(H(F),\Lambda)^\omega$, and we have
$$\Rep^{\widehat\unip}(H(F),\Lambda)^\omega\subseteq \Rep_\fingen^\unip(H(F),\Lambda)\subseteq \Rep^{\widehat\unip}(H(F),\Lambda)\cap \Rep_\fingen(H(F),\Lambda).$$
Therefore there are fully faithful embeddings 
$$\Rep^{\widehat\unip}(H(F))\hookrightarrow \Ind\Rep_\fingen^\unip(H(F))\hookrightarrow\Ind\Rep_\fingen(H(F),\Lambda).$$

The category $\Rep^{\widehat\unip}(H(F),\Lambda)$ admits a finite collection of compact projective generators: Let $P$ be a parahoric subgroup of $H(F)$ with Levi quotient $L_P$. Let $\widetilde{R}_{\dot{w},\hat{u}}^T$ 
%\Xinwen{Need a notation reflecting unipotent}
be the Deligne--Lusztig induction of the indecomposible unipotent monodromic tilting object $\mathrm{Til}^{\mathrm{mon}}_{\dot{w},\hat{u}}$ indexed by $w\in W_{L_P}$ as in \cite[\S 4.4.2]{Tame}.  Then objects of the form $\cInd_P^{G(F)}\widetilde{R}^T_{\dot{w},\hat{u}}$ form a set of compact projective generators in $\Rep^{\widehat\unip}(H(F),\Lambda)$.

\begin{defn}
    Let $\Shv^{\widehat\unip}(G(F),\Lambda)$ (resp. $\Shv^\unip_\fingen(\Isoc_G,\Lambda)$)
     be the subcategory of $\Shv(\Isoc_G,\Lambda)$ (resp. $\Shv_\fingen(\Isoc_G,\Lambda)$) consisting of those objects $\cF$ with $(i_b)^!\cF$ lies is $\Rep^{\widehat\unip}(G_b(F),\Lambda)$ (resp. $\Rep^\unip_\fingen(G_b(F),\Lambda)$) for any $b$. 
\end{defn}

The following properties will be needed.
\begin{prop}[{\cite[Proposition 4.122]{Tame}}]\label{prop-unip-local-Langlands-cat}
    \begin{enumerate}
        \item The functors $(i_b)^*,(i_b)_*,(i_b)^!,(i_b)_!$ preserve the subcategory $\Shv^{\widehat\unip}(\Isoc_G,\Lambda)$ (resp. $\Ind\Shv_\fingen^\unip(\Isoc_G,\Lambda)$) of $\Shv(\Isoc_G,\Lambda)$ (resp. $\Ind\Shv_\fingen(\Isoc_G,\Lambda)$). A similar open and closed stratification as in Theorem \ref{thm-local-Langlands-cat}(1) holds for $\Shv^{\widehat\unip}(\Isoc_G,\Lambda)$ (resp. $\Ind\Shv^\unip_\fingen(\Isoc_G,\Lambda)$).
        \item An object $\cF\in \Shv(\Isoc_G,\Lambda)$ (resp. $\cF\in\Ind\Shv_\fingen(\Isoc_G,\Lambda)$) is unipotent if and only if $(i_b)^*\cF\in \Rep^{\widehat\unip}(G_b(F),\Lambda)$ (resp. $(i_b)^*\cF\in \Ind\Rep^{\unip}_\fingen(G_b(F),\Lambda)$) for all $b$.
        \item The embedding $\Shv^{\widehat\unip}(\Isoc_G,\Lambda)\hookrightarrow\Shv(\Isoc_G,\Lambda)$ (resp. $\Ind\Shv_\fingen^\unip(\Isoc_G,\Lambda)\hookrightarrow\Ind\Shv_\fingen(\Isoc_G,\Lambda)$) admits a right adjoint $\cP^{\widehat\unip}$ (resp. $\cP^\unip$) called the unipotent projector. Moreover, $\cP^{\widehat\unip}$ (resp. $\cP^\unip$) preserves compact objects (resp. finitely generated objects) and commutes with $(i_b)^!$ and $(i_b)_*$.
    \end{enumerate}
\end{prop}

Let $\Psi^\unip\colon \Ind\Shv_\fingen^\unip(\Isoc_G,\Lambda)\to \Shv^{\widehat\unip}(\Isoc_G,\Lambda)$ denote the natural functor. As $\Shv^\unip_\fingen(\Isoc_G,\Lambda)$ contains $\Shv^{\widehat\unip}(\Isoc_G,\Lambda)^\omega$, there is a continuous fully faithful functor
$$\Psi^{L,\unip}\colon \Shv^{\widehat\unip}(\Isoc_G,\Lambda)\hookrightarrow\Ind\Shv^\unip_\fingen(\Isoc_G,\Lambda)$$
left adjoint ot $\Psi^\unip$.

\subsubsection{$t$-structures on the local Langlands category}\label{subsection-t-str}

There are two natural $t$-structures on $\Shv(\Isoc_G,\Lambda)$ glued from the standard $t$-structures $(\Rep(G_b(F),\Lambda)^{\leq 0},\Rep(G_b(F),\Lambda)^{\geq 0})$ on $\Rep(G_b(F),\Lambda)$.

We recall the perverse $t$-structure on $\Shv(\Isoc_G,\Lambda)$.

\begin{prop}[{\cite[Proposition 3.95]{Tame}}]
    Let $\Shv(\Isoc_G,\Lambda)^{p,\leq0}\subseteq \Shv(\Isoc_G,\Lambda)$
    denote the subcategory generated under colimits and extensions by objects of the form
    $$(i_b)_!\cInd_K^{G_b(F)}\Lambda[n-\langle2\rho,\nu_b\rangle],\quad b\in B(G),\ n\geq 0,\ K\subseteq G_b(F)\text{ pro-$p$ open compact subgroup}.$$
    Then $\Shv(\Isoc_G,\Lambda)^{p,\leq0}$ form a connective part of an admissible $t$-structure   $$(\Shv(\Isoc_G,\Lambda)^{p,\leq0},\Shv(\Isoc_G,\Lambda)^{p,\geq0})$$ on $\Shv(\Isoc_G,\Lambda)$ called the \emph{perverse $t$-structure}. The coconnective part of the perverse $t$-structure can be described as
    $$\Shv(\Isoc_G,\Lambda)^{p,\geq 0}=\{\cF\in \Shv(\Isoc_G,\Lambda)| (i_b)^!\cF\in \Rep(G_b(F),\Lambda)^{\geq \langle2\rho,\nu_b\rangle}\}.$$
    When $\Lambda=\overline{\QQ}_\ell$,
    the perverse $t$-structure restricts to a bounded $t$-structure on $\Shv(\Isoc_G,\Lambda)^\omega$, and
    $$\Shv(\Isoc_G,\Lambda)^{p,\leq 0}\cap \Shv(\Isoc_G,\Lambda)^\omega=\{\cF\in \Shv(\Isoc_G,\Lambda)| (i_b)^*\cF\in \Rep(G_b(F),\Lambda)^{\leq \langle2\rho,\nu_b\rangle}\}.$$
\end{prop}
\begin{rmk}
    In \cite{Tame}, we worked with more general weights of $G$ in the place of $2\rho$. However, we will only need the $2\rho$-case in the sequel.
\end{rmk}

Aside the perverse $t$-structure on $\Shv(\Isoc_G,\Lambda)$, there is another natural $t$-structre on $\Shv(\Isoc_G,\Lambda)$ called the \emph{exotic} $t$-structure, which will be more relevant to us.

\begin{prop}[{\cite[Proposition 3.110, Proposition 3.111]{Tame}}]\label{prop-adm-dual-exotic}
    Let $\Shv(\Isoc_G,\Lambda)^{e,\leq 0}$ (resp. $\Shv(\Isoc_G,\Lambda)^{e,\geq 0}$) denote the subcategory of objects $\cF$ such that
    $$(i_b)^!\cF\in \Rep(G_b(F),\Lambda)^{\leq \langle2\rho,\nu_b\rangle}\quad \text{(resp. $(i_b)^\sharp\cF\in \Rep(G_b(F),\Lambda)^{\geq \langle2\rho,\nu_b\rangle}$)}$$
    for all $b\in B(G)$. 
    \begin{enumerate}
    \item The pair $(\Shv(\Isoc_G,\Lambda)^{e,\leq 0},\Shv(\Isoc_G,\Lambda)^{e,\geq 0})$ form an acessible $t$-structure on $\Shv(\Isoc_G,\Lambda)$, called the \emph{exotic $t$-structure}. The exotic $t$-structure on $\Shv(\Isoc_G,\Lambda)$ restricts to a $t$-structure on $\Shv(\Isoc_G,\Lambda)^\Adm$.
    \item The admissible duality $(\DD_{\Isoc_G}^\can)^\Adm$ interchanges $\Shv(\Isoc_G,\Lambda)^{e,\leq 0}\cap \Shv(\Isoc_G,\Lambda)^\Adm$ and $\Shv(\Isoc_G,\Lambda)^{e,\geq 0}\cap \Shv(\Isoc_G,\Lambda)^\Adm$.
    \end{enumerate}
\end{prop}

The exotic $t$-structure naturally restricts to a $t$-structure on the subcategory $\Shv^{\widehat\unip}(\Isoc_G,\Lambda)$. Let $(\Shv^{\widehat\unip}(\Isoc_G,\Lambda)^{e,\leq 0},\Shv^{\widehat\unip}(\Isoc_G,\Lambda)^{e,\geq 0})$ denote the resulting $t$-structure. Then $\cF\in \Shv^{\widehat\unip}(\Isoc_G,\Lambda)$ is connective (resp. coconnective) if and only if
$$(i_b)^!\cF\in\Rep^{\widehat\unip}(G_b(F))^{\leq \langle2\rho,\nu_b\rangle} \quad\text{(resp. $\cP^{\widehat\unip}((i_b)^\sharp\cF)\in\Rep^{\widehat\unip}(G_b(F))^{\geq \langle2\rho,\nu_b\rangle}$)}$$
for all $b\in B(G)$. It is not hard to see that $\cP^{\widehat\unip}$ is exotic $t$-exact and $(\DD_{\Isoc_G}^{\widehat\unip,\can})^\Adm$ interchanges $\Shv^{\widehat\unip}(\Isoc_G,\Lambda)^{e,\leq 0}\cap \Shv^{\widehat\unip}(\Isoc_G,\Lambda)^\Adm$ and $\Shv^{\widehat\unip}(\Isoc_G,\Lambda)^{e,\geq 0}\cap \Shv^{\widehat\unip}(\Isoc_G,\Lambda)^\Adm$.

\subsubsection{Stack of local Langlands parameters}

Fix $\sqrt{q}\in \Lambda$. 

Let $\hat{G}$ be the dual group of $G$ over $\Lambda$. It is equipped with a pinning $(\hat{B},\hat{T},\hat{e})$.
The tuple $(\hat{G},\hat{B},\hat{T},\hat{e})$ carries an action of the Weil group $W_F$. Fix a lift $\phi\in W_F$ of the arithmetic Frobenius. Let $I_F\subseteq W_F$ be the inertia subgroup and $P_F\subseteq I_F$ be the wild inertia subgroup. Then $W_F/I_F\cong \phi^\ZZ$. Fix a topological generator $\tau\in I_F/P_F$. Then $\langle \phi,\tau\rangle$ generates a dense subgroup of $W_F$ under the relation $\phi\tau\phi^{-1}=\tau^q$.
%\Xiangqian{Maybe it is better to use arithmetic Frobenius?}

 By assumption, the action of $I_F$ on $\hat{G}$ is trivial, and the $\phi$-action on $\hat{G}$ has finite order. Define the Langlands dual group as the semi-direct product
$${}^LG\coloneqq \hat{G}\rtimes \langle\phi\rangle.$$
Let $R$ be a $\Lambda$-algebra. A \emph{Langlands parameter} (or \emph{$L$-parameter}) of ${}^LG$ with coefficients in $R$ is a strongly continuous homomorphism
    $\varphi\colon W_F\to {}^LG(R)$
such that the composition of $\varphi$ with the projection ${}^LG(R)\to \langle \phi\rangle$ is the natural projection $W_F\to \langle \phi\rangle$. We refer to \cite[\S 2.1.1]{Tame} for the definition of strong continuity.
    
Let $\Loc_{{}^LG,F}$ denote the (derived) moduli stack of local $L$-parameters of ${}^LG$. It is known (see \cite[Proposition 3.1.6]{Zhu-coherent-sheaf}) that $\Loc_{{}^LG,F}$ a disjoint union of classical Artin stacks of finite type over $\Lambda$. It is equidimensional of dimension $0$, and is local complete intersection with trivial dualizing complex.

Let $R$ be a $\Lambda$-algebra. We say an $L$-parameter $\varphi\colon W_F\to {}^LG(R)$ \emph{tame} if $\varphi$ is trivial on the wild inertia subgroup. We say $\varphi$ \emph{unipotent} if $\varphi$ is tame and for each geometric point $x=\Spec K$ of $\Spec R$, the element $\varphi_x(\tau)\in \hat{G}(K)$ is unipotent. Let $\Loc_{{}^LG,F}^{\widehat\unip}$ (resp. $\Loc_{{}^LG,F}^{\mathrm{tame}}$) be the substack of $\Loc_{{}^LG,F}$ classifying unipotent (resp. tame) $L$-parameters. The substacks $\Loc_{{}^LG,F}^{\widehat\unip}$ and $\Loc_{{}^LG,F}^{\mathrm{tame}}$ are open and closed in $\Loc_{{}^LG,F}$ and of finite type over $\Lambda$. The stack $\Loc^{\mathrm{tame}}_{{}^LG,F}$ can be written as
$$\Loc^{\mathrm{tame}}_{{}^LG,F}\simeq \{(x,y)\in \hat{G}\times\hat{G}| xyx^{-1}=\phi(y)^q\}/\hat{G}$$
where the morphism sends an $L$-parameter $\varphi$ to $(\varphi(\phi),\varphi(\tau))$. Similarly, we have
$$\Loc^{\widehat\unip}_{{}^LG,F}\simeq \{(x,y)\in \hat{G}\times\cU^\wedge_{\hat{G}}| xyx^{-1}=\phi(y)^q\}/\hat{G}$$
where $\cU^\wedge_{\hat{G}}$ is the formal completion of $\hat{G}$ along the unipotent cone $\cU_{\hat{G}}$ .

Let $\phi=\phi_{\cU^\wedge_{\hat{G}}/\hat{G}}\colon\cU^\wedge_{\hat{G}}/\hat{G}\to \cU^\wedge_{\hat{G}}/\hat{G}$ be the automorphism defined by 
$$\phi_{\cU^\wedge_{\hat{G}}/\hat{G}}(g)=\phi(g)^{1/q}, \quad g\in \cU^\wedge_{\hat{G}}.$$
Note that $g\mapsto g^{q}\colon \cU^\wedge_{\hat{G}}\to \cU^\wedge_{\hat{G}}$ is an isomorphism, thus $\phi_{\cU^\wedge_{\hat{G}}/\hat{G}}$ is well-defined.
Then $\Loc^{\widehat\unip}_{{}^LG,F}$ is isomorphic to the stack $\cL_\phi(\cU^\wedge_{\hat{G}}/\hat{G})$ of $\phi$-fixed points: If $X$ is a stack with an automorphism $\phi\colon X\to X$, let 
$$\cL_{\phi}(X)\coloneqq X\times_{\mathrm{id}\times\phi,X\times X,\Delta}X$$
denote the stack of $\phi$-fixed points. 

\subsubsection{Unipotent categorical local Langlands}\label{subsubsection-unipotent-cll}

We recall the unipotent categorical local Langlands established in \cite{Tame}. Fix a pinning $(G,B,T,e)$ of $G$ defined over $O_F$ and a nontrivial additive character $\kappa_F\to\Lambda^\times$.

\begin{thm}[{\cite[Theorem 5.3]{Tame}}]\label{thm-unipotent-cllc}
	If $\Lambda$ is of characteristic $\ell$, then assume that $\ell$ is large relative to $G$. More precisely, assume that $\ell$ is bigger than the Coxeter number of any simple factors of $G$, and $\ell\neq 19$ (resp. $\ell\neq 31$) if $G$ has a simple factor of type $\mathsf{E}_7$ (resp. $\mathsf{E}_8$).
	There is a fully faithful embedding
	$$\LL^{\unip}_G\colon \Shv^{\unip}_\fingen(\Isoc_G,\Lambda)\hookrightarrow \Coh(\Loc_{{}^LG,F}^{\widehat\unip}).$$
	Moreover, if $\Lambda=\overline\QQ_\ell$, then $\LL^{\unip}$ is an equivalence of categories.
\end{thm}

We recall some main ingredients of the proof which will be needed later. Recall the Bezrukavnikov equivalence
$$\BB_G^{\unip}\colon (\Ind\Shv_{\fingen}(\Iw\backslash LG/\Iw,\Lambda),\star) \simeq  (\Ind\Coh(S^\unip_{\hat{G}}),\star) $$
where $S^\unip_{\hat{G}}=\hat{U}/\hat{B}\times_{\hat{G}/\hat{G}}\hat{U}/\hat{B}$ is the Steinberg variety, $\Iw\backslash LG/\Iw$ is the affine Hecke stack, and $\star$ are convolution products on two sides. The equivalence $\BB_G^{\unip}$ is compatible with the (relative and arithmetic) Frobenius actions on two sides. Taking categorical traces defines a commutative diagram
$$\begin{tikzcd}[column sep=huge]
    (\Ind\Shv_{\fingen}(\Iw\backslash LG/\Iw,\Lambda) \ar[r,"\Ch_{LG,\phi}^\unip"]\ar[d,"\BB_G^\unip"swap,"\simeq"] & \Ind\Shv^{\unip}_{\fingen}(\Isoc_G,\Lambda) \ar[d,hook,"\LL_G^{\unip}"] \\
    \Ind\Coh(S^\unip_{\hat{G}}) \ar[r,"\Ch_{{}^L{G},\phi}^\unip"] & \Ind\Coh(\Loc^{\widehat\unip}_{{}^LG,F}).
\end{tikzcd}$$
The functor $\Ch^\unip_{LG,\phi}$ is defined by $\Ch^\unip_{LG,\phi}=\Nt_{*}\circ \delta^{!}$ using the correspondence
$$\Iw\backslash LG/\Iw \xleftarrow{\delta}\Sht_{\cI}^\loc\xrightarrow{\Nt}\Isoc_G.$$
Here, $\delta$ sends a shtuka $\phi^*\cE\dashrightarrow\cE$ to the inverse modification $\cE\dashrightarrow \phi^*\cE$.

On the spectral side, define
$$\widetilde\Loc_{{}^LG,F}^{\unip}\coloneqq \Loc_{{}^LG,F}^{\widehat\unip}\times_{\hat{G}/\hat{G}} \hat{U}/\hat{B}$$
where the morphism $\Loc_{{}^LG,F}^{\widehat\unip}\to \hat{G}/\hat{G}$ is defined by evaluating at $\tau$. Denote by $\pi^\unip\colon\widetilde\Loc_{{}^LG,F}^{\unip}\to \Loc_{{}^LG,F}^{\widehat\unip} $ the natural morphism. There is a natural morphism
$$\delta^\unip\colon \widetilde\Loc_{{}^LG,F}^{\unip}\to S^\unip_{\hat{G}}.$$
On the level of classical points, the stack $\widetilde\Loc_{{}^LG,F}^{\unip}$ classifies pairs $(x,y,g\hat{B})$ with $(x,y)\in \Loc_{{}^LG,F}^{\widehat\unip}$ and $y\in g\hat{U}g^{-1}$. Then $\delta^\unip$ sends $(x,y,g\hat{B})$ to $(y,g\hat{B},\phi^{-1}(x^{-1}g)\hat{B})$. The functor $\Ch^\unip_{{}^LG,\phi}$ is defined by
$\Ch^\unip_{{}^LG,\phi}= (\pi^\unip)_*\circ (\delta^\unip)^!.$

For $w\in \widetilde{W}$, let $j_w\colon \Fl_{G,w}\hookrightarrow \Fl_G$ be the affine Schubert cell associated with $w$. Denote $\Sht_{\cI,w}^\loc=\delta^{-1}(\Iw\backslash LG_w/\Iw)$. Note that $\Sht_{\cI,w}^\loc$ classifies local shtukas $\phi^*\cE\dashrightarrow\cE$ bounded by $w^{-1}$.  Let $LG_w$ denote the preimage of $ \Fl_{G,w}$ in $LG$. Let the normalized standard sheaf (resp. normalized costandard sheaf) be
$$\Delta_w\coloneqq (j_w)_! \omega_{\Iw\backslash\Fl_{G,w}}(-\frac{l(w)}{2})[-l(w)]\quad\text{resp. }\nabla_w\coloneqq (j_w)_* \omega_{\Iw\backslash\Fl_{G,w}}(-\frac{l(w)}{2})[-l(w)].$$
We recall the definition of Wakimoto sheaves. If $\lambda\in \XX_\bullet(T)^+$ is dominant, set $W_\lambda\coloneqq \nabla_{\lambda}$. For general weights, $W_\lambda$ is determined by requiring $W_\lambda\star W_\mu=W_{\lambda+\mu}$ for any $\lambda$ and $\mu$ in $\XX_\bullet(T)$. In particular, if $\lambda\in -\XX_\bullet(T)^+$ is anti-dominant, there is an isomorphism $W_\lambda\simeq \Delta_{\lambda}$. 

On the spectral side, let $\iota\colon \hat{U}/\hat{B}\to S^\unip_{\hat{G}}$ denote the diagonal morphism. We have
$$\BB^\unip_G(W_\lambda)= \iota_* \omega_{\hat{U}/\hat{B}}(\lambda)$$
for $\lambda\in \XX_\bullet(T)$. We note that $\omega_{\hat{U}/\hat{B}}\simeq \cO_{\hat{U}/\hat{B}}[-\dim \hat{T}]$ is trivial up to a shift.

Let $B(G)_{\unr}\coloneqq \mathrm{Im}(B(T)\to B(G))$. Elements in $B(G)_{unr}$ are called \emph{unramified}. By \cite[Lemma 4.2.3]{XZ-cycles}, there is a bijection
$$B(G)_\unr \xrightarrow{\simeq} \XX_\bullet(T)^+_\phi.$$
Here $\XX_\bullet(T)^+_\phi$ is defined as in \S\ref{subsubsection-rel-root-data}. Under this bijection, the Kottwitz map
$$\XX_\bullet(T)^+_\phi\cong B(G)_\unr\subseteq B(G)\hookrightarrow \XX_\bullet(T)^{\phi,+}_\QQ\times \pi_1(G)_\phi$$
sends an element $\mu\in \XX_\bullet(T)^+_\phi$ to $(\mu^\diamond,[\mu])$ where $\mu^\diamond= \frac{1}{m}\sum_{i=0}^{m-1}\phi^i(\mu)$ for some $m$ with $\phi^m(\mu)=\mu$, and $[\mu]$ is the image of $\mu$ under the projection $\XX_\bullet(T)_\phi\to\pi_1(G)_\phi.$ Moreover, there is a commutative diagram
$$\begin{tikzcd}
    B(T) \ar[r,"\simeq"]\ar[d] & \XX_\bullet(T)_\phi \ar[d] \\
    B(G)_\unr \ar[r,"\simeq"] & \XX_\bullet(T)_\phi/W_0\cong \XX_\bullet(T)_\phi^+,
\end{tikzcd}$$
where $W_0=W^\phi$ is the relative Weyl group.

\begin{rmk}\label{rmk-phi-coinv-surj}
	The natural map $\XX_\bullet(T)^+\to \XX_\bullet(T)_\phi^+$ is \emph{not} surjective in general. If the center of $G$ connected, or equivalently the derived subgroup of $\hat{G}$ is simply-connected, then $\XX_\bullet(T)^+\to \XX_\bullet(T)_\phi^+$ is surjective. See \cite[Remark 4.2.5]{XZ-cycles}.
\end{rmk}

For $\lambda\in \XX_\bullet(T)$, let $b_\lambda$ be the image of $\lambda^{-1}(\varpi)\in G(\breve{F})$ in $B(G)_\unr$. Assume that $\lambda$ is dominant or anti-dominant, then $t_\lambda\in \widetilde{W}$ is $\phi$-straight in the sense of \cite{He-Nie-minimal-length}. By \cite[Proposition 3.13]{Tame}, we have
$$\Sht^\loc_{\cI,t_\lambda}\cong \BB I_{b_\lambda}$$
where $I_{b_\lambda}\subseteq G_{b_\lambda}(F)$ is the Iwahori subgroup. There is a commutative diagram
$$\begin{tikzcd}
    \Iw\backslash LG_{t_\lambda} /\Iw\ar[r,leftarrow]\ar[d] & \Sht^\loc_{\cI,t_\lambda}\cong \BB I_{b_\lambda} \ar[d]\arrow[ld, phantom, very near start, "\square"]\ar[r] & \Isoc_{G,b_\lambda}\cong \BB G_{b_\lambda}(F)\ar[d,"i_{b_\lambda}"]\\
    \Iw\backslash LG/\Iw \ar[r,leftarrow,"\delta"] & \Sht^\loc_\cI \ar[r,"\Nt"] & \Isoc_G,
\end{tikzcd}$$
with the left square Cartesian. Let $\lambda\in \XX_\bullet(T)^+$ be a dominant coweight. We see that 
$$\Ch^\unip_{LG,\phi}(W_\lambda)\simeq (i_{b_\lambda})_*(\cInd_{I_{b_\lambda}}^{G_{b_{\lambda}}(F)}\Lambda)[-\langle2\rho,\nu_{b_\lambda}\rangle]$$
and 
$$\Ch^\unip_{LG,\phi}(W_{w_0(\lambda)})\simeq (i_{b_\lambda})_!(\cInd_{I_{b_\lambda}}^{G_{b_{\lambda}}(F)}\Lambda) [-\langle2\rho,\nu_{b_\lambda}\rangle]$$
Here, we use the equation $\langle2\rho,\nu_{b_\lambda}\rangle=\langle2\rho,\lambda^\diamond\rangle=\langle 2\rho,\lambda\rangle$.

On the spectral side, define 
$$\Loc_{{}^LB.F}^\unip\coloneqq \widetilde{\Loc}_{{}^LG,F}^{\widehat\unip}\times_{S^\unip_{\hat{G}}}\hat{U}/\hat{B}.$$
Let 
$$\fq^\unip\colon \Loc_{{}^LB,F}^\unip\to \Loc_{{}^LG,F}^{\widehat{\unip}}$$
denote the natural morphism. There is a natural isomorphism
$\Loc_{{}^LB.F}^\unip\simeq \cL_\phi(\hat{U}/\hat{B}).$ It defines a natural morphism $\Loc_{{}^LB.F}^\unip\to \hat{U}/\hat{B}$. For $\lambda\in \XX_\bullet(T)$, let $\omega_{\Loc_{{}^LB.F}^\unip}(\lambda)$ denote the canonical sheaf on $\Loc_{{}^LB.F}^\unip$ twisted by the character $\lambda$ of $\hat{B}$. Because $\Loc_{{}^LB.F}^\unip$ is a $\phi$-fixed points stack, there is a natural isomorphism $\cO_{\Loc_{{}^LB.F}^\unip}\simeq \omega_{\\Loc_{{}^LB.F}^\unip}$.

\begin{prop}\label{prop-Ch-Wakimoto}
    Let $\lambda\in \XX_\bullet(T)^+$ be a dominant coweight. Let $b_\lambda\in B(G)$ be the associated element. There are isomorphisms
    $$\LL_G^{\unip}((i_{b_\lambda})_*(\cInd_{I_{b_\lambda}}^{G_{b_{\lambda}}(F)}\Lambda)[-\langle2\rho,\nu_{b_\lambda}\rangle])\simeq (\fq^\unip)_*\omega_{\Loc_{{}^LB,F}^\unip}(\lambda)$$
    and
    $$\LL_G^{\unip}((i_{b_\lambda})_!(\cInd_{I_{b_\lambda}}^{G_{b_{\lambda}}(F)}\Lambda)[-\langle2\rho,\nu_{b_\lambda}\rangle])\simeq (\fq^\unip)_*\omega_{\Loc_{{}^LB,F}^\unip}(w_0(\lambda)).$$
\end{prop}
\begin{proof}
    There is a commutative diagram
    $$\begin{tikzcd}
        \hat{U}/\hat{B} \ar[r,leftarrow]\ar[d,"\iota"] & \Loc^\unip_{{}^LB,F} \ar[d]\ar[rd,"\fq^\unip"]\ar[ld,phantom,"\square", very near start] \\
        S_{\hat{G}}^\unip \ar[r,leftarrow,"\delta^\unip"] & \widetilde{\Loc}^\unip_{{}^LG,F} \ar[r,"\pi^\unip"] & \Loc^{\widehat\unip}_{{}^LG,F}
    \end{tikzcd}$$
    with the left square Cartesian. By base change, we have
    $$\Ch^\unip_{{}^LG,\phi}(\iota_*\omega_{\hat{U}/\hat{B}}(\mu))\simeq (\fq^\unip)_*\omega_{\Loc^\unip_{{}^LB,F}}(\mu)$$
    for any $\mu\in \XX^\bullet(T)$.
\end{proof}

By definition, for two objects $\cF_1,\cF_2\in \Ind\Shv_\fingen(\Iw\backslash LG/\Iw,\Lambda)$, there is a canonical isomorphism
$$\can\colon \Ch^\unip_{LG,\phi}(\phi_*\cF_1\star\cF_2)\simeq \Ch^\unip_{LG,\phi}(\cF_2\star \cF_1).$$
In fact, let $\Hk_\bullet(\Sht_\cI^\loc)$ denote the \v{C}ech nerve of $\Sht_\cI^\loc\to \Isoc_G$. Then $\Hk_1(\Sht_\cI^\loc)$ classifies iterated local shtukas $\phi^*\cE_0\stackrel{\beta_2}\dashrightarrow \cE_1\stackrel{\beta_2}\dashrightarrow \cE_0$. There is a commutative diagram
$$\begin{tikzcd}[column sep=huge]
    \Hk_1(\Sht^\loc_\cI) \ar[d,"\delta_1"]\ar[r,"\pFr"] & \Hk_1(\Sht^\loc_\cI) \ar[d,"\delta_1"] \\
    (\Iw\backslash LG/\Iw)^2 \ar[r,"\sw\circ(1\times\phi)"] &  (\Iw\backslash LG/\Iw)^2, 
\end{tikzcd}$$
where $\pFr$ is the partial Frobenius sending $(\phi^*\cE_0\stackrel{\beta_2}\dashrightarrow \cE_1\stackrel{\beta_1}\dashrightarrow \cE_0)$ to $(\phi^*\cE_1\stackrel{\phi(\beta_1)}\dashrightarrow \phi^*\cE_0\stackrel{\beta_2}\dashrightarrow \cE_1)$, and $\delta_1$ sends $(\phi^*\cE_0\stackrel{\beta_2}\dashrightarrow \cE_1\stackrel{\beta_1}\dashrightarrow \cE_0)$ to $(\beta_2^{-1},\beta_1^{-1})$. Hence there is a canonical isomorphism
$$\pFr^!(\delta_1)^!(\phi_*\cF_1\boxtimes\cF_2)\simeq (\delta_1)^!(\cF_2\boxtimes\cF_1).$$
Pushforward to $\Isoc_G$ defines the canonical isomorphism.

Recall that connected components of $\Iw\backslash LG/\Iw$ is bijective to $\pi_1(G)=\XX_\bullet(T)/\ZZ\hat\Delta$. Paring with the half sum of positive roots $\rho$ defines a function $\langle\rho,-\rangle\colon \pi_1(G)\to \frac12\ZZ/\ZZ$. We say that $\alpha\in \pi_1(G)$ is even if $\langle\rho,\alpha\rangle=0$, and $\alpha\in \pi_1(G)$ is odd if $\langle\rho,\alpha\rangle=\frac12$. For $\cF\in\Ind\Shv_\fingen(\Iw\backslash LG/\Iw,\Lambda)$, denote $d(\cF)=1$ if $\cF$ is supported on the odd components of $\Iw\backslash LG/\Iw$, and $d(\cF)=0$ if $\cF$ is supported on the even components of $\Iw\backslash LG/\Iw$. The commutative constraint of the functor $\Ch^\unip_{LG,\phi}$ is given by
$$\sigma_{\cF_1,\cF_2}=(-1)^{d(\cF_1)d(\cF_2)}\cdot\can\colon \Ch^\unip_{LG,\phi}(\phi_*\cF_1\star\cF_2)\simeq \Ch^\unip_{LG,\phi}(\cF_2\star \cF_1)$$
if $\cF_1,\cF_2$ are supported only on the even components or odd components. In general we define $\sigma_{\cF_1,\cF_2}$ be linear combination. The factor $(-1)^{d(\cF_1)d(\cF_2)}$ is to make it compatible with the commutative constraint of the central functor, cf. \cite[\S 3.5.1]{AR-central}.

One can also consider the unipotent monodromic affine Hecke category $\Shv_{\hat{u}\mbox{-}\mon}(\Iw^u\backslash LG/\Iw^u,\Lambda)$ as in \cite[\S 4.2.3]{Tame}, where $\Iw^u=\ker(\Iw\to T)$. There are unipotent monodromic standard (resp. costandard) objects $\nabla^\mon_{\dot{w},\hat{u}}$ (resp. $\Delta^\mon_{\dot{w},\hat{u}}$) and unipotent monodromic Wakimoto sheaves $W^\mon_{\lambda,\hat{u}}$ in $\Shv_{\hat{u}\mbox{-}\mon}(\Iw^u\backslash LG/\Iw^n,\Lambda)$ for $w\in\widetilde{W}$ and $\lambda\in \XX_\bullet(T)$. There is a unipotent monodromic affine Deligne--Lusztig functor
$$\Ch^{\hat{u}\mbox{-}\mon}_{LG,\phi}\colon \Shv_{\hat{u}\mbox{-}\mon}(\Iw^u\backslash LG/\Iw^u,\Lambda)\to \Shv(\Isoc_G,\Lambda)$$
defined by pull-push along
$$\Iw^u\backslash LG/\Iw^u\xleftarrow{\delta^u} \widetilde{\Sht}_\cI^\loc\xrightarrow{\Nt^u} \Isoc_G,$$
where $\widetilde{\Sht}_\cI^\loc=LG/\Ad_\phi\Iw^u$.
By \cite[Corollary 4.67]{Tame}, if $w\in\widetilde{{W}}$ is $\phi$-straight with image $b\in B(G)$, then
$$\Ch^{\hat{u}\mbox{-}\mon}_{LG,\phi}(\nabla^\mon_{w,\hat{u}})\simeq (i_b)_*\cInd_{\widetilde{I}_b}^{G_b(F)}\Lambda\quad\text{and}\quad\Ch^{\hat{u}\mbox{-}\mon}_{LG,\phi}(\Delta^\mon_{w,\hat{u}})\simeq (i_b)_!\cInd_{\widetilde{I}_b}^{G_b(F)}\Lambda,$$
where $\widetilde{I}_b=I_b$ if $\Lambda$ has characteristic 0, and $\widetilde{I}_b=I_b^{(\ell)}$ is the maximal prime-to-$\ell$ subgroup of $I$ if $\Lambda$ has characteristic $\ell$. There is a unipotent monodromic version of Bezrukavnikov's equivalence
$$\BB^{\hat{u}\mbox{-}\mon}_G\colon \Shv_{\hat{u}\mbox{-}\mon}(\Iw^u\backslash LG/\Iw^u,\Lambda)\simeq \Ind\Coh(S^{\widehat\unip}_{\hat{G}}),$$
where $S^{\widehat\unip}_{\hat{G}}=\hat{U}^\wedge/\hat{B}\times_{\cU_{\hat{G}}^\wedge/\hat{G}}\hat{U}^\wedge/\hat{B}$ is the formal completion of $\hat{B}/\hat{B}\times_{\hat{G}/\hat{G}}\hat{B}/\hat{B}$ along the unipotent cone. There is a corresponding spectral Deligne--Lusztig induction functor
$$\Ch^{\hat{u}\mbox{-}\mon}_{{}^LG,\phi}\colon \Ind\Coh(S^{\widehat\unip}_{\hat{G}})\to \Ind\Coh(\Loc^{\widehat\unip}_{{}^LG,F})$$
defined by pull-push along
$$S^{\widehat\unip}_{\hat{G}}\xleftarrow{\delta^{\widehat\unip}} \widetilde\Loc^{\widehat\unip}_{{}^LG,F}\xrightarrow{\pi^{\widehat\unip}}\Loc^{\widehat\unip}_{{}^LG,F}$$
for $\widetilde\Loc^{\widehat\unip}_{{}^LG,F}\coloneqq\Loc^{\widehat\unip}_{{}^LG,F}\times_{\hat{G}/\hat{G}}\hat{U}^\wedge/\hat{B}$. Then there is a commutative diagram
$$\begin{tikzcd}
    \Shv_{\hat{u}\mbox{-}\mon}(\Iw^u\backslash LG/\Iw^u,\Lambda) \ar[r,"\Ch_{LG,\phi}^{\hat{u}\mbox{-}\mon}"]\ar[d,"\BB_G^{\hat{u}\mbox{-}\mon}"swap,"\simeq"] & \Shv^{\widehat\unip}(\Isoc_G,\Lambda) \ar[d,hook,"\LL_G^{\unip}\circ\Psi^{L,\unip}"] \\
    \Ind\Coh(S^{\widehat\unip}_{\hat{G}}) \ar[r,"\Ch_{{}^L{G},\phi}^{\hat{u}\mbox{-}\mon}"] & \Ind\Coh(\Loc^{\widehat\unip}_{{}^LG,F}).
\end{tikzcd}$$

\subsubsection{Spectral action}
We define spectral action on the unipotent part of the local Langlands category. Let $\Perf(\Loc^{\widehat\unip}_{{}^LG,F})$ be the category of perfect complexes over $\Loc^{\widehat\unip}_{{}^LG,F}$. There is a symmetric monoidal action of $\Perf(\Loc^{\widehat\unip}_{{}^LG,F})$ on $\Ind\Coh(\Loc^{\widehat\unip}_{{}^LG,F})$.

\begin{prop}\label{prop-spectral-action}
    There is a (unique) action $\star$ of $\Perf(\Loc^{\widehat\unip}_{{}^LG,F})$ on $\Ind\Shv_{\fingen}^{\unip}(\Isoc_G,\Lambda)$ (resp. $\Shv^{\widehat\unip}(\Isoc_G,\Lambda)$) making the functor $\LL^{\unip}_G$ (resp. $\LL^{\unip}_G\circ \Psi^{L,\unip}$) linear over $\Perf(\Loc^{\widehat\unip}_{{}^LG,F})$. Moreover, the action of $\Perf(\Loc^{\widehat\unip}_{{}^LG,F})$ on $\Shv^{\widehat\unip}(\Isoc_G,\Lambda)$ preserves compact objects and admissible objects.
\end{prop}
\begin{proof}
    The first statement was proved in \cite{Tame}. Here we give a different proof.
    Pullback along $\Loc^{\widehat\unip}_{{}^LG,F}\to \BB\hat{G}$ defines a symmetric monoidal functor $\Perf(\BB \hat{G})\to \Perf(\Loc^{\widehat\unip}_{{}^LG,F})$.
    By \cite[Theorem VIII.5.2]{FS}, the category $\Perf(\Loc^{\widehat\unip}_{{}^LG,F})$ is generated under cones and retracts by the image of $\Perf(\BB \hat{G})$. It suffices to show that for $V\in \Perf(\BB \hat{G})^\heartsuit$, the action of $V$ on $\Ind\Coh(\Loc^{\widehat\unip}_{{}^LG,F})$ preserves the essential image of $\LL^{\unip}_G$ (resp. $\LL^{\unip}_G\circ\Psi^{L,\widehat\unip}$), and preserves compact objects, finitely generated objects and admissible objects in $\Shv^{\widehat\unip}(\Isoc_G,\Lambda)$. By \cite[Proposition 4.95]{Tame}, the category $\Ind\Shv_{\fingen}^{\widehat\unip}(\Isoc_G,\Lambda)$ is compactly generated by objects $\Ch^{\unip}_{{}^LG,\phi}(\cF)$ for $\cF\in \Shv_{\fingen}(\Iw\backslash LG/\Iw,\Lambda)$. For $\cF\in\Shv_{\fingen}(\Iw\backslash LG/\Iw,\Lambda)$, we have
    $$V\star \Ch^{\unip}_{{}^LG,\phi}(\cF)\simeq \Ch^{\unip}_{{}^LG,\phi}(Z_V\star \cF),$$
    where $Z_V\in \Shv_{\fingen}(\Iw\backslash LG/\Iw,\Lambda)$ is the central sheaf associated to $V$. It follows that $V\star(-)$ preserves the essential image of $\Ind\Shv_{\fingen}^{\widehat\unip}(\Isoc_G,\Lambda)$ and preserves finitely generated objects. The action of $V\star(-)$ preserves $\Shv(\Isoc_G,\Lambda)$ and compact objects in it can be proved similarly using the unipotent monodromic affine Hecke category $\Shv^{\hat{u}\mbox{-}\mathrm{mon}}(\Iw^u\backslash LG/\Iw^u,\Lambda)$ and unipotent monodromic central sheaves. Let $V^\vee\in \Perf(\BB \hat{G})$ be the dual representation of $V$. The actions $V\star()$ and $V^\vee\star(-)$ on $\Shv^{\widehat\unip}(\Isoc_G,\Lambda)$ are left and right adjoint to each other. Thus $V\star(-)$ also preserves admissible objects in $\Shv^{\widehat\unip}(\Isoc_G,\Lambda)$.
\end{proof}
\begin{rmk}
    The spectral action on the tame part of the local Langlands category $\Shv(\Isoc_G,\Lambda)$ can also be defined via taking categorical trace of the central functor constructed in \cite{ALWY-mixed-central}. We will not explore this approach in this article as it is not needed.
\end{rmk}

\subsubsection{Compatibility with isogenies}

Let $G\to G'$ be a central isogeny between reductive groups over $F$. Assume  the kernel of $G\to G'$ has order invertible in $\Lambda$. We add a superscript $(-)'$ on the stacks and morphisms associated to $G'$. There are natural injections $\XX_\bullet(T)\hookrightarrow \XX_\bullet(T')$ and $\widetilde{W}\hookrightarrow\widetilde{W}'$.  There is a commutative diagram
$$\begin{tikzcd}
    \Iw\backslash LG/\Iw  \ar[d,"\kappa"]  & \Sht^\loc_{\cI} \ar[l,"\delta"swap]\ar[r,"\Nt"]\ar[d,""] & \Isoc_G \ar[d,"\kappa_{\Isoc}"] \\
    \Iw'\backslash LG'/\Iw' & \Sht^\loc_{\cI'} \ar[l,"\delta'"swap]\ar[r,"\Nt'"] & \Isoc_{G'}.
\end{tikzcd}$$
Note that $\kappa^!\Delta_w\simeq \Delta_w$ and $\kappa^!\nabla_w\simeq \nabla_w$ for $w\in \widetilde{W}$. Similarly, $\kappa^!W_\lambda=W_\lambda$ for $\lambda\in \XX_\bullet(T)$.

Similarly we have a commutative diagram
$$\begin{tikzcd}
    \Iw^u\backslash LG/\Iw^u \ar[d,"\kappa^u"]  & \widetilde\Sht^\loc_{\cI} \ar[l,"\delta^u"swap]\ar[r,"\Nt^u"]\ar[d] & \Isoc_G \ar[d,"\kappa_{\Isoc}"] \\
    \Iw^{u\prime}\backslash LG'/\Iw^{u\prime} & \widetilde\Sht^\loc_{\cI'} \ar[l,"\delta^{u\prime}"swap]\ar[r,"\Nt^{u\prime}"] & \Isoc_{G'}.
\end{tikzcd}$$
As the kernel of $G\to G'$ has order invertible in $\Lambda$, pullback along $T\to T'$ defines an equivalence $\Shv_{\hat{u}\mbox{-}\mon}(T')\xrightarrow{\sim}\Shv_{\hat{u}\mbox{-}\mon}(T)$. We see that $(\kappa^u)^!\Delta_{\dot{w},\hat{u}}^\mon\simeq \Delta_{\dot{w},\hat{u}}^\mon$ and $(\kappa^u)^!\nabla_{\dot{w},\hat{u}}^\mon\simeq \nabla_{\dot{w},\hat{u}}^\mon$ for $w\in \widetilde{W}$ and $(\kappa^u)^!W^\mon_{\lambda,\hat{u}}\simeq W^\mon_{\lambda,\hat{u}}$ for $\lambda\in \XX_\bullet(T)$.

Recall that $\pi_0(\Iw\backslash LG/\Iw)=\pi_1(G)$ and $\pi_0(\Iw'\backslash LG'/\Iw')=\pi_1(G')$. There is a open and closed embedding $\Fl_\cI\hookrightarrow\Fl_{\cI'}$ corresponding to the union of connected components indexed by $\pi_1(G)\subseteq \pi_1(G')$. 

\begin{lemma}\label{lemma-llc-isogeny}
    Let $\cF\in \Ind\Shv_\fingen(\Iw'\backslash LG'/\Iw',\Lambda)$ (resp. $\cF\in \Shv_{\hat{u}\mbox{-}\mon}(\Iw'\backslash LG'/\Iw',\Lambda)$) be an object. Then $\Ch^\unip_{LG,\phi}(\kappa^!\cF)$ (resp. $\Ch^{\hat{u}\mbox{-}\mon}_{LG,\phi}(\kappa^!\cF)$) is a direct summand of $(\kappa_{\Isoc})^!\Ch^\unip_{LG',\phi}(\cF)$ (resp. $(\kappa_{\Isoc})^!\Ch^{\hat{u}\mbox{-}\mon}_{LG',\phi}(\cF)$).
\end{lemma}
\begin{proof}
    Let $\cF\in\Ind\Shv_\fingen(\Iw'\backslash LG/\Iw')$. 
    Consider the commutative diagram
    $$\begin{tikzcd}
        \Iw\backslash LG/\Iw \ar[d,"\kappa"] & \Sht^\loc_{\cI} \ar[l,"\delta"swap]\ar[d,hook,"\kappa_1"]\ar[rd,"\Nt"]  \\
        \Iw'\backslash LG'/\Iw' \ar[d,equal] & X \ar[l,"\delta_0"swap]\ar[r,"\Nt_0"]\ar[d,"\kappa_2"]\ar[rd,phantom,"\square",very near start] & \Isoc_G \ar[d,"\kappa_\Isoc"] \\
        \Iw'\backslash LG'/\Iw' & \Sht^\loc_{\cI'} \ar[l,"\delta'"swap]\ar[r,"\Nt'"] & \Isoc_{G'},
    \end{tikzcd}$$
    where the lower right square is Cartesian. We can write
    $$X=LG\times^{LG} \Fl_{\cI'}\quad\text{and}\quad \Sht^\loc_\cI=LG\times^{LG}\Fl_\cI$$
    where $LG$ acts on $LG$ by $\phi$-twisted conjugation. Then natural morphism $\Fl_\cI\hookrightarrow \Fl_{\cI'}$ is a $LG$-equivariant open and closed embedding. It follows that $\kappa_1$ is an open and closed embedding. By base change, we have
    $$(\kappa_\Isoc)^!\Ch^\unip_{LG',\phi}(\cF)=(\kappa_{\Isoc})^!(\Nt')_*(\delta')^!\cF\simeq (\Nt_0)_*(\delta_0)^!\cF.$$
    The object $(\kappa_1)_*(\kappa_1)^!(\delta_0)^!\cF$ is a direct summand of $(\delta_0)^!\cF$. Therefore 
    $$\Ch^\unip_{LG,\phi}(\kappa^!\cF)=\Nt_*\delta^!\kappa^!\cF\simeq(\Nt_0)_*(\kappa_1)_*(\kappa_1)^!(\delta_0)^!\cF$$
    is a direct summand of $\Ch^\unip_{LG',\phi}(\cF)$.

    For unipotent monodromic sheaves, there is a similar diagram 
    $$\begin{tikzcd}
        \Iw^u\backslash LG/\Iw^u \ar[d,"\kappa^u"] & \widetilde\Sht^\loc_{\cI} \ar[l,"\delta^u"swap]\ar[d,"\kappa_1^u"]\ar[rd,"\Nt^u"]  \\
        \Iw^{u\prime}\backslash LG'/\Iw^{u\prime} \ar[d,equal] & \widetilde{X} \ar[l,"\delta_0^u"swap]\ar[r,"\Nt_0^u"]\ar[d,"\kappa_2^u"]\ar[rd,phantom,"\square",very near start] & \Isoc_G \ar[d,"\kappa_\Isoc"] \\
        \Iw^{u\prime}\backslash LG'/\Iw^{u\prime} & \widetilde\Sht^\loc_{\cI'} \ar[l,"\delta^{u\prime}"swap]\ar[r,"\Nt^{u\prime}"] & \Isoc_{G'},
    \end{tikzcd}$$
    The morphism $\kappa_1^u$ can be written as $LG\times^{LG}\widetilde\Fl_\cI\to LG\times^{LG}\widetilde\Fl_{\cI'}$ where $\widetilde\Fl_\cI=LG/\Iw^u$. Then $\widetilde\Fl_{\cI}\to\Fl_{\cI}$ (resp. $\widetilde\Fl_{\cI'}\to\Fl_{\cI'}$) is a $T$-torsor (resp. $T'$-torsor). As the kernel of $G\to G'$ has order invertible in $\Lambda$, pushforward along $\widetilde\Fl_\cI\to \widetilde\Fl_{\cI'}$ defines an equivalence of categories $\Shv_{\hat{u}\mbox{-}\mon}(\widetilde\Fl_{\cI},\Lambda)\simeq\Shv_{\hat{u}\mbox{-}\mon}(\widetilde\Fl_{\cI',\pi_1(G)},\Lambda)$ where $\widetilde\Fl_{\cI',\pi_1(G)}$ is the union of connected components indexed by $\pi_1(G)$. Thus $\Ch^{\hat{u}\mbox{-}\mon}_{LG,\phi}(\kappa^!\cF)$ is a direct summand of $(\kappa_{\Isoc})^!\Ch^{\hat{u}\mbox{-}\mon}_{LG',\phi}(\cF)$ by a similar argument.
\end{proof}

\subsection{Generic part of local Langlands categories}

Denote $Z_{{}^LG,F}^{\widehat\unip}=H^0(\Loc_{{}^LG,F}^{\widehat\unip},\cO)$. %By \cite{DHKM-L-par}, 
Recall that $\Lambda$-points of 
$$C_{{}^LG,F}^{\widehat\unip}\coloneqq \Spec Z_{{}^LG,F}^{\widehat\unip}$$ 
are in bijection with semisimple unramified $L$-parameters of ${}^LG$ over $\Lambda$. 

\subsubsection{Generic unramified $L$-parameters}

Let $\varphi\colon W_F\to {}^LG(\Lambda)$ be a semisimple unramified $L$-parameter. After conjugation, we may assume that $\varphi$ factors through ${}^LT\subseteq {}^LG$. Recall that in Definition \ref{def-generic-unramified}, we say that $\varphi$ is generic (resp. strongly generic) if $H^2(W_F,(\hat{\fg}/\hat{\ft})_\varphi)=0$ (resp. $R\Gamma(W_F,(\hat{\fg}/\hat{\ft})_\varphi)=0$).

%\Xinwen{In fact, since this definition already appears in the introduction, there seems no need to repeat it here again. Maybe even move the definition of the regular ss parameter later to the introduction.}

Recall that for a $\phi$-orbit $\cO\subseteq \Phi(\hat{G},\hat{T})$, we denote
$$\hat\alpha_\cO=\sum_{\hat{\gamma}\in \cO}\hat\gamma \in \XX^\bullet(\hat{T})^\phi.$$

\begin{lemma}\label{lemma-generic-root}
    Let $\varphi\colon W_F\to {}^LG(\Lambda)$ be a semisimple unramified $L$-parameter factors through ${}^LT(\Lambda)$. Denote $\varphi(\phi)=x\phi\in {}^LT(\Lambda)=\hat{T}(\Lambda)\rtimes \langle\phi\rangle$.
    \begin{enumerate}
        \item $\varphi$ is generic if and only if for any $\phi$-orbit $\cO\subseteq \Phi(\hat{G},\hat{T})$, we have
        $$\hat\alpha_\cO(x)\neq \left\{\begin{aligned}
            &q^{|\cO|},\quad &&\text{if $\cO$ has type $\mathsf{A}$ or  $\mathsf{BC^-}$}, \\
            &-q^{|\cO|},\quad &&\text{if $\cO$ has type $\mathsf{BC^+}$}.
        \end{aligned}\right.$$
        \item $\varphi$ is strongly generic if and only if for any $\phi$-orbit $\cO\subseteq \Phi(\hat{G},\hat{T})$, we have $$\hat\alpha_\cO(x)\neq \left\{\begin{aligned}
            &1,q^{|\cO|},\quad &&\text{if $\cO$ has type $\mathsf{A}$ or  $\mathsf{BC^-}$}, \\
            &-1,-q^{|\cO|},\quad &&\text{if $\cO$ has type $\mathsf{BC^+}$}.
        \end{aligned}\right.$$
    \end{enumerate}
\end{lemma}
\begin{rmk}
    If $G$ is split over $F$, then a semisimple unramified $L$-parameter $\varphi\colon W_F\to \hat{G}(\Lambda)$ with $\varphi(\phi)=x\in \hat{T}$ is generic (resp. strongly generic) if and only $\hat{\alpha}(x)\neq q$ (resp. $\hat{\alpha}(x)\neq 1,q$) for all $\hat\alpha\in \Phi(\hat{G},\hat{T})$.
\end{rmk}
%\Xinwen{Maybe make the split case explicitly.}
\begin{proof}
   There is a direct sum decomposition
   $$(\hat{\fg}/\hat{\ft})_\varphi=\bigoplus_{\cO\subseteq \Phi(\hat{G},\hat{T})} \big(\bigoplus_{\hat\gamma\in \cO} \hat{\fg}_{\hat\gamma}\big)$$
   as $W_F$-representations. The action of $\phi\in W_F$ on $\big(\bigoplus_{\hat\gamma\in \cO} \hat{\fg}_{\hat\gamma}\big)$ is given by
   $$\phi((z_{\hat\gamma})_{\hat{\gamma}\in \cO})= (\hat\gamma(x)\phi(z_{\phi^{-1}(\hat{\gamma})}))_{\hat\gamma\in \cO}.$$
   By \cite[Lemma 5.1.8]{XZ-vector}, we know that for a $\phi$-orbit $\cO$ and $\hat\gamma\in \cO$, we have $\phi^{|\cO|}(z_{\hat\gamma})=z_{\hat\gamma}$ if $\cO$ has type $\mathsf{A}$ or  $\mathsf{BC^-}$, and $\phi^{|\cO|}(z_{\hat\gamma})=-z_{\hat\gamma}$ if $\cO$ has type $\mathsf{BC^+}$. Thus $H^0(W_F,(\hat\fg/\hat\ft)_\varphi)=0$ if and only if $\hat\alpha_\cO(x)\neq 1$ for $\cO$ has type $\mathsf{A}$ or  $\mathsf{BC^-}$ and $\hat\alpha_\cO(x)\neq -1$ for $\cO$ has type $\mathsf{BC^+}$. By local Tate duality, $H^2(W_F,(\hat\fg/\hat\ft)_\varphi)=0$ if and only if $\hat\alpha_\cO(x)\neq q^{|\cO|}$ for $\cO$ has type $\mathsf{A}$ or  $\mathsf{BC^-}$ and $\hat\alpha_\cO(x)\neq -q^{|\cO|}$ for $\cO$ has type $\mathsf{BC^+}$.
\end{proof}

\subsubsection{Generic locus in the stack of unipotent $L$-parameters}\label{subsubsection-generic-stack-L-par}

We recall that we let $\hat{A}=\hat{T}/(\phi-1)\hat{T}$ denote the $\phi$-coinvariants of $\hat{T}$. Let $\hat{A}^\gen$ (resp. $\hat{A}^\sgen$) denote the generic (resp. strongly generic) locus in $\hat{A}$. By Lemma \ref{lemma-generic-root}, $\hat{A}^\gen$ (resp. $\hat{A}^\sgen$) is the open complement of the Cartier divisor cut out by 
$$\prod_{\cO\text{ type $\mathsf{A}$ or $\mathsf{BC^-}$}}(\hat\alpha_\cO-q^{|\cO|})\prod_{\cO\text{ type $\mathsf{BC^+}$}}(\hat\alpha_\cO+q^{|\cO|})$$
$$\text{resp. }\prod_{\cO\text{ type $\mathsf{A}$ or $\mathsf{BC^-}$}}(\hat\alpha_\cO-q^{|\cO|})(\hat\alpha_\cO-1)\prod_{\cO\text{ type $\mathsf{BC^+}$}}(\hat\alpha_\cO+q^{|\cO|})(\hat\alpha_\cO+1).$$
The open dense subscheme $\hat{A}^\gen$ (resp. $\hat{A}^\sgen$) is stable under $W_0$-action. Let $\hat{A}^\gen\git W_0$ (resp. $\hat{A}^\sgen\git W_0$) denote the GIT quotient.

Recall that there is an isomorphism $\Loc_{{}^LG,F}^{\widehat\unip}\simeq  \cL_\phi(\cU^\wedge_{\hat{G}}/\hat{G})$. We have a $\phi$-equivariant commutative diagram
$$\begin{tikzcd}
    \BB \hat{B} \ar[r]\ar[d] & \hat{U}/\hat{B} \ar[d] \\
    \BB \hat{G} \ar[r] & \cU^\wedge_{\hat{G}}/\hat{G},
\end{tikzcd}$$
where the horizontal morphisms are defined by the identity in $\hat{U}$ (resp. $\cU^\wedge_{\hat{G}}$). Taking $\phi$-fixed points, we obtain a commutative diagram
$$\begin{tikzcd}
    \hat{B}\phi/\hat{B} \ar[r,"\unr_B"]\ar[d,"\pi_\phi"swap] & \Loc^\unip_{{}^LB,F} \ar[d,"\fq^\unip"] \\
    \hat{G}\phi/\hat{G} \ar[r,"\unr_G"] & \Loc^{\widehat\unip}_{{}^LG,F}.
\end{tikzcd}$$
Here, $\pi_\phi$ is the twisted Grothendieck--Springer resolution defined in \S\ref{subsection-twisted-GS}. The morphism $\unr_G$ induces a morphism
$$\unr_G^\coarse\colon \hat{A}\git W_0\cong \hat{G}\phi\git\hat{G}\to C^{\widehat\unip}_{{}^LG,F}$$
on the GIT quotients. On the other hand, evaluating a parameter at $\phi$ defines a morphism $\ev_\phi\colon \Loc^{\widehat\unip}_{{}^LG,F}\to \hat{G}\phi/\hat{G}$. Let  $$\ev_\phi^\coarse\colon C^{\widehat\unip}_{{}^LG,F}\to \hat{A}\git W_0$$ 
denote induced map on coarse quotients. Then $\ev_\phi^\coarse$ is a right inverse of $\unr_G^\coarse$. It follows that the map $\unr_G^\coarse$ identifies $\hat{A}\git W_0$ as the underlying reduced subscheme of $C^{\widehat\unip}_{{}^LG,F}$. In particular, the set of $\Lambda$-points of them are identified. 

Now, let $\xi$ be such a $\Lambda$-point of $\hat{A}\git W_0$, or equivalently a $\Lambda$-point of $C^{\widehat\unip}_{{}^LG,F}$. 

\begin{defn}
    Let
    $V^\wedge_\xi\coloneqq \Loc^{\widehat\unip}_{{}^LG,F}\times_{C^{\widehat\unip}_{{}^LG,F}}(C^{\widehat\unip}_{{}^LG,F})^\wedge_\xi$
    denote the formal completion of $\Loc^{\widehat\unip}_{{}^LG,F}$ along the preimage of $\xi$. Let
    $$i_\xi^\wedge\colon V^\wedge_\xi\hookrightarrow \Loc^{\widehat\unip}_{{}^LG,F}$$
    denote the embedding.

    Let $W^\wedge_\xi\coloneqq \Loc^{\unip}_{{}^LB,F}\times_{C^{\widehat\unip}_{{}^LG,F}}(C^{\widehat\unip}_{{}^LG,F})^\wedge_\xi$ denote the formal completion of $\Loc^{\unip}_{{}^LB,F}$ along the preimage of $\xi$. Let $\fq^\unip_\xi\colon W^\wedge_\xi\to V^\wedge_\xi$ denote the completion of $\fq^\unip$.
\end{defn}

Let $(\hat{G}\phi/\hat{G})^\wedge_\xi$ and $(\hat{B}\phi/\hat{B})^\wedge_\xi$ denote the completion along preimages of $\xi$ respectively. There is a commutative diagram
$$\begin{tikzcd}
    (\hat{B}\phi/\hat{B})^\wedge_\xi \ar[r,"\unr_{B,\xi}"]\ar[d,"\pi_{\phi,\xi}"swap] & W^\wedge_\xi \ar[d,"\fq^\unip_\xi"] \\
    (\hat{G}\phi/\hat{G})_\xi^\wedge \ar[r,"\unr_{G,\xi}"] & V^\wedge_\xi.
\end{tikzcd}$$

\begin{prop}\label{prop-generic-Loc-equal-GS}
    Let $\xi\in (\hat{A}^\gen\git W_0)(\Lambda)$ be a generic point in $\hat{A}\git W_0$. Then the morphism $$\unr_{B,\xi}\colon(\hat{B}\phi/\hat{B})^\wedge_\xi\to W^\wedge_\xi$$ is an isomorphism, and the morphism $$\unr_{G,\xi}\colon (\hat{G}\phi/\hat{G})^\wedge_\xi\to V^\wedge_\xi$$ induces an isomorphism on the underlying reduced stack. If $\Lambda=\overline{\QQ}_\ell$, then $\unr_{G,\xi}$ is also an isomorphism.
\end{prop}
\begin{proof}
    We first prove the statement for $\unr_{B,\xi}$. The formal stack $W_\xi^\wedge$ is quasi-smooth of virtual dimension 0. Thus it suffices to show that the underlying classical stack $(W^\wedge_{\xi})_\mathrm{cl}$ is isomorphic to $(\hat{B}\phi/\hat{B})^\wedge_\xi$. We have
    $$(\Loc_{{}^LB,F}^\unip)_{\mathrm{cl}}=\{(x,y)\in\hat{B}\times \hat{U}| xyx^{-1}=\phi(y)^q\}/\hat{B}.$$
    For a positive root $\hat\gamma\in \Phi(\hat{G},\hat{T})^+$, let $\hat{U}_{-\hat\gamma}\simeq \GG_a$ denote the root group in $\hat{U}$ (Recall that $\hat\gamma\in \Phi(\hat{G},\hat{T})^+$ is defined using the opposite Borel $B^+$). Write $\hat{\gamma}=\sum_{\hat\alpha\in \hat\Delta}n_{\hat\alpha}\hat\alpha$ as a positive sum of simple roots. Let $\mathrm{ht}(\hat\gamma)=\sum_{\hat\alpha\in\hat\Delta}n_{\hat\alpha}$ denote the height of $\hat\gamma$. For $i\geq 1$, define the normal subgroup 
    $$\hat{U}_i=\prod_{\mathrm{ht}(\hat\gamma)\geq i}\hat{U}_{-\hat\gamma}\subseteq \hat{U}\subseteq \hat{B}.$$
    Then $\hat{U}_i/\hat{U}_{i+1}\cong \prod_{\mathrm{ht}(\hat\gamma)=i}\hat{U}_{-\hat{\gamma}}$ is commutative. Let $R$ be an ordinary $\Lambda$-algebra, and $(x,y)$ be an $R$-point of $\hat{B}\times\hat{U}$ satisfying $xyx^{-1}=\phi(y)^q$. There is a natural morphism $\hat{B}\to\hat{T}\to\hat{A}\git W_0$. Assume that $x$ lies in the completion of $\hat{B}$ along the preimage of $\xi$. We need to show that $y$ is trivial. We prove by induction that $y$ lies in $\hat{U}_{i}$ for any $i$. Assume that $y$ lies in $\hat{U}_{i}$. Let $\bar{y}$ denote the image of $y$ in $\hat{U}_{i}/\hat{U}_{i+1}$. The adjoint action of $\hat{B}$ on $\hat{U}_{i}/\hat{U}_{i+1}$ factors through $\hat{T}$, and $\hat{T}$ acts on the component $\hat{U}_{=\hat\gamma}$ by $-\hat\gamma$. The $\phi$-action on $\hat{U}_{i}/\hat{U}_{i+1}$ is given by permuting the components. Write $\bar{y}=(\bar{y}_{-\hat\gamma})_{\mathrm{ht}(\hat\gamma)=i}$. Then the relation $xyx^{-1}=\phi(y)^q$ implies that
    $$\hat\gamma(x)^{-1}\bar{y}_{-\hat\gamma}= q\phi(\bar{y}_{-\phi^{-1}(\hat{\gamma})})$$
    for any $\hat\gamma$ with $\mathrm{ht}(\hat{\gamma})=i$. By \cite[Lemma 5.1.8]{XZ-vector}, for each $\phi$-orbit $\cO\subseteq \Phi(\hat{G},\hat{T})^{+,\mathrm{ht}=i}$ and $\hat\gamma\in \cO$, we have 
    $$(\prod_{\hat\gamma'\in \cO}\hat\gamma'(x)-q^{-|\cO|})\bar{y}_{-\hat\gamma}=0$$
    if $\cO$ has type $\mathsf{A}$ or $\mathsf{BC}^-$, and 
    $$(\prod_{\hat\gamma'\in \cO}\hat\gamma'(x)+ q^{-|\cO|})\bar{y}_{-\hat\gamma}=0$$
    if $\cO$ has type $\mathsf{BC}^+$.
    By Lemma \ref{lemma-generic-root}, for each $\phi$-orbit $\cO\subseteq \Phi(\hat{G},\hat{T})^{+,\mathrm{ht}=i}$, the element $\prod_{\hat\gamma\in\cO} \hat\gamma(x)- q^{-|\cO|}$ (resp. $\prod_{\hat\gamma\in\cO} \hat\gamma(x)+ q^{-|\cO|}$) is invertible in $R$. It follows that $\bar{y}=0$, and thus $y\in\hat{U}_{i+1}$.

    By the proof of \cite[Proposition 2.42]{Tame}, we know that the morphism $\Loc^{\unip}_{{}^LB,F}\to \Loc^{\widehat\unip}_{{}^LG,F}$ is surjective. It follows that $\unr_{G,\xi}\colon (\hat{G}\phi/\hat{G})^\wedge_\xi\to V^\wedge_\xi$ is surjective on $\Lambda$-points. Also note that $\unr_{G,\xi}$ is a closed embedding. Thus $\unr_{G,\xi}$ induces an isomorphism on the underlying reduced stacks. 

    If $\Lambda=\overline{\QQ}_\ell$, we need to show that $\unr_{G,\xi}$ is an isomorphism. It suffices to show the tangent complexes of two stacks agree. Let $x\in (\hat{G}\phi/\hat{G})^\wedge_\xi$ be a $\Lambda$-point. Let $\varphi_x$ be the unramified $L$-parameter sending $\phi$ to $x\phi\in {}^LG(\Lambda)$. The tangent complex of $\hat{G}\phi/\hat{G}$ at $x$ is given by 
    $$\hat\fg\xrightarrow{1-\Ad_{x\phi}}\hat{\fg}$$
    placed in degrees $[-1,0]$.
    The tangent complex of $\Loc^{\widehat\unip}_{{}^LG,F}$ at $\varphi_x$ is given by $R\Gamma(W_F,\hat\fg_{\varphi_x})[1]$. The differential of $\unr_{G,\xi}$ at $x$ is given by
    $$[\hat\fg\xrightarrow{1-\Ad_{x\phi}}\hat{\fg}]\simeq R\Gamma(\phi^\ZZ,H^0(I_F,\hat\fg_{\varphi_x}))[1]\to R\Gamma(W_F,\hat\fg_{\varphi_x})[1].$$
    The cone of this map is $R\Gamma(\phi^\ZZ,H^1(I_F,\hat{\fg}_{\varphi_x}))$. Because the $I_F$-action on $\hat{\fg}_{\varphi_x}$ is trivial, we have
    $$H^1(I_F,\hat{\fg}_{\varphi_x})\simeq \hat\fg_{\varphi_x}(-1)$$
    as $\phi^\ZZ$-modules. We may assume that $x$ lies in $\hat{B}$. Then $\hat{\fg}_{\varphi_x}$ admits a $W_F$-invariant filtration with graded pieces $(\hat{\fg}/\hat{\fb})_{\varphi_x}$, $\hat{\ft}_{\varphi_x}$ and $\hat{\fu}_{\varphi_x}$. We can further filter $\hat\fu_{\varphi_x}$ and $(\hat\fg/\hat\fb)_{\varphi_x}$ with respect to the height of roots. By similar arguments as above, we see that $R\Gamma(\phi^\ZZ, \hat{\fu}_{\varphi_x}(-1))$ and $R\Gamma(\phi^\ZZ,(\hat{\fg}/\hat\fb)_{\varphi_x}(-1))$ are acyclic. Because $\Lambda$ has characteristic 0 and $\Ad_{x\phi}$ acts on $\hat\ft$ by a finite order automorphism. Hence $R\Gamma(\phi^\ZZ,\hat\ft_{\varphi_x}(-1))$ is acyclic. Together we see that $R\Gamma(\phi^\ZZ,\hat{\fg}_{\varphi_x}(-1))$ is acyclic. Thus $\unr_{G,\xi}$ induces an isomorphism on tangent complexes. 
\end{proof}

\subsubsection{Generic part of local Langlands categories}

Let $A$ be a(n ordinary) commutative $\Lambda$-algebra. 
Let $Z\subseteq \Spec A$ be a closed subset such that the open complement $U:=\Spec A\backslash Z\subseteq \Spec A$ is quasi-compact. Equivalently, $Z$ can be endowed with a scheme structure defined by a finitely generated ideal of $A$.

In this case, it is well-known that there is a localization sequence of $A$-linear categories (in the sense of \cite[Definition 7.26]{Tame})
\begin{equation*}\label{eq: localization sequence for perfect complex}
   \QCoh(\Spec A)_Z\to\QCoh(\Spec A)\to \QCoh(U),
\end{equation*}
where $\QCoh(\Spec A)_Z$ is the category of $A$-modules that are supported on $Z$. In addition, $\QCoh(\Spec A)_Z$ is compactly generated by $\Perf(\Spec A)_Z$ consisting of perfect $A$-modules that restrict to zero on $U$.

%\begin{rmk}
%    We write $Z^\wedge$ for the formal completion of $\Spec A$ along $Z$. (This is independent of the choice of the scheme structure on $Z$.) Then there is a canonical equivalence $\QCoh(Z^\wedge)\cong\QCoh(\Spec A)_Z$. \Xinwen{Need to check this.}\Xiangqian{I guess $\QCoh(Z^\wedge)$ should be the category of derived $I$-complete $A$-modules, so these two are not the same.}
%\end{rmk}

Now let $\bC$ be an $A$-linear (presentable stable) category. We define the categories
$$\bC_{U}\coloneqq\bC\otimes_{\QCoh(\Spec A)}\QCoh(U)\quad\text{and}\quad \bC_Z\coloneqq \bC\otimes_{\QCoh(\Spec A)}\QCoh(\Spec A)_Z.$$
By abuse of terminology, we will say an object in $\bC$ is supported at $Z$ if it belongs to the full subcategory $\bC_Z$.

By tensoring $\bC$ the above localization sequence, we obtain a localization sequence
$$\bC_Z\xrightarrow{i} \bC\xrightarrow{j} \bC_{U}$$
of $A$-linear categories. Note that if $\bC$ is compactly generated by a set of compact generators $\{c_i\}_i$, then $\bC_Z$ is compactly generated by
$\{c_i\boxtimes_A E\}_{i, E\in\Perf(\Spec A)_Z}$. 

%Clearly, if the ideal defining $Z$ is generated by $a_1,\dots, a_n\in A$, then $\bC_Z=\bC_{\{a_1=0\}}\cap\cdots\cap\bC_{\{a_n=0\}}$. 

We let $j^R$ denote the right adjoint of $j$, which is a fully faithful embedding.
\begin{lemma}\label{lem: localization to open subset}
    %Suppose $Z$ can be defined one equation $a=0$ for some $a\in A$. Then 
    %\[
    %j^R(j(x))\cong \colim (x\xrightarrow{a} x\xrightarrow{a}\cdots).
    %\]
    %In addition, $j^R(\bC_U)$ 
    %consist of objects $x\in \bC$ such that $x\xrightarrow{a}x$ is an isomorphism. 
    We have a natural isomorphism of functors $j^R\circ j\cong (-)\otimes_A R\Gamma(U,\cO_U)$.
\end{lemma}
\begin{proof}
    This is clear when $\bC=\QCoh(\Spec A)$. The general case then follows immediately. 
\end{proof}
   
\begin{eg}\label{ex: support conditin for indcoh}
    Suppose $A$ is of finite type over $\Lambda$.
    Let $X$ be an algebraic stack over $\Spec A$ and let $\bC$ be the category of ind-coherent sheaves on $X$. We write $X_Z^\wedge=X\times_{\Spec A}Z^\wedge$ and $X_U=X\times_{\Spec A}U$.
    Then $\bC_Z=\Ind\Coh(X_Z^\wedge)$ and $\bC_{U}=\Ind\Coh(X_U)$.
\end{eg}

%We say that objects in $\bC_Z$ are supported on $Z$.  By \cite[Proposition 3.5.5]{AG-singular-supp}, 
%if $Z=\{a=0\}$ for $a\in A$, then   

We have the following description of objects in $\bC_Z$.

\begin{lemma}\label{lemma-supp}
    Assume that $\bC$ is compactly generated. Let $\{c_s\}_s$ be a set of compact generators. Then an object $x\in \bC$ lies in $\bC_Z$ if and only if $\Hom(c_s,x)\in\QCoh(\Spec A)_Z$ for any $c_s$.
   % \begin{enumerate}
    %    \item A compact object $c\in \bC$ lies in $\bC_Z$ if and only if the $A$-module $\Hom(c,c)$ belongs to $\QCoh(\Spec A)_Z$.
    %    \item An object $x\in \bC$ lies in $\bC_Z$ if and only if $\Hom(c_s,x)\in\QCoh(\Spec A)_Z$ for any $c_s$.
      %  \item The category $\bC_Z$ is compactly generated by $\bC_Z\cap \bC^\omega$. 
   % \end{enumerate}
\end{lemma}
\begin{proof}
We have $x\in \bC_Z$ if and only if $j^R(j(x))=x\otimes_A R\Gamma(U,\cO_U))=0$ if and only if for every $c_s$ 
\[
\Hom(c_s,x\otimes_A R\Gamma(U,\cO_U))=\Hom(c_s,x)\otimes_AR\Gamma(U,\cO_U)=0.
\]\end{proof}

Now suppose $\bC'\subset \bC$ is a full presentable stable subcategory of $\bC$. Then $\bC'$ is automatically equipped with an $A$-linear structure. It follows that we have the similar localization sequence for $\bC'$.
As $\QCoh(\Spec A)$ is rigid, the natural functors $\bC'_Z\to \bC_Z$ and $\bC'_U\to \bC_U$ are fully faithful (e.g. see \cite[Lemma 7.98(3)]{Tame}).
Clearly, we have $\bC'_Z=\bC_Z\cap \bC'$.

More generally, if $\bC_1$ and $\bC_2$ are two $A$-linear categories, with an $A$-linear functor $F\colon \bC_1\to \bC_2$, then $F$ restricts to a functor $(\bC_1)_Z$ to $(\bC_2)_Z$.

\begin{lemma}\label{lemma-supp-functors}
If $F$ admits a left $F^L$ (resp. continuous right $F^R$) adjoint, then $F^L$ (resp. $F^R$) sends $(\bC_2)_Z$ into $(\bC_1)_Z$.
    %Assume that $\bC_1$ and $\bC_2$ are two $A$-linear categories. Let $F\colon \bC_1\to\bC_2$ be an $A$-linear continuous functor. Let $Z\subseteq \Spec A$ be a closed subset.
    %\begin{enumerate}
    %    \item If $F$ is a fully faithful embedding, then $(\bC_1)_Z=(\bC_2)_Z\cap \bC_1$.
    %    \item If $F$ admits a left adjoint $F^L$, then $F^L$ sends $(\bC_2)_Z$ into $(\bC_1)_Z$.
    %    \item If $F$ admits a continuous right adjoint $F^R$, then $F^R$ sends $(\bC_2)_Z$ into $(\bC_1)_Z$.
    %\end{enumerate}
\end{lemma}
\begin{proof}
    %(1) follows from \cite[\S 3.3.10]{AG-singular-supp}. (2) and (3) follow 
    As the symmetric monoidal category $\QCoh(\Spec A)$ is rigid, both $F^L$ and $F^R$ are $A$-linear. The lemma then is clear.
\end{proof}

\begin{lemma}\label{lemma-tensor-supp}
    Let $A_i$ for $i=1,2$ be commutative $\Lambda$-algebras and $\bC_i$ be $A_i$-linear categories. Let $Z_i\subseteq\Spec A_i$ be closed subschemes with quasi-compact open complements. Denote $A=A_1\otimes A_2$, $\bC=\bC_1\otimes_\Lambda\bC_2$, and $Z=Z_1\times Z_2$. Then there is a natural equivalence 
    $$\bC_{Z}\simeq (\bC_1)_{Z_1}\otimes_\Lambda (\bC_2)_{Z_2}.$$
\end{lemma}
\begin{proof}
    It suffices to prove this for $\bC_i=\QCoh(\Spec A_i)$. Clearly there is a full embedding $$(\bC_1)_{Z_1}\otimes_\Lambda(\bC_2)_{Z_2}\hookrightarrow \bC_Z.$$ Assume that the ideal defining $Z_i$ is generated by $f_{i,1},\dots,f_{i,n_i}\in A_i$ for $i=1,2$. It is well-known that $\QCoh(\Spec A_i)_{Z_i}\subseteq\QCoh(\Spec A_i)$ is generated by the Koszul complex $K_i=\mathrm{Kos}(A_i,f_{i,1},\dots,f_{i,n_i})$. Similarly $\QCoh(\Spec A)_Z$ is generated by $K=K_1\otimes K_2$. Therefore the above embedding is essentially surjective.
\end{proof}

Now we specialize the above general discussion to the cases we are interested in.
The category $\Ind\Coh(\Loc^{\widehat\unip}_{{}^LG,F})$ is $Z^{\widehat\unip}_{{}^LG,F}$-linear. Let $\xi\in (\hat{A}\git W_0)(\Lambda)=C^{\widehat\unip}_{{}^LG,F}(\Lambda)$ be a closed point. The fully faithful embedding
$$(i_\xi^\wedge)_*\colon \Ind\Coh(V_\xi^\wedge)\hookrightarrow \Ind\Coh(\Loc^{\widehat\unip}_{{}^LG,F})$$
identifies $\Ind\Coh(V_\xi^\wedge)$ with the subcategory $\Ind\Coh(\Loc^{\widehat\unip}_{{}^LG,F})_\xi\subseteq \Ind\Coh(\Loc^{\widehat\unip}_{{}^LG,F})$, as explained in Example \ref{ex: support conditin for indcoh}.

The category $\Ind\Shv_{\fingen}^{\unip}(\Isoc_G,\Lambda)$ (resp.  $\Shv^{\widehat\unip}(\Isoc_G,\Lambda)$) admits a unique $Z^{\widehat\unip}_{{}^LG,F}$-linear structure making functor $\LL^\unip_G$ (resp. $\LL^\unip_G\circ\Psi^{L,\unip}$) $Z^{\widehat\unip}_{{}^LG,F}$-linear. Let $\Shv^{\widehat\unip}(\Isoc_G,\Lambda)_\xi$ (resp. $\Ind\Shv^{\unip}_\fingen(\Isoc_G,\Lambda)_\xi$) denote the subcategory of objects that are supported at the closed point $\xi\in C^{\widehat\unip}_{{}^LG,F}$. Then $\Psi^{L,\unip}$ restricts to a fully faithful functor
$$\Psi^{L,\unip}_\xi\colon \Shv^{\widehat\unip}(\Isoc_G,\Lambda)_\xi\hookrightarrow  \Ind\Shv_\fingen^{\unip}(\Isoc_G,\Lambda)_\xi$$
and the unipotent local Langlands functor induces a fully faithful functor
$$\LL^{\unip}_{G,\xi}\colon \Ind\Shv^{\unip}_{\fingen}(\Isoc_G,\Lambda)_\xi\to \Ind\Coh(V^\wedge_\xi).$$

Let $R_\xi$ (resp. $R_{\xi,\Ind\fingen}$) denote the right adjoint of the fully faithful embedding 
$$i_\xi\colon \Shv^{\unip}(\Isoc_G,\Lambda)_\xi\hookrightarrow \Shv^{\unip}(\Isoc_G,\Lambda) \quad\text{resp.}\quad i_{\xi,\Ind\fingen}\colon \Ind\Shv_\fingen^{\unip}(\Isoc_G,\Lambda)_\xi\hookrightarrow \Ind\Shv_\fingen^{\unip}(\Isoc_G,\Lambda).$$ For an object $\cF$ in $\Shv^{\widehat\unip}(\Isoc_G,\Lambda)$ (resp. $\Ind\Shv_\fingen^{\unip}(\Isoc_G,\Lambda)$), denote
$$\cF_\xi\coloneqq i_\xi\circ R_\xi(\cF)\quad \textrm{resp.}\quad i_{\xi,\Ind\fingen}\circ R_{\xi,\Ind\fingen}(\cF).$$
There is a natural morphism $\cF_\xi\to \cF$ by adjunction.

\begin{lemma}
    The subcategory $\Shv^{\widehat\unip}(\Isoc_G,\Lambda)_\xi$ (resp. $\Ind\Shv_\fingen^{\unip}(\Isoc_G,\Lambda)_\xi$) is closed under the spectral action defined in Proposition \ref{prop-spectral-action}. Moreover, the functor $\cF\mapsto \cF_\xi$ commutes with the spectral action.
\end{lemma}
\begin{proof}
    The first statement is clear as $\Ind\Coh(V^\wedge_\xi)\subseteq \Ind\Coh(\Loc^{\widehat\unip}_{{}^LG,F})$ is closed under the action of $\Perf(\Loc^{\widehat\unip}_{{}^LG,F})$. After applying the local Langlands functor, the functor $\cF\mapsto \cF_\xi$ is given by 
    $$(i_\xi^\wedge)_*(i_\xi^\wedge)^!\colon \Ind\Coh(\Loc^{\widehat\unip}_{{}^LG,F})\to \Ind\Coh(\Loc^{\widehat\unip}_{{}^LG,F}),$$
    which commutes with the action of $\Perf(\Loc^{\widehat\unip}_{{}^LG,F})$. This proves the second statement.
\end{proof}

Let $\cF\in \Shv^{\widehat\unip}(\Isoc_G,\Lambda)^\Adm$ be an admissible object. By adjunction, there is a morphism
$$\bigoplus_{\xi}\cF_\xi\to \cF$$
where $\xi$ runs through $\Lambda$ points of $C^{\widehat\unip}_{{}^LG,F}$. Denote $\Shv^{\widehat\unip}(\Isoc_G,\Lambda)^\Adm_{\xi}=\Shv^{\widehat\unip}(\Isoc_G,\Lambda)^\Adm\cap \Shv^{\widehat\unip}(\Isoc_G,\Lambda)_\xi$.

\begin{prop}\label{prop-adm-supp}
    Let $\cF\in \Shv^{\widehat\unip}(\Isoc_G,\Lambda)^\Adm$ be an admissible object. Then natural morphisms
    $$\prod_\xi\cF_\xi\leftarrow\bigoplus_\xi\cF_\xi\to \cF$$
    are isomorphisms.
\end{prop}
\begin{proof}
    Let $C\in \Shv(\Isoc_G,\Lambda)$ be a compact object. Then $\Hom(C,\cF)$ is a perfect $\Lambda$-module. Thus there is a direct sum decomposition
    $$\Hom(C,\cF)=\bigoplus_\xi \Hom(C,\cF)_\xi$$
    by the support of $Z^{\widehat\unip}_{{}^LG,F}$-action. We claim that there is a natural isomorphism $\Hom(C,\cF_\xi)\xrightarrow{\sim}\Hom(C,\cF)_\xi$. If fact, by Lemma \ref{lem: localization to open subset} there is an isomorphism
    $$\cF_\xi \simeq \mathrm{fiber}\big(\cF\to \cF\otimes_{Z^{\widehat\unip}_{{}^LG,F}}R\Gamma(C_{{}^LG,F}^{\widehat\unip}\backslash\{\xi\},\cO)\big).$$
    Therefore 
    $$\Hom(C,\cF_\xi)\simeq\mathrm{fiber}\big(\Hom(C,\cF)\to \Hom(C,\cF)\otimes_{Z^{\widehat\unip}_{{}^LG,F}}R\Gamma(C_{{}^LG,F}^{\widehat\unip}\backslash\{\xi\},\cO))\big)\simeq \Hom(C,\cF)_\xi.$$
    Hence $\Hom(C,\cF_\xi)=0$ for all but finitely many $\xi$. It follows that  $$\prod_\xi\Hom(C,\cF_\xi)\xleftarrow{\simeq}\bigoplus_\xi\Hom(C,\cF_\xi)\xrightarrow{\simeq}\Hom(C,\cF)$$
    are isomorphisms.
\end{proof}

Let $b\in B(G)$ be an element. The fully faithful embedding $$(i_b)_!\colon\Ind\Rep^\unip_\fingen(G_b(F),\Lambda)\hookrightarrow \Ind\Shv^{\unip}_\fingen(\Isoc_G,\Lambda)$$ 
endows $\Ind\Rep^\unip_\fingen(G_b(F),\Lambda)$ with a  $Z^{\widehat\unip}_{{}^LG,F}$-linear structure. If $\xi$ is a $\Lambda$-point of $\hat{A}\git W_0$, let 
$$\Ind\Rep_\fingen^\unip(G_b(F),\Lambda)_\xi\subseteq \Ind\Rep^\unip_\fingen(G_b(F),\Lambda)$$ 
denote the subcategory of objects supported at $\xi$.

\begin{lemma}\label{lemma-6-functor-supp}
    The functors $(i_b)_!$, $(i_b)_*$, $\cP^{\unip}\circ(i_b)_\flat$, $(i_b)^!$, $(i_b)^*$, and $\cP^{\unip}\circ(i_b)^\sharp$ preserves objects supported at $\xi$.
\end{lemma}
\begin{proof}
    The claim for $(i_b)_!$ is trivial. Let $V\in \Ind\Rep^\unip_\fingen(G_b(F),\Lambda)_\xi$. Let $C\in \Shv^{\unip}_\fingen(\Isoc_G,\Lambda)$ be a finitely generated object. Then
    $$\Hom(C,\cP^{\unip}\circ(i_b)_\flat V)\simeq \Hom((i_b)^!C,V)$$
    is supported at $\xi$ by Lemma \ref{lemma-supp}, as $(i_b)^!C$ is finitely generated. The claim for $(i_b)_*$ is similar. The claim for rest functors follows from Lemma \ref{lemma-supp-functors}.
\end{proof}

\begin{lemma}\label{lemma-rep-supp-t-str}
    Let $b\in B(G)$. The standard $t$-structure on $\Rep^{\widehat\unip}(G_b(F),\Lambda)$ restricts to a $t$-structure on $\Rep^{\widehat\unip}(G_b(F),\Lambda)_\xi$ such that
    $$\Rep^{\widehat\unip}(G_b(F),\Lambda)_\xi^{\leq 0}=\Rep^{\widehat\unip}(G_b(F),\Lambda)^{\leq 0}\cap\Rep^{\widehat\unip}(G_b(F),\Lambda)_\xi,$$
    $$\Rep^{\widehat\unip}(G_b(F),\Lambda)_\xi^{\geq 0}=\Rep^{\widehat\unip}(G_b(F),\Lambda)^{\geq 0}\cap\Rep^{\widehat\unip}(G_b(F),\Lambda)_\xi.$$
\end{lemma}
\begin{proof}
    It suffices to show that the truncation functors $\tau^{\leq0}$ and $\tau^{\geq0}$ preserves objects supported at $\xi$. Let $C$ be a compact projective generator in $\Rep^{\widehat\unip}(G_b(F),\Lambda)$. Then a representation $V\in\Rep^{\widehat\unip}(G_b(F),\Lambda)$ is connective (resp. coconnective) if and only if $\Hom(C,V)$ is connective (resp. coconnective). If $V$ is supported at $\xi$, then the $Z^{\widehat\unip}_{{}^LG,F}$-module $\Hom(C,V)$ is supported at $\xi$. It follows that
    $$\Hom(C,\tau^{\leq0}V)\simeq \tau^{\leq0}\Hom(C,V)$$
    is supported at $\xi$. Hence $\tau^{\leq0}(\cF)$ lies in $\Rep^{\widehat\unip}(G_b(F),\Lambda)_\xi$. Same proof works for $\tau^{\geq0}$.
\end{proof}

By Lemma \ref{lemma-6-functor-supp} and Lemma \ref{lemma-rep-supp-t-str}, the exotic $t$-structure and the perverse $t$-structure on $\Shv^{\widehat\unip}(\Isoc_G,\Lambda)$ restricts to well-defined $t$-structures on the subcategory $\Shv^{\widehat\unip}(\Isoc_G,\Lambda)_\xi$ such that 
$$\Shv^{\widehat\unip}(\Isoc_G,\Lambda)^{p,\leq 0\text{ or }\geq0}_\xi=\Shv^{\widehat\unip}(\Isoc_G,\Lambda)^{p,\leq 0\text{ or }\geq0}\cap \Shv^{\widehat\unip}(\Isoc_G,\Lambda)_\xi$$
and
$$\Shv^{\widehat\unip}(\Isoc_G,\Lambda)^{e,\leq 0\text{ or }\geq0}_\xi=\Shv^{\widehat\unip}(\Isoc_G,\Lambda)^{e,\leq 0\text{ or }\geq0}\cap \Shv^{\widehat\unip}(\Isoc_G,\Lambda)_\xi.$$

\begin{rmk}
    Let $\bC$ be a $Z^{\widehat\unip}_{{}^LG,F}$-linear category, and let
 $\xi\in C^{\widehat\unip}_{{}^LG,F}(\Lambda)$. Let $(C^{\widehat\unip}_{{}^LG,F})_{(\xi)}$ denote the localization of $C^{\widehat\unip}_{{}^LG,F}$ at $\xi$. Then in addition to the subcategory $\bC_\xi$ of objects supported at $\xi$ as discussed above, there is also a subcategory of objects localized at $\xi$, given by \[
\bC_{(\xi)}:=\bC\otimes_{\QCoh(C^{\widehat\unip}_{{}^LG,F})} \QCoh((C^{\widehat\unip}_{{}^LG,F})_{(\xi)}),
    \]
    which contains $\bC_\xi$.
    This was considered in \cite{Hamann-Lee-vanishing} and \cite{Hansen-Beijing-notes} (see \cite[Definition A.1]{Hamann-Lee-vanishing}). 
    When $\bC$ is the local Langlands category, the intersections of the subcategory of admissible objects (or ULA objects, in the terminology of \cite{FS}) with these two categories agree. We choose to work with $\bC_\xi$ since it is easier to describe compact generators in $\bC_\xi$, as we shall see below. %\Xinwen{Is the writing ok?}\Xiangqian{ok}
\end{rmk}

\subsection{$t$-exactness of the local Langlands functor}\label{subsection-t-exact-Hecke}
In this subsection, we show that the spectral action is \emph{exotic} $t$-exact when restricted to $\Shv^{\widehat\unip}(\Isoc_G,\Lambda)_\xi$ provided that $\xi$ is generic.

\subsubsection{Supports of objects in $\Ind\Shv_\fingen^{\widehat\unip}(\Isoc_G,\Lambda)_\xi$.}

Let $\bar\lambda\in \XX_\bullet(T)^+_\phi$ be an dominant element. Let $b\in B(G)_\unr$ denote the image of $-\bar\lambda$ under the canonical morphism $\XX_\bullet(T)_\phi\simeq B(T)\to B(G)_\unr$. Fix a lifting $\lambda\in \XX_\bullet(T)$ of $\bar\lambda$. Denote
$$V_b=(i_b)^!\Ch^\unip_{LG,\phi}(W_\lambda)[\langle2\rho,\nu_b\rangle]\quad\text{resp.}\quad V_b'=(i_b)^!\Ch^\unip_{LG,\phi}(W_{w_0(\lambda)})[\langle2\rho,\nu_b\rangle].$$ 
Similarly, we denote
$$\widetilde{V}_b=(i_b)^!\Ch^{\hat{u}\mbox{-}\mon}_{LG,\phi}(W^\mon_{\lambda,\hat{u}})[\langle2\rho,\nu_b\rangle]\quad\text{resp.}\quad \widetilde{V}_b'=(i_b)^!\Ch^{\hat{u}\mbox{-}\mon}_{LG,\phi}(W^\mon_{w_0(\lambda),\hat{u}})[\langle2\rho,\nu_b\rangle].$$ 
If there exists a dominant lift $\lambda\in \XX_\bullet(T)^+$ of $\bar\lambda$, by Proposition \ref{prop-Ch-Wakimoto}, we have $$V_b\simeq V_b'\simeq \cInd_{I_b}^{G_b(F)}\Lambda\quad\text{and}\quad\widetilde{V}_b\simeq \widetilde{V}'_b\simeq \cInd_{\widetilde{I}_b}^{G_b(F)}\Lambda,$$ 
where $\widetilde{I}_b=I_b$ if $\Lambda$ is of characteristic 0 and $\widetilde{I}_b=I_b^{(\ell)}$ if $\Lambda$ is of characteristic $\ell$.

\begin{lemma}\label{lemma-wakimoto-ch}
    There are isomorphisms
    $$\Ch^\unip_{LG,\phi}(W_\lambda)\simeq (i_b)_*V_b[-\langle2\rho,\nu_b\rangle]\quad\text{and}\quad\Ch^\unip_{LG,\phi}(W_{w_0(\lambda)})\simeq (i_b)_!V_b'[-\langle2\rho,\nu_b\rangle],$$
    $$\Ch^{\hat{u}\mbox{-}\mon}_{LG,\phi}(W^\mon_{\lambda,\hat{u}})\simeq (i_b)_*\widetilde{V}_b[-\langle2\rho,\nu_b\rangle]\quad\text{and}\quad\Ch^{\hat{u}\mbox{-}\mon}_{LG,\phi}(W^\mon_{w_0(\lambda),\hat{u}})\simeq (i_b)_!\widetilde{V}'_b[-\langle2\rho,\nu_b\rangle].$$
    Moreover, $V_b$, $V_b'$, $\widetilde{V}_b$, and $\widetilde{V}'_b$ lies in $\Rep_\fingen^{\unip}(G_b(F),\Lambda)^\heartsuit$, and $\widetilde{V}_b$, $\widetilde{V}'_b$ are compact projective in $\Rep^{\widehat\unip}(G_b(F),\Lambda)^\heartsuit$.
\end{lemma}
\begin{proof}
    Denote $G'=G_\ad\times G_\ab$ with a central isogeny $G\to G'$. Let $\bar{\lambda}'\in \XX_\bullet(T')_\phi^+$ (resp. $\lambda'\in \XX_\bullet(T')$) be the image of $\bar\lambda$ (resp. $\lambda$). We can choose a dominant coweight $\lambda'_\dom\in \XX_\bullet(T')^+$ mapping to $\bar\lambda'$ as $G'$ has connected center. Then $\lambda'_\dom=\lambda'+\phi(\nu)-\nu$ for some $\nu\in \XX_\bullet(T')$. Hence $W_{\lambda'_\dom}\simeq \phi_*(W_{\nu})\star W_{\lambda'}\star W_{-\nu}$. Note that $W_{\lambda'_\dom}\simeq \nabla_{\lambda'_\dom}$. We have
    $$\Ch^\unip_{LG',\phi}(W_{\lambda'})\simeq \Ch^\unip_{LG',\phi}(W_{\lambda_{\dom}'})=(i_{b'})_*\cInd_{I_{b'}'}^{G_{b'}'(F)}\Lambda[-\langle2\rho,{\nu_b'}\rangle].$$
    
    By Proposition \ref{lemma-llc-isogeny}, we know that $\Ch^\unip_{LG,\phi}(W_\lambda)$ is a direct summand of $(\kappa_\Isoc)^!\Ch^{\unip}_{LG',\phi}(W_{\lambda'}).$ Note that
    $$(\kappa_\Isoc)^!\Ch^{\unip}_{LG',\phi}(W_{\lambda'})=\bigoplus_{b_\alpha}(i_{b_\alpha})_*\bigg(\cInd_{I'_{b'}}^{G'_{b'}(F)}\Lambda\bigg|_{G_{b_\alpha}(F)}\bigg)[-\langle2\rho,{\nu_{b_\alpha}}\rangle],$$
    where $b_\alpha$ runs though elements in $B(G)$ mapping to $b'$. All the $b_\alpha$'s appearing in the formula are unramified. Moreover, the Newton point $\nu_{b_\alpha}$ are the same for different $b_\alpha$ and the Kottwitz point $\kappa_G(b_\alpha)$ are distinct. Therefore $\Ch^\unip_{LG,\phi}(W_\lambda)$ is a direct summand of $(i_b)_*\bigg(\cInd_{I'_{b'}}^{G'_{b'}(F)}\Lambda\bigg|_{G_{b}(F)}\bigg)[-\langle2\rho,{\nu_{b_\alpha}}\rangle]$. Thus $\Ch^\unip_{LG,\phi}(W_\lambda)=(i_b)_*V_b[-\langle2\rho,\nu_b\rangle]$, and $V_b$ is concentrated in degree 0. Similarly, we can prove the claim for $W_{w_0(\lambda)}$.

    For $\Ch^{\hat{u}\mbox{-}\mon}_{LG,\phi}(W^\mon_{\lambda,\hat{u}})$, by the same argument in the unipotent monodromic setting, we can show that $\Ch^{\hat{u}\mbox{-}\mon}_{LG,\phi}(W^\mon_{\lambda,\hat{u}})$ is a direct summand of $(i_b)_*\cInd_{\widetilde{I'}_{b'}}^{G'_{b'}(F)}\Lambda\bigg|_{G_b(F)}$. Note that the kernel of $G\to G'$ has order invertible in $\Lambda$, thus we can use Lemma \ref{lemma-llc-isogeny}. Therefore $\Ch^{\hat{u}\mbox{-}\mon}_{LG,\phi}(W^\mon_{\lambda,\hat{u}})\simeq (i_b)_*\widetilde{V}_b$ and $\widetilde{V}_b$ is compact projective in $\Rep(G_b(F),\Lambda)^\heartsuit$. Similar for $\Ch^{\hat{u}\mbox{-}\mon}_{LG,\phi}(W^\mon_{w_0(\lambda),\hat{u}})$.
\end{proof}

By the proof of Proposition \ref{prop-Ch-Wakimoto}, for $\lambda\in \XX^\bullet(T)$, we have
    $$\LL^{\unip}_G(\Ch^{\unip}_{LG,\phi}(W_\lambda))\simeq (\fq^\unip)_*\omega_{\Loc^{\unip}_{{}^LB,F}}(\lambda).$$
    The object $\omega_{\Loc^{\unip}_{{}^LB,F}}(\lambda)$ is the pullpack of $\omega_{\hat{T}\phi/\hat{T}}(\lambda)$ along the natural morphism
    $$\fp^\unip\colon \Loc^{\unip}_{{}^LB,F}\simeq \cL_\phi(\hat{U}/\hat{B})\to \hat{T}\phi/\hat{T}\simeq \cL_\phi(\BB\hat{T}).$$ 
    Using $\hat{T}\phi/\hat{T}\simeq \hat{A}\times \BB\hat{T}^\phi$, we know that the object $(\fq^\unip)_*\omega_{\Loc^{\unip}_{{}^LB,F}}(\lambda)$ carries a natural action of the commutative $\Lambda$-algebra $\cO(\hat{A})$. It follows that $\Ch^{\unip}_{LG,\phi}(W_\lambda)$ also carries an action of $\cO(\hat{A})$. Let $\xi\in (\hat{A}\git W_0)(\Lambda)$. Consider the $\cO(\hat{A})$-module
    $$D_\xi\coloneqq \bigoplus_{\chi}\Lambda_\chi$$
    where $\chi$ runs through $\Lambda$-points of $\hat{A}$ that maps to $\xi$, and $\Lambda_\chi$ is the skyscraper $\cO(\hat{A})$-module supported at $\chi$. Let $b\in B(G)_\unr$ be the image of $\bar\lambda^{-1}$. Fix a lift $\lambda\in \XX_\bullet(T)$ of $\bar\lambda$. By Lemma \ref{lemma-wakimoto-ch}, the object
    $$(i_b)_*V_b[-\langle2\rho,\nu_b\rangle]\simeq \Ch^\unip_{LG,\phi}(W_\lambda) \quad\text{resp.}\quad(i_b)_!V_b'[-\langle2\rho,\nu_b\rangle]\simeq \Ch^\unip_{LG,\phi}(W_{w_0(\lambda)})$$
    carries an action of $\cO(\hat{A})$. Define the $G_b(F)$-representation
    $$M_{\xi,b}\coloneqq V_b\otimes_{\cO(\hat{A})}D_\xi\quad\text{resp.}\quad M'_{\xi,b}=V_b'\otimes_{\cO(\hat{A})}D_\xi.$$

The following lemma will be needed latter.

\begin{lemma}\label{lemma-V_b-fil}
    The $G_b(F)$-representation $\widetilde{V}_b$ admits a finite filtration with graded pieces isomorphic to $V_b$.
\end{lemma}
\begin{proof}
    We need to check that $\Ch^{\hat{u}\mbox{-}\mon}_{LG,\phi}(W^\mon_{\lambda,\hat{u}})$ has a finite filtration by $\Ch^\unip_{LG,\phi}(W_\lambda)$. We can do this on the spectral side. We have
    $$\begin{tikzcd}
        \hat{T}\phi/\hat{T} \ar[d,"i_e"] & \Loc_{B,F}^{\unip} \ar[d]\ar[l,"\fp^\unip"swap]\ar[rd,"\fq^\unip"]\ar[ld,phantom,"\square",very near start]  \\
        \cL_\phi(\hat{T}^\wedge_e/\hat{T}) & \Loc_{B,F}^{\widehat\unip} \ar[l,"\fp^{\widehat\unip}"swap]\ar[r,"\fq^{\widehat\unip}"] & \Loc_{G,F}^{\widehat\unip},
    \end{tikzcd}$$
    where the left Cartesian square is defined by taking $\phi$-fixed stacks of
    $$\begin{tikzcd}
        \BB\hat{T} \ar[d] & \hat{U}/\hat{B}\ar[l]\ar[d]\ar[ld,phantom,"\square",very near start] \\
        \hat{T}^\wedge_e/\hat{T} & \hat{U}^\wedge/\hat{B}. \ar[l]
    \end{tikzcd}$$
    Note that $\cL_\phi(\hat{T}^\wedge_e/\hat{T})=\hat{T}\phi/\hat{T}\times \cL_\phi(\hat{T}^\wedge_e)$ and $\cL_\phi(\hat{T}^\wedge_e)$ is a nilpotent thickening of $\Spec \Lambda$. In fact, if $\Lambda$ has characteristic 0, then $\cL_\phi(\hat{T}^\wedge_e)=\Spec\Lambda$, and  if $\Lambda$ has characteristic $\ell$, then $\cL_\phi(\hat{T}^\wedge_e)=\Spec\Lambda[T(\kappa_F)/T(\kappa_F)^{(\ell)}]$, where $T(\kappa_F)^{(\ell)}$ is the maximal prime-to-$\ell$ subgroup of $T(\kappa_F)$. We know that
    $$\LL^{\unip}_G(\Ch^{\hat{u}\mbox{-}\mon}_{LG,\phi}(W^\mon_{\lambda,\hat{u}}))=(\fq^{\widehat\unip})_*(\fp^{\widehat\unip})^!\omega_{\cL_\phi(\hat{T}^\wedge_e/\hat{T})}(\lambda)$$
    and
    $$\LL^{\unip}_G(\Ch^{\unip}_{LG,\phi}(W_{\lambda}))=(\fq^\unip)_*(\fp^\unip)^!\omega_{\hat{T}\phi/\hat{T}}(\lambda)\simeq (\fq^{\widehat\unip})_*(\fp^{\widehat\unip})^!(i_e)_*\omega_{\hat{T}\phi/\hat{T}}(\lambda).$$
    Note that $\omega_{\cL_\phi(\hat{T}^\wedge_e/\hat{T})}\simeq\cO_{\cL_\phi(\hat{T}^\wedge_e/\hat{T})}$ and $\omega_{\hat{T}\phi/\hat{T}}\simeq\cO_{\hat{T}\phi/\hat{T}}$ as the stacks $\cL_\phi(\hat{T}^\wedge_e/\hat{T})$ and $\hat{T}\phi/\hat{T}$ are  $\phi$-fixed point stacks. It follows that
    $$(i_e)_*\omega_{\cL_\phi(\hat{T}^\wedge_e/\hat{T})}(\lambda)\simeq \omega_{\cL_\phi(\hat{T}^\wedge_e/\hat{T})}(\lambda)\otimes_{\cO(\cL_\phi(\hat{T}^\wedge_e))}\Lambda.$$
    Thus $\omega_{\cL_\phi(\hat{T}^\wedge_e/\hat{T})}(\lambda)$ admits a finite filtration by $(i_e)_*\omega_{\cL_\phi(\hat{T}^\wedge_e/\hat{T})}(\lambda)$ as $\cO(\cL_\phi(\hat{T}^\wedge_e))$ admits a finite filtration by $\Lambda$. The claim for $\widetilde{V}_b$ follows from this.
\end{proof}

If we restrict to the subcategory $\Ind\Shv^{\unip}_\fingen(\Isoc_G,\Lambda)_\xi$ of objects supported at $\xi$ for $\xi$ generic, then the functor $\LL^{\unip}_{G,\xi}$ becomes an equivalence. Moreover, we can write down generators of $\Ind\Shv^{\unip}_\fingen(\Isoc_G,\Lambda)_\xi$ explicitly.

\begin{prop}\label{prop-generic-equiv}
    Let $\xi\in (\hat{A}^\gen\git W_0)(\Lambda)$ be a generic element. Then the fully faithful functor
    $$\LL^{\unip}_{G,\xi}\colon \Ind\Shv_\fingen^{\unip}(\Isoc_G,\Lambda)_\xi\to \Ind\Coh(V^\wedge_\xi)$$
    is an equivalence of categories. Moreover, the category $\Ind\Shv_\fingen^{\unip}(\Isoc_G,\Lambda)_\xi$ is compactly generated by the collection of objects $\{(i_b)_*M_{\xi,b}\}_{b\in B(G)_\unr}$, or the collection of objects $\{(i_b)_!M'_{\xi,b}\}_{b\in B(G)_\unr}$.
\end{prop}
\begin{proof}
    By Proposition \ref{prop-generic-Loc-equal-GS}, the closed embedding
    $$\unr_{G,\xi}\colon (\hat{G}\phi/\hat{G})^\wedge_\xi\hookrightarrow V^\wedge_\xi$$
    defines an isomorphism on the underlying reduced stack. Hence objects of the form $(\unr_{G,\xi})_*\cF$ for $\cF\in\Coh((\hat{G}\phi/\hat{G})_\xi^\wedge)$ compactly generate the category $\Ind\Coh(V^\wedge_\xi)$.

    There is a natural morphism $\hat{B}\phi/\hat{B}\to \hat{A}$. For $\chi\in \hat{A}(\Lambda)$, let $(\hat{B}\phi/\hat{B})^\wedge_\chi$ denote the formal completion of $\hat{B}\phi/\hat{B}$ along the preimage of $\chi$. Hence there is a disjoint union decomposition
    $$(\hat{B}\phi/\hat{B})^\wedge_\xi=\bigsqcup_{\chi} (\hat{B}\phi/\hat{B})^\wedge_\chi$$
    where $\chi$ runs through $\Lambda$-points of $\hat{A}$ mapping to $\xi$.
    Moreover, there is a commutative diagram
    $$\begin{tikzcd}[column sep=huge]
        \bigsqcup_{\chi}\hat{A}^\wedge_\chi\times \BB\hat{T}^\phi \ar[r,leftarrow,"\gamma_{\phi,\xi}"] & \bigsqcup_{\chi}(\hat{B}\phi/\hat{B})^\wedge_\chi \ar[r,"\simeq"swap,"\unr_{B,\xi}"]\ar[d,"\pi_{\phi,\xi}"] & W^\wedge_\xi \ar[d,"\fq^\unip_\xi"]\ar[ll,bend right,"\fp^\unip_\xi"swap] \\
        & (\hat{G}\phi/\hat{G})^\wedge_\xi \ar[r,"\unr_{G,\xi}"] & V^\wedge_\xi.
    \end{tikzcd}$$
    By Proposition \ref{prop-generators-GS}, objects of the form
    $$(\pi_{\phi,\xi})_*(\gamma_{\phi,\xi})^!\cF_{\bar\lambda},\quad \cF_{\bar\lambda}\coloneqq \bigoplus_\chi (\iota_\chi)_*\cO_{\BB\hat{T}}(\bar\lambda)$$
    for $\bar\lambda\in \XX_\bullet(T)_\phi^+$ (resp. $\bar\lambda\in -\XX_\bullet(T)_\phi^+$)
    generate the category $\Ind\Coh((\hat{G}\phi/\hat{G})^\wedge_\xi)$ under colimits. Thus the objects 
    $$\Xi_{\xi,\bar\lambda}\coloneqq (\unr_{G,\xi})_*(\pi_{\phi,\xi})_*(\gamma_{\phi,\xi})^!\cF_{\bar\lambda},\quad\bar\lambda\in \XX_\bullet(T)_\phi$$
    generate the category $\Ind\Coh(V^\wedge_\xi)$.

    Let $b\in B(G)_\unr$ associated to $\bar\lambda\in \XX_\bullet(T)_\phi^+$. Fix a lift $\lambda\in \XX_\bullet(T)$ of $\lambda$. Then we have
    $$\begin{aligned}
        \LL^{\unip}_G(\Ch^\unip_{LG,\phi}(W_\lambda)\otimes_{\cO(\hat{A})}D_\xi)&\simeq (\fq^\unip)_*\omega_{\Loc^{\widehat\unip}_{{}^LG,F}}(\lambda)\otimes_{\cO(\hat{A})}D_\xi \\ &\simeq (\fq^\unip)_*(\fp^\unip)^!(\omega_{\hat{T}\phi/\hat{T}}(\lambda)\otimes_{\cO(\hat{A})}D_\xi)\simeq \Xi_{\xi,\bar{\lambda}}.
    \end{aligned}$$
    Here the last isomorphism follows from the canonical isomorphism $\omega_{\hat{T}\phi/\hat{T}}\simeq\cO_{\hat{T}\phi/\hat{T}}$. In particular, we have
    $$\LL^{\unip}_{G}((i_b)_*M_{\xi,b}[-\langle2\rho,\nu_b\rangle])\simeq \Xi_{\xi,\bar{\lambda}}\quad\text{and}\quad\LL^{\unip}_{G}((i_b)_!M_{\xi,b}'[-\langle2\rho,\nu_b\rangle])\simeq \Xi_{\xi,\bar{w_0\lambda}}.$$
    Therefore the functor $\LL^{\unip}_{G,\xi}$ is essentially surjective and the objects $(i_b)_*M_{\xi,b}$ (resp. $(i_b)_!M_{\xi,b}'$) generate the category $\Ind\Shv^{\unip}_\fingen(\Isoc_G,\Lambda)_\xi$ under colimits.
\end{proof}

\begin{cor}\label{cor-supp-generic}
    Let $\xi\in (\hat{A}^\gen\git W_0)(\Lambda)$ be a generic element. Then for any object $\cF\in \Ind\Shv_\fingen^{\unip}(\Isoc_G,\Lambda)_\xi$, the $!$-stalk $(i_b)^!\cF$ and the $*$-stalk $(i_b)^*\cF$ at $\Isoc_{G,b}$ vanish for every $b\in B(G)\setminus B(G)_\unr$. 
\end{cor}
\begin{proof}
    By Proposition \ref{prop-generic-equiv}, the category $\Ind\Shv_\fingen^{\unip}(\Isoc_G,\Lambda)_\xi$ is compactly generated by objects
    $$(i_{b})_*M_{\xi,b},\quad b\in B(G)_\unr.$$
    Note that the $!$-stalks of such objects at every $b\in B(G)\setminus B(G)_\unr$ is trivial. 
    Therefore the $!$-stalk of any object in $\Ind\Shv_\fingen^{\widehat\unip}(\Isoc_G,\Lambda)_\xi$ at any $b\in B(G)\setminus B(G)_\unr$  is trivial. Similarly, the category $\Ind\Shv_\fingen^{\unip}(\Isoc_G,\Lambda)_\xi$ is compactly generated by objects
    $$(i_{b})_!M'_{\xi,b},\quad b\in B(G)_\unr.$$
    This proves the statement for $*$-stalks.
\end{proof}

\begin{cor}\label{cor-detect-t-str-supp}
    Let $b\in B(G)_\unr$ and $\xi\in (\hat{A}^\gen\git W_0)(\Lambda)$ be a generic element. An object $M\in \Rep^{\widehat\unip}(G_b(F),\Lambda)_\xi$ lies in $\Rep^{\widehat\unip}(G_b(F),\Lambda)_\xi^{\leq 0}$ (resp. $\Rep^{\widehat\unip}(G_b(F),\Lambda)_\xi^{\geq 0}$) if and only if
    $$\Hom(\widetilde{V}_b,M)$$
    lies in $\Lambda\mbox{-}\mathrm{mod}^{\leq 0}$ (resp. $\Lambda\mbox{-}\mathrm{mod}^{\geq 0}$). 
\end{cor}
\begin{rmk}\label{rmk-detect-t-str}
    As mentioned earlier, if there is a dominant coweight $\lambda\in \XX_\bullet(T)^+$ mapping to $b$, then we have $\widetilde{V}_b\simeq \cInd_{\widetilde{I}_b}^{G_b(F)}\Lambda$. For example, the identity element $1\in B(G)$ can always be lifted to $0\in \XX_\bullet(T)^+$, thus an object $M\in\Rep^{\widehat\unip}(G(F),\Lambda)_\xi$ is connective (resp. coconnective) if and only if $\Hom(\cInd_{\widetilde{I}}^{G(F)}\Lambda,M)$ is connective (resp. coconnective). In general, if $Z_G$ is connected, then every element in $B(G)_\unr$ can be lifted to a dominant coweight by Remark \ref{rmk-phi-coinv-surj}.
\end{rmk}
\begin{proof}
    By Lemma \ref{lemma-wakimoto-ch}, the object $\widetilde{V}_b$ is compact projective. It suffices to show that  
    $$\Hom(\widetilde{V}_b,-)\colon \Rep^{\widehat\unip}(G_b(F),\Lambda)_\xi\to \Lambda\mbox{-}\mathrm{mod}$$
    is conservative. By Corollary \ref{cor-supp-generic}, the representation $M_{\xi,b}\in \Rep^\unip_\fingen(G(F),\Lambda)_\xi$ generates $\Rep^{\widehat\unip}(G_b(F),\Lambda)_\xi$ under colimits. It suffices to write $M_{\xi,b}$ as a colimit of $\widetilde{V}_b$. By the proof of Lemma \ref{lemma-V_b-fil}, we know that
    $$V_b\simeq \widetilde{V}_b\otimes_{\cO(\cL_\phi(\hat{T}^\wedge_e))}\Lambda$$
    is a colimit of $\widetilde{V}_b$. As $M_{\xi,b}=V_b\otimes_{\cO(\hat{A})}D_\xi$ is a colimit of $V_b$, we see that $M_{\xi,b}$ is a colimit of $\widetilde{V}_b$.
\end{proof}

\begin{cor}
    Let $b\in B(G)$ and $\pi$ be an irreducible unipotent $G_b(F)$-representation. If the associated semisimple $L$-parameter $\varphi_\pi^\mathrm{ss}$ of $\pi$ constructed in \cite[Theorem 5.19]{Tame} is generic, then $b$ is unramified and $\pi$ admits nonzero $I_b$-invariants.
\end{cor}
\begin{proof}
    By Corollary \ref{cor-supp-generic}, we know that $b$ must be unramified. By Corollary \ref{cor-detect-t-str-supp}, we know that $\Hom(V_b,\pi)\neq 0$. However, $V_b$ is a retract of $\cInd_{I'_{b'}}^{G'_{b'}(F)}\Lambda\bigg|_{G_b(F)}$, which is a finite direct sum of $\cInd_{I_b}^{G_b(F)}\Lambda$. This implies that $\Hom(\cInd_{I_b}^{G_b(F)}\Lambda,\pi)\neq 0$.
\end{proof}

If moreover the element $\xi$ is \emph{strongly generic}, then a stronger statements holds. We show that sheaves on different Newton strata glued trivially.

\begin{prop}\label{prop-split-sreg}
    Let $\xi\in (\hat{A}^\sgen\git W_0)(\Lambda)$ be a strongly generic element.
    \begin{enumerate}
        \item The natural transforms
            $$(i_b)_!\to (i_b)_*\to\cP^{\unip}\circ(i_b)_\flat $$
            defined by adjunctions are equivalences when restricted to the subcategory $\Ind\Rep^{\unip}_\fingen(G_b(F),\Lambda)_\xi$. 
        \item There are natural equivalences
        $$(i_b)^!\simeq\cP^{\unip}\circ(i_b)^\sharp\rightarrow (i_b)^!$$ 
             when restricted to the subcategory $\Ind\Shv^{\unip}_\fingen(\Isoc_G,\Lambda)_\xi$. Here $\cP^{\unip}\circ(i_b)^\sharp\rightarrow (i_b)^!$ is defined by adjunctions, and $(i_b)^*\simeq\cP^{\unip}\circ(i_b)^\sharp$ is defined by the zigzag
             $$(i_b)^*\to \cP^{\unip}\circ(i_b)^\sharp(i_b)_*(i_b)^*\leftarrow \cP^{\unip}\circ(i_b)^\sharp.$$
        \item For $\cF\in\Ind\Shv^{\unip}_\fingen(\Isoc_G,\Lambda)_\xi$, the natural morphism
            $$\bigoplus_{b\in B(G)_\unr} (i_b)_!(i_b)^!\cF\to \cF$$
            is an isomorphism. Therefore there is a canonical direct sum decomposition
            $$\Ind\Shv^{\unip}_\fingen(\Isoc_G,\Lambda)_\xi\simeq \bigoplus_{b\in B(G)_\unr}\Ind\Rep^{\unip}_\fingen(G_b(F),\Lambda)_\xi.$$
    \end{enumerate}
\end{prop}
\begin{rmk}
    Similar results in the Fargues--Scholze setting were conjectured by Hansen in \cite[Conjecture 2.1.8]{Hansen-Beijing-notes}. For toral $L$-parameters, some cases of Hansen's conjecture follow from the results of \cite{Hamann-Eisenstein}, which relies on the compatibility of Fargues--Scholze parameters and endoscopic classifications.
\end{rmk}
\begin{proof}
    We first show that $(i_b)_!\to (i_b)_*$ is an equivalence when restricted to $\Ind\Rep^{\unip}_\fingen(G_b(F),\Lambda)_\xi$. 
    The category $\Ind\Shv^{\unip}_{\fingen}(\Isoc_G,\Lambda)_\xi$ is compactly generated by objects $(i_b)_*M_{\xi,b}$, $b\in B(G)_\unr$ by the proof of Corollary \ref{cor-supp-generic}. It follows that $\Ind\Rep_\fingen^{\unip}(G_b(F),\Lambda)_\xi$ is compactly generated by $M_{\xi,b}$ for $b\in B(G)_\unr$. Thus it suffices to check that the natural morphism
    $$(i_b)_!M_{\xi,b}\to (i_b)_*M_{\xi,b}$$
    is an isomorphism. It suffices to show that $(i_{b'})^*(i_b)_*M_{\xi,b}$ vanishes for any $b'\neq b$.

    First assume that $\hat{G}$ is almost simple and simply-connected. Fix $b\in B(G)_\unr$. We know that
    $$\LL^{\unip}_{G}((i_b)_*M_{\xi,b}[-\langle2\rho,\nu_b])\simeq \Xi_{\xi,\bar\lambda}\quad\text{and}\quad\LL^{\unip}_{G}((i_b)_!M'_{\xi,b}[-\langle2\rho,\nu_b])\simeq \Xi_{\xi,\overline{w_0\lambda}}.$$
    We claim that there is an isomorphism $\Xi_{\xi,\bar\lambda}\simeq \Xi_{\xi,\bar {w_0\lambda}}$.
    Fix a lift $x\in \hat{T}(\Lambda)$ mapping to $\xi$. By \cite[Lemma 5.2.11]{XZ-vector}, the element $x$ is $\phi$-regular. Hence $\hat G_x\simeq \hat{T}^{\phi}$ is a torus, and $(W_0)_{\bar{x}}$ is trivial. Let $\bar{x}$ be the image of $x$ in $\hat{A}$. By Proposition \ref{prop-adjoint-quot-endoscopy} and Proposition \ref{prop-Groth-endoscopy}, we have
    $$(\hat{G}\phi/\hat{G})^\wedge_\xi\simeq (\hat{T}^{\phi})^\wedge_e/ \hat G_x$$ and $$(\hat{B}\phi/\hat{B})^\wedge_\xi\simeq\bigsqcup_{w\in W_0} (\hat T^{\phi})^\wedge_e \times \BB \hat T^{\phi}$$
    where $\chi$ runs through $\Lambda$-points of $\hat{A}$ mapping to $\xi$. The twisted Grothendieck--Springer resolution $\pi_{\phi,\xi}$ is given by
    $$\pi_{\phi,\xi}=(w)_{w\in W_0}\colon\bigsqcup_{w\in W_0} (\hat T^{\phi})^\wedge_{e} \times \BB \hat T^{\phi}\to (\hat T^{\phi})^\wedge_{e} \times \BB \hat T^{\phi}$$
    where $w$ acts on $(\hat T^{\phi})^\wedge_{e} \times \BB \hat T^{\phi}$ via the Weyl action on $\hat{T}^{\phi}$. Denote $\Lambda_e=(\iota_e)_*\cO_{\BB \hat{T}^\phi}\in \Coh((\hat{T}^{\phi})^\wedge_{e} \times \BB \hat T^{\phi})$ where $\iota_e\colon\BB \hat{T}^\phi\hookrightarrow (\hat{T}^{\phi})^\wedge_e\times \BB\hat{T}^\phi$ is the closed embedding at identity. We have
    $$\begin{aligned}
        (\pi_{\phi,\xi})_*\Lambda_e(\bar\lambda)& \simeq \bigoplus_{w\in W_0}(\iota_e)_*\Lambda_e(w(\bar{\lambda})) \\
        &\simeq \bigoplus_{w\in W_0}(\iota_e)_*\Lambda_e(w(\bar{w_0\lambda}))\\
        &\simeq (\pi_{\phi,\xi})_*\Lambda_e(\bar{w_0\lambda})
    \end{aligned}$$
    It follows that there is an isomorphism
    $$(i_b)_*M_{\xi,b}\simeq (i_b)_!M'_{\xi,b}.$$
    Therefore all the $*$-stalks of $(i_b)_*M_{\xi,b}$ vanishes except at $b$. In general, we can pass to $G'=G_\ad\times G_\ab$. By the argument in the proof of Proposition \ref{prop-generators-GS}, the claim holds for $G'$. As in Lemma \ref{lemma-wakimoto-ch}, we see that the $*$-stalks of $(i_b)_*M_{\xi,b}$ vanishes except at $b$. Hence the natural morphism $(i_b)_!M_{\xi,b}\to (i_b)_*M_{\xi,b}$ is an isomorphism. This proves that $(i_b)_!\xrightarrow{\sim}(i_b)_*$. Taking right adjoint to $(i_b)_!\xrightarrow{\sim} (i_b)_*$ implies that $\cP^{\unip}\circ(i_b)^\sharp\xrightarrow{\sim}(i_b)^!$.
    
    The category $\Ind\Shv^{\unip}_\fingen(\Isoc_G,\Lambda)$ is compactly generated by objects of the form $(i_{b'})_!V$ for $b'\in B(G)$ and $V\in \Rep_\fingen(G_{b'}(F),\Lambda)$. We have
    $$\begin{aligned}
        \Hom((i_{b'})_!V,\bigoplus_{b\in B(G)_\unr} (i_b)_!(i_{b})^!\cF) &\simeq \bigoplus_{b\in B(G)_\unr} \Hom((i_{b'})_!V, (i_{b})_!(i_{b})^!\cF) \\
        &\simeq \bigoplus_{b\in B(G)_\unr} \Hom((i_{b'})_!V, (i_{b})_*(i_{b})^!\cF) \\
        &\simeq \Hom((i_{b'})_!V,(i_{b'})_*(i_{b'})^!\cF) \\
        &\simeq \Hom((i_{b'})_!V,(i_{b'})_!(i_{b'})^!\cF) \\
        &\simeq\Hom(V,(i_{b'})^!\cF) \\
        &\simeq\Hom((i_{b'})_!V,\cF)
    \end{aligned}$$
    if $b\in B(G)_\unr$. Here the second and the fourth isomorphism follow from $(i_b)_!\simeq (i_b)_*$. If $b$ is not unramified, then $\Hom((i_b)_!V,\cF)=0$ by Corollary \ref{cor-supp-generic}. This proves (3).
    
    It follows that 
    $$(i_b)^*\cF\simeq \bigoplus_{b'\in B(G)_\unr}(i_b)^*(i_{b'})_!(i_{b'})^!\cF\simeq (i_b)^*(i_b)_!(i_b')^!\cF\simeq (i_b)^!\cF.$$
    Thus $\cP^{\unip}\circ(i_b)_\flat$ is isomorphic to $(i_b)_*$ as they are both right adjoint of $(i_b)^!\simeq (i_b)^*$. 
\end{proof}

By Proposition \ref{prop-split-sreg}, the exotic $t$-structure and the perverse $t$-structure on $\Shv^{\widehat\unip}(\Isoc_G,\Lambda)_\xi$ agree provided that $\xi$ is strongly regular.

\subsubsection{$t$-exactness of Hecke operators} Let $V\in\Rep(\hat{G})^\heartsuit$ be a finite dimensional representation. Define the Hecke operators
$$T_V\colon \Ind\Shv^{\unip}_{\fingen}(\Isoc_G,\Lambda)\to\Ind\Shv^{\unip}_{\fingen}(\Isoc_G,\Lambda)$$
$$\text{resp.}\quad T_V\colon \Shv^{\widehat\unip}(\Isoc_G,\Lambda)\to \Shv^{\widehat\unip}(\Isoc_G,\Lambda)$$
by $T_V(\cF)=\widetilde{V}\star \cF$, where $\star$ is the spectral action defined in Proposition \ref{prop-spectral-action}, and $\widetilde{V}$ is the vector bundle on $\Loc^{\widehat\unip}_{{}^LG,F}$ associated to $V$. Let $V^\vee$ be the dual representation. Then $T_V$ and $T_{V^\vee}$ are left and right adjoint to each other. 

The following theorem is inspired by \cite[Conjecture 2.4.1]{Hansen-Beijing-notes}. Recall that a (finite dimensional) $\hat{G}$-representation $V$ is called \emph{tilting} if it admits a Weyl filtration and a good filtration. If $\Lambda=\overline{\QQ}_\ell$ has characteristic 0, then every $\hat{G}$-representation is tilting.

\begin{thm}\label{thm-Hecke-t-exact}
    Let $V\in \Rep(\hat{G})^\heartsuit$ be a finite dimensional tilting representation. Let $\xi\in (\hat{A}^\gen\git W_0)(\Lambda)$ be a generic point. Then the Hecke operator
    $$T_V\colon \Shv^{\widehat\unip}(\Isoc_G,\Lambda)_\xi\to \Shv^{\widehat\unip}(\Isoc_G,\Lambda)_\xi$$
    is exotic $t$-exact.
\end{thm}
\begin{proof}
    Because $T_V$ is left and right adjoint to $T_{V^\vee}$, it suffices to show that $T_V$ preserves the coconnective part. Let $b\in B(G)_\unr$. Recall the representation
    $$M_{\xi,b}=V_b\otimes_{\cO(\hat{A})}D_\xi\in \Rep_\fingen^{\unip}(G_b(F),\Lambda).$$
\begin{lemma}\label{lemma-generic-unr-detect-t-str}
    An object $\cF\in \Shv^{\widehat\unip}(\Isoc_G,\Lambda)_\xi$ lies in $\Shv^{\widehat\unip}(\Isoc_G,\Lambda)_\xi^{e,\geq 0}$ if and only if
    $$\Hom((i_b)_*M_{\xi,b},\cF)\in \Lambda\mbox{-}\mathrm{mod}^{\geq \langle2\rho,\nu_b\rangle}$$
    for all $b\in B(G)_\unr$.
\end{lemma}
\begin{proof}
    The representation $M_{\xi,b}$ lies in $\Rep(G_b(F),\Lambda)^{\leq0}$. Thus the ``only if'' part follows from definition.

    For the ``if'' part, we need to show that $N_b\coloneqq\cP^{\widehat\unip}\circ(i_b)^\sharp\cF$ lies in $ \Rep^{\widehat\unip}(G_b(F),\Lambda)^{\geq \langle2\rho,\nu_b\rangle}$ for any $b\in B(G)$. If $b\notin B(G)_\unr$, then the category $\Rep^{\widehat\unip}(G_b(F),\Lambda)_\xi$ is trivial by Corollary \ref{cor-supp-generic}. Thus we may assume $b\in B(G)_\unr$. By Corollary \ref{cor-detect-t-str-supp}, it suffices to show that $\Hom(\widetilde{V}_b,N_b)$ lies in $\Lambda\mbox{-}\mathrm{mod}^{\geq \langle2\rho,\nu_b\rangle}$. By Lemma \ref{lemma-V_b-fil}, $\widetilde{V}_b$ is a finite extension of $V_b$. Thus it suffices to show $\Hom(V_b,N_b)\in \Lambda\mbox{-}\mathrm{mod}^{\geq \langle2\rho,\nu_b\rangle}$. By definition, we have $M_{\xi,b}=V_b\otimes_{\cO(\hat{A})}D_\xi$. Therefore we have
    $$\Hom(M_{\xi,b},N_b)\simeq \Hom_{\cO(\hat{A})}(D_\xi,\Hom(V_b,N_b)).$$
    Now the claim follows from Lemma \ref{lemma-hom-alg} below, by noting that $D_\xi=\bigoplus_\chi\Lambda_\chi$ and each $\Lambda_\chi$ is cutting out by a regular sequence in $\cO(\hat{A})$.
\end{proof}

\begin{lemma}\label{lemma-hom-alg}
    Let $R$ be a ring. Let $x_1,\dots,x_n$ be a regular sequence in $R$. Let $M$ be a complex of $R$-modules that is supported on the closed subset $V(x_1,\dots,x_n)$ of $\Spec R$. Then $M$ is coconective if and only if
    $$\Hom(R/(x_1,\dots,x_n),M)$$
    is coconnective.
\end{lemma}
\begin{proof}
    If $n=1$, then there is a short exact sequence
    $$0\to \mathrm{coker}(H^{i-1}(M)\xrightarrow{x} H^{i-1}(M))\to H^i(\Hom(R/x_1,M))\to \ker(H^i(M)\xrightarrow{x} H^i(M))\to 0.$$
    As $H^i(M)$ is supported on $V(x_1)$, the kernel and cokernel of $H^i(M)\xrightarrow{x}H^i(M)$ is non-zero if and only if $H^i(M)$ is non-zero. This proves the case $n=1$. In general, it follows from induction on $n$.
\end{proof}

Let $\bar\lambda\in \XX_\bullet(T)^+_\phi$ be a dominant coweight. By Proposition \ref{cor-tilting-tensor-twisted-case}, the object $V^\vee\otimes \Xi_{\xi,\bar\lambda}$ is a retract of an object that admits a filtration by retracts of $\Xi_{\xi,\bar\mu}$ for $\bar\mu\in \XX_\bullet(T)^+_\phi$. Therefore the object $T_{V^\vee}((i_b)_*M_{\xi,b}[-\langle2\rho,\nu_b\rangle])$ is a retract of an object that admits a filtration by retracts of $(i_{b'})_*M_{b',\xi}[-\langle2\rho,\nu_{b'}\rangle]$ for $b'\in B(G)_\unr$. Hence if $\cF\in\Shv^{\widehat\unip}(\Isoc_G,\Lambda)_\xi^{e,\geq0}$, then
$$\Hom((i_b)_*M_{\xi,b}[-\langle2\rho,\nu_b\rangle],T_V(\cF))\simeq \Hom(T_{V^\vee}((i_b)_*M_{\xi,b}[-\langle2\rho,\nu_b\rangle]),\cF)$$
lies in $\Lambda\mbox{-}\mathrm{mod}^{\geq 0}$. Thus $T_V(\cF)$ lies in $\Shv^{\widehat\unip}(\Isoc_G,\Lambda)_\xi^{e,\geq0}$ by Lemma \ref{lemma-generic-unr-detect-t-str}.
\end{proof}

If $\Lambda=\overline{\QQ}_\ell$, we can prove a stronger result. 
\begin{thm}\label{thm-char-0-t-str}
    Assume $\Lambda=\overline{\QQ}_\ell$. Then the functor
    $$\LL^{\unip}_{G,\xi}\colon\Shv^{\widehat\unip}(\Isoc_G,\Lambda)_\xi\xrightarrow{\simeq} \Ind\Coh(V^\wedge_\xi)$$
    is compatible with the exotic $t$-structure on the left hand side and the standard $t$-structure on the right hand side.
\end{thm}
\begin{rmk}
    We expect similar results hold true when $\Lambda=\overline{\FF}_\ell$. %In fact, if one can show that $\LL^{\unip}_G$ sends the Iwahori--Whittaker representation to the structure sheaf on $\Loc^{\widehat\unip}_{{}^LG,F}$, then the same method applies.
   % \Xinwen{Probably over the generic locus, this is not that difficult to prove? The coh springer sheaf restricted to this part is just $O\otimes_{O_{\hat{T}//W}}O_{\hat{T}}$. Probably can match $W$ action with the finite Hecke algebra action? Probably then anti-spherical module is given by trivial rep? Another thought is instead of AB equivalence, one probably can use geometric Casselman-Shalika?}
\end{rmk}

\begin{proof}
    Recall that in characteristic 0, the fully faithful embedding $\Shv^{\widehat\unip}(\Isoc_G,\Lambda)\to \Ind\Shv^{\unip}_{\fingen}(\Isoc_G,\Lambda)$ is an equivalence.
    Let $\mathrm{IW}^\unip\in \Rep^{\widehat\unip}(G(F),\Lambda)$ denote the unipotent component of the Iwahori--Whitakker representation. By \cite[Theorem 5.3 (2)]{Tame}, we have
    $$\LL^{\unip}_G((i_1)_* \mathrm{IW}^\unip)=\cO_{\Loc^{\widehat\unip}_{{}^LG,F}}.$$
    Let $\cF\in\Shv^{\widehat\unip}(\Isoc_G,\Lambda)_\xi$ be a connective (resp. coconnective) object. Then $\LL^{\unip}_{G}(\cF)$ is connective (resp. coconnective) if and only if the cohomology
    $$R\Gamma(\Loc^{\widehat\unip}_{{}^LG,F},V\otimes\LL^{\unip}_{G}(\cF))$$
    is connective (resp. coconnective) for any finite dimensional representation $V\in\Rep(\hat{G})^\heartsuit$. This is equivalent to
    $$\Hom((i_1)_*\mathrm{IW}^\unip, T_V(\cF))\simeq \Hom(\mathrm{IW}^\unip, (i_1)^!T_V(\cF))$$
    being connective (resp. coconnective). By Theorem \ref{thm-Hecke-t-exact}, the object $T_V(\cF)$ is connective (resp. coconnective). The functor $(i_1)^!$ is right $t$-exact by definition, and is left $t$-exact as $i_1$ is a closed embedding and hence $(i_1)^!\simeq (i_1)^\sharp$. It follows that $(i_1)^!T_V(\cF)$ is connective (resp. coconnective). The theorem follows as $\mathrm{IW}^\unip$ is compact projective.
\end{proof}

\subsubsection{Generic part of the cohomology of affine Deligne--Lusztig varieties.}\label{subsubsection-adlv}

Fix a dominant coweight $\mu\in \XX_\bullet(T)^+$. Let $\mu^*=-w_0(\mu)$. Recall the definition of admissible set
$$\Adm(\mu^*)=\{w\in\widetilde{W}| w\leq t_{\lambda}, \text{for some }\lambda\in W\mu^*\}.$$
Let $\Sht^{\loc}_{\cI,\mu}$ is the closed substack of $\Sht^\loc_\cI$ classifying local shtukas $(\cE,\phi_\cE)$ where the modification $\phi_\cE$ is bounded by $\Adm(\mu^*)\subseteq \widetilde{W}$. 
\begin{rmk}
    Our notation is consistent with \cite{XZ-cycles}, \cite{PR-p-adic-shtuka} and \cite{DHKZ-igusa}, but different from \cite{Tame}, where the stack of local shtukas bounded by $\Adm(\mu^*)$ is denoted by $\Sht^\loc_{\cI,\mu^*}$.
\end{rmk}
Let $B(G,\mu^*)\subseteq B(G)$ be the subset of elements $b$ with $\nu_b\leq \mu^\diamond$ and $\kappa(b)=[\mu]$.
For $b\in B(G,\mu^*)$, recall the definition of affine Deligne--Lusztig variety:
$$X_{\cI,\mu}(b)\coloneqq \{g\Iw\in \Fl_{\cI} | g^{-1}b\phi(g)\text{ bounded by }\Adm(\mu^*)\}.$$
Then $X_{\cI,\mu}(b)$ carries an action of $G_b(F)$. Moreover, it fits into a Cartesian diagram
$$\begin{tikzcd}
    G_b(F)\backslash X_{\cI,\mu}(b) \ar[r]\ar[d] & \Sht^\loc_{\cI,\mu} \ar[d,"\Nt"] \\
    \BB G_b(F) \ar[r,"i_b"] & \Isoc_{G}.
\end{tikzcd}$$
Let 
$$\iota_{\mu,b}\colon  X_{\cI,\mu}(b)\to \Iw\backslash LG/\Iw$$ 
be the natural morphism sending $g$ to $\phi(g)^{-1}b^{-1}g$. 
Let $Z_\mu^\Til\in \Shv_\fingen(\Iw\backslash LG/\Iw,\Lambda)$ be the central sheaf associated to the indecomposible tilting representation $V^\Til_\mu\in \Rep(\hat{G})$ with highest weight $\mu$.
By base change, we have as isomorphism
$$(i_b)^!T_{V^\Til_\mu}(i_1)_*(\cInd_I^{G(F)}\Lambda)= (i_b)^!\Nt_*\delta^!Z^\Til_\mu\simeq R\Gamma_c(X_{\cI,\mu}(b),(\iota_{\mu,b})^!Z_\mu^\Til)$$
in $\Rep(G_b(F),\Lambda)$. In particular, $R\Gamma_c(X_{\cI,\mu}(b),(\iota_{\mu,b})^!Z_\mu^\Til)$ carries a natural action of $W_E\times H_I\times G_b(F)$. For $\xi\in (\hat{A}\git W_0)(\Lambda)$, let $R\Gamma_c(X_{\cI,\mu}(b),(\iota_{\mu,b})^!Z^\Til_\mu)_{(\xi)}$ denote the localization of $R\Gamma_c(X_{\cI,\mu}(b),(\iota_{\mu,b})^!Z^\Til_\mu)$ at the maximal ideal of $Z(H_I)\simeq \cO(\hat{A}\git W_0)$ defining $\xi$.

\begin{rmk}
    Suppose $\mu$ is minuscule. Then $V_\mu^\Til=V_\mu$ is the highest weight representation. We expect that $R\Gamma_c(X_{\cI,\mu}(b),(\iota_{\mu,b})^!Z^\Til_\mu)$ is equal to the compactly supported cohomology of the associated local Shimura variety $\Sht_{G,b,\mu,I}$ defined in \cite[\S 24.1]{SW-Berkeley-notes}
\end{rmk}

\begin{thm}\label{thm-coh-vanishing-ADLV}
    If $\xi$ is generic, then $R\Gamma_c(X_{\cI,\mu}(b),(\iota_{\mu,b})^!Z_\mu^\Til)_{(\xi)}$ sits in degrees less than or equal to $\langle2\rho,\nu_b\rangle$. If moreover $\xi$ is strongly generic, then $R\Gamma_c(X_{\cI,\mu}(b),(\iota_{\mu,b})^!Z_\mu^\Til)_{(\xi)}$ is concentrated in degree $\langle2\rho,\nu_b\rangle$.
\end{thm}
\begin{proof}
    By Proposition \ref{prop-split-sreg} and Theorem \ref{thm-Hecke-t-exact}, we know that the completion at $\xi$
    $$R\Gamma_c(X_{\cI,\mu}(b),(\iota_{\mu,b})^!Z^\Til_\mu)\otimes_{\cO(\hat{A}\git W_0)}\cO((\hat{A}\git W_0)^\wedge_\xi)\simeq (i_b)^!T_{V^\Til_\mu}(i_1)_*(\cInd_I^{G(F)}\Lambda\otimes_{\cO(\hat{A}\git W_0)}\cO((\hat{A}\git W_0)^\wedge_\xi))$$
    is concentrated in degrees less than or equal to $\langle2\rho,\nu_b\rangle$ if $\xi$ is generic, and is concentrated in degree $\langle2\rho,\nu_b\rangle$ if $\xi$ is strongly generic.
    As $\cO((\hat{A}\git W_0)^\wedge_\xi))$ is a faithfully flat $\cO(\hat{A}\git W_0)_{(\xi)}$-module, we see that $R\Gamma_c(X_{\cI,\mu}(b),(\iota_{\mu,b})^!Z^\Til_\mu)_{(\xi)}$ satisfies the same bound on cohomological degrees.
\end{proof}

\section{Generic part of the cohomology of Shimura varieties}\label{section-shimura}
In this section, we prove our main theorems about cohomology of Shimura varieties. We will let $p$ be an odd prime in this section. We will fix an isomorphism $\CC\simeq \overline{\QQ}_p$.

\subsection{Igusa sheaves}
We first fix a few notations and terminologies that will be used throughout this section. Let $\sG$ be a reductive group over $\QQ$. Let $\sX$ be a $\sG(\RR)$-conjugacy classes of homomorphisms $h\colon \Res_{\CC/\RR}\GG_m\to \sG_\RR$ that satisfies
\begin{enumerate}[(SV1)]
    \item $\Ad\circ h\colon \Res_{\CC/\RR}\GG_m\to \GL(\fg)$ endows $\fg$ with a Hodge structure of type $(-1,1),(0,0),(1,-1)$.
    \item $\Ad\circ h(i)$ is a Cartan involution on $\sG_\RR^{\mathrm{ad}}$.
\end{enumerate}
Then $(\sG,\sX)$ form a Shimura datum. In the literature, a Shimura datum is usually required to satisfy an additional axiom
\begin{enumerate}[(SV3)]
    \item $\sG_\RR^{\mathrm{ad}}$ does not contain a $\QQ$-factor whose $\RR$-points are compact. %\Xinwen{Probably can remove this assumption, by \cite{XZ-cycles}.}
\end{enumerate}
However it is not needed for our purpose as in \cite{XZ-cycles}. The composition $\GG_{m,\CC}\hookrightarrow (\Res_{\CC/\RR}\GG_m)_\CC\xrightarrow{h}\sG_\CC$ defines a $G(\CC)$-conjugacy classes $\{\mu\}$ of cocharacters of $\sG_\CC$. The reflex field $\sE\subseteq \CC$ is the defining field the the conjugacy class $\{\mu\}$. Recall that $(\sG,\sX)$ is of Hodge type if there exists an embedding $(\sG,\sX)\hookrightarrow (\mathsf{H},\sX_{\mathsf{H}})$ where $(\mathsf{H},\sX_{\mathsf{H}})$ is a Siegal Shimura datum. Recall that a Shimura datum $(\sG,\sX)$ is of abelian type if there exists a Shimura datum $(\sG',\sX')$ of Hodge type and a central isogeny $\sG'_\der\to \sG_\der$ which inducing an isomorphism $(\sG'_\ad,\sX'_\ad)\simeq (\sG_\ad,\sX_\ad)$.

Let $(\sG,\sX)$ be a Shimura datum of Hodge type with reflex field $\sE$. Write $G=\sG_{\QQ_p}$. Let $\cI$ be an Iwahori model of $G$ over $\ZZ_p$. Denote $I=\cI(\ZZ_p)\subseteq G(\QQ_p)$. Let $K^p\subseteq \sG(\AAA_f^p)$ be a neat open compact subgroup. Let $\sSh_{K^pI}(\sG,\sX)$ denote the Shimura variety associated to $(\sG,\sX)$ with level $K^pI$. Then $\sSh_{K^pI}(\sG,\sX)$ is a quasi-projective smooth variety over $\sE$.

The fixed isomorphsim $\CC\simeq\overline{\QQ}_p$ induces a $p$-adic place $v$ of $\sE$. Let $E$ denote the local field $\sE_v$ and $O_E$ denote the ring of integers in $E$. Let $\{\mu\}$ denote the $G(\overline\QQ_p)$-conjugacy class of cocharacters of $G_{\overline\QQ_p}$ defined by the isomorphism $\CC\simeq \overline{\QQ}_p$. Then $E$ is identified with the local reflex field of the conjugacy class $\{\mu\}$. Let $\mu$ denote the dominant representative of $\{\mu\}$ in $\XX_\bullet(T)$, where $T$ is the canonical Cartan subgroup of $G$. 

By \cite{PR-p-adic-shtuka}, there is a normal flat integral model $\scS_{K^pI}(\sG,\sX)$ over $O_E$ of $\sSh_{K^pI}(\sG,\sX)$ satisfying certain properties. Let $k=\overline\FF_p$ and $\Sh_{\mu}=\Sh_{\mu,K^p}\coloneqq \scS_{K^pI}(\sG,\sX)^\perf_k$ denote the perfection of the special fiber of $\scS_{K^pI}(\sG,\sX)$ over $k$. Then there exists a crystalline period map
$$\loc_p\colon \Sh_\mu\to \Sht^{\loc}_{\cI,\mu}$$
by \cite{SYZ-EKOR} and \cite{Hoff-parahoric-display}.
Here $\Sht^\loc_{\cI,\mu}\subseteq \Sht^\loc_\cI$ is defined in \S\ref{subsubsection-adlv}. 
Let $\Isoc_{G,\leq\mu^*}$ denoted the closed substack of $\Isoc_G$ associated to the subset $B(G,\mu^*)\subseteq B(G)$.

The following result on existing of perfect Igusa stacks in crucial to us.

\begin{thm}\label{thm-igusa}
    There exists a perfect stack $\Igs=\Igs_{K^p}$ over $k$ such that there exists a Cartesian diagram
    $$\begin{tikzcd}
        \Sh_\mu \ar[r,"\loc_p"]\ar[d,"\Nt^\glob"swap] \ar[rd,phantom,"\square", very near start]& \Sht^\loc_{\cI,\mu} \ar[d,"\Nt"] \\
        \Igs \ar[r,"\loc_p^0"] & \Isoc_{G,\leq\mu^*}
    \end{tikzcd}$$
    of perfect \'etale stacks over $k$. In addition, $\Nt$ is ind-pfp proper and $\loc_p$ is pseudo coh. pro-smooth.
\end{thm}
\begin{proof}
    This is proved in \cite[\S 6.5.3]{DHKZ-igusa} modulo the difference between \'etale topology and $h$-topology. See also \cite[Proposition 6.4]{Tame} for the proof.
    The morphism $\loc_p$ is pseudo coh. pro-smooth by \cite[Lemma 6.2]{Tame}. %In particular, the morphism $\Nt^\glob$ is ind-proper and surjective. 
\end{proof}

\subsubsection{Sheaves on the perfect Igusa stack}

Let $\ell\neq p$ be a prime. Let $\Lambda=\overline\QQ_\ell$ or $\overline\FF_\ell$ as before. We recall the discussion on $\Shv(\Igs,\Lambda)$ in \cite[\S 6.1.3]{Tame}.

The morphism $\Sht_{\cI,\mu}^\loc\to \Isoc_{G,\leq \mu^*}$ defines a \v{C}ech nerve $\Hk_\bullet(\Sht^\loc_{\cI,\mu})$. Pullback of the groupoid $\Hk_\bullet(\Sht^\loc_{\cI,\mu})$ to $\Sh_\mu$ defines a groupoid $\Hk_\bullet(\Sh_\mu)$ over $\Sh_\mu$ by Theorem \ref{thm-igusa}. Then $\Igs$ is identified with the \'etale sheafification of the geometric realization of $\Hk_\bullet(\Sh_\mu)$. It follows that 
$$\Shv(\Igs,\Lambda)\simeq |\Shv(\Hk_\bullet(\Sh_\mu),\Lambda)|$$
where the transitioning functors are given by $*$-pushforwards. By \cite[Lemma 6.7]{Tame}, the category $\Shv(\Igs,\Lambda)$ is compactly generated by objects of the form $(\Nt^\glob)_*\cF$ for $\cF\in\Shv_c(\Sh_\mu,\Lambda)$. There is a canonical duality
$$\DD^\can_{\Igs}\colon \Shv(\Igs,\Lambda)^\vee \xrightarrow{\sim} \Shv(\Igs,\Lambda)$$
on $\Shv(\Igs,\Lambda)$ induced by the Verdier duality on $\Shv(\Sh_\mu,\Lambda)$. Let 
$$(\DD^\can_{\Igs})^\Adm\colon (\Shv(\Igs,\Lambda)^\Adm)^{\mathrm{op}}\xrightarrow{\simeq}\Shv(\Igs,\Lambda)^\Adm$$
denote the admissible realization of $\DD^\can_{\Igs}$. By \cite[Lemma 6.7]{Tame}, the dualizing sheaf $\omega_{\Igs}$ lies in $\Shv(\Igs,\Lambda)^\Adm$.

\begin{defn}\label{def-omega-can}
    Define the object    $$\omega_{\Igs}^\can\coloneqq (\DD^\can_{\Igs})^\Adm(\omega_{\Igs})\in \Shv(\Igs,\Lambda)^\Adm.$$
\end{defn}

For $\cF\in\Shv(\Sh_\mu,\Lambda)$, there is a canonical isomorphism
$$\Hom((\Nt^\glob)_*\cF,\omega_{\Igs}^\can)= R\Gamma(\Sh_\mu,\cF)^\vee.$$
As explained in \cite[Lemma 6.20]{Tame}, we have $\omega_{\Igs}^\can\simeq \omega_{\Igs}$ if  $\Sh_\mu$ is proper.

\begin{rmk}\label{rmk-omega-can-on-scheme}
    If $X$ is a scheme of finite type over $k$, we set $\omega_X^\can\coloneqq (\DD_X)^\Adm(\omega_X)$, where $\DD_X:\Shv(X)^\vee\cong\Shv(X)$ is the usual Verdier duality. There is a natural isomorphism $\omega_X^\can\simeq f^\sharp\omega_{\Spec k}$ for $f\colon X\to \Spec k$, where $f^\sharp$ is the right adjoint of $f_*$. For an object $\cF\in\Shv(X,\Lambda)$, we have
    $$\Hom(\cF,\omega_X^\can)\simeq R\Gamma(X,\cF)^\vee.$$
    If particular, if $\cF$ is constructible, then 
    $$\Hom(\cF,\omega_X^\can)\simeq R\Gamma_c(X,(\DD_X)^\omega(\cF)).$$
    
    The object $\omega_{\Igs}^\can$ is the descent of $\omega_{\Sh_\mu}^\can$ along the morphism $\Nt^\glob\colon \Sh_\mu\to \Igs$, in the sense that $\omega_{\Sh_\mu}^{\can}\cong (\Nt^\glob)^! \omega_{\Igs}^{\can}$.
\end{rmk}

By \cite[Proposition 6.12]{Tame}, the commutative diagram
$$\begin{tikzcd}
    \Shv(\Isoc_{G,\leq \mu^*},\Lambda) \ar[r,"(\loc_p^0)^!"]\ar[d,"(\Nt)^!"swap] & \Shv(\Igs,\Lambda) \ar[d,"(\Nt^\glob)^!"] \\
    \Shv(\Sht^\loc_{\cI,\mu},\Lambda) \ar[r,"(\loc_p)^!"] & \Shv(\Sh_\mu,\Lambda)
\end{tikzcd}$$
is right adjointable in $\mathrm{Lincat}_\Lambda$. This means that the natural morphism
$$(\Nt)^!\circ(\loc_p^0)_\flat\to (\loc_p)_\flat\circ (\Nt^\glob)^!$$
defined by adjunction is an isomorphism. %Taking left adjoint, this implies that
%$$(\Nt^\glob)_*\circ(\loc_p^0)^!\to (\loc_p)^!\circ (\Nt)_*$$
%is an isomorphism. 

In \cite[Proposition 6.9, Remark 6.10]{Tame}, we defined the object
$$\fI\coloneqq (\loc_p^0)_\flat\omega_{\Igs}\in \Shv(\Isoc_{G,\leq\mu^*},\Lambda)$$
called the \emph{Igusa sheaf}. We will also need a variant of it in the sequel.

\begin{defn}\label{def-Igusa-sheaf}
    Define the \emph{$!$-Igusa sheaf}
    $$\fI^\can\coloneqq (\loc_p^0)_\flat \omega_{\Igs}^\can \in \Shv(\Isoc_{G,\leq \mu^*},\Lambda).$$
\end{defn}

We have the following results on the $!$-Igusa sheaf.

\begin{prop}\label{prop-!-Igusa-dual}
    There is a natural isomorphism
    $$\fI^\can\simeq (\DD^\can_{\Isoc_{G,\leq\mu^*}})^\Adm(\fI).$$
\end{prop}
\begin{proof}
    By \cite[Lemma 6.8]{Tame}, we have a  canonical isomorphism
    $$(\DD_{\Igs}^\can)^\omega\circ (\loc_p^0)^!\simeq (\loc_p^0)^!\circ (\DD^\can_{\Isoc_{G,\leq\mu^*}})^\omega.$$
    By \cite[Lemma 7.40]{Tame}, it induces a canonical isomorphism
    $$(\loc_p^0)_\flat\circ (\DD_{\Igs}^\can)^\Adm\simeq (\DD^\can_{\Isoc_{G,\leq\mu^*}})^\Adm\circ (\loc_p^0)_\flat.$$
    The claim now follows by applying the above isomorphism to $\omega_{\Igs}$.
\end{proof}

\begin{prop}\label{prop-!-Igusa-to-Igusa}
    There is a natural morphism $\fI^\can\to \fI$.
\end{prop}
\begin{proof}
    The morphism is induced by a natural morphism $\omega_{\Igs}^\can\to \omega_{\Igs}$ we now define. As explained in Remark \ref{rmk-omega-can-on-scheme}, we have $(\Nt^\glob)^!\omega^\can_{\Igs}\simeq \omega^\can_{\Sh_\mu}$, where $\omega^\can_{\Sh_\mu}=\pi^\sharp \omega_{\Spec k}$ for $\pi\colon\Sh_\mu\to \Spec k$.
    Over $\Sh_\mu$, there is a natural morphism
    $$\omega_{\Sh_\mu}^\can\to \omega_{\Sh_\mu}$$
    defined by taking right adjoint of the natural transfrom $\pi_!\to \pi_*$. On each level of the \v{C}ech nerve $\Hk_\bullet(\Sh_\mu)$, we also have canonical morphisms
    $$\omega_{\Hk_\bullet(\Sh_\mu)}^\can\to \omega_{\Hk_\bullet(\Sh_\mu)}$$
    as all the transition morphisms in $\Hk_\bullet(\Sh_\mu)$ are ind-proper. It descends to a natural morphism $\omega_{\Igs}^\can\to \omega_{\Igs}$.
\end{proof}

\subsubsection{Well-positioned subschemes}
To compute the stalks of the (!)-Igusa sheaves, we need to recall the notion of well-positioned subsets defined in \cite{Lan-Stroh-compactification-subsch}. Before that, we need to recall compactifications of Shimura varieties. As we are mainly interested in the special fiber of Shimura variety, we shall only work with the special fiber. 

Denote $K=K^pI\subseteq \sG(\AAA_f)$. Let
$$S_K(\sG,\sX)\coloneqq \scS_K(\sG,\sX)\otimes_{O_E}k$$
denote the special fiber of $\scS_K(\sG,\sX)$. Let $S^\mini_K(\sG,\sX)$ denote the special fiber of the integral minimal compactification of $\scS_K(\sG,\sX)$ defined in \cite[Theorem 3]{Pera-integral-compactification}. Then $S^\mini_K(\sG,\sX)$ is projective and normal over $k$. The boundary of $S^\mini_K(\sG,\sX)$ admits a stratification by locally closed subschemes indexed by the set of cusp labels $\Cusp_K(\sG,\sX)$. For $\Phi\in \Cusp_K(\sG,\sX)$, let $S_\Phi$ denote the associated stratum. Thus
$$S^\mini_K(\sG,\sX)=S_K(\sG,\sX)\sqcup\bigsqcup_{\Phi\in\Cusp(\sG,\sX)}S_\Phi.$$
Each stratum $S_\Phi$ is a finite quotient of a smaller Shimura variety of Hodge type.

We fix a complete admissible rpcd cone decomposition $\Sigma=\{\Sigma_\Phi\}_{\Phi\in\Cusp_K(\sG,\sX)}$ as in \cite{Pera-integral-compactification}. By \cite[Theorem 1]{Pera-integral-compactification}, there is a toroidal compactification $S^\tor_K(\sG,\sX)$
of $S_K(\sG,\sX)$ associated to $\Sigma$. There is a surjective proper morphism
$$\oint_{K,\Sigma}\colon S^\tor_K(\sG,\sX)\to S^\mini_K(\sG,\sX).$$

We describe the boundary of $S^\tor_K(\sG,\sX)$ following \cite[Proposition 2.1.2]{Lan-Stroh-compactification-subsch}. Let $\Phi\in\Cusp_K(\sG,\sX)$. There is a subset $\Sigma^+_\Phi\subseteq \Sigma_\Phi$ equipped with an action of an arithmetic group $\Gamma_\Phi$. The preimage $(\oint_{K,\Sigma})^{-1}(S_\Phi)$ admits a stratification $\{S_{\Phi,[\sigma]}\}$ indexed by $[\sigma]\in \Sigma_\Phi^+/\Gamma_\Phi$.
There is a torus $\sT_\Phi$ over $\Spec\ZZ$, a proper surjective morphism $C_\Phi\to S_\Phi$ and an $\sT_\Phi$-torsor $\Xi_\Phi\to C_\Phi$. For each element $\sigma\in \Sigma_\Phi$, there is an affine toroidal embedding
$$
\Xi_\Phi\hookrightarrow\Xi_\Phi(\sigma)
$$
relative to $C_\Phi$, and a $\sT_\Phi$-invariant closed subscheme 
$$
\Xi_{\Phi,\sigma}\hookrightarrow \Xi_\Phi(\sigma).
$$ 
If $\sigma\in \Sigma^+_\Phi$ with image $[\sigma]\in \Sigma^+_\Phi/\Gamma_\Phi$, there is a canonical isomorphism
$$
(\Xi_\Phi(\sigma))^\wedge_{\Xi_{\Phi,\sigma}}\simeq  (S^\tor_K(\sG,\sX))^\wedge_{S_{\Phi,[\sigma]}}
$$
between formal completions, and which restricts to a canonical isomorphism $\Xi_{\Phi,\sigma}\simeq S_{\Phi,[\sigma]}.$
By \cite[Proposition 2.1.2 (9)]{Lan-Stroh-compactification-subsch}, for each closed point $x\in S_{\Phi,[\sigma]}$, there exists a correspondence
$$S^\tor_K(\sG,\sX)\leftarrow U\rightarrow \Xi_\Phi(\sigma)$$
such that two morphisms are \'etale and $U\to S^\tor_K(\sG,\sX)$ is an \'etale neighborhood of $x$. Moreover, the stratification on $U$ induced from $S^\tor_K(\sG,\sX)$ is equal to the pullback of the stratification on $\Xi_\Phi(\sigma)$ indexed by faces of $\sigma$. In particular, the preimage of $S_K(\sG,\sX)$ and the preimage of $\Xi_\Phi$ agree in $U$.

We recall the definition of well-positioned subschemes.
\begin{defn}\label{def-well-positioned-subscheme}
    Let $Y\subseteq S_K(\sG,\sX)$ be a locally closed subscheme. We say that $Y$ is \emph{well-positioned} if there exists a collection
    $$Y_\Phi^\natural\subseteq S_\Phi,\quad \Phi\in\Cusp_K(\sG,\sX)$$
    of locally closed subschemes satisfying the following condition:
    Let $\Phi\in \Cusp_K(\sG,\sX)$, $\sigma\in \Sigma_\Phi^+$ and let 
        $$\Spf R\subseteq (S^\tor_K(\sG,\sX))^\wedge_{S_{\Phi,[\sigma]}}$$
    be an affine open formal subscheme. Denote $W=\Spec R$. Let $W^\circ$ denote the preimage of $S_K(\sG,\sX)$ in $W$.  Then the pullback of $Y$ along $W^\circ\to S_K(\sG,\sX)$ is equal to the pullback of $Y_\Phi^\natural$ along the composition
    $$W^\circ\subseteq W\to \Xi_\Phi(\sigma)\to C_\Phi\to S_\Phi.$$
    Here we use the canonical isomorphism $(\Xi_\Phi(\sigma))^\wedge_{\Xi_{\Phi,\sigma}}\simeq(S^\tor_K(\sG,\sX))^\wedge_{S_{\Phi,[\sigma]}}$.
\end{defn}

Let $Y$ be a well-positioned subscheme of $S_K(\sG,\sX)$ with associated collection $\{Y_\Phi^\natural\}$. Let $\overline{Y}$ denote the schematic closure of $Y$ in $S_K(\sG,\sX)$. Let $\overline{Y}^\mini$ (resp. $\overline{Y}^\tor$) denote the schematic closure of $\overline{Y}$ in $S^\mini_K(\sG,\sX)$ (resp. $S^\tor_K(\sG,\sX)$). Denote $Y_0=\overline{Y}\backslash Y$ and let $Y_0^\mini$ (resp. $Y_0^\tor$) denote the closure of $Y_0$ in $S^\mini_K(\sG,\sX)$ (resp. $S^\tor_K(\sG,\sX)$). Define the partial minimal (resp. toroidal) compactification of $Y$ by
$$Y^\mini=\overline{Y}^\mini\backslash Y_0^\mini,\quad\text{resp.}\quad Y^\tor=\overline{Y}^\tor\backslash Y_0^\tor.$$
For each $\Phi\in\Cusp_K(\sG,\sX)$ and $[\sigma]\in \Sigma_\Phi^+/\Gamma_\Phi$, let $Y_{\Phi,[\sigma]}$ denote $Y^\tor\times_{S^\tor_K(\sG,\sX)}S_{\Phi,[\sigma]}$. For $\sigma\in \Sigma_\Phi^+$, let $Y^\natural_?$ denote the pullback of $Y^\natural_\Phi$ to $?$ for $?=\Xi_\Phi(\sigma)$, $\Xi_{\Phi,\sigma}$, or $C_\Phi$.
By \cite[Theorem 2.3.2 (5)]{Lan-Stroh-compactification-subsch}, For $\sigma\in \Sigma_\Phi^+$ with image $[\sigma]\in \Sigma_\Phi^+/\Gamma_\Phi$, there is an canonical isomorphism
$$(Y^\natural_{\Xi_\Phi(\sigma)})^\wedge_{Y_{\Xi_{\Phi,\sigma}}^\natural}\simeq (Y^\tor)^\wedge_{Y_{\Phi,\sigma}}$$
induced by the canonical isomorphism $(\Xi_\Phi(\sigma))^\wedge_{\Xi_{\Phi,\sigma}}\simeq(S^\tor_K(\sG,\sX))^\wedge_{S_{\Phi,[\sigma]}}$. In particular, it induces a canonical isomorphism $Y_{\Xi_{\Phi,\sigma}}^\natural\simeq Y_{\Phi,\sigma}$.

The notion of well-position subschemes is generalized in \cite[\S 2.3]{Mao-Hodge-well-positioned} in order to treat the case of Igusa schemes. Let $Y\subseteq S_K(\sG,\sX)$ be a well-positioned subscheme. Let $\widetilde{Y}^\tor\to Y^\tor$ be a morphism between schemes. Let $\widetilde{Y}$ denote the pullback of $\widetilde{Y}^\tor$ along $Y\hookrightarrow Y^\tor$.
For $\Phi\in\Cusp_K(\sG,\sX)$, let $\widetilde{Y}_{C_\Phi}^\natural$ be a scheme over $Y_{C_\Phi}^\natural$.

\begin{defn}
    We say that $\widetilde{Y}$ is a \emph{well-positioned} with respect to $\{\widetilde{Y}_{C_\Phi}^\natural\}_{\Phi\in\Cusp_K(\sG,\sX)}$ if $\widetilde{Y}\to Y$ is affine and flat, and for any $\Phi\in\Cusp_K(\sG,\sX)$, $\sigma\in \Sigma_\Phi^+$, and any affine open subscheme 
    $$\Spf(R)\subseteq (Y^\tor)^\wedge_{Y_{\Phi,\sigma}} \simeq(Y^\natural_{\Xi_\Phi(\sigma)})^\wedge_{Y^\natural_{\Xi_{\Phi,\sigma}}},$$ 
    the pullback of $\widetilde{Y}^\tor\to Y^\tor$ and the pullback of $\widetilde{Y}^\natural_{C_\Phi}\to Y_{C_\Phi}^\natural$ to $W=\Spec(R)$ are isomorphic.
\end{defn}

The following Lemma is a slight generalization of \cite[Lemma 2.25]{Mao-Hodge-well-positioned}.

\begin{lemma}\label{lemma-well-positioned-etale-nbhd}
    Assume that $\widetilde{Y}^\tor\to Y^\tor$ is finitely presented. For $\Phi\in\Cusp_K(\sG,\sX)$, $\sigma\in \Sigma^+_\Phi$ and a closed point $x\in Y_{\Phi,\sigma}$, there exists an \'etale neighborhood $U\to S^\tor_K(\sG,\sX)$ of $x$ and an \'etale morphism $U\to \Xi_\Phi(\sigma)$ such that the following conditions hold:
    \begin{enumerate}
        \item The pullback of $\widetilde{Y}^\tor\to S^\tor_K(\sG,\sX)$ and the pullback of $\widetilde{Y}_{C_\Phi}^\natural\to C_\Phi$ to $U$ are isomorphic.
        \item The stratification on $U$ induced from $S_K^\tor(\sG,\sX)$ coincides with the stratification on $U$ induced from $\Xi_\Phi(\sigma)$.
    \end{enumerate}
\end{lemma}
\begin{proof}
    By \cite[Proposition 2.1.2 (9)]{Lan-Stroh-compactification-subsch}, we can first choose $U$ so that (2) holds. By Lemma \ref{lemma-Artin-approx}, we can take a further \'etale neighborhood so that (1) holds. Note that by \cite[Lemma 2.35]{Mao-Hodge-well-positioned}, the morphisms $\widetilde{Y}^\tor\to Y^\tor$ and $\widetilde{Y}_{C_\Phi}^\natural\to Y_{C_\Phi}^\natural$ are also affine and flat.
\end{proof}

\begin{lemma}\label{lemma-Artin-approx}
    Let $X_1,X_2,Y_1,Y_2$ be schemes over $k$ of finite type with affine morphisms $f_i\colon Y_i\to X_i$ for $i=1,2$. Let $x_1\in X$, $x_2\in X_2$ be closed points. Assume that we are given an isomorphism $\cO_{X_1,x_1}^\wedge\simeq \cO_{X_2,x_2}^\wedge$ and an isomorphism $Y_1\times_{X_1}\Spf(\cO_{X_1,x_1}^\wedge)\simeq Y_2\times_{X_2}\Spf(\cO_{X_2,x_2}^\wedge)$ over it. Then there exists a correspondence
    $$X_1\leftarrow U\rightarrow X_2$$
    that are \'etale neighborhood of $x_i\in X_i$ for $i=1,2$, such that there is an isomorphism
    $$Y_1\times_{X_1}U\simeq Y_2\times_{X_2}U$$
    over $U$.
\end{lemma}
\begin{proof}
    By \cite[Corollary 2.6]{Artin-approximation}, we can assume that $X_1=X_2=\Spec A$ is affine and $x_1=x_2=x$. Let $\widehat{A}$ denote the completion of $A$ at $x$ and $A^h$ denote the henselization of $A$ at $x$. Write $Y_1=\Spec \frac{A[x_1,\dots,x_n]}{(f_1,\dots,f_{n'})}$ and $Y_2=\Spec \frac{A[y_1,\dots, y_m]}{(g_1,\dots,g_{m'})}$. Thus we have an isomorphism 
    $$\widehat{A}[x_1,\dots,x_n]/(f_1,\dots,f_{n'})\simeq \widehat{A}[y_1,\dots,y_m]/(g_1,\dots,g_{m'}).$$
    Write $x_i=H_i(y_1,\dots,y_m)$ and $y_j=G_j(x_1,\dots,x_n)$ for $H_i,G_j\in\widehat{A}[Y_1,\dots,Y_m]$. The condition that $(H_i)_{i=1,\dots n}$ and $(G_j)_{j=1,\dots m}$ define morphisms between $Y_1$ and $Y_2$ and are inverse to each other can be checked by a finitely many polynomial conditions on the coefficients of $H_i,G_j$. By Artin approximation (\cite[Theorem 1.10]{Artin-approximation}), we can find polynomials $H_i',G_j'$ with coefficients in $A^h$ such that the same conditions hold. Thus there is an isomorphism $Y_1\times_{X_1}\Spec A^h\simeq Y_2\times_{X_2}\Spec A^h$. Such an isomorphism is defined over an \'etale neighborhood of $x$ in $\Spec A$.
\end{proof}

\begin{prop}\label{prop-partial-compactly-supp}
    Let $Y$ be a well-positioned subscheme of $S_K(\sG,\sX)$. Let $\widetilde{Y}\to Y$ be a well-positioned morphism such that $\widetilde{Y}^\tor\to Y^\tor$ is finitely presented. Consider the commutative diagram
    $$\begin{tikzcd}
        \widetilde{Y} \ar[r,"j_{\widetilde{Y}^\tor}"]\ar[d,"i_{\widetilde{Y}}"] & \widetilde{Y}^\tor \ar[d,"i_{\widetilde{Y}^\tor}"]\\
        S_K(\sG,\sX) \ar[r,"j^\tor"] & S_K^\tor(\sG,\sX),
    \end{tikzcd}$$
    the natural morphisms
    $$(j^\tor)_!(i_{\widetilde{Y}})_*\Lambda\to (i_{\widetilde{Y}^\tor})_*(j_{\widetilde{Y}^\tor})_!\Lambda$$
    $$(i_{\widetilde{Y}^\tor})_!(j_{\widetilde{Y}^\tor})_*\Lambda\to (j^\tor)_*(i_{\widetilde{Y}})_!\Lambda$$
    defined by adjunctions are isomorphisms.
\end{prop}
\begin{proof}
    These morphisms are clearly isomorphisms when restricted to $S_K(\sG,\sX)$. If $x\in S^\tor_K(\sG,\sX)$ doesn't lies in the closure of ${Y}^\tor$, then the stalks of two sides are both zero. Thus it suffices to check the statement around the boundary of $Y^\tor$. Let $x\in Y^\tor\backslash Y$ be a closed point. Assume that $x$ lies in $Y_{\Phi,[\sigma]}=Y^\tor\times_{S^\tor_K(\sG,\sX)}S_{\Phi,[\sigma]}$ for $\Phi\in\Cusp_K(\sG,\sX)$ and $[\sigma]\in \Sigma^+_\Phi/\Gamma_\Phi$. Choose a representative $\sigma\in \Sigma_\Phi^+$ of $[\sigma]$. By Lemma \ref{lemma-well-positioned-etale-nbhd}, there is an \'etale neighborhood $U\to S^\tor_K(\sG,\sX)$ of $x$ and an \'etale morphism $U\to \Xi_\Phi(\sigma)$ such that:
    \begin{enumerate}
        \item The stratification on $S^\tor_K(\sG,\sX)$ and the stratification on $\Xi_\Phi(\sigma)$ agree after pullback to $U$.
        \item The pullback of $\widetilde{Y}^\tor\to S_K^\tor(\sG,\sX)$ and the pullback of $Y^\natural_{C_\Phi}\to C_\Phi$ coincide after pullback to $U$.
    \end{enumerate}
    We may take a further \'etale cover of $U$ so that the $\sT_\Phi$-torsor $\Xi_\Phi\to C_\Phi$ splits. Let $\sT_\Phi(\sigma)$ denote the toric variety associated to $\sigma$. Then we obtain an \'etale morphism
    $$U\to \sT_\Phi(\sigma)\times C_\Phi$$
    such that the preimage of $\sT_\Phi\times C_\Phi$ in $U$ agrees with the preimage of $S_K(\sG,\sX)$ in $U$, and the pullback of $\sT_\Phi\times \widetilde{Y}_{C_\Phi}^\natural$ to $U$ agrees with the pullback of $\widetilde{Y}^\tor$ to $U$. Therefore the pullback of $\sT_\Phi\times\widetilde{Y}_{C_\Phi}^\natural$ in $U$ agrees with the pullback of $\widetilde{Y}$ in $U$. Thus it suffices to show that the natural morphisms
    $$(a\times\id)_!(\id\times b)_*\Lambda\to (\id\times b)_*(a\times\id)_!\Lambda$$
    $$(\id\times b)_!(a\times\id)_*\Lambda\to (a\times\id)_*(\id\times b)_!\Lambda$$
    are isomorphisms for
    $$\begin{tikzcd}
        \sT_\Phi\times\widetilde{Y}_{C_\Phi}^\natural \ar[r,"a\times\id"]\ar[d,"\id\times b"] & \sT_\Phi(\sigma)\times\widetilde{Y}_{C_\Phi}^\natural \ar[d,"\id\times b"]\\
        \sT_\Phi\times C_\Phi \ar[r,"a\times\id"] & \sT_\Phi(\sigma)\times C_\Phi,
    \end{tikzcd}$$
    where $a\colon \sT_\Phi\hookrightarrow \sT_\Phi(\sigma)$ and $b\colon \widetilde{Y}_{C_\Phi}^\natural\to C_\Phi$ are natural morphisms. By K\"unneth formula, we have
    $$(a\times\id)_!(\id\times b)_*\Lambda\simeq a_!\Lambda\boxtimes b_*\Lambda\simeq (\id\times b)_*(a\times\id)_!\Lambda$$
    $$(a\times\id)_*(\id\times b)_!\Lambda\simeq a_*\Lambda\boxtimes b_!\Lambda\simeq (\id\times b)_!(a\times\id)_*\Lambda$$
    and hence finish the proof.
\end{proof}

%\begin{lemma}\label{lemma-product-pushforward}
%    Let $X,Y,Z$ be schemes of finite type over $k$. Let $f\colon X\to Y$ be a morphism. Let $\cF\in \Shv_c(X,\Lambda)$ and $\cG\in \Shv_c(Z,\Lambda)$ be constructible sheaves. Then there are natural isomorphisms
%    $$(f\times\id_Z)_!(\cF\boxtimes \cG)\simeq f_!\cF\boxtimes \cG,$$
%    $$(f\times\id_Z)_*(\cF\boxtimes \cG)\simeq f_*\cF\boxtimes \cG.$$
%\end{lemma}
%\begin{proof}
%    This follows from the general six functor formalism as explained in \cite[\S 8.2.1]{Tame}. To benefit readers who are not familiar with the abstract formalism,
%    we prove the first isomorphism, using the proper base change and the porjection formula (which might be more familiar).  Namely, we have
%    $$\begin{aligned}
%        (f\times\id_Z)_!(\cF\boxtimes \cG)&\simeq (f\times\id_Z)_!((\mathrm{pr}_1)^*\cF\otimes (\mathrm{pr}_2)^*\cG) \\
%        &\simeq (f\times\id_Z)_!((\mathrm{pr}_1)^*\cF\otimes (f\times\id_Z)^*(\mathrm{pr}_2)^*\cG) \\
%        &\simeq (f\times\id_Z)_!(\mathrm{pr}_1)^*\cF\otimes (\mathrm{pr}_2)^*\cG\\
%        &\simeq (\mathrm{pr}_1)^*f_!\cF\otimes (\mathrm{pr}_2)^*\cG\\
%        & \simeq f_!\cF\boxtimes \cG.
%    \end{aligned}$$
%\end{proof}

\subsubsection{Stalks of Igusa sheaves}\label{subsubsection-connective-Igusa}
As $(\loc_p^0)^!$ preserves compact objects, its right adjoint $(\loc_p^0)_\flat$ preserves admissible objects. Therefore the object $\fI^\can$ is admissible. For $b\in B(G,\mu^*)$, let $\Ig_b$ denote the Igusa scheme defined in \cite[Definition 6.18]{Mao-Hodge-well-positioned}, with partial minimal compactification 
$$j_{\Ig_b^\mini}\colon \Ig_b\hookrightarrow \Ig_b^\mini$$
defined in \cite[Definition 6.25]{Mao-Hodge-well-positioned}. For any open compact subgroup $K_b\subseteq I_b$, there is an Igusa variety $\Ig_{b,K_b}$ of level $K_b$ and a partial minimal compactification $j_{\Ig_{b,K_b}^\mini}\colon\Ig_{b,K_b}\hookrightarrow\Ig_{b,K_b}^\mini$. Varying $K_b$, we obtain a finite \'etale (resp. finite flat) tower $\{\Ig_{b,K_b}\}_{K_b}$ (resp. $\{\Ig_{b,K_b}^\mini\}_{K_b}$). Then $\Ig_b$ (resp. $\Ig_b^\mini$) is equal to the inverse limit of the tower $\{\Ig_{b,K_b}\}_{K_b}$ (resp. $\{\Ig_{b,K_b}^\mini\}_{K_b}$).

We define the partially compactly supported cohomology groups
$$R\Gamma_{c-\partial}(\Ig_b,\Lambda)\coloneqq R\Gamma(\Ig_b^\mini,(j_{\Ig_b^\mini})_!\Lambda_{{\Ig_b}})$$
and
$$R\Gamma_{\partial-c}(\Ig_b,\Lambda)\coloneqq R\Gamma_c(\Ig_b^\mini,(j_{\Ig_b^\mini})_*\Lambda_{{\Ig_b}}).$$
As each $\Ig_{b,K_b}$ (resp. $\Ig_{b,K_b}^\mini$) are qcqs, we have
$$R\Gamma_{c-\partial}(\Ig_b,\Lambda)\simeq \lim_{\substack{\longrightarrow\\ K_b}}R\Gamma(\Ig_{b,K_b}^\mini,(j_{\Ig_{b,K_b}^\mini})_!\Lambda_{{\Ig_{b,K_b}}})$$
and
$$R\Gamma_{\partial-c}(\Ig_b,\Lambda) \simeq\lim_{\substack{\longrightarrow\\ K_b}}R\Gamma_c(\Ig_{b,K_b}^\mini,(j_{\Ig_{b,K_b}^\mini})_*\Lambda_{{\Ig_{b,K_b}}}).$$
Note that $\Ig_b$ and $\Ig_{b}^\mini$ carries natural actions of $G_b(\QQ_p)$. Therefore $R\Gamma_{c-\partial}(\Ig_b,\Lambda)$ and $R\Gamma_{\partial-c}(\Ig_b,\Lambda)$ are admissible $G_b(\QQ_p)$-representations. 

\begin{prop}\label{prop-igusa-sheaf-stalk}
    For $b\in B(G,\mu^*)$, there are natural isomorphisms
    $$(i_b)^!\fI^\can\simeq R\Gamma_{c-\partial}(\Ig_b,\Lambda)$$ 
    $$(i_b)^\sharp\fI\simeq R\Gamma_{\partial-c}(\Ig_b,\Lambda)$$
    in $\Rep(G_b(\QQ_p),\Lambda)$.
\end{prop}
\begin{rmk}
    It is proved in \cite[Proposition 6.14]{Tame} that $(i_b)^!\fI\simeq R\Gamma(\Ig_b,\Lambda)$. On the other hand, it is not hard to show that $(i_b)^\sharp\fI^\can\simeq R\Gamma_c(\Ig_b,\Lambda)$. 
\end{rmk}
\begin{proof}
    Let $b\in B(G,\mu^*)$. Let $w\in \widetilde{W}$ be a $\phi$-straight element mapping to $b\in B(G,\mu^*)$. By \cite[Theorem 3.1]{He-Kottwitz-Rapoport-conj}, we know that $w$ lies in $\Adm(\mu^*)$. By \cite[Lemma 2.4.3]{Kim-central-leaf}, we can choose a completely slope divisible lifting $\dot{w}\in N_G(T)(\breve{\QQ}_p)$ of $w$. Therefore there is an isomorphism $\Sht^\loc_{\cI,w}\simeq \BB I_b$. Let
    $$i_w\colon \BB I_b\simeq \Sht^\loc_{\cI,w}\to \Sht^\loc_{\cI,\mu}$$
    denote the locally closed embedding. Take a pro-$p$ open compact subgroup $K_b\subseteq G_b(\QQ_p)$ contained in $I_b$. Consider the commutative diagram
    $$\begin{tikzcd}
        \Ig_{b,K_b}^\perf\ar[r]\ar[d]\ar[dd,bend right,"i_{\Ig_{b,K_b}}"swap]\ar[rd,phantom,"\square",very near start] & \BB K_b\ar[d]\ar[dd,bend left,"i_{w,K_b}"]\\
        \cC^\perf_b \ar[r]\ar[d,"i_{C_b}"]\ar[rd,phantom,"\square",very near start] &\BB I_{b} \ar[d,"i_w"swap] \\
        \Sh_\mu \ar[r,"\loc_p"]\ar[d,"\Nt^\glob"swap]\ar[rd,phantom,"\square",very near start]&\Sht^\loc_{\cI,\mu} \ar[d,"\Nt"] \\
        \Igs \ar[r,"\loc_p^0"] & \Isoc_{G,\mu^{-1}}.
    \end{tikzcd}$$
    where all the squares are Cartesian. The fiber $\cC_b^\perf$ is the perfection of the central leaf $\cC_b\subseteq S_K(\sG,\sX)$ defined by $\dot{w}$, and $\Ig^\perf_{b,K_b}$ is the perfection of the Igusa variety $\Ig_{b,K_b}$ of level $K_b$. We have
    $$\begin{aligned}
        \Hom(\Nt_*(i_{w,K_b})_!\omega_{\BB K_b},\fI^\can) &\simeq \Hom((\loc_p^0)^!\Nt_*(i_{w,K_b})_!\omega_{\BB K_b},\omega_{\Igs}^\can) \\
        &\simeq \Hom((\Nt^\glob)_*(\loc_p)^!(i_{w,K_b})_!\omega_{\BB K_b},\omega_{\Igs}^\can)\\
        &\simeq R\Gamma(\Sh_\mu,(\loc_p)^!(i_{w,K_b})_!\omega_{\BB K_b})^\vee\\
        &\simeq R\Gamma(\Sh_\mu,(i_{\Ig_{b,K_b}})_!\omega_{\Ig^\perf_{b,K_b}})^\vee\\
        &\simeq R\Gamma_c(\Sh_\mu,(i_{\Ig_{b,K_b}})_*\Lambda_{\Ig^\perf_{b,K_b}}).
    \end{aligned}$$
    Here, the fourth isomorphic follows as $\loc_p$ is pseudo coh. pro-smooth. By \cite[Lemma 5.14]{Mao-Hodge-well-positioned}, $\cC_b$ is a well-positioned subscheme of $S_K(\sG,\sX)$. By \cite[Proposition 6.21]{Mao-Hodge-well-positioned}, the morphism $\Ig_{b,K_b}\to \cC_b$ is well-positioned. Let $\Ig_{b,K_b}^\tor$ denote the partial toroidal compactification of $\Ig_{b,K_b}$. Then $\Ig_{b,K_b}^\mini$ is the relative normalization of $\cC_b^\mini$ in $\Ig_{b,K_b}^\tor$. By \cite[Proposition 6.21]{Mao-Hodge-well-positioned}, the morphism $\Ig_{b,K_b}^\tor\to \cC^\tor_b$ is finite \'etale. Therefore the morphism $\Ig_{b,K_b}^\tor\to \Ig_{b,K_b}^\mini$ is proper. Consider the commutative diagram
    $$\begin{tikzcd}
        \Ig_{b,K_b} \ar[r,"j_{\Ig_{b,K_b}^\tor}"]\ar[rr,bend left,"j_{\Ig_{b,K_b}^\mini}"]\ar[d,"i_{\Ig_{b,K_b}}"] & \Ig^\tor_{b,K_b} \ar[r]\ar[d,"i_{\Ig^\tor_{b,K_b}}"] & \Ig^\mini_{b,K_b} \ar[d,"i_{\Ig^\mini_{b,K_b}}"] \\
        S_K(\sG,\sX) \ar[r,"j^\tor"]\ar[rr,bend right,"j"] & S_K^\tor(\sG,\sX) \ar[r,"\oint_{K,\Sigma}"] & S_K^\mini(\sG,\sX).
    \end{tikzcd}$$
    By Proposition \ref{prop-partial-compactly-supp}, we have
    $$\begin{aligned}
        (i_{\Ig_{b,K_b}^\mini})_*(j_{\Ig_{b,K_b}^\mini})_!\Lambda_{\Ig_{b,K_b}}&\simeq(\oint_{K,\Sigma})_!(i_{\Ig^\tor_{b,K_b}})_*(j_{\Ig_{b,K_b}^\tor})_!\Lambda_{\Ig_{b,K_b}} \\ &\simeq (\oint_{K,\Sigma})_!(j^\tor)_!(i_{\Ig_{b,K_b}})_*\Lambda_{\Ig_{b,K_b}} \\ &\simeq j_!(i_{\Ig_{b,K_b}})_*\Lambda_{\Ig_{b,K_b}}.
    \end{aligned}$$
    It follows that
    $$\begin{aligned}
        R\Gamma_c(\Sh_\mu,(i_{\Ig_{b,K_b}})_*\Lambda_{\Ig^\perf_{b,K_b}}) &\simeq R\Gamma(S_K^\mini(\sG,\sX),j_!(i_{\Ig_{b,K_b}})_*\Lambda_{\Ig_{b,K_b}}) \\&\simeq R\Gamma(S_K^\mini(\sG,\sX),(i_{\Ig_{b,K_b}^\mini})_*(j_{\Ig_{b,K_b}^\mini})_!\Lambda_{\Ig_{b,K_b}}) \\ &\simeq R\Gamma(\Ig_{b,K_b}^\mini,(j_{\Ig_{b,K_b}^\mini})_!\Lambda_{\Ig_{b,K_b}}).
    \end{aligned}$$
    Using that $\Nt_*(i_{w,K_b})_!\omega_{\Sht^\loc_w}\simeq (i_b)_!\cInd_{K_b}^{G_b(\QQ_p)}\Lambda,$ we have
    $$\Hom((i_b)_!\cInd_{K_b}^{G_b(\QQ_p)}\Lambda,\fI^\can)\simeq R\Gamma(\Ig_{b,K_b}^\mini,(j_{\Ig_{b,K_b}^\mini})_!\Lambda_{\Ig_{b,K_b}}).$$
    It implies that $(i_b)^!\fI^\can\simeq R\Gamma_{c-\partial}(\Ig_b,\Lambda)$.

    The statement for $\fI$ follows from Proposition \ref{prop-!-Igusa-dual}, as $\Ig_{b,K_b}$ is smooth of dimension $\langle2\rho,\nu_b\rangle$. Alternatively, we have
    $$\begin{aligned}
        \Hom(\Nt_*(i_{w,K_b})_*\omega_{\BB K_b},\fI) &\simeq \Hom((\loc_p^0)^!\Nt_*(i_{w,K_b})_*\omega_{\BB K_b},\omega_{\Igs}) \\
        &\simeq \Hom((\Nt^\glob)_*(\loc_p)^!(i_{w,K_b})_*\omega_{\BB K_b},\omega_{\Igs})\\
        &\simeq \Hom((\loc_p)^!(i_{w,K_b})_*\omega_{\BB K_b},\omega_{\Sh_\mu})\\
        &\simeq \Hom((i_{\Ig_{b,K_b}})_*\omega_{\Ig_{b,K_b}^\perf},\omega_{\Sh_\mu})\\
        &\simeq R\Gamma(\Sh_\mu,(i_{\Ig_{b,K_b}})_!\Lambda_{\Ig^\perf_{b,K_b}}),
    \end{aligned}$$
    where the last isomorphism follows from taking Verdier dual. By the same argument as above using Proposition \ref{prop-partial-compactly-supp}, we have
    $$R\Gamma(\Sh_\mu,(i_{\Ig_{b,K_b}})_!\Lambda_{\Ig^\perf_{b,K_b}})\simeq R\Gamma_c(\Ig_{b,K_b}^\mini,(j_{\Ig_{b,K_b}^\mini})_*\Lambda_{\Ig_{b,K_b}}).$$
    Using that $\Nt_*(i_{w,K_b})_*\omega_{\Sht^\loc_w}\simeq (i_b)_*\cInd_{K_b}^{G_b(\QQ_p)}\Lambda,$ we obtain 
    $$(i_b)^\sharp\fI\simeq R\Gamma_{\partial-c}(\Ig_b,\Lambda)$$
    as desired.
\end{proof}

\begin{cor}\label{cor-igusa-sheaf-exotic}
    The object $\fI^\can$ (resp. $\fI$) lies in $\Shv(\Isoc_{G,\leq\mu^*},\Lambda)^{e,\leq 0}$ (resp. $\Shv(\Isoc_{G,\leq\mu^*},\Lambda)^{e,\geq 0}$).
\end{cor}
\begin{rmk}
    If $\sSh_{K}(\sG,\sX)$ is proper, this is proved in \cite[Proposition 6.21]{Tame}. In fact, if $\sSh_{K}(\sG,\sX)$ is proper, then $\fI\simeq\fI^\can$ is self-dual under $(\DD^\can_{\Isoc_G})^\Adm$ and hence lies in the heart of the exotic $t$-structure.
\end{rmk}
\begin{proof}
    By Proposition \ref{prop-igusa-sheaf-stalk}, we need to show that
    $$(i_b)^!\fI^\can\simeq R\Gamma_{c-\partial}(\Ig_b,\Lambda)$$
    is concentrated in degrees less than or equal to $\langle2\rho,\nu_b\rangle$. Write
    $$R\Gamma_{c-\partial}(\Ig_b,\Lambda)=\lim_{\substack{\longrightarrow\\ K_b}}R\Gamma(\Ig_{b,K_b}^\mini,(j_{\Ig_{b,K_b}^\mini})_!\Lambda_{\Ig_{b,K_b}})$$
    as in the proof of Proposition \ref{prop-igusa-sheaf-stalk}. By \cite[Corollary 6.26]{Mao-Hodge-well-positioned}, the partial minimal compactification $\Ig_{b,K_b}^\mini$ is affine of dimension $\langle2\rho,\nu_b\rangle$. By Artin vanishing, we see that $R\Gamma(\Ig^\mini_{b,K_b},(j_{\Ig_{b,K_b}^\mini})_!\Lambda_{\Ig_{b,K_b}})$ is concentrated in degrees $[0,\langle2\rho,\nu_b\rangle]$. This gives the desired upper bound. The claim for $\fI$ follows from Proposition \ref{prop-adm-dual-exotic} and Proposition \ref{prop-!-Igusa-dual}.
\end{proof}

\subsection{Cohomology of Shimura varieties}\label{subsection-Igusa-stack-coh-formula}

\subsubsection{A spectral description of cohomology of Shimura varieties}

We study cohomology of Shimura varieties using the perfect Igusa stack and the categorical local Langlands correspondence. 

From now on, we assume that $G$ splits over an unramified extension of $\QQ_p$. If $\Lambda$ has characteristic $\ell$, we assume that $\ell$ is bigger than the Coxeter number of any simple factors of $G$. Note that $G$ appears in a Shimura variety of Hodge type, thus it does not contain simple factors of type $\mathsf{E}_7$ or $\mathsf{E}_8$. 

By \cite[\S 4.2]{Zhu-coherent-sheaf}, there is a groupoid $\mathbf{TS}_G$ attached to $G$. Assume that $\mathbf{TS}_G$ is non-empty and we choose an element $t\in \mathbf{TS}_G$. This amounts to choosing
\begin{itemize}
    \item A pinned quasi-split group $(G^*,B^*,T^*,e^*)$ over $\QQ_p$;
    \item an isomorphism $\eta\colon G_{\breve{\QQ}_p}\simeq G^*_{\breve{\QQ}_p}$ and an element $b\in G^*(\breve{\QQ}_p)$ such that $\eta\circ\phi\circ\eta^{-1}=\Ad_b\circ \phi$.
\end{itemize}
In particular, the group $G^*$ is unramified over $\QQ_p$. The pinning on $G^*$ defines an Iwahori model $\cI^*$ of $G^*$. After $\phi$-conjugating $b$ by an element in $G^*(\breve{\QQ}_p)$, we may assume that $b=\dot{w}$ is a lift of a length 0 element $w\in \widetilde{W}$. Therefore $\cI=\eta^{-1}(\cI^*)$ is an Iwahori model of $G$. If $\cF\in\Shv_\fingen(\Iw\backslash LG/\Iw,\Lambda)$ is an object, we denote by $\eta\cF\coloneqq \eta_*\cF\in \Shv_\fingen(\Iw^*\backslash LG^*/\Iw^*)$ the corresponding object. We have an isomorphism
$$\eta_{\dot{w}}\colon \Isoc_G\xrightarrow{\sim}\Isoc_{G^*},\quad g\mapsto \eta(g)\dot{w}.$$
By Theorem \ref{thm-unipotent-cllc} there is a fully faithful embedding
$$\LL^{\unip}_{G^*}\colon \Ind\Shv_\fingen^{\unip}(\Isoc_{G^*},\Lambda)\hookrightarrow \Ind\Coh(\Loc^{\widehat\unip}_{{}^LG,\QQ_p}).$$
We obtain a fully faithful embedding
$$\Ind\Shv_\fingen^\unip(\Isoc_G,\Lambda)\stackrel{(\eta_{\dot{w}})_*}\simeq \Ind\Shv^\unip_\fingen(\Isoc_{G^*},\Lambda)\stackrel{\LL_{G^*}^\unip}{\hookrightarrow}\Ind\Coh(\Loc^{\widehat\unip}_{{}^LG,\QQ_p}).$$
We denote by $\LL_G^\unip$ the above composition. 

The isomorphism $\eta_{\dot{w}}$ induces an isomorphism $G\simeq G_b$, where $G_b$ is the extended pure inner form of $G^*$ associated to $b=\dot{w}\in B(G)$. The Iwahori model $\cI^*$ defines an Iwahori model $\cI$ of $G$. Let $I=\cI(\ZZ_p)\subseteq G(\QQ_p)$, which is identified with $I_b\subseteq G_b(\QQ_p)$. Thus we have
$$(\eta_{\dot{w}})_*(i_1)_*\cInd_I^{G(\QQ_p)}\Lambda\simeq (i_b)_*\cInd_{I_b}^{G_b(\QQ_p)}\Lambda.$$
Let
$$\fA_{b}\coloneqq \LL^{\unip}_{G}((i_1)_*\cInd_I^{G(\QQ_p)})\simeq \LL^\unip_{G^*}((i_b)_*\cInd_{I_b}^{G_b(\QQ_p)}\Lambda)$$
the coherent sheaf associated to $(i_1)_*\cInd_I^{G(\QQ_p)}\Lambda$.

If $G$ is unramified, then we can choose $\dot{w}=1$ and hence $G=G^*$. In this case, we have
$$\fA_1= \CohSpr^\unip_{{}^LG},$$
where $\CohSpr^\unip_{{}^LG}$ is the unipotent spectral Springer sheaf is defined as 
$$\CohSpr^\unip_{{}^LG}\coloneqq (\fq^\unip)_*\omega_{\Loc^{\unip}_{{}^LB,\QQ_p}}\in  \Coh(\Loc^{\widehat\unip}_{{}^LG,\QQ_p}).$$

Recall that in \cite[Theorem 6.16]{Tame}, we define the unipotent spectral Igusa sheaf
$$\fI^\unip_\spec\coloneqq \LL^{\unip}_{G}(\cP^{\unip}\circ\Psi^L((i_{\leq\mu^*})_*\fI)))$$
in $\Ind\Coh(\Loc^{\widehat\unip}_{{}^LG,\QQ_p})$, where $i_{\leq\mu^*}\colon \Isoc_{G,\leq\mu^*}\hookrightarrow \Isoc_G$ is the closed embedding, $\Psi^L\colon \Shv(\Isoc_G,\Lambda)\hookrightarrow \Ind\Shv_\fingen(\Isoc_G,\Lambda)$ is the natural fully faithful functor, and $\cP^\unip\colon\Ind\Shv_\fingen(\Isoc_G,\Lambda)\to\Ind\Shv^\unip_\fingen(\Isoc_G,\Lambda)$ is the unipotent projector. We can also define the $!$-version.

\begin{defn}\label{def-igusa-sheaf}
    Define the unipotent coherent $!$-Igusa sheaf 
    $$\fI^{\can,\unip}_\spec\coloneqq \LL^{\unip}_{G}(\cP^{\unip}\circ\Psi^L((i_{\leq\mu^*})_*\fI^\can)))$$
    in $\Ind\Coh(\Loc^{\widehat\unip}_{{}^LG,\QQ_p})$.
\end{defn}

By Proposition \ref{prop-!-Igusa-to-Igusa}, there is a natural morphism $$\fI^{\can,\unip}_\spec\to \fI^{\unip}_\spec.$$
Let $V_\mu$ denote the highest weight representation of $\hat{G}$ of highest weight $\mu$. Let $\widetilde{V}_\mu$ be the vector bundle on $\Loc^{\widehat\unip}_{{}^LG,F}$ associated to $V_\mu$.

\begin{prop}\label{prop-coh-of-Sh-igusa-sheaf}
    There are canonical isomorphisms
    $$\Hom(\widetilde{V}_\mu\otimes \fA_b,\fI^{\can,\unip}_\spec)\simeq  R\Gamma_c(\sSh_{K^pI}(\sG,\sX)_{\overline{E}},\Lambda)(d/2)[d],$$
    $$\Hom(\widetilde{V}_\mu\otimes \fA_b,\fI^{\unip}_\spec)\simeq  R\Gamma(\sSh_{K^pI}(\sG,\sX)_{\overline{E}},\Lambda)(d/2)[d],$$
    where $d=\dim\sSh_{K^pI}(\sG,\sX)_{\overline{E}}$. Moreover, the natural morphism $\fI^{\can,\unip}_\spec\to \fI^{\unip}_\spec$ induces the natural morpshim
    $$R\Gamma_c(\sSh_{K^pI}(\sG,\sX)_{\overline{E}},\Lambda)(d/2)[d]\to R\Gamma(\sSh_{K^pI}(\sG,\sX)_{\overline{E}},\Lambda)(d/2)[d]$$
    on cohomologies.
\end{prop}
\begin{proof}
    Let $Z_\mu$ denote the central sheaf in $\Shv_\fingen(\Iw\backslash LG/\Iw,\Lambda)$ associated to the highest weight representation $V_\mu$. Recall the natural morphism $\delta\colon \Sht_{\cI,\mu}^\loc\to \Iw\backslash LG/\Iw$. By \cite[Theorem 6.16]{Tame}, there is a canonical isomorphism
    $$\Hom(\Nt_*\delta^!Z_\mu,\fI) \simeq R\Gamma(\sSh_{K^pI}(\sG,\sX)_{\overline{E}},\Lambda)(d/2)[d].$$
    Similarly, we have
    $$\begin{aligned}
        \Hom(\Nt_*\delta^!Z_\mu,\fI^\can) & \simeq \Hom((\loc_p^0)^!\Nt_*\delta^!Z_\mu,\omega^\can_{\Igs}) \\
        &\simeq\Hom((\Nt^\glob)_*(\loc_p)^!\delta^!Z_\mu,\omega_{\Igs}^\can) \\
        &\simeq R\Gamma_c(\Sh_\mu, (\DD_{\Sh_\mu})^\omega((\loc_p)^!\delta^!Z_\mu)) \\
        &\simeq R\Gamma_c(\sSh_{K^pI}(\sG,\sX)_{\overline{E}},\Lambda)(d/2)[d].
    \end{aligned}$$
    Here, for the last isomorphism, we use that $(\loc_p)^!\delta^!Z_\mu\in\Shv_c(\Sh_\mu,\Lambda)$ is isomorphic to nearby cycle sheaf $R\Psi(\Lambda(d/2)[d])$ over $\Sh_\mu$, and by \cite[Corollary 4.6]{Lan-Stroh-nearby-cycle-II}, the compact supported cohomology of $R\Psi\Lambda$ computes the compact supported cohomology of $\sSh_{K^pI}(\sG,\sX)_{\overline{E}}$. By definition, the natural morphism $\fI^\can\to \fI$ induces the natural morphism 
    $$R\Gamma_c(\Sh_\mu,R\Psi(\Lambda(d/2)[d]))\to R\Gamma(\Sh_\mu,R\Psi(\Lambda(d/2)[d]))$$
    on cohomologies.
    
    By the argument of \cite[Corollary 6.15]{Tame} and Lemma \ref{lemma-igusa-bounded-below}, there are canonical isomorphisms
    $$\Hom(\Nt_*\delta^!Z_\mu,\Psi^L((i_{\leq\mu^*})_*\fI))\simeq R\Gamma(\sSh_{K^pI}(\sG,\sX)_{\overline{E}},\Lambda)(d/2)[d],$$
    $$\Hom(\Nt_*\delta^!Z_\mu,\Psi^L((i_{\leq\mu^*})_*\fI^\can))\simeq R\Gamma_c(\sSh_{K^pI}(\sG,\sX)_{\overline{E}},\Lambda)(d/2)[d].$$
    Here, the left Hom-space is computed in $\Ind\Shv_\fingen(\Isoc_G,\Lambda)$. It suffices to show that there is a canonical isomorphism
    $$\LL^\unip_G(\Nt_*\delta^!Z_\mu)\simeq \widetilde{V}_\mu\otimes \fA_b.$$
    There is a commutative diagram
    $$\begin{tikzcd}
        \Iw\backslash LG/\Iw \ar[r,"{}_{\dot{w}^{-1}}\eta","\simeq"swap] & \Iw^*\backslash LG^*/\Iw^* \\
        \Sht^\loc_{\cI} \ar[d,"\Nt"swap]\ar[u,"\delta"]\ar[r,"\eta_{\dot{w}}","\simeq"swap] & \Sht^\loc_{\cI^*} \ar[d,"\Nt^*"]\ar[u,"\delta^*"swap] \\
        \Isoc_G \ar[r,"\eta_{\dot{w}}","\simeq"swap] & \Isoc_{G^*},
    \end{tikzcd}$$
    where ${}_{\dot{w}^{-1}}\eta(g)=\dot{w}^{-1}g$. In particular, we have
    $$({}_{\dot{w}^{-1}}\eta)_* Z_\mu\simeq \Delta_{w^{-1}}\star \eta Z_\mu,$$
    where $\eta Z_\mu$ is the central sheaf for $G^*$.
    It follows that
    $$(\eta_{\dot{w}})_*\Nt_*\delta^!Z_\mu\simeq (\Nt^*)_*(\delta^*)^!(\Delta_{w^{-1}}\star \eta Z_\mu)\simeq T_{V_\mu}( (i_b)_*\cInd_{I_b}^{G_b(\QQ_p)}\Lambda)$$
    as $(\Nt^*)_*(\delta^*)^!\Delta_{w^{-1}}\simeq (i_b)_*\cInd_{I_b}^{G_b(\QQ_p)}\Lambda$. Therefore we have
    $$\LL_G^\unip(\Nt_*\delta^!Z_\mu)\simeq \LL^\unip_{G^*}(T_{V_\mu}((i_b)_*\cInd_{I_b}^{G_b(\QQ_p)}\Lambda)))\simeq \widetilde{V}_\mu\otimes \fA_b$$
    as desired.
\end{proof}

\begin{lemma}\label{lemma-igusa-bounded-below}
    The object $\fI^\can$ is bounded below with respect to the perverse $t$-structure on $\Shv(\Isoc_{G,\leq\mu^*},\Lambda)$.
\end{lemma}
\begin{proof}
    Let $b\in B(G,\mu^*)$. By Proposition \ref{prop-igusa-sheaf-stalk}, we have
    $$(i_b)^!\fI^\can\simeq R\Gamma_{c-\partial}(\Ig_b,\Lambda)$$
    is concentrated in degrees $[0,\langle2\rho,\nu_b\rangle]$. In particular, it is bounded below.
\end{proof}

By transport of structures, the object $\fA_b$ carries a natural \emph{right} action of the Iwahori--Hecke algebra
$$H_I\coloneqq \Lambda[I\backslash G(\QQ_p)/I].$$
The vector bundle $\widetilde{V}_\mu$ admits a tautological \emph{left} $W_E$-action, where $W_E$ is the Weil group of $E$. The (!)-Igusa sheaf $\fI$ (resp. $\fI^\can$) carries natural action of the prime-to-$p$ Hecke algebra
$$H_{K^p}\coloneqq \Lambda[K^p\backslash \sG(\AAA_f^p)/K^p]$$
as the Igusa stack $\Igs$ carries prime-to-$p$ Hecke actions. It follows that $\fI^\unip_\spec$ and $\fI^{\can,\unip}_\spec$ also admit actions of $H_{K^p}$.
On the other hand, the cohomology $R\Gamma_c(\sSh_{K^pI}(\sG,\sX)_{\overline{E}},\Lambda)(d/2)[d]$ (resp. $R\Gamma(\sSh_{K^pI}(\sG,\sX)_{\overline{E}},\Lambda)(d/2)[d]$) carries a usual $W_E\times H_I\times H_{K^p}$-action. The main result of this section is the following local-global compatibility.

\begin{thm}\label{thm-local-global-compatibility}
    The isomorphisms $$\Hom(\widetilde{V}_\mu\otimes\fA_b,\fI^{\can,\unip}_\spec)\simeq R\Gamma_c(\sSh_{K^pI}(\sG,\sX)_{\overline{E}},\Lambda)(d/2)[d]$$
    $$\Hom(\widetilde{V}_\mu\otimes\fA_b,\fI^{\unip}_\spec)\simeq R\Gamma(\sSh_{K^pI}(\sG,\sX)_{\overline{E}},\Lambda)(d/2)[d]$$
    are compatible with $W_E\times H_I\times H_{K^p}$-actions.
\end{thm}

\begin{rmk}
    We normalize the actions so the both $W_E$ and $H_I$ act on the \emph{left} of both sides. 
\end{rmk}

The proof of Theorem \ref{thm-local-global-compatibility} will be given in \S\ref{subsection-S=T}. 

We give a first application of Theorem \ref{thm-local-global-compatibility}, which allows us to decompose the Igusa sheaf with respect to the Hecke eigen-systems. Let $S$ be a finite set of finite places of $\QQ$ containing $p$ and $\ell$ such that $\sG$ is unramified away from $S$, and the open compact subgroup $K^p$ decompose into $K^p=K^SK_S^p$ with $K^S\subseteq \sG(\AAA_f^S)$ is a product of \emph{hyperspecial} subgroups. Hence
$$H_{K^S}\coloneqq\Lambda [K^S\backslash \sG(\AAA_f^S)/K^S]$$
is a commutative algebra over $\Lambda$. 

The unipotent spectral center $Z^{\widehat\unip}_{{}^LG,\QQ_p}=H^0(\Loc^{\widehat\unip}_{{}^LG,\QQ_p},\cO)$ acts on $\cInd_I^{G(\QQ_p)}\Lambda$, and in particular $\cO(\hat{G}\phi\git \hat{G})$ acts on $\cInd_I^{G(\QQ_p)}\Lambda$ via 
$$\ev^{\coarse}_\phi\colon \cO(\hat{G}\phi\git \hat{G})\to Z^{\widehat\unip}_{{}^LG,\QQ_p}.$$
We will prove later in Proposition \ref{prop-S=T-center} that when $b=1$, this action agrees with the action of the center $Z(H_I)\simeq \cO(\hat{G}\phi\git\hat{G})$.
\begin{lemma}\label{lem: decomposition of Igs sheaf}
    There is a direct sum decomposition
    $$\fI^{\unip}_\spec=\bigoplus_{\xi_p,\xi^p}(\fI^{\unip}_\spec)_{\xi_p,\xi^p}$$
    where $\xi_p$ runs through $\Lambda$-points of $\hat{G}\phi\git \hat{G}$ and $\xi^p$ runs through $\Lambda$-points of $\Spec H_{K^S}$, such that
    $$\Hom(\widetilde{V}_\mu\otimes \fA_b,(\fI^{\unip}_\spec)_{\xi_p,\xi^p})\simeq R\Gamma(\Sh_{K^pI}(\sG,\sX)_{\overline{E},\Lambda},\Lambda)_{\xi_p,\xi^p}(d/2)[d].$$
    Here $R\Gamma(\Sh_{K^pI}(\sG,\sX)_{\overline{E},\Lambda},\Lambda)_{\xi_p,\xi^p}$ is the direct summand of $R\Gamma(\Sh_{K^pI}(\sG,\sX)_{\overline{E},\Lambda},\Lambda)$ where the action of $\cO(\hat{G}\phi\git \hat{G})\otimes H_{K^S}$ is supported at the point $(\xi_p,\xi^p)$. Moreover, $(\fI^{\unip}_\spec)_{\xi_p,\xi^p}$ is non-zero for finitely many $(\xi_p,\xi^p)$. Similar statements hold for $\fI^{\can,\unip}_\spec$.
\end{lemma}
\begin{proof}
    We work with a more general setup. Let $\bC$ be a compactly generated $\Lambda$-linear category. Let $A\in\bC^\Adm$ be an admissible object, or equivalently, $\Hom(C,A)$ is a perfect $\Lambda$-module for any compact $C\in\bC^\omega$. Assume that $A$ carries an action of a commutative $\Lambda$-algebra $R$. Then we for each $\Lambda$-point $\xi$ of $\Spec R$, we define
    $$A_\xi\coloneqq A\otimes_{R}R_{(\xi)}.$$
    Then we have
    $$\Hom(C,A_{\xi})=\Hom(C,A)\otimes_RR_{(\xi)}=\Hom(C,A)_\xi$$
    for any $C\in\bC^\omega$. Here $\Hom(C,A)_\xi$ is the direct summand of $\Hom(C,A)$ where the action of $R$ is supported at $\xi$. The morphism
    $$A\to \bigoplus_\xi A_\xi$$
    induces an isomorphism
    $$\Hom(C,A)\xrightarrow\sim \bigoplus_\xi\Hom(C,A_\xi)\simeq \bigoplus_\xi\Hom(C,A)_\xi$$
    for any $C\in\bC^\omega$. Therefore we have $A\simeq \bigoplus A_\xi$. If we further assume that $A$ is contained in the $\Lambda$-linear subcategory $\langle C_i\rangle_{i=1,\dots,n} \subseteq \bC$ generated under colimits by finitely many compact objects $C_i$, then $A_{\xi}=0$ for all but finitely many $\xi$. It follows as $\Hom(C_i,A)$ is supported at finitely many $\xi$, and hence $\bigoplus_{i=1,\dots,n}\Hom(C_i,A_\xi)=0$ for all but finitely many $\xi$.

    Now apply this to $\bC=\Ind\Coh(\Loc^{\widehat\unip}_{{}^LG,\QQ_p})$ and $A=\fI^{\unip}_{\spec}$ (resp. $A=\fI^{\can,\unip}_{\spec}$), we obtain the desired decomposition. Note that $\fI^{\unip}_{\spec}$ (resp. $\fI^{\can,\unip}_\spec$) lies in the essential image of
    $$\Ind\Shv_\fingen^{\unip}(\Isoc_{G,\leq\mu^*},\Lambda)\stackrel{(i_{\leq\mu^*})_*}\hookrightarrow\Ind\Shv_\fingen^{\unip}(\Isoc_{G},\Lambda)\stackrel{\LL_G^\unip}\hookrightarrow\Ind\Coh(\Loc^{\widehat\unip}_{{}^LG,\QQ_p}),$$
    which is generated by finitely many compact objects.
\end{proof}

\subsubsection{Cohomological correspondences}\label{subsubection-coh-corr}
We recall the properties of cohomological correspondences following \cite[Appendix A.2]{XZ-cycles}. Let
$$X_1\xleftarrow{c_l} C\xrightarrow{c_r} X_2$$
be a correspondence of perfect stacks with $c_l$ and $c_r$ are representable pfp separated. Let $\cF_1\in\Shv(X_1,\Lambda)$ and $\cF_2\in\Shv(X_2,\Lambda)$. A \emph{cohomological correspondence} from $(X_1,\cF_1)$ to $(X_2,\cF_2)$ supported on $C$ is a morphism $u\colon (c_l)^*\cF_1\to (c_r)^!\cF_2.$

\begin{eg}\label{eg-coh-correspondence}
    Here are some examples of functoriality of cohomological correspondences. One can define the functoriality in more general situations, but we shall only need the following cases in the sequel. For more general discussions, see \cite[Appendix A.2]{XZ-cycles}.
\begin{enumerate}
    \item  Let
        $$\begin{tikzcd}
        X_1 \ar[d,"f"] & C \ar[d,"h"]\ar[l,"c_l"swap]\ar[r,"c_r"] & X_2 \ar[d,"g"] \\
        Y_1 & D \ar[l,"d_l"swap]\ar[r,"d_r"] & Y_2
        \end{tikzcd}$$
    be a commutative diagram. Let $u\colon (Y_1,\cF_1)\to (Y_2,\cF_2)$ be a correspondence. 
    
    If the left diagram is Cartesian, then we define the $!$-pullback $h^!(u)\colon (X_1,f^!\cF_1)\to(X_2,g^!\cF_2)$ by the composition
    $$h^!(u)\colon (c_l)^*f^!\cF\to h^!(d_l)^*\cF\xrightarrow{u}h^!(d_r)^!\cF\simeq (c_r)^!g^!\cF,$$
    where the first morphism is defined by adjunction. 
    
    If the right diagram is Cartesian, then we define the $*$-pullback $h^*(u)\colon (X_1,f^*\cF_1)\to(X_2,g^*\cF_2)$ by the composition
    $$h^*(u)\colon  (c_l)^*f^*\cF_1\simeq h^*(d_l)^*\cF_1\xrightarrow{u} h^*(d_r)^!\cF_2\to (c_r)^!g^*\cF_2,$$
    where the third morphism is defined by adjunction.
    \item Assume $X_1,X_2$ are separated pfp perfect schemes. Let $u\colon (X_1,\cF_1)\to (X_2,\cF_2)$ be a cohomological correspondence supported on $C$. Then we define the verdier dual $\DD(u)$ by the composition
    $$\DD(u)\colon (c_r)^*(\DD_{X_2})^\omega(\cF_2)\simeq (\DD_C)^\omega((c_r)^!\cF_2)\xrightarrow{u}(\DD_C)^\omega((c_l)^*\cF_1)\simeq (c_l)^!(\DD_{X_1})^\omega(\cF_1).$$
    \item Assume that there is a commutative diagram
    $$\begin{tikzcd}
        C \ar[r,"c_r"]\ar[d,"c_l"] & X_2 \ar[d,"g"] \\ X_1 \ar[r,"f"] & Z
    \end{tikzcd}.$$
    Let $u\colon (X_1,\cF_1)\to (X_2,\cF_2)$ be a cohomological correspondence. 
    
    If $c_l$ is proper, we define the $!$-pushforward morphism by
    $$f_!\cF_1\to f_!(c_l)_*(c_l)^*\cF_1\xrightarrow{u}f_!(c_l)_*(c_r)^!\cF_2\simeq g_!(c_r)_!(c_r)^!\cF_2\to g_!\cF_2,$$
    where the first and the last morphism are given by adjunction.

    If $c_r$ is proper, we define the $*$-pushforward morphism by
    $$f_*\cF_1\to f_*(c_l)_*(c_l)^*\cF_1\xrightarrow{u}f_*(c_l)_*(c_r)^!\cF_2\simeq g_*(c_r)_*(c_r)^!\cF_2\to g_*\cF_2,$$
    where the first and the last morphism are given by adjunction.

    In particular, if $X_1$, $X_2$ are separated pfp schemes, then we can define
    $$R\Gamma_c(u)\colon R\Gamma_c(X_1,\cF_1)\to R\Gamma_c(X_2,\cF_2)$$
    $$R\Gamma(u)\colon R\Gamma(X_1,\cF_1)\to R\Gamma(X_2,\cF_2)$$
    provided that $c_l$ and $c_r$ are proper.
\end{enumerate}
\end{eg}

%\Xinwen{Is it intentionally to put the following paragraph here? Before and after are discussions in general setup and this paragraph is specific to Igs?}\Xiangqian{All the following are specific to Igs.} 
Recall the Cartesian diagram 
$$\begin{tikzcd}
    \Sh_\mu \ar[r,"\loc_p"]\ar[d,"\Nt^\glob"swap] & \Sht^\loc_{\cI,\mu} \ar[d,"\Nt"] \\
    \Igs \ar[r,"\loc_p^0"] & \Isoc_{G,\leq\mu^*}
\end{tikzcd}$$
in Theorem \ref{thm-igusa}. By the proof of Proposition \ref{prop-coh-of-Sh-igusa-sheaf} and \cite[Theorem 6.16]{Tame}, we have
$$\Hom(\Nt_*\cF,(\loc_p^0)_\flat\omega_\Igs^\can)\simeq R\Gamma_c(\Sh_\mu,(\DD_{\Sh_\mu})^\omega((\loc_p)^!\cF))$$
$$\Hom(\Nt_*\cF,(\loc_p^0)_\flat\omega_\Igs)\simeq R\Gamma(\Sh_\mu,(\DD_{\Sh_\mu})^\omega((\loc_p)^!\cF))$$
for $\cF\in\Shv_\fingen(\Sht^\loc_{\cI,\mu},\Lambda)$. 

Now let $D$ be a perfect stack together with proper morphisms $d_l,d_r\colon D\to \Sht^\loc_{\cI,\mu}$ such that the diagram
$$\begin{tikzcd}
    D \ar[r,"d_r"]\ar[d,"d_l"swap] & \Sht^\loc_{\cI,\mu} \ar[d,"\Nt"] \\ \Sht^\loc_{\cI,\mu} \ar[r,"\Nt"] & \Isoc_{G,\leq\mu^*}
\end{tikzcd}$$
commutes. In practice, $D$ will be taken to be certain closed substack of $\Hk_1(\Sht^\loc_{\cI,\mu})$. Pulled back to $\Sh_\mu$, we obtain a commutative diagram
$$\begin{tikzcd}
    \Sh_\mu \ar[d,"\loc_p"] & C \ar[ld,phantom,"\square"very near start]\ar[rd,phantom,"\square"very near start]\ar[d,"\loc_{p,C}"]\ar[l,"c_l"swap]\ar[r,"c_r"] & \Sh_\mu \ar[d,"\loc_p"] \\
    \Sht^\loc_{\cI,\mu} & D \ar[l,"d_l"swap]\ar[r,"d_r"] & \Sht^\loc_{\cI,\mu}
\end{tikzcd}$$
with two squares Cartesian. Let 
$$u\colon (d_l)^*\cF_1\to (d_r)^!\cF_2$$
be a cohomological correspondence between $\cF_1,\cF_2\in\Shv_\fingen(\Sht^\loc_{\cI,\mu},\Lambda)$ supported on $D$. By Example \ref{eg-coh-correspondence} (1) and (2), we can define the cohomological correspondence 
$$\DD((\loc_p)^!(u))\colon (c_r)^*(\DD_{\Sh_\mu})^\omega((\loc_p)^!\cF_2)\to (c_l)^!(\DD_{\Sh_\mu})^\omega((\loc_p)^!\cF_1)$$
from $(\Sh_\mu,(\DD_{\Sh_\mu})^\omega((\loc_p)^!\cF_2))$ to $(\Sh_\mu,(\DD_{\Sh_\mu})^\omega((\loc_p)^!\cF_1)$. By Example \ref{eg-coh-correspondence} (3), we obtain morphisms
$$R\Gamma_c(\DD((\loc_p)^!(u)))\colon R\Gamma_c(\Sh_\mu,(\DD_{\Sh_\mu})^\omega((\loc_p)^!\cF_2))\to R\Gamma_c(\Sh_\mu,(\DD_{\Sh_\mu})^\omega((\loc_p)^!\cF_1)),$$
$$R\Gamma(\DD((\loc_p)^!(u)))\colon R\Gamma(\Sh_\mu,(\DD_{\Sh_\mu})^\omega((\loc_p)^!\cF_2))\to R\Gamma(\Sh_\mu,(\DD_{\Sh_\mu})^\omega((\loc_p)^!\cF_1)).$$
On the other hand, by Example \ref{eg-coh-correspondence} (3), we define the pushforward morphism
$$\Nt_*(u)\colon \Nt_*\cF_1\to \Nt_*\cF_2.$$

\begin{lemma}\label{lemma-coh-correspondence}
    Under the isomorphism $\Hom(\Nt_*\cF_i,(\loc_p^0)_\flat\omega_\Igs^\can)\simeq R\Gamma_c(\Sh_\mu,(\DD_{\Sh_\mu})^\omega((\loc_p)^!\cF_i))$ (resp. $\Hom(\Nt_*\cF_i,(\loc_p^0)_\flat\omega_\Igs)\simeq R\Gamma(\Sh_\mu,(\DD_{\Sh_\mu})^\omega((\loc_p)^!\cF_i))$), $i=1,2$, the morphism $R\Gamma_c(\DD((\loc_p)^!(u)))$ (resp. $R\Gamma(\DD((\loc_p)^!(u)))$) agrees with the morphism induced by $\Nt_*(u)$.
\end{lemma}
\begin{proof}
    We only prove the case of $\omega_\Igs^\can$ as the other case can be proved similarly.
    The morphism $\Nt_*(u)$ is given by the composition
    $$\Nt_*\cF_1\xrightarrow[(\mathrm{A})]{} \Nt_*(d_l)_*(d_l)^*\cF_1\xrightarrow[(\mathrm{B})]{u}\Nt_*(d_l)_*(d_r)^!\cF_2\simeq \Nt_*(d_r)_*(d_r)^!\cF_2\xrightarrow[(\mathrm{C})]{} \Nt_*\cF_2.$$
    There is a morphism 
    $$\begin{aligned}
        \Hom(\Nt_*(d_l)_*(d_l)^*\cF_1,(\loc_p^0)_\flat\omega_\Igs^\can) &\simeq R\Gamma_c(\Sh_\mu,(\DD_{\Sh_\mu})^\omega((\loc_p)^!(d_l)_*(d_l)^*\cF_1)) \\
        &\simeq R\Gamma_c(\Sh_\mu,(\DD_{\Sh_\mu})^\omega((c_l)_*(\loc_{p,C})^!(d_l)^*\cF_1))\\
        &\to R\Gamma_c(\Sh_\mu,(c_l)_!(c_l)^!(\DD_{\Sh_\mu})^\omega((\loc_p)^!\cF_1))
    \end{aligned}$$
    where the third arrow is induced by the $(c_l)^*(\loc_p)^!\to (\loc_{p,C})^!(d_l)^*$. Applying the functor $\Hom(-,(\loc_p^0)_\flat\omega_\Igs^\can)$ to the morphism labeled by (A) induces a commutative diagram
    $$\begin{tikzcd}
        \Hom(\Nt_*(d_l)_*(d_l)^*\cF_1,(\loc_p^0)_\flat\omega_\Igs^\can) \ar[r,"(\mathrm{A})"]\ar[d] & \Hom(\Nt_*\cF_1,(\loc_p^0)_\flat\omega_\Igs^\can) \ar[d,"\simeq"] \\
        R\Gamma_c(\Sh_\mu,(c_l)_!(c_l)^!(\DD_{\Sh_\mu})^\omega((\loc_p)^!\cF_1)) \ar[r] & R\Gamma_c(\Sh_\mu,(\DD_{\Sh_\mu})^\omega((\loc_p)^!\cF_1))
    \end{tikzcd}$$
    where the lower arrow is defined by the counit morphism $(c_l)_!(c_l)^!\to \id$. There are isomorphisms
    $$\begin{aligned}
        \Hom(\Nt_*(d_r)_*(d_r)^!\cF_2,f_\flat\omega_\Igs^\can) &\simeq R\Gamma_c(\Sh_\mu,(\DD_{\Sh_\mu})^\omega((\loc_p)^!(d_r)_*(d_r)^!\cF_2)) \\
        &\simeq R\Gamma_c(\Sh_\mu,(c_r)_*(c_r)^*(\DD_{\Sh_\mu})^\omega((\loc_p)^!\cF_2))
    \end{aligned}$$
    such that the diagram
    $$\begin{tikzcd}
        \Hom(\Nt_*\cF_2,(\loc_p^0)_\flat \omega_Z^\can) \ar[r,"(\mathrm{C})"]\ar[d,"\simeq"] & \Hom(\Nt_*(d_r)_*(d_r)^!\cF_2,(\loc_p^0)_\flat \omega_\Igs^\can) \ar[d,"\simeq"]\ar[r,"\simeq"]\ar[d,"\simeq"] & \Hom(\Nt_*(d_l)_*(d_r)^!\cF_2,(\loc_p^0)_\flat \omega_\Igs^\can)\ar[d,"\simeq"] \\ 
        R\Gamma_c(\Sh_\mu,(\DD_{\Sh_\mu})^\omega(\cF_2)) \ar[r] & R\Gamma_c(\Sh_\mu,(c_r)_*(c_r)^*(\DD_{\Sh_\mu})^\omega((\loc_p)^!\cF_2)) \ar[r,"\simeq"] & R\Gamma_c(\Sh_\mu,(c_l)_*(c_r)^*(\DD_{\Sh_\mu})^\omega((\loc_p)^!\cF_2))
    \end{tikzcd}$$
    commutes, where the lower left arrow is defined by the unit morphism $\id\to (c_r)_*(c_r)^*$. Finally, we have a commutative diagram
    $$\begin{tikzcd}[column sep=huge]
        \Hom(\Nt_*(d_l)_*(d_r)^!\cF_2,(\loc_p^0)_\flat \omega_\Igs^\can) \ar[r,"u","(\mathrm{B})"swap]\ar[d,"\simeq"] &  \Hom(\Nt_*(d_l)_*(d_l)^*\cF_1,(\loc_p^0)_\flat\omega_\Igs^\can) \ar[d] \\
        R\Gamma_c(\Sh_\mu,(c_l)_*(c_r)^*(\DD_{\Sh_\mu})^\omega((\loc_p)^!\cF_2)) \ar[r,"\DD((\loc_p)^!(u))"] & R\Gamma_c(\Sh_\mu,(c_l)_!(c_l)^!(\DD_{\Sh_\mu})^\omega((\loc_p)^!\cF_1))
    \end{tikzcd}$$
    where the lower arrow is defined by the cohomological correspondence $\DD((\loc_p)^!(u))$. Combine the diagrams above, we finish the proof.
\end{proof}

\subsubsection{Compatibility with $W_E$-actions}
We prove the $W_E$-actions part in Theorem \ref{thm-local-global-compatibility}. 

\begin{prop}\label{prop-compatibility-Weil-group}
    The isomorphisms in Proposition \ref{prop-coh-of-Sh-igusa-sheaf} are compatible with $W_E$-actions.
\end{prop}
\begin{proof}
We only deal with first isomorphism for compactly supported cohomology as the second isomorphism can be treated in the same way.
Assume $[E:\QQ_p]=s$. It suffices to compare the actions of the arithmetic Frobenius $\phi^s$ and the actions of the tame inertia generator $\tau$ on both sides.

We first deal with the Frobenius.  The highest weight representation $V_\mu$ carries an natural action of $\hat{G}\rtimes \langle\phi^s\rangle\subseteq {}^LG$. The action of $\phi^s$ on $V_\mu$ defines an isomorphism
$$\phi^s_{V_\mu}\colon (\phi^s)_*V_\mu\xrightarrow{\sim} V_\mu,$$
where $(\phi^s)_*V_\mu$ is the $\hat{G}$-representation defined by the composition $\hat{G}\xrightarrow{\phi^{-s}}\hat{G}\to\GL(V_\mu)$. For any $\hat{G}$-representation $V$, there is a canonical isomorphism $\can\colon V\otimes\cO_{\hat{G}\phi/\hat{G}} \xrightarrow{\sim} \phi_*V\otimes\cO_{\hat{G}\phi/\hat{G}}$ over $\hat{G}\phi/\hat{G}$ defined by sending $f\in\mathrm{Map}(\hat{G}\phi,V)$ to 
$$\can(f)(g\phi)=\phi^{-1}(g)f(g\phi)\in \mathrm{Map}(\hat{G}\phi,\phi_*V).$$
Pulling back along $\ev_\phi\colon\Loc_{{}^LG,\QQ_p}^{\widehat\unip}\to \hat{G}\phi/\hat{G}$, we obtain a canonical isomorphism $\can\colon \widetilde{V}\xrightarrow{\sim}\widetilde{\phi_*V}$. The composition
$$\act_{\phi^s}\colon\widetilde{V}_\mu\xrightarrow{\can^s}\widetilde{(\phi^s)_*V_\mu}\xrightarrow{\phi^s_{V_\mu}}\widetilde{V}_\mu$$
is identified with the tautological action of $\phi^s$ on $\widetilde{V}_\mu$.

\begin{lemma}
    The action of $\act_{\phi^s}\otimes\id$ on 
    $$\Hom(\widetilde{V}_\mu\otimes \fA_b,\fI^{\can,\unip}_\spec)\simeq  R\Gamma_c(\sSh_{K^pI}(\sG,\sX)_{\overline{E}},\Lambda)(d/2)[d]$$
    is identified with the \emph{geometric} Frobenius action on $ R\Gamma_c(\sSh_{K^pI}(\sG,\sX)_{\overline{E}},\Lambda)(d/2)[d]$.
\end{lemma}
\begin{proof}
    The central sheaf $Z_\mu\in\Shv_\fingen(\Iw\backslash LG/\Iw,\Lambda)$ carries a natural Weil structure $F\colon Z_\mu\xrightarrow{\sim}(\phi^s)_*Z_\mu$. Recall the natural isomorphism
    $${}_{\dot{w}^{-1}}\eta\colon \Iw\backslash LG/\Iw\simeq \Iw^*\backslash LG^*/\Iw^*$$
    defined by ${}_{\dot{w}^{-1}}\eta(g)=\dot{w}^{-1}g$. We have
    $$({}_{\dot{w}^{-1}}\eta)_*Z_\mu\simeq \Delta_{w^{-1}}\star\eta Z_\mu$$
    and
    $$({}_{\dot{w}^{-1}}\eta)_*\phi_*(Z_\mu)\simeq \Delta_{w^{-1}}\star(\Delta_{w}\star \phi_*\eta Z_\mu\star \Delta_{w^{-1}})\simeq \phi_*\eta Z_\mu\star \Delta_{w^{-1}}\xrightarrow[\simeq]{\tau_{\eta Z_\mu,\Delta_{w^{-1}}}} \Delta_{w^{-1}}\star \phi_*\eta Z_\mu,$$
    where the $\tau_{-,-}$ the commutative constraint on central sheaves. The canonical Weil structure on $Z_\mu$ is compatible with the canonical Weil structure on $\eta F\colon \eta Z_\mu\xrightarrow{\sim}(\phi^s)_*\eta Z_\mu$ under the above isomorphisms up to $(-1)^{d(Z_\mu)d(\Delta_{w^{-1}})}$. Under the equivalence $\BB^\unip_{G^*}$, $\eta F$ corresponds to the inverse of $\phi_{V_\mu}^s$ on $\iota_*\omega_{\hat{U}/\hat{B}}\otimes V_\mu$. 
    
    By the construction of categorical trace, for $\cF\in\Shv_\fingen(\Iw\backslash LG/\Iw,\Lambda)$, there is a canonical isomorphism
    $$\can\colon \Nt_*\delta^!\cF\xrightarrow{\sim} \Nt_*\delta^!\phi_*\cF.$$
    defined by the commutative diagram
    $$\begin{tikzcd}
        \Iw\backslash LG/\Iw \ar[d,"\phi"] & \Sht^\loc_\cI \ar[l,"\delta"swap]\ar[rd,"\Nt"]\ar[d,"\phi"] \\
        \Iw\backslash LG/\Iw & \Sht^\loc_\cI \ar[l,"\delta"swap]\ar[r,"\Nt"] & \Isoc_G.
    \end{tikzcd}$$
    Here we use that the relative Frobenius on $\Isoc_G$ is canonically trivial. We have
    $$(\eta_{\dot{w}})_*(\Nt_*\delta^!\cF)\simeq (\Nt^*)_*(\delta^*)^!(\Delta_{w^{-1}}\star \eta\cF)=\Ch^\unip_{LG^*,\phi}(\Delta_{w^{-1}}\star \eta\cF)$$
    and
    $$\begin{aligned}
        (\eta_{\dot{w}})_*(\Nt_*\delta^!\phi_*\cF)&\simeq (\Nt^*)_*(\delta^*)^!(\Delta_{w^{-1}}\star (\Delta_w\star \phi_*\eta\cF\star \Delta_{w^{-1}}))\\ &\simeq (\Nt^*)_*(\delta^*)^!(\phi_*\eta\cF\star \Delta_{w^{-1}}).
    \end{aligned}$$
    One can check that the isomorphism $(\eta_{\dot{w}})_*(\can)$ agrees with the composition
    $$(\eta_{\dot{w}})_*(\Nt_*\delta^!\cF)\simeq (\Nt^*)_*(\delta^*)^!(\Delta_{w^{-1}}\star \eta\cF)\xrightarrow{\can}(\Nt^*)_*(\delta^*)^!(\phi_*\eta\cF\star \Delta_{w^{-1}})\simeq(\eta_{\dot{w}})_*(\Nt_*\delta^!\phi_*\cF),$$
    where $\can$ is defined using partial Frobenius, and is equal to $ (-1)^{d(\cF)d(\Delta_{w^{-1}})}\sigma_{\eta\cF,\Delta_{w^{-1}}}$ where $\sigma_{-,-}$ is the commutative constraint of categorical traces as in \S\ref{subsubsection-unipotent-cll}. 
    
    By the above discussion, we know that the composition 
    $$\act_{\phi^s}^\geo\colon \Nt_*\delta^!Z_\mu\xrightarrow{\can^s}\Nt_*\delta^!(\phi^s)_*Z_\mu\xrightarrow{F^{-1}}\Nt_*\delta^!Z_\mu$$
    is identified with
    $$\act_{\phi^s}\otimes\id\colon \widetilde{V}_\mu\otimes \fA_b\xrightarrow{\sim} \widetilde{V}_\mu\otimes \fA_b.$$
    under the functor $\LL_G^{\unip}$.

    Therefore we only need to identify the action on the right hand side induced by $\act_{\phi^s}^\geo$. The morphism $\act_{\phi^s}^\geo$ is induced by the cohomological correspondence
    $$F^{-1}\colon \delta^!Z_\mu \to (\phi^s)^*(\delta^!Z_\mu)\simeq(\phi^s)^!(\delta^!Z_\mu)$$
    supported on 
    $$\Sht^\loc_{\cI,\mu} \xleftarrow{\id} \Sht^\loc_{\cI,\mu} \xrightarrow{\phi^s}\Sht^\loc_{\cI,\mu}.$$
    By \cite[Proposition 5.2.15]{DHKZ-igusa} (see also \cite[Remark 6.11]{Tame}), the relative Frobenius on $\Igs$ is canonically trivial. Therefore the commutative diagram
    $$\begin{tikzcd}
        \Sh_\mu \ar[r,"\phi^s"]\ar[d,"\loc_p"swap] & \Sh_\mu \ar[d,"\loc_p"] \\
        \Sht_{\cI,\mu} \ar[r,"\phi^s"] & \Sht_{\cI,\mu}
    \end{tikzcd}$$
    is Cartesian. Recall that the Weil sheaf $(\loc_p)^!\delta^!Z_\mu$ is identified with the nearby cycle sheaf $R\Psi(\Lambda(d/2)[d])$ on $\Sh_\mu$. 
    By Lemma \ref{lemma-coh-correspondence}, the action of $\act^\geo_{\phi^s}$ on 
    $$\Hom(\Nt_*\delta^!Z_\mu,\fI^\can)\simeq R\Gamma_c(\Sh_\mu, (\DD_{\Sh_\mu})^\omega (R\Psi(\Lambda(d/2)[d])))$$
    is induced by the cohomological correspondence 
    $$\DD((\loc_p)^!F^{-1})\colon (\phi^s)^*(\DD_{\Sh_\mu})^\omega(R\Psi(\Lambda(d/2)[d]))\to (\DD_{\Sh_\mu})^\omega(R\Psi(\Lambda(d/2)[d]))$$
    supported on 
    $$\Sh_\mu\xleftarrow{\phi^s}\Sh_\mu\xrightarrow{\id}\Sh_\mu.$$
    The nearby cycle $R\Psi(\Lambda(d/2)[d])$ is self-dual under Verdier duality. Thus the above cohomological correspondence is identified with 
    $$F\colon (\phi^s)^*R\Psi(\Lambda(d/2)[d])\to R\Psi(\Lambda(d/2)[d])$$
    defined by the Weil structure on $R\Psi(\Lambda(d/2)[d])$. Now the claim follows from the usual theory of Weil sheaves.
\end{proof}

Note that the action of the geometric Frobenius on $\Hom(\widetilde{V}_\mu\otimes \fA_b,\fI^{\can,\unip}_\spec)$ is induced by $\act_{\phi^s}\otimes\id$. Thus the Frobenius actions are compatible.

Now we consider the $\tau$-actions. Let
$m_{Z_\mu}\colon Z_\mu\to Z_\mu$
denote the monodromy action. As explained in \cite[Corollary 4.21]{ALWY-mixed-central}, there is a unipotent action of $I_{\QQ_p}$ on $Z_\mu$. Then $m_{Z_\mu}$ is given by the action of $\tau\in I_{\QQ_p}$. When pulled back to $\Sh_\mu$, the action of $m_{Z_\mu}$ on $R\Psi(\Lambda(d/2)[d])$ equals to the usual action of $\tau$ on nearby cycles. Thus the action of $\tau$ on $R\Gamma_c(\Sh_\mu,R\Psi(\Lambda(d/2)[d]))\simeq \Hom(\Nt_*\delta^!Z_\mu,\fI^\can)$ is induced by $m_{Z_\mu}^{-1}$. On the other hand, the equivalence $\BB^\unip_{G^*}$ matches the monodromy action on $\eta Z_\mu$ with  the tautological automorphism
$$\act_\tau\colon \iota_*\omega_{\hat{U}/\hat{B}}\otimes V_\mu\xrightarrow{\sim}\iota_*\omega_{\hat{U}/\hat{B}}\otimes V_\mu$$
defined by the $\hat{G}$-action on $V_\mu$. After applying the functor $\Ch^\unip_{{}^LG,\phi}$ to $({}_{\dot{w}^{-1}}\eta)_*Z_\mu\simeq\Delta_{w^{-1}}\star \eta Z_\mu$, the automorphism $\act_\tau$ becomes the tautological action of $\tau$ on $\widetilde{V}_\mu\otimes\fA_b$. Thus the isomorphism in Proposition \ref{prop-coh-of-Sh-igusa-sheaf} is also compatible with $\tau$-actions.
\end{proof}

\subsection{$S=T$ for Shimura varieties of Iwahori level structure}\label{subsection-S=T}
In this subsection, we generalize the $S=T$ result for Shimura varieties in \cite{Wu-S-equal-T} to the Iwahori level, and prove the $H_I$-actions part in Theorem \ref{thm-local-global-compatibility}.

\subsubsection{Hecke stacks and $p$-adic shtukas}
We recall the theory of $p$-adic shtukas following \cite{SW-Berkeley-notes} and \cite{PR-p-adic-shtuka}.

Let $\Perfd^\Aff_k$ be the category of affinoid perfectoid algebras over $k$. For $S=\Spa(R,R^+)\in \Perfd^\Aff_k$, we denote
$$\cY_{S,[0,\infty)}\coloneqq \Spa(W(R^+))\backslash \{[\varpi]=0\},$$
where $\varpi\in R^+$ is a pseudo-uniformizer and $[\varpi]\in W(R^+)$ is the Teichm\"uller lift of $\varpi$. We also denote
$$\cY_{S,(0,\infty)}\coloneqq \Spa(W(R^+))\backslash \{p[\varpi]=0\}.$$
Note that for $S\in\Perfd^\Aff_k$, the absolute $q$-Frobenius on $S$ lifts to Frobenius endomorphisms $\phi_S$ on $\cY_{S,[0,\infty)}$ and $\cY_{S,(0,\infty)}$. By \cite[Proposition 11.3.1]{SW-Berkeley-notes}, an untilt $S^\sharp$ of $S$ defines a closed Cartier divisor $S^\sharp\hookrightarrow \cY_{S,[0,\infty)}$. Moreover, the untilt $S^\sharp$ is defined over $\QQ_p$ if and only if the associated Cartier divisor factors through $\cY_{S,(0,\infty)}$.

For $S\in\Perfd_k^\Aff$ and an untilt $S^\sharp$, let $B^+(S^\sharp)$ denote the completion of $\cO_{\cY_{S,[0,\infty)}}$ at the ideal $\cI_{S^\sharp}$ defining $S^\sharp$. Denote $B(S^\sharp)=B^+(S^\sharp)[1/\cI_{S^\sharp}]$. If $S^\sharp$ has characteristic $0$, then $B^+(S^\sharp)=B^+_\mathrm{dR}(S^\sharp)$; and if $S^\sharp=S$ has characteristic $p$, then $B^+(S)=W(S)$. 

\begin{defn}
    Define the $p$-adic local Hecke stack $\cHk=\cHk_\cI$ to be the $v$-stack sending $S\in \Perfd_k^\Aff$ to the datam $(S^\sharp,\cE_1,\cE_2,\beta)$ where
    \begin{enumerate}
        \item $S^\sharp$ is an untilt of $S$.
        \item $\cE_1$ and $\cE_2$ are $\cI$-torsors over $\Spec B^+(S^\sharp)$.
        \item $\beta\colon \cE_2|_{\Spec B(S^\sharp)}\xrightarrow{\sim}\cE_1|_{\Spec B(S^\sharp)}$ is an isomorphism of $G$-torsors.
    \end{enumerate}
\end{defn}

Let $L^+_{\cY_{[0,\infty)}}\cI$ (resp. $L_{\cY_{[0,\infty)}}G$) denote the positive loop group (resp. loop group) over $\Spd\breve\ZZ_p$ sending $S$ to an untilt $S^\sharp$ together with an element in $\cI(B^+(S^\sharp))$ (resp. $G(B(S^\sharp))$).
Let $\cGr=\cGr_\cI\coloneqq L_{\cY_{[0,\infty)}}G/L^+_{\cY_{[0,\infty)}}\cI$ denote the Beilinson--Drinfeld affine Grassmanian over $\Spd\breve\ZZ_p$. Then there is an isomorphism $\cHk\simeq L^+_{\cY_{[0,\infty)}}\cI\backslash \cGr$. 

We recall the definition of $p$-adic shtukas.

\begin{defn}
    Let $S\in\Perfd_k^\Aff$ with an untilt $S^\sharp\hookrightarrow \cY_{S,[0,\infty)}$. An $\cI$-shtuka over $(R,R^+)$ with one leg at $S^\sharp$ is a pair $(\cE,\phi_\cE)$ where
\begin{enumerate}
    \item $\cE$ is an $\cI$-torsor over $\cY_{S,[0,\infty)}$,
    \item $\phi_\cE$ is an $\cI$-torsor isomorphism
        $$\phi_\cE\colon (\phi_{S}^*\cE)|_{\cY_{S,[0,\infty)}\backslash S^\sharp}\xrightarrow{\sim}\cE_{\cY_{S,[0,\infty)}\backslash S^\sharp}$$
        that is meromorphic along $S^\sharp$.
\end{enumerate}
Let $\cSht=\cSht_\cI$ denote the $v$-stack over $\Spd\breve\ZZ_p$ classifies $\cI$-shtukas with one leg.
\end{defn}

\begin{defn}
    Let $(\cE_1,\phi_{\cE_1}), (\cE_2,\phi_{\cE_2})$ be two $\cI$-shtukas over $S$ with one leg at $S^\sharp$. A \emph{quasi-isogeny} between $(\cE_1,\phi_{\cE_1})$ and $(\cE_2,\phi_{\cE_2})$ is an isomorphism
$\gamma\colon \cE_1|_{\cY_{S,(0,\infty)}}\xrightarrow{\sim}\cE_2|_{\cY_{S,(0,\infty)}}$
that is meromorphic along $S$ and satisfies $\phi_{\cE_1}\circ\phi^*_S(\gamma)=\gamma\circ \phi_{\cE_2}$.
\end{defn}
 
\begin{defn}
    Let $n\geq 0$. We define the $v$-stack $\Hk_n(\cSht)$ by sending $S\in \Perfd_k^\Aff$ to the groupoid of data
$$(S^\sharp, (\cE_i,\phi_{\cE_i})_{i=0,\dots,n},(\gamma_j)_{j=1,\dots,n})$$
where
\begin{enumerate}
    \item $S^\sharp$ is an untilt of $S$.
    \item For $i=0,\dots,n$, $(\cE_i,\phi_{\cE_i})$ is an $\cI$-shtuka with obe leg at $S^\sharp$.
    \item For $j=1,\dots,n$, $\gamma_j\colon (\cE_j,\phi_{\cE_j})\dashrightarrow (\cE_{j-1},\phi_{\cE_{j-1}})$ is a quasi-isogeny.
\end{enumerate}
\end{defn}

More concretely, the stack $\Hk_n(\cSht)$ classifies commutative diagrams of the form
$$\begin{tikzcd}
    \phi^*\cE_n \ar[d,dashrightarrow,"\phi_{\cE_n}"]\ar[r,dashrightarrow,"\phi^*(\gamma_n)"] & \phi^*\cE_{n-1} \ar[r,dashrightarrow,"\phi^*(\gamma_{n-1})"]\ar[d,dashrightarrow,"\phi_{\cE_{n-1}}"] & \cdots \ar[r,dashrightarrow,"\phi^*(\gamma_1)"] & \phi^*\cE_0 \ar[d,dashrightarrow,"\phi_{\cE_0}"] \\
    \cE_n \ar[r,dashrightarrow,"\gamma_n"] & \cE_{n-1} \ar[r,dashrightarrow,"\gamma_{n-1}"] & \cdots \ar[r,dashrightarrow,"\gamma_1"] & \cE_0 
\end{tikzcd}$$
where $\phi_{\cE_i}$ are modifications at $S^\sharp$ and $\gamma_j$ are modifications at $S$. The collection $\Hk_\bullet(\cSht)$ clearly form a groupoid in $v$-stacks. 

Let $\mu\in \XX_\bullet(T)$ be a dominant coweight. We say that an $\cI$-shtuka $(\cE,\phi_\cE)$ is bounded by $\mu$ if the modification $\phi_\cE^{-1}$ is bounded by the $v$-local model $\MM_{\cI,\mu}^v$ as in \cite[Definition 2.4.2]{PR-p-adic-shtuka}. We denote by $\cSht_\mu$ the closed substack of $\cSht$ classifying $\cI$-shtukas with one leg bounded by $\mu$. We can similarly define a simplicial stack $\Hk_\bullet(\cSht_\mu)$ by requiring each $\phi_{\cE_i}^{-1}$ bounded by $\MM_{\cI,\mu}^v$. Denote by $|\Hk_\bullet(\cSht_\mu)|$ the geometric realization of $\Hk_\bullet(\cSht_\mu)$ as $v$-stacks.

Let $C$ denote the completion of $\overline{\QQ}_p$. Let $O_C$ denote the ring of integers of $C$. There are natural morphisms $\Spd C\to \Spd O_C\leftarrow \Spd k$ and $\Spd O_C\to \Spd\breve\ZZ_p$.

Let $\Gr_{C}$ denote the $B^+_\mathrm{dR}$-affine Grassmannian over $C$ as in \cite[\S 19]{SW-Berkeley-notes}. For $\mu\in \XX_\bullet(T)^+$, let $\Gr_{C,\mu^*}$ denote the closed Schubert cell associated to $\mu^*$. We have the following result over generic fiber.

\begin{prop}\label{prop-shtuka-char-0}
    For $n\geq 0$, there is a natural isomorphism
    $$\Hk_n(\cSht_{C})\simeq I\backslash G(\QQ_p)\times^I G(\QQ_p)\times^I\cdots\times^I\Gr_{G,C},$$
    where there are $n$ factors of $G(\QQ_p)$. If $\mu\in \XX_\bullet(T)^+$, the above isomorphism restricts to an isomorphism
    $$\Hk_n(\cSht_{C,\mu})\simeq I\backslash G(\QQ_p)\times^I G(\QQ_p)\times^I\cdots\times^I\Gr_{C,\mu^*}.$$
    In particular, we have $\cSht_C\simeq I\backslash\Gr_{G,C}$ and $\cSht_{C,\mu}\simeq I\backslash\Gr_{C,\mu^*}$. 
\end{prop}
\begin{proof}
    Let $S\in\Perfd_k^\Aff$ with an untilt $S^\sharp$ over $C$. By \cite[Proposition 2.5.1]{PR-p-adic-shtuka}, giving an $\cI$-shtuka $(\cE,\phi_\cE)$ with one leg at $S^\sharp$ is equivalent to giving an $I$-torsor $\PP_\cE$ over $S$ together with a $I$-equivariant morphism $\PP_\cE\to \Gr_{C}$. Moreover, two shtukas $(\cE_1,\phi_{\cE_1})$ and $(\cE_2,\phi_{\cE_2})$ are $p$-isogenous if and only there is an isomorphism $\PP_{\cE_1}\times^IG(\QQ_p)\simeq \PP_{\cE_2}\times^IG(\QQ_p)$ of $G(\QQ_p)$-torsors that is compactible with the morphisms to $\Gr_{G,C}$. This proves the first isomorphism. The bounded-by-$\mu$ case follows from \cite[Proposition 2.5.2]{PR-p-adic-shtuka}.
\end{proof}

By Proposition \ref{prop-shtuka-char-0}, we know that there is an isomorphism
$$|\Hk_\bullet(\cSht_{C,\mu})|\simeq G(\QQ_p)\backslash\Gr_{C,\mu^*}$$
over the generic fiber.

\subsubsection{$p$-quasi-isogeny stack}
Let $X$ be a scheme over $\breve\ZZ_p$. We recall two ways to associate a $v$-sheaf to $X$:
\begin{enumerate}
    \item Let $X^\diamond$ be the $v$-sheafification of the presheaf sending $(R,R^+)\in\Perfd_k^\Aff$ to an untilt $(R^\sharp,R^{\sharp+})$ and a morphism $\Spec R^{\sharp+}\to X$.
    \item Let $X^\Diamond$ be the $v$-sheafification of the presheaf sending $(R,R^+)\in\Perfd_k^\Aff$ to an untilt $(R^\sharp,R^{\sharp+})$ and a morphism $\Spec R^{\sharp}\to X$.
\end{enumerate}
If $X$ is a scheme over $k$, then the $v$-sheaves $X^\diamond$ and $X^\Diamond$ depends only on the perfection $X^\perf$ of $X$. We can extend the above definitions to perfect stacks over $k$. There is a natural morphism $j_X\colon X^\diamond\to X^\Diamond$ of $v$-stacks. If $X$ is separated scheme of finite type over $\breve\ZZ_p$, then $j_X$ is an open immersion. If moreover $X$ is proper, then $j_X$ is an isomorphism.

\begin{prop}
    There is are natural isomorphisms $(\Fl_\cI)^\Diamond\simeq\cGr_{\cI,k}$, $(\Iw)^\Diamond\simeq(L^+_{\cY_{S,[0,\infty)}}\cI)_k$, and $(\Iw\backslash LG/\Iw)^\Diamond\simeq \cHk_{\cI,k}$.
\end{prop}
\begin{proof}
    Clear from definitions.
\end{proof}

By \cite[Theorem VII]{DHKZ-igusa}, there is a Cartesian diagram
$$\begin{tikzcd}
    \scS_{K^pI}(\sG,\sX)_{O_C}^\diamond \ar[r,"\loc_p"]\ar[d,"\BL^\glob"swap] & \cSht_{O_C,\mu} \ar[d,"\BL"] \\
    \Igs \ar[r,"\loc_p^0"] & \Bun_{G,\leq\mu^*}
\end{tikzcd}$$
%\Xinwen{Do you want to use $\pi_{HT}$ to be consistent with their notation?}\Xiangqian{Maybe keep the notation?} \Xinwen{I'm fine either way.}
of $v$-stacks over $k$, where $\Bun_G$ is the moduli stack of $G$-bundles over the Fargues--Fontaine curve. Note that the notations of morphisms are different with \emph{loc. cit.}. We use $\loc_p$ to denote the crystalline period map $\pi_{\crys,\cI}$ and $\loc_p^0$ to denote the Hodge--Tate period map $\bar\pi_\HT$ of the Igusa stack, so that the notations is consistent with Theorem \ref{thm-igusa}. After applying the reduction functor in \cite{GL-meromorphic}, we obtain the $h$-sheafification of the diagram in Theorem \ref{thm-igusa}. For simplicity, we denote $\scS^\diamond_{O_C}=\scS_{K^pI}(\sG,\sX)_{O_C}^\diamond$ in this subsection.

\begin{prop}\label{prop-p-isog-stack}
    The pullback of the groupoid $\Hk_\bullet(\cSht_{O_C,\mu})$ defines a groupoid $\Hk_\bullet(\scS^\diamond_{O_C})$ over $\scS^\diamond_{O_C}$.
\end{prop}
\begin{proof}
    The groupoid $\Hk_\bullet(\cSht_{O_C,\mu})$ is identified with the \v{C}ech nerve of $\cSht_{O_C,\mu}\to |\Hk_\bullet(\cSht_{O_C,\mu})|$. The morphism $\BL$ is defined by sending an $\cI$-shtuka $(\cE,\phi_\cE)$ over $S$ to its restriction $\cE_{\cY_{S,[r,\infty)}}$ for $r\gg 0$ together with $\phi$-structure. Therefore $\BL$ factors through a morphism $|\Hk_\bullet(\cSht_{O_C,\mu})|\to \Bun_{G,\leq\mu^*}$. Therefore the pullback of $\Hk_\bullet(\cSht_{O_C,\mu})$ to $\scS^\diamond_{O_C}$ is identified with the \v{C}ech nerve of $$\scS^\diamond_{O_C}\to \Igs\times_{\Bun_{G,\leq\mu^*}} |\Hk_\bullet(\cSht_{O_C,\mu})|.$$
    In particular, there is a commutative diagram
    $$\begin{tikzcd}
        \scS_{O_C}^\diamond \ar[d] & \Hk_1(\scS_{O_C}^\diamond) \ar[l]\ar[r]\ar[d]\ar[ld,phantom,"\square",very near start]\ar[rd,phantom,"\square",very near start] & \scS_{O_C}^\diamond \ar[d] \\
        \cSht_\mu & \Hk_1(\cSht_\mu) \ar[l]\ar[r] & \cSht_\mu,
    \end{tikzcd}$$
    where both squares are Cartesian.
\end{proof}

We can view $\Hk_1(\scS_{O_C}^\diamond)$ as the $v$-stack which classifies $p$-quasi-isogenies between points in $\scS_{O_C}^\diamond$.

\subsubsection{$S$-operators}

In this subsection, whenever we have a $v$-stack $X_{O_C}$ over $\Spd O_C$, we denote by
$$X_k\xrightarrow{i} X_{O_C}\xleftarrow{j} X_C$$
the embedding of the special fiber and generic fiber respectively.
%\Xinwen{Do we need torsioin coefficient in the sequel for have a good theory of $D_{\et}$ in analytic setting? I guess it would not affect $S=T$ as it can be checked at integral coefficient level and then $Z/\ell^n$ level for every $n$.}\Xiangqian{I think it is not needed, Scholze introduce the notion of etale $\ZZ_\ell$-sheaves with 6-functors over diamonds, which agrees with the limit of $\ZZ/\ell^n\ZZ$-sheaves. As long as we are not dealing with things like $\BB G(F)$, this is no problem. e.g. \cite{ALWY-mixed-central} works with $\Lambda=\overline\FF_\ell$ or $\overline\QQ_\ell$.} \Xinwen{Then I'm worrying about $\BB I$. The $Z_l$ sheaf defined via limit is not what one wants. See discussions between 10.8 and 10.9 in \cite{Tame}. For $\Iw\backslash LG/\Iw$, this is ok.}\Xiangqian{I have rearrange this part. Please check.}

We first recall the definition of central sheaves in \cite{ALWY-mixed-central}. Let $L$ be a finite extension of $\QQ_\ell$ with a uniformizer $\varpi_L\in O_L$. Denote $\Lambda=O_L/\varpi_L^n$ for some integer $n\geq 1$. Let $D_\et(-,\Lambda)$ denote the \'etale sheaf theory over $v$-stacks defined in \cite[Definition 14.13]{Scholze-diamonds}. Let $\cS_\mu\in D_\et^\ULA(\cHk_{C},\Lambda)$ denote the perverse sheaf associated to $V_\mu$ under the geometric Satake equivalence in \cite[Chapter VI]{FS}. As $\mu$ is minuscule, $\cS_\mu$ is simply a twist of constant sheaf on $\cHk_{C,\mu}$. By \cite[Proposition 6.12]{AGLR-local-model}, the pullback
$$j^*\colon D_\et^\ULA(\cHk_{O_C},\Lambda)\to D_\et^\ULA(\cHk_{C},\Lambda)$$
is an equivalence of categories with inverse given by $j_*$. Thus the object
$$\cZ_\mu\coloneqq j_*\cS_\mu$$
is universally locally acyclic relative to $\cHk_{O_C}\to\Spd O_C$. 

 \begin{rmk}
We recall that by definition 
\[
D_\et^\ULA(\cHk_k,\Lambda)=\colim_{w}\,\colim_{n}\, D_\et^{\ULA}((\Iw_n\backslash LG_{\leq w}/\Iw)^{\Diamond},\Lambda),
\]
where $LG_{\leq w}$ is the closure of $LG_w$ in $LG$, and $\Iw_n=L^{n}\cI$ is the $n$-th jet group of $\cI$ such that the action of $\Iw$ on $\Fl_{\cI,\leq w}=LG_{\leq w}/\Iw$ factors through the action of $\Iw_n$. As proved in \cite[Proposition 6.7, Proposition A.5]{AGLR-local-model}, the natural comparison functor from \cite[\S 27]{Scholze-diamonds} induces an equivalence $D_\et^{\ULA}((\Iw_n\backslash LG_{\leq w}/\Iw)^{\Diamond},\Lambda)\cong \Shv^*_c(\Iw_n\backslash LG_{\leq w}/\Iw,\Lambda)$, where $\Shv^*_c$ is the category of usual constructible sheaves on the perfect algebraic stack $\Iw_n\backslash LG_{\leq w}/\Iw=\Iw_n\backslash \Fl_{\cI,\leq w}$.

 By \cite[Remark 10.89, Remark 10.131]{Tame}, there is a canonical equivalence
$$\Shv_\fingen(\Iw_n\backslash LG_{\leq w}/\Iw,\Lambda)\simeq \Shv^*_{c}(\Iw_n\backslash LG_{\leq w}/\Iw,\Lambda),$$
fitting into the following commutative diagram
$$\begin{tikzcd}
    \Shv_\fingen(\Iw_n\backslash LG/\Iw,\Lambda)\ar[r,"\sim"]\ar[d,"t^!"swap] & \Shv^*_c(\Iw_n\backslash LG/\Iw,\Lambda) \ar[d,"t^*"] \\
    \Shv_\fingen(\Fl_{\cI,\leq w},\Lambda) \ar[r,"\sim"] & \Shv_c^*(\Fl_{\cI,\leq w},\Lambda),
\end{tikzcd}$$
where $t\colon \Fl_{\cI}\to \Iw\backslash LG/\Iw$ is the projection, and where the lower arrow is the canonical equivalence $\Shv_\fingen(\Fl_\cI,\Lambda)=\Shv_c^*(\Fl_{\cI,\leq w},\Lambda)^{\mathrm{op}}\simeq \Shv_c^*(\Fl_{\cI,\leq w},\Lambda)$ given by the usual Verider duality of $\Shv_c^*(\Fl_{\cI,\leq w},\Lambda)$. 

It follows that we have a canonical equivalence $$
D_{\et}^\ULA(\cHk_k,\Lambda)\cong \Shv_\fingen(\Iw\backslash LG/\Iw,\Lambda).$$ 

Under this equivalence, the pullback $Z_\mu^*\coloneqq i^*\cZ_\mu\in D_\et^\ULA(\cHk_k,\Lambda)$ corresponds to an object in $\Shv_\fingen(\Iw\backslash LG/\Iw,\Lambda)$, which is (by defintion) the central sheaf $Z_\mu$. Note that $Z_\mu$ belongs to $\Shv_{\fingen}(\Iw\backslash LG_{\leq w}/\Iw,\Lambda)$ for some $w$. By abuse of notations, we shall also write $Z_\mu^*$ for the corresponding object in $\Shv_c^*(\Iw\backslash LG_{\leq w}/\Iw,\Lambda)$ under the above mentioned equivalence from \cite[Remark 10.89, Remark 10.131]{Tame}.

%\Xinwen{I changed $D_{\et,c}$ to $\Shv_c^*$. I made the replacement in later places, but not sure whether have changed }
%Let $Z_\mu^*\in D_{\et,c}(\Iw\backslash LG/\Iw,\Lambda)$ denote the object associated to $Z_\mu$. By \cite[Proposition 4.3]{ALWY-mixed-central}, the analytification functor
%$$D_\et(\Iw\backslash LG/\Iw,\Lambda)\xrightarrow{\sim} D_\et(\cHk_k,\Lambda)$$
%is an equivalence that restricts to an equivalence
%$$D_{\et,c}(\Iw\backslash LG/\Iw,\Lambda)\xrightarrow{\sim} D_\et^\ULA(\cHk_k,\Lambda).$$
%Under this equivalence, the object $Z_\mu^*$ matches with $i^*\cZ_\mu$.     
\end{rmk}
 
Consider the morphism
$$\delta\colon \cSht_{O_C,\mu}\to \cHk_{O_C}$$
sending a $p$-adic shtuka $\phi_\cE\colon\phi^*\cE\dashrightarrow \cE$ with one leg at $S^\sharp$ to the modification $(\phi_\cE)^{-1}\colon \cE|_{\Spec B^+(S^\sharp)}\dashrightarrow\phi^*\cE|_{\Spec B^+(S^\sharp)}$. The goal of this subsection is to define the $S$-operator as a cohomological self-correspondence on $(\cSht_\mu,\delta^*\cZ_\mu)$. 

We define some iterated Hecke stacks following \cite{AR-central}. Let $\sy\in \{\sx,\szero,\sx\cup \szero\}$ be a symbol. Let $S\in\Perfd^\Aff_k$ together with an untilt $S^\sharp$ over $O_C$. Denote
$$\Gamma_\sy=\left\{\begin{aligned}
    & S^\sharp, &&\text{if }\sy=\sx,\\
    & S, &&\text{if }\sy=\szero,\\
    & S^\sharp\cup S, &&\text{if }\sy=\sx\cup \szero,
\end{aligned}\right.$$
as Zariski closed subspaces of $\cY_{S,[0,\infty)}$. We denote by $B_{\sx\cup\szero}^+(S^\sharp)$ the completion of $\cO_{\cY_{S,[0,\infty)}}$ at $S^\sharp\cup S$.

\begin{defn}
    For a sequence $\underline{\sy}=(\sy_1,\dots,\sy_n)$ of symbols in $\{\sx,\szero,\sx\cup \szero\}$, define the iterated Hecke stack $\cHk_{O_C}(\underline{\sy})$ over $\Spd O_C$ by sending $S\in\Perfd^\Aff_k$ to the groupoid of $(S^\sharp,(\cE_t)_{t=0,\dots,n},(\beta_t)_{t=1,\dots,n})$ where
    \begin{enumerate}
        \item $S^\sharp$ is an untilt of $S$ over $O_C$.
        \item For $t=0,\dots,n$, $\cE_t$ is an $\cI$-torsors over $\Spec B_{x\cup p}^+(S^\sharp)$.
        \item For $t=1,\dots,n$, 
            $$\beta_t\colon \cE_t|_{\Spec B_{x\cup p}^+(S^\sharp)\backslash \Gamma_{\sy_t}}\xrightarrow{\sim}\cE_{t-1}|_{\Spec B_{x\cup p}^+(S^\sharp)\backslash \Gamma_{\sy_t}}$$
            is an isomorphism of $\cI$-torsors.
    \end{enumerate}
\end{defn}

There is a natural morphism $p_{\underline{\sy}}\colon \cHk_{O_C}(\underline{\sy})\to\prod_{i=1}^n\cHk_{O_C}(\sy_i)$. Similarly, we can define the stack of iterated shtuka.

\begin{defn}
    For a sequence $\underline{\sy}=(\sy_1,\dots,\sy_n)$ of symbols in $\{\sx,\szero,\sx\cup \szero\}$, define the $v$-stack $\cSht_{O_C}(\underline{\sy})$ by sending $S\in\Perfd^\Aff_k$ to the groupoid of the following data:
    \begin{enumerate}
        \item An untilt $S^\sharp$ of $S$ over $O_C$.
        \item For $t=0,\dots,n$, an $\cI$-torsor $\cE_t$ over $\cY_{S,[0,\infty)}$.
        \item For $t=1,\dots,n$, an isomorphism
        $$\beta_t\colon \cE_{t}|_{\cY_{S,[0,\infty)}\backslash \Gamma_\sy}\xrightarrow\sim  \cE_{t-1}|_{\cY_{S,[0,\infty)}\backslash \Gamma_\sy}$$
        of $\cI$-torsors that is meromorphic along $\Gamma_\sy$.
        \item An isomorphism $\alpha\colon \cE_n\simeq \phi_S^*\cE_0$.
    \end{enumerate}
\end{defn}

If $\underline{\sy}=(\sy_1,\dots,\sy_n)$ is a sequence of symbols, we denote $\underline{\sy}^c=(\sy_n,\dots,\sy_1)$. Let $$\delta_{\underline{\sy}}\colon \cSht_{O_C}(\underline{\sy})\to \cHk_{O_C}(\underline{\sy}^c)$$
be the morphism sending $(S^\sharp, (\cE_i)_{i=0,\dots,n},(\beta_i)_{i=1,\dots,n},\alpha)$ to $(S^\sharp,(\cE_{n-t})_{t=0,\dots,n},((\beta_{n+1-t})^{-1})_{t=1,\dots,n}))$.

Assume that $\underline{\sy}=(\sy_1,\dots,\sy_n)$ is a sequences of symbols with exactly one $\sy_t$ containing $\sx$. As explained in \cite[\S 2.3.7 and \S 2.3.8]{AR-central},
the special fiber $\cHk_k(\underline{\sy})$ is identified with the $n$-iterated affine Hecke stack 
$$(\underbrace{\Iw\backslash LG\times^\Iw LG\times^\Iw\cdots \times^\Iw LG/\Iw}_{n\text{ terms of }LG})^\Diamond.$$
The generic fiber $\cHk_C(\underline{\sy})$ splits into a product
$$\cHk_C\times (\underbrace{\Iw\backslash LG\times^\Iw LG\times^\Iw\cdots \times^\Iw LG/\Iw}_{m\text{ terms of }LG})^\Diamond,$$
where $m$ is the number of those $\sy_t$ containing $\szero$. There is a natural embedding $\cHk_{O_C}(\sx)\hookrightarrow\cHk_{O_C}(\sx\cup \szero)$ identifying $\cHk_{O_C}(\sx)$ with the closure of 
$$\cHk_{C}\times (\BB\Iw)^\Diamond\hookrightarrow \cHk_C\times (\Iw\backslash LG/\Iw)^\Diamond\simeq \cHk_C(\sx\cup\szero).$$

We study some special cases that will be needed later. Consider $\cHk_{O_C}(\szero,\sx\cup \szero)$. We have 
$$\cHk_{C}(\szero,\sx\cup \szero)\simeq \cHk_C\times (\Iw\backslash LG\times^\Iw LG/\Iw)^\Diamond\quad\text{and}\quad\cHk_{k}(\szero,\sx\cup \szero)\simeq (\Iw\backslash LG\times^\Iw LG/\Iw)^\Diamond.$$ Let $\cF_1,\cF_2\in D_\et^\ULA(\cHk_k,\Lambda)$ be two objects. %\Xinwen{Changed $D_{\et,c}$ to $\Shv_c^*$. Pleae check.} 
Denote 
$$\cF_1\widetilde\boxtimes \cF_2=\Delta^*(\cF_1\boxtimes \cF_2)\in D_\et^\ULA((\Iw\backslash LG\times^\Iw LG/\Iw)^\Diamond,\Lambda)$$
for $\Delta\colon \Iw\backslash LG\times^\Iw LG/\Iw\to (\Iw\backslash LG/\Iw)^2$ sending $(\cE_2\dashrightarrow\cE_1\dashrightarrow\cE_0)$ to $(\cE_1\dashrightarrow\cE_0,\cE_2\dashrightarrow\cE_1)$. Consider the sheaf $\cS_\mu\boxtimes (\cF_1\widetilde\boxtimes \cF_2)\in D_{\et}^\ULA(\cHk_{C}(\szero,\sx\cup \szero),\Lambda).$ Then the object 
$$j_*(\cS_\mu\boxtimes (\cF_1\widetilde\boxtimes \cF_2))\in D_{\et}^\ULA(\cHk_{O_C}(\szero,\sx\cup \szero),\Lambda)$$ 
is universally locally acyclic over $\Spd O_C$.

\begin{lemma}\label{lemma-special-fiber-iterated}
    The object $i^*j_*(\cS_\mu\boxtimes (\cF_1\widetilde\boxtimes \cF_2))\in D_\et^\ULA(\cHk_k(\szero,\sx\cup \szero),\Lambda)$ is identified with
    $\cF_1\widetilde\boxtimes (Z^*_\mu\star \cF_2)$
    under the isomorphism $\cHk_{k}(\szero,\sx\cup \szero)\simeq (\Iw\backslash LG\times^\Iw LG/\Iw)^\Diamond$.
\end{lemma}
\begin{proof}
    Consider the stack $\cHk_{O_C}(\szero,\sx,\szero)$. We have
    $$\cHk_{C}(\szero,\sx,\szero)\simeq \cHk_C\times (\Iw\backslash LG\times^\Iw LG/\Iw)^\Diamond$$
    and
    $$\cHk_{k}(\szero,\sx,\szero)\simeq (\Iw\backslash LG\times^\Iw LG\times^\Iw LG/\Iw)^\Diamond.$$ By the same proof of \cite[Lemma 2.4.2]{AR-central}, we know that
    $$i^*j_*(\cS_\mu\boxtimes (\cF_1\widetilde\boxtimes\cF_2))\simeq \cF_1\widetilde{\boxtimes}Z^*_\mu\widetilde{\boxtimes}\cF_2$$
    over $\cHk_{k}(\szero,\sx,\szero)$. There is a morphism 
    $$m\colon \cHk_{O_C}(\szero,\sx,\szero)\to \cHk_{O_C}(\szero,\sx\cup \szero)$$
    by composing the last two modifications. Over the generic fiber, $m$ is an isomorphism, and over the special fiber, $m$ is given by composing the last two modifications. Pushing forward along $m$, we see that
    $$i^*j_*(\cS_\mu\boxtimes (\cF_1\widetilde\boxtimes \cF_2))\simeq \cF_1\widetilde\boxtimes m_*(Z_\mu^*\widetilde\boxtimes\cF_2)\simeq\cF_1\widetilde\boxtimes (Z^*_\mu\star \cF_2)$$
    over $\cHk_{k}(\szero,\sx\cup \szero)$ as desired.
\end{proof}

Similarly, we can consider $\cHk_{O_C}(\sx\cup \szero,\szero)$ and we have 
$$i^*j_*(\cS_\mu\boxtimes (\cF_1\widetilde{\boxtimes}\cF_2))\simeq (Z_\mu^*\star \cF_1)\widetilde{\boxtimes} \cF_2$$
over $\cHk_{k}(\sx\cup \szero,\szero)\simeq(\Iw\backslash LG\times^\Iw LG/\Iw)^\Diamond$.

\begin{defn}
    Let $\underline{\sy}=(\sy_1,\dots,\sy_n)$ and $\underline{\sy}'=(\sy'_1,\dots,\sy'_{n'})$ be two sequences of symbols in $\{\sx,\szero,\sx\cup \szero\}$. Define the $v$-stack $\cHk_{O_C}(\underline{\sy}|\underline{\sy}')$ over $\Spd O_C$ that classifies the following data for $S\in\Perfd^\Aff_k$:  
    \begin{enumerate}
        \item An untilt $S^\sharp$ of $S$ over $O_C$.
        \item A point $(S^\sharp,(\cE_t)_{0\leq t\leq n},(\beta_t)_{1\leq t\leq n})$ of $\cHk_{O_C}(\underline{\sy})$.
        \item A point $(S^\sharp,(\cE'_s)_{0\leq s\leq n'},(\beta'_s)_{1\leq s\leq n'})$ of $\cHk_{O_C}(\underline{\sy}')$.
        \item Two isomorphisms $\alpha_0\colon \cE_0\simeq \cE_0'$ and $\alpha_1\colon \cE_n\simeq \cE_{n'}'$ such that the diagram
        $$\begin{tikzcd}
            \cE_n \ar[r,dashrightarrow,"\beta_n"]\ar[d,"\alpha_1"swap,"\simeq"] & \cE_{n-1} \ar[r,dashrightarrow,"\beta_{n-1}"] & \cdots \ar[r,dashrightarrow,"\beta_1"] & \cE_0 \ar[d,"\alpha_0","\simeq"swap] \\
            \cE'_{n'} \ar[r,dashrightarrow,"\beta'_{n'}"] & \cE'_{n'-1} \ar[r,dashrightarrow,"\beta'_{n'-1}"] & \cdots \ar[r,dashrightarrow,"\beta'_1"] & \cE'_0
        \end{tikzcd}$$
        commutes.
    \end{enumerate}
\end{defn}

Now we can define the $S$-correspondences. The discussions in Example \ref{eg-coh-correspondence} about cohomological correspondences obviously generalized to the setting of $v$-stacks. Let $\cF\in D_{\et}^\ULA(\cHk_k,\Lambda)$ %\Xinwen{Here should be $\cHk_k$ or $\Iw\backslash LG/\Iw$? $D^{\ULA}_{\et}(\cHk_k)$ or $\Shv^*_c(\Iw\backslash LG/\Iw)$?}\Xiangqian{I didn't distinguish $D_{\et,c}(\Iw\backslash LG/\Iw)$ and $D_\et^\ULA(\cHk_k)$ here. It should be $D_\et^\ULA(\cHk_k)$ everywhere rigorously.} 
be an object together with an endomorphism $F\colon \cF\to\phi_*\cF$. By \cite[Lemma 4.42]{Tame} (see also \cite[Proposition VI.8.3]{FS}), the (right) dual $\cF^\vee$ of $\cF$ is identified with $\sw^*\DD(\cF)$, where $\DD$ is the Verdier dual and $\sw\colon \cHk_k\to \cHk_k$ is defined by switching two $\cI$-torsors. Let $Y\subseteq \Iw\backslash LG/\Iw$ be a  $\phi$-stable quasi-compact closed substack such that the objects $\cF$, $\cF^\vee$ are supported on $Y$.

Consider the closed substack $\cHk^{Y,\mu}_{C}(\szero,\sx\cup \szero)$ of $\cHk_{C}(\szero,\sx\cup \szero)$ where the two modifications in characteristic $p$ are bounded by $Y$ and the modification over $C$ is bounded by $\mu$. Then the object $\cS_\mu\boxtimes (\cF\widetilde\boxtimes \cF^\vee)$ is supported on $\cHk^{Y,\mu}_{C}(\szero,\sx\cup \szero)$. Let $\cHk^{Y,\mu}_{O_C}(\szero,\sx\cup \szero)$ be the $v$-closure of $\cHk^{Y,\mu}_{C}(\szero,\sx\cup \szero)$ in $\cHk_{O_C}(\szero,\sx\cup \szero)$. Let $\cHk_{O_C}^\mu(\sx)$ denote the $v$-closure of $\cHk_{C,\mu}\times (\BB\Iw)^\Diamond\subseteq \cHk_C(\sx)$ in $\cHk_{O_C}(\sx)$. Consider the Cartesian diagram
$$\begin{tikzcd}
    \cHk^{Y,\mu}_{O_C}(\sx|(\szero,\sx\cup \szero)) \ar[r,"h_r^\loc"]\ar[d,"h_l^\loc"swap]\ar[rd,phantom,"\square",very near start] & \cHk^{Y,\mu}_{O_C}(\szero,\sx\cup \szero) \ar[d,"m"] \\
    \cHk_{O_C}^\mu(\sx) \ar[r,"s"] & \cHk_{O_C}(\sx\cup \szero).
\end{tikzcd}$$
The stack $\cHk^{Y,\mu}_{O_C}(\sx|(\szero,\sx\cup \szero))$ is a closed substack of $\cHk_{O_C}(\sx|(\szero,\sx\cup \szero))$ defined earlier. The morphisms $h^\loc_l$ and $h^\loc_r$ are proper.

We define a cohomological correspondence 
$$\rC^\loc_+\colon(\cHk_{O_C}^\mu(\sx),j_*(\cS_\mu\boxtimes\Lambda))\to (\cHk^{Y,\mu}_{O_C}(\szero,\sx\cup \szero),j_*(\cS_\mu\boxtimes (\cF\widetilde\boxtimes \cF^\vee)))$$
supported on $\cHk^{Y,\mu}_{O_C}(\sx|(\szero,\sx\cup \szero))$ as follows: Using $\cHk_{C}(\sx\cup\szero)\simeq \cHk_C\times (\Iw\backslash LG/\Iw)^\Diamond$, we have
$$s_*(\cS_\mu\boxtimes \Lambda)\simeq \cS_\mu\boxtimes \Delta_1$$
and
$$m_*(\cS_\mu\boxtimes (\cF\widetilde\boxtimes \cF^\vee))\simeq \cS_\mu\boxtimes (\cF\star\cF^\vee).$$
The unit morphism $\Delta_1\to \cF\star\cF^\vee$ induces a morphism
$s_*(\cS_\mu\boxtimes \Lambda)\to m_*(\cS_\mu\boxtimes (\cF\widetilde{}\boxtimes\cF^\vee))$. Applying $j_*$ defines a morphism $s_*j_*(\cS_\mu\boxtimes \Lambda)\to m_*j_*(\cS_\mu\boxtimes (\cF\widetilde{}\boxtimes\cF^\vee))$. By base change, we obtain a cohomological correspondence
$$\rC_+^\loc\colon (h_l^\loc)^* j_*(\cS_\mu\boxtimes \Lambda)\to (h_r^\loc)^!j_*(\cS_\mu\boxtimes (\cF\widetilde{\boxtimes}\cF^\vee)).$$
There is a commutative diagram
$$\begin{tikzcd}
    \cSht_{O_C,\mu} \ar[d,"\delta_{\sx}"] & \cSht^{Y,\mu}_{O_C}(\sx|(\sx\cup\szero,\szero)) \ar[d]\ar[l,"h_l"swap]\ar[r,"h_r"]\ar[ld,phantom,"\square", very near start]\ar[rd,phantom,"\square", very near start] & \cSht_{O_C}^{Y,\mu}(\sx\cup\szero,\szero) \ar[d,"\delta_{(\sx\cup\szero,\sx)}"] \\
    \cHk_{O_C}^\mu(\sx) & \cHk^{Y,\mu}_{O_C}(\sx|(\szero,\sx\cup\szero)) \ar[l,"h_l^\loc"swap]\ar[r,"h_r^\loc"] & \cHk_{O_C}^{Y,\mu}(\szero,\sx\cup\szero)
\end{tikzcd}$$
such that both squares are Cartesian. Note that $(\delta_\sx)^*(j_*(\cS_\mu\boxtimes\Lambda))\simeq \delta^*\cZ_\mu$ over $\cSht_{O_C,\mu}$. By Example \ref{eg-coh-correspondence} (1), we can define a cohomological correspondence
$$\rC_+\colon (h_l)^* \delta^*\cZ_\mu\to (h_r)^!(\delta_{(\sx\cup\szero,\sx)})^*j_*(\cS_\mu\boxtimes(\cF\widetilde{\boxtimes}\cF^\vee))$$
from $(\cSht_{O_C,\mu},\delta^!\cZ_\mu)$ to $(\cSht_{O_C}^{Y,\mu}(\sx\cup\szero,\szero),(\delta_{(\sx\cup\szero,\sx)})^*j_*(\cS_\mu\boxtimes(\cF\widetilde{\boxtimes}\cF^\vee)))$. We call $\rC_+$ the \emph{creation correspondence}.

By the similar considerations, there is a commutative diagram
$$\begin{tikzcd}
    \cSht_{O_C}^{Y,\mu}(\szero,\sx\cup\szero) \ar[d,"\delta_{(\szero,\sx\cup\szero)}"] & \cSht_{O_C}^{Y,\mu}((\szero,\sx\cup\szero)|\sx) \ar[l,"h_l"swap]\ar[r,"h_r"]\ar[ld,phantom,"\square", very near start]\ar[rd,phantom,"\square", very near start]\ar[d] & \cSht_{O_C,\mu} \ar[d,"\delta_\sx"] \\
    \cHk_{O_C}^{Y,\mu}(\sx\cup\szero,\szero) & \cHk^{Y,\mu}_{O_C}((\sx\cup\szero,\szero)|\sx) \ar[l,"h_l^\loc"swap]\ar[r,"h_r^\loc"] & \cHk_{O_C}(\sx)
\end{tikzcd}$$
with both square Cartesian. The counit morphism $\cF^\vee\star\cF\to \Delta_1$ defines a cohomological correspondence
$$\rC_-^\loc\colon (h_l^\loc)^*j_*(\cS_\mu\boxtimes(\cF^\vee\widetilde\boxtimes\cF))\to (h_r^\loc)^!j_*(\cS_\mu\boxtimes\Lambda)$$
from $( \cHk_{O_C}^{Y,\mu}(\sx\cup\szero,\szero),j_*(\cS_\mu\boxtimes(\cF^\vee\widetilde\boxtimes\cF)))$ to $(\cHk_{O_C}(\sx),j_*(\cS_\mu\boxtimes\Lambda))$. Pulling back to shtukas defines the \emph{annihilation correspondence}
$$\rC_-\colon (h_l)^*(\delta_{(\szero,\sx\cup\szero)})^*j_*(\cS_\mu\boxtimes(\cF^\vee\widetilde\boxtimes\cF))\to (h_r)^!\delta^*\cZ_\mu$$
from $( \cSht_{O_C}^{Y,\mu}(\szero,\sx\cup\szero),(\delta_{(\szero,\sx\cup\szero)})^*j_*(\cS_\mu\boxtimes(\cF^\vee\widetilde\boxtimes\cF)))$ to $(\cSht_{O_C,\mu},\delta^*\cZ_\mu)$.

Define the partial Frobenius
$$\pFr\colon \cSht_{O_C}(\szero,\sx\cup\szero)\to \cSht_{O_C}(\sx\cup\szero,\sx)$$ by sending
$$(\begin{tikzcd}[column sep=small]\phi^*\cE_0\ar[r,"\beta_2","\sx\cup\szero"swap,dashrightarrow] &\cE_1 \ar[r,"\beta_1","\szero"swap,dashrightarrow] &\cE_0\end{tikzcd})\mapsto (\begin{tikzcd}[column sep=small]\phi^*\cE_1\ar[r,"\phi^*(\beta_1)","\szero"swap,dashrightarrow] &\phi^*\cE_0 \ar[r,"\beta_2","\sx\cup\szero"swap,dashrightarrow] &\cE_0\end{tikzcd}).$$
Here the symbols below the arrows indicate the positions of legs. There is a natural commutative diagram
$$\begin{tikzcd}[column sep=huge]
    \cSht_{O_C}(\szero,\sx\cup\szero) \ar[r,"\pFr"]\ar[d,"q_{(\szero,\sx\cup\szero
    )}"swap] & \cSht_{O_C}(\sx\cup\szero,\szero) \ar[d,"q_{(\sx\cup\szero,\szero
    )}"] \\
    \cHk_{O_C}(\sx\cup\szero)\times\cHk_{O_C}(\szero) \ar[r,"(\phi\times\id)\circ\sw"] & \cHk_{O_C}(\szero)\times\cHk(\sx\cup\szero)
\end{tikzcd}$$
where $q_{\underline{\sy}}$ is the composition $\cSht_{O_C}(\underline{\sy})\xrightarrow{\delta_{\underline{\sy}}}\cHk_{O_C}(\underline{\sy}^c)\xrightarrow{p_{\underline{\sy}^c}}\prod_{i=1}^n\cHk_{O_C}(\sy_{n+1-i})$ for $\underline{\sy}=(\sy_1,\dots,\sy_n)$. Consider $j_*(\cS_\mu\boxtimes\cF^\vee)\in D_\et^\ULA(\cHk_{O_C}(\sx\cup\szero),\Lambda)$ and $\phi_*\cF\boxtimes\Lambda\in D_\et^\ULA(\cHk_{O_C}(\szero),\Lambda)$ where $\cHk_{O_C}(\szero)\simeq (\Iw\backslash LG/\Iw)^\Diamond\times\Spd O_C$. The commutative diagram above induces an isomorphism
$$\pFr^*(q_{(\sx\cup\szero,\szero)})^*((\phi_*\cF\boxtimes\Lambda)\boxtimes(j_*(\cS_\mu\boxtimes\cF^\vee)))\simeq (q_{(\szero,\sx\cup\szero)})^*((j_*(\cS_\mu\boxtimes\cF^\vee))\boxtimes (\cF\boxtimes\Lambda)).$$
We see that there is a canonical isomorphism
$$\can\colon \pFr^*(\delta_{(\sx\cup\szero,\szero)})^*j_*(\cS_\mu\boxtimes(\phi_*\cF\widetilde{\boxtimes}\cF^\vee))\simeq (\delta_{(\szero,\sx\cup\szero)})^*j_*(\cS_\mu\boxtimes(\cF^\vee\widetilde{\boxtimes}\cF)).$$
The Weil structure $F\colon \cF\to\phi_*\cF$ defines a cohomological correspondence
$$\rC_{F}\colon \pFr^*(\delta_{(\sx\cup\szero,\szero)})^*j_*(\cS_\mu\boxtimes(\cF\widetilde{\boxtimes}\cF^\vee))\xrightarrow{F}\pFr^*(\delta_{(\sx\cup\szero,\szero)})^*j_*(\cS_\mu\boxtimes(\phi_*\cF\widetilde{\boxtimes}\cF^\vee))\xrightarrow[\simeq]{\varepsilon_\cF\cdot \can} (\delta_{(\szero,\sx\cup\szero)})^*j_*(\cS_\mu\boxtimes(\cF^\vee\widetilde{\boxtimes}\cF))$$
supported on 
$$\cSht^{Y,\mu}_{O_C}(\sx\cup\szero,\szero)\xleftarrow{\pFr}\cSht^{Y,\mu}_{O_C}(\szero,\sx\cup\szero)\xrightarrow{\id}\cSht^{Y,\mu}_{O_C}(\szero,\sx\cup\szero),$$
where we denote $\varepsilon_\cF=(-1)^{d(\cF)(d(\cF)+d(\Delta_{w^{-1}}))}$.
We recall that $d(\cF)$ is $1$ if $\cF$ is supported on odd components of $\Iw\backslash LG/\Iw$, and $d(\cF)=0$ if $\cF$ is supported on even components of $\Iw\backslash LG/\Iw$. This normalization is to make the $S$-correspondence compatible with the abstract $S$-operator defined in \cite[(7.65)]{Tame}.

Consider the Cartesian diagram
$$\begin{tikzcd}
    && \Hk_1^Y(\cSht_{O_C,\mu})\ar[ld,"f_l"swap]\ar[rd,"f_r"] \\
    & \cSht_{O_C}^{Y,\mu}(\sx|(\sx\cup\szero,\szero))\ar[ld,"h_l"swap]\ar[d,"h_r"] && \cSht^{Y,\mu}_{O_C}(\sx|(\szero,\sx\cup\szero)) \ar[d,"h_l"]\ar[rd,"h_r"] \\
    \cSht_{O_C,\mu} & \cSht_{O_C}^{Y,\mu}(\sx\cup\szero,\szero) && \cSht_{O_C}^{Y,\mu}(\szero,\sx\cup\szero)\ar[ll,"\pFr"swap] & \cSht_{O_C,\mu}
\end{tikzcd}$$
where $\Hk_1^Y(\Sht_{O_C,\mu})$ is a closed substack of $\Hk_1(\cSht_{O_C,\mu})$. Here $f_l$ and $f_r$ are defined by sending
$$\begin{tikzcd}
    \phi^*\cE_2 \ar[r,dashrightarrow,"\phi_{\cE_2}"]\ar[d,dashrightarrow,"\phi^*(\gamma)"swap] &\cE_2 \ar[d,dashrightarrow,"\gamma"] \\
    \phi^*\cE_1 \ar[r,dashrightarrow,"\phi_{\cE_1}"] & \cE_1 
\end{tikzcd} \in \Hk_1^Y(\Sht_{O_C,\mu}) $$
to $$\begin{tikzcd}
    (\phi^*\cE_1\ar[r,dashrightarrow,"\phi^*(\gamma)^{-1}"]\ar[rr,bend right,dashrightarrow,"\phi_{\cE_1}"] & \phi^*\cE_2 \ar[r,dashrightarrow,"\gamma\circ\phi_{\cE_2}"] & \cE_1)
\end{tikzcd}\quad\text{and}\quad\begin{tikzcd}
    (\phi^*\cE_2\ar[r,dashrightarrow,"\gamma\circ\phi_{\cE_2}"]\ar[rr,bend right,dashrightarrow,"\phi_{\cE_2}"] & \cE_1 \ar[r,dashrightarrow,"\gamma^{-1}"] & \cE_2)
\end{tikzcd}$$ 
respectively.

\begin{defn}\label{def-excursion}
    Let $\cF\in D_{\et}^\ULA(\cHk_k,\Lambda)$ %\Xinwen{$\Iw\backslash LG/\Iw$ should be $\cHk_{k}$?}
    with $F\colon \cF\to\phi_*\cF$ as above. Define the \emph{$S$-correspondence} $\rS^\corr_\cF$ as the composition
    $$\rS^\corr_\cF\coloneqq\rC_-\circ\rC_F\circ\rC_+\colon (\cSht_{O_C,\mu},\delta^*\cZ_\mu)\to (\cSht_{O_C,\mu},\delta^*\cZ_\mu)$$
    supported on $\Hk_1^Y(\cSht_{O_C,\mu})$.
\end{defn}

Now let $\Lambda$ be one of $O_L$, $O_L/\varpi^n_L$, and $L$ for a finite extension $L/\QQ_\ell$. Let $\cF\in \Shv_\fingen(\Iw\backslash LG/\Iw,\Lambda)$ be an object together with $F\colon\cF\to\phi_*\cF$. Let $Y\subseteq \Iw\backslash LG/\Iw$ be a $\phi$-stable quasi-compact closed substack containing the support of $\cF$ and $\cF^\vee$. We denote by 
$$\mu\circ Y\coloneqq m\circ \delta^{-1}(\Iw\backslash \Fl_{\cI,\mu}\times Y)\subseteq \Iw\backslash LG/\Iw$$
the convolution substack of $\Iw\backslash \Fl_{\cI,\mu}$ and $Y$.
Note that $Z_\mu\star \cF^\vee$ is supported on $\mu\circ Y$.
Similarly, we have a Cartesian diagram over the special fiber:
$$\begin{tikzcd}
    && \Hk_1^Y(\Sht^\loc_{\mu})\ar[ld,"f_l"swap]\ar[rd,"f_r"] \\
    & \Sht^{\loc,0}_{\mu|(\mu\circ Y,Y)}\ar[ld,"h_l"swap]\ar[d,"h_r"] && \Sht^{\loc,0}_{(Y,\mu\circ Y)|\mu} \ar[d,"h_l"]\ar[rd,"h_r"] \\
    \Sht_{\mu}^\loc & \Sht^\loc_{(\mu\circ Y,Y)} && \Sht^\loc_{(Y,\mu\circ Y)}\ar[ll,"\pFr"swap] & \Sht_{\mu}^\loc,
\end{tikzcd}$$
where $\Sht^{\loc}_{\mu|(\mu\circ Y,Y)}$ (resp. $\Sht^{\loc,0}_{(Y,\mu\circ Y)|\mu}$) classifies 
$$\begin{tikzcd}
    (\phi^*\cE_0\ar[r,dashrightarrow,"\beta_2"]& \cE_1 \ar[r,dashrightarrow,"\beta_1"] & \cE_0)
\end{tikzcd}$$
with $(\beta_1)^{-1}$ bounded by $\mu\circ Y$ (resp. $Y$) and $(\beta_2)^{-1}$ bounded by $Y$ (resp. $\mu\circ Y$), and $\Sht^{\loc,0}_{\mu|(\mu\circ Y,Y)}\subseteq \Sht^{\loc}_{\mu|(\mu\circ Y,Y)}$ (resp. $\Sht^{\loc,0}_{(Y,\mu\circ Y)|\mu}\subseteq \Sht^{\loc}_{(Y,\mu\circ Y)}$) is the closed substack where the composition $(\beta_1\circ\beta_2)^{-1}$ is bounded by $\mu$. There are morphisms
$$q_{(\mu\circ Y,Y)}\colon \Sht^\loc_{(\mu\circ Y,Y)}\to (\Iw\backslash LG/\Iw)^2 \quad\text{resp.}\quad q_{(Y,\mu\circ Y)}\colon \Sht^\loc_{(Y,\mu\circ Y)}\to (\Iw\backslash LG/\Iw)^2$$
as before. We define the cohomological correspondence
$$\rC_+^!\colon (\Sht_\mu^\loc,\delta^!Z_\mu)\to (\Sht^\loc_{(\mu\circ Y,Y)},(q_{(\mu\circ Y,Y)})^!(\cF\boxtimes(Z_\mu\star \cF^\vee)))$$
using the composition
$$Z_\mu\xrightarrow{\mathrm{unit}}\cF\star\cF^\vee\star Z_\mu\stackrel{(-1)^{d(Z_\mu)d(\cF)}\tau_{Z_\mu,\cF}}{\simeq} \cF\star Z_\mu\star\cF^\vee,$$
where $\tau_{Z_\mu,\cF}$ is the commutative constraint on the central sheaf. We can similarly define cohomological correspondences
$$\rC_F^!\colon (\Sht^\loc_{(\mu\circ Y,Y)},(q_{(\mu\circ Y,Y)})^!(\cF{\boxtimes}(Z_\mu\star \cF^\vee)))\to (\Sht^\loc_{(Y,\mu\circ Y)},(q_{(Y,\mu\circ Y)})^!((Z_\mu\star\cF^\vee){\boxtimes}\cF))$$
and
$$\rC_-^!\colon (\Sht^\loc_{(Y,\mu\circ Y)},(q_{(Y,\mu\circ Y)})^!((Z_\mu\star\cF^\vee){\boxtimes}\cF))\to (\Sht^\loc_\mu, \delta^!Z_\mu).$$
Then we obtain the $S$-correspondence
$$\rS^{\corr,!}_\cF=\rC_-^!\circ\rC_F^!\circ\rC_+^!\colon (\Sht^\loc_\mu,\delta^!Z_\mu)\to (\Sht^\loc_\mu,\delta^!Z_\mu)$$
over the special fiber (using the sheaf theory $\Shv_\fingen(-,\Lambda)$).

By Example \ref{eg-coh-correspondence} (3), $S_{\cF,k}^{\corr,!}$ induces an endomorphism
$$\Nt_*(\rS^{\corr,!}_{\cF})\colon \Nt_*\delta^!Z_\mu\to \Nt_*\delta^!Z_\mu$$
in $\Shv^{\unip}_\fingen(\Isoc_G,\Lambda)$. 
Denote
$[-]_\phi= \Ch^\unip_{LG^*,\phi}\colon \Ind\Shv_\fingen(\Iw^*\backslash LG^*/\Iw^*,\Lambda)\to \Ind\Shv_\fingen^\unip(\Isoc_{G^*},\Lambda)$
for simplicity.
Applying the functor
$$(\eta_{\dot{w}})_*\colon \Shv_\fingen^\unip(\Isoc_{G},\Lambda)\xrightarrow{\sim} \Shv_\fingen^\unip(\Isoc_{G^*},\Lambda),$$
we obtain an endomorphism
$$(\eta_{\dot{w}})_*\Nt_*(\rS^{\corr,!}_{\cF})\colon [\Delta_{w^{-1}}\star \eta Z_\mu]_\phi \to [\Delta_{w^{-1}}\star \eta Z_\mu]_\phi.$$

We recall the definition of $S$-operators in \cite[(7.65)]{Tame}. The endomorphism $F\colon \cF\to \phi_*F$ induces an endomorphism
$$\eta F\colon \Delta_{w^{-1}}\star\eta\cF \to \phi_*\eta\cF\star \Delta_{w^{-1}}$$
in $\Shv_\fingen(\Iw^*\backslash LG^*/\Iw^*,\Lambda)$. Define the $S$-operator $S_{\eta\cF}$ as the composition
$$\begin{aligned}
    &[\Delta_{w^{-1}}\star \eta Z_\mu]_\phi\xrightarrow{\mathrm{unit}}[\Delta_{w^{-1}}\star \eta Z_\mu\star\eta \cF\star\eta\cF^\vee]_\phi\stackrel{\tau_{\eta Z_\mu,\eta\cF}}{\simeq}[\Delta_{w^{-1}}\star\eta\cF\star\eta Z_\mu\star \eta\cF^\vee]  \\ &\xrightarrow{\eta F}[\phi_*\eta\cF\star \Delta_{w^{-1}}\star\eta Z_\mu\star \eta\cF^\vee] \xrightarrow[\simeq]{\sigma_{\eta\cF,\Delta_{w^{-1}}\star\eta Z_\mu\star\eta\cF^\vee}} [\Delta_{w^{-1}}\star\eta Z_\mu\star\eta\cF^\vee\star \eta\cF]\xrightarrow{\mathrm{counit}}[\Delta_{w^{-1}}\star Z_\mu].
\end{aligned}$$
Here $\tau_{-,-}$ is the commutative constraint for the central sheaves, cf. \cite[\S 3.5.1]{AR-central}, and $\sigma_{-,-}$ is the commutative constraint in the categorical trace defined in \S\ref{subsubsection-unipotent-cll}.. Take $\mu=0$ in the above definition gives the $S$-operator
$$S_{\eta\cF}^0\colon (i_b)_*\cInd_{I_b}^{G_b(\QQ_p)}\Lambda\to(i_b)_*\cInd_{I_b}^{G_b(\QQ_p)}\Lambda$$
as $[\Delta_{w^{-1}}]_\phi=(i_b)_*\cInd_{I_b}^{G_b(\QQ_p)}\Lambda$.
More precisely, it is defined as the composition
$$[\Delta_{w^{-1}}]_\phi\xrightarrow{\mathrm{unit}}[\Delta_{w^{-1}}\star\eta\cF\star\eta\cF^\vee]_\phi\xrightarrow{\eta F}[\phi_*\eta\cF\star\Delta_{w^{-1}}\star\eta\cF^\vee]_\phi\stackrel{\sigma_{\Delta_{w^{-1}}\star \eta\cF,\eta\cF^\vee}}{\simeq} [\Delta_{w^{-1}}\star \eta\cF^\vee\star\eta\cF]_\phi\xrightarrow{\mathrm{counit}}[\Delta_{w^{-1}}]_\phi.$$ 
If $\Lambda=\overline\QQ_\ell$ or $\overline\FF_\ell$, then the same construction works over the spectral side. Recall that $\LL^{\unip}_{G}(\Nt_*\delta^!Z_\mu)=V_\mu\otimes\fA_b$. Let
$$S_{\eta\cF}^{\spec}\colon V_\mu\otimes\fA_b\to V_\mu\otimes\fA_b\quad\text{resp.}\quad S_{\eta\cF}^{0,\spec}\colon \fA_b\to \fA_b$$
denote the corresponding abstract $S$-operators associated to $\BB^{\unip}_{G^*}(\eta\cF)$ and $\BB^\unip_{G^*}(\eta F)$. The functor $\LL^{\unip}_{G^*}$ matches $S_{\eta\cF}$ (resp. $S_{\eta\cF}^0$) with $S_{\eta\cF}^\spec$ (resp. $S_{\eta\cF}^{0,\spec}$). Moreover, it is clear from definition that 
$$S_{\eta\cF}^{\spec}=\id \otimes S_{\eta\cF}^{0,\spec}.$$

\begin{lemma}\label{lemma-excursion-operator}
    The endomorphism $(\eta_{\dot{w}})_*\Nt_*(\rS^{\corr,!}_{\cF})$ defined by $S$-correspondence agrees with the $S$-operator $S_{\eta\cF}$.  
\end{lemma}
\begin{proof}
    By Lemma \ref{lemma-special-fiber-iterated}, the endomorphism $(\eta_{\dot{w}})_*\Nt_*(\rS^{\corr,!}_{\cF})$ is given by the composition
    $$[\Delta_{w^{-1}}\star \eta Z_\mu]_\phi\xrightarrow{\rC_+} [\Delta_{w^{-1}}\star\eta\cF\star\eta Z_\mu\star\eta\cF^\vee]_\phi \xrightarrow{\rC_F} [\Delta_{w^{-1}}\star\eta Z_\mu\star\eta\cF^\vee\star \eta\cF]_\phi\xrightarrow{\rC_-} [\Delta_{w^{-1}}\star \eta Z_\mu]_\phi,$$
    where the three arrows are defined by applying $(\eta_{\dot{w}})_*$ to pushforwards of $\rC_+$, $\rC_F$, and $\rC_-$ to $\Isoc_G$ respectively. We have
    \begin{itemize}
        \item The first morphism is induced by
            $$\Delta_{w^{-1}}\star \eta Z_\mu\xrightarrow{\mathrm{unit}}\Delta_{w^{-1}}\star\eta \cF\star\eta \cF^\vee\star \eta Z_\mu\stackrel{(-1)^{d(Z_\mu)d(\cF)}\tau_{\eta Z_\mu,\eta\cF}}{\simeq} \Delta_{w^{-1}}\star \eta\cF\star\eta Z_\mu\star\eta\cF^\vee.$$
        \item The second morphism is given by
            $$[\Delta_{w^{-1}}\star \eta\cF\star \eta Z_\mu\star\eta\cF^\vee]_\phi\xrightarrow{[\eta F]_\phi}[\phi_*\eta\cF\star\Delta_{w^{-1}}\star \eta Z_\mu\star\eta\cF^\vee]_\phi\stackrel{\varepsilon_\cF\cdot\can}{\simeq}[\Delta_{w^{-1}}\star\eta Z_\mu\star\eta \cF^\vee\star\eta\cF]_\phi,$$
            where $\can$ is the canonical isomorphism defined by partial Frobenius. Note that 
            $$\varepsilon_\cF\cdot\can=(-1)^{d(\cF)d(Z_\mu)}\sigma_{\eta\cF,\Delta_{w^{-1}}\star\eta Z_\mu\star\eta\cF^\vee}.$$ 
        \item The last morphism is induced by
            $$\Delta_{w^{-1}}\star \eta Z_\mu\star\eta\cF^\vee\star\eta\cF\xrightarrow{\mathrm{counit}}\Delta_{w^{-1}}\star \eta Z_\mu.$$
    \end{itemize}
    This agrees with the $S_{\eta\cF}$.
\end{proof}

By \cite[Remark 10.130]{Tame}, there are canonical equivalences
$$\Shv_\fingen(\Sht^\loc_\mu,\Lambda)\simeq \Shv^*_{c}(\Sht^\loc_\mu,\Lambda)\quad\text{and}\quad\Shv_\fingen(\Hk_1(\Sht^\loc_\mu),\Lambda)\simeq \Shv^*_{c}(\Hk_1(\Sht^\loc_\mu),\Lambda).$$
Moreover, under the above equivalences, the object $\delta^!Z_\mu\in \Shv_\fingen(\Sht^\loc_\mu,\Lambda)$ matches with $\delta^*Z^*_\mu\in \Shv^*_{c}(\Sht^\loc_\mu,\Lambda)$. We need the following statement.

\begin{lemma}\label{lemma-sht-analytification}
    Assume that $\Lambda$ is torsion. There are canonical fully faithful embeddings $(c_{\Sht_\mu^{\loc}})^*: \Shv_c^*(\Sht_\mu^{\loc},\Lambda)\to D_{\et}((\Sht_\mu^{\loc})^{\Diamond},\Lambda)$ and $(c_{\Hk_1^Y(\Sht_\mu^{\loc})})^*: \Shv_c^*(\Hk_1^Y(\Sht_\mu^{\loc}),\Lambda)\to D_{\et}((\Hk_1^Y(\Sht_\mu^{\loc}))^{\Diamond},\Lambda)$
\end{lemma}
\begin{proof}
   We first notice that if $X$ is a perfect standard placid algebraic space (in the sense of \cite[Definition 10.61]{Tame}) which admits a cohomologically pro-unipotent morphism $X\to X_0$ with $X_0$ pfp over $k$, then there is a canonical fully faithful embedding
   \[
   \Shv_c^*(X,\Lambda)\simeq D_\et(X,\Lambda)\xrightarrow{(c_X)^*}D_\et(X^{\Diamond},\Lambda),
   \]
   where the first equivalence is as in \cite[Remark 10.60]{Tame}, and the second functor is the comparison functor from \cite[\S 27]{Scholze-diamonds}. 
   
   Applying this to the \v{C}ech nerve of $LG_{\mu}\to \Sht_\mu^{\loc}$, and via the \'etale descent, we obtain the desired fully faithful embedding. Similar for $\Hk_1^Y(\Sht^\loc_\mu)$.
\end{proof}

There is a natural morpshim $\cSht_{k,\mu}\to (\Sht^\loc_\mu)^\Diamond$ defined by restricting a shtuka with one leg at $S$ to the formal neighborhood of $S$. It defines Cartesian diagrams
$$\begin{tikzcd}
    \cSht_{k,\mu} \ar[d] & \Hk_1^Y(\cSht_{k,\mu})\ar[r]\ar[l]\ar[d]\ar[ld,phantom,"\square",very near start]\ar[rd,phantom,"\square",very near start] & \cSht_{k,\mu} \ar[d]\\
    (\Sht^\loc_\mu)^\Diamond & \Hk_1^Y(\Sht^\loc_\mu)^\Diamond \ar[r]\ar[l]& (\Sht^\loc_\mu)^\Diamond.
\end{tikzcd}$$
Let $\Lambda=O_L/\varpi_L^n$ be torsion. By Lemma \ref{lemma-sht-analytification}, we can define the analytification of $\rS_\cF^{\corr,!}$ as a cohomological correspondence
$$(\rS^{\corr,!}_\cF)^\Diamond\colon ((\Sht^\loc_\mu)^\Diamond,(\delta^\Diamond)^*Z_\mu^*)\to ((\Sht^\loc_\mu)^\Diamond,(\delta^\Diamond)^*Z_\mu^*)$$
supported on $(\Hk_1^Y(\Sht^\loc_\mu))^\Diamond$. On the other hand, let $\cF^*\in D_\et^\ULA(\cHk_k,\Lambda)$ be the object associated to $\cF$. We have defined the $S$-correspondence $\rS_{\cF^*}^{\corr}$ on $(\cSht_{O_C,\mu},\delta^*\cZ_\mu)$ supported on $\Hk_1^Y(\cSht_{O_C,\mu})$

\begin{lemma}\label{lemma-S-operator-comparison}
    The pullback of $(\rS^{\corr,!}_\cF)^\Diamond$ to $\cSht_{k,\mu}$ agrees with the restriction of $\rS_{\cF^*}^{\corr}$ to the special fiber.
\end{lemma}
\begin{proof}
    We can define the analytification $(\rC_+^!)^\Diamond$,  $(\rC_-^!)^\Diamond$, and $(\rC_F^!)^\Diamond$ of $\rC_+^!$, $\rC_-^!$, and $\rC_F^!$ respectively. 
    The restriction of $\rC_+$ (resp. $\rC_-$) to the special fiber agrees with the pullback of $(\rC_+^!)^\Diamond$ (resp. $(\rC_-^!)^\Diamond$) to $\cSht_{k,\mu}$ as they are both pulled back from the cohomological correspondences between local Hecke stacks induced by the unit (resp. counit) morphism of $\cF^*$. The claim for $\rC_F$ follows from the commutative diagram
    $$\begin{tikzcd}
        \cSht^{Y,\mu}_k(\sx\cup\szero,\szero) \ar[d] & \cSht^{Y,\mu}_k(\szero,\sx\cup\szero) \ar[d]\ar[l,"\pFr"swap] \\
        (\Sht^\loc_{(\mu\circ Y,Y)})^\Diamond \ar[d]  & (\Sht^\loc_{(Y,\mu\circ Y)})^\Diamond \ar[l,"\pFr^\Diamond"swap] \ar[d] \\
        (\cHk_k)^2 & (\cHk_k)^2. \ar[l,"(\phi\times\id)\circ \sw"swap]
    \end{tikzcd}$$
    We see that the restriction of $\rS^\corr_{\cF^*}$ to the special fiber agrees with the pullback of $(\rS^{\corr,!}_\cF)^\Diamond$.
\end{proof}
%\Xinwen{Not sure whether needs say more to justify, or even make it into a lemma.}

\subsubsection{Compatibility with Hecke actions}
In this subsection, we prove the $H_I$-actions part in Theorem \ref{thm-local-global-compatibility}. We first prove that over the generic fiber, the $S$-correspondences defined in Definition \ref{def-excursion} agree with the Hecke correspondences. Let $\Lambda$ be one of $O_L/\varpi_L^n$, $O_L$, and $L$ for $L/\QQ_\ell$ finite.

Let $\Sht^\loc_1$ be the perfect stack classifies $\cI$-local shtukas bounded by $1\in\widetilde{W}$ and $\Hk_1(\Sht^\loc_1)$ be the corresponding stack of Hecke-shtukas. Then the correspondence $$\Sht_1^\loc\xleftarrow{c_l} \Hk_1(\Sht^\loc_1)\xrightarrow{c_r}\Sht^\loc_1$$ is identified with the Hecke correspondence
$$\BB I\xleftarrow{c_l} I\backslash G(\QQ_p)/I\xrightarrow{c_r} \BB I.$$ 
Let $\cF\in \Shv_\fingen(\Iw\backslash LG/\Iw,\Lambda)$ with $F\colon\cF\xrightarrow{\sim}\phi_*\cF$. Repeating the construction of $\rS^{\corr,!}_\cF$ with $\mu=0$, we define the $S$-correspondence 
$$\rS^{\corr,0,!}_\cF\colon (\Sht^\loc_1,\omega_{\Sht^\loc_1})\to (\Sht^\loc_1,\omega_{\Sht^\loc_1})$$
supported on certain quasi-compact closed substack of $\Hk_1(\Sht^\loc_1)$. Similarly, for $\cF\in D_\et^\ULA(\cHk_k,\Lambda)$ and $F\colon \cF\to\phi_*\cF$, we can also define a cohomological correspondence
$$\rS^{\corr,0}_\cF\colon ((\Sht^\loc_1)^\Diamond,\Lambda)\to ((\Sht^\loc_1)^\Diamond,\Lambda).$$
%\Xinwen{I changed $D_{\et,c}$ to $\Shv_c^*$ here. Is it correct?}

On the other hand, for any double coset $IgI\in I\backslash G(\QQ_p)/I$, there is a canonical correspondence
$$\rT^{\corr,!}_{IgI}\colon (c_{g,l})^*\omega_{\BB I} \xrightarrow{\sim} (c_{g,r})^!\omega_{\BB I}$$
supported on $\BB I\xleftarrow{c_{g,l}} I\backslash IgI/I\xrightarrow{c_{g,r}} \BB I.$ By linear combination, for any element $f\in H_I$, we can define the Hecke cohomological correspondence 
$$\rT^{\corr,!}_f\colon (\BB I,\omega_{\BB I})\to (\BB I,\omega_{\BB I})$$
supported on the support of $f$.

For $\cF\in \Shv_\fingen(\Iw\backslash LG/\Iw,\Lambda)$ and $F\colon \cF\to\phi_*\cF$, let $t^!\cF\in \Shv_\fingen(\Fl_\cI,\Lambda)$ be the pullback of $\cF$. For $g\in G(\QQ_p)/I$, let $(t^!\cF)_g$ denote the $*$-stalk of $t^!\cF$ at $g$. By Grothendieck sheaf-function correspondence, the Weil structure on $\cF$ defines a function
$$f_\cF(g)=\tr(\phi^{-1}|(t^!\cF)_g),\quad g\in I\backslash G(\QQ_p)/I$$
in $H_I$ by taking geometric Frobenius traces.

\begin{prop}\label{prop-S=T-no-leg}
    For $\cF\in \Shv_\fingen(\Iw\backslash LG/\Iw,\Lambda)$ and $F\colon \cF\to\phi_*\cF$, there is an isomorphism
    $\rS^{\corr,0,!}_\cF\simeq \varepsilon_\cF\cdot\rT^{\corr,!}_{f_\cF}$
    of cohomological correspondences from $(\Sht_1^\loc,\omega_{\Sht_1^\loc})$ to $(\Sht_1^\loc,\omega_{\Sht_1^\loc})$. Here $\varepsilon_\cF=(-1)^{d(\cF)(d(\cF)+d(\Delta_{w^{-1}}))}$.
\end{prop}
\begin{proof}
    The proof is same to the proof of \cite[Proposition 6.3.3]{XZ-cycles}.
\end{proof}

Now we can prove $S=T$ for $(i_1)_*\cInd_I^{G(\QQ_p)}\Lambda\in \Shv^\unip_\fingen(\Isoc_G,\Lambda)$.

\begin{prop}\label{prop-S=T-cllc}
    The $S$-operator $S_{\eta\cF}^0\colon (i_1)_*\cInd_I^{G(\QQ_p)}\Lambda\to(i_1)_*\cInd_I^{G(\QQ_p)}\Lambda$ agrees with the action of $\varepsilon_\cF\cdot f_\cF\in H_I$ on $\cInd_I^{G(\QQ_p)}\Lambda$.
\end{prop}
\begin{proof}
    This follows from Lemma \ref{lemma-excursion-operator} and Proposition \ref{prop-S=T-no-leg}.
\end{proof}

Now take $\Lambda=\overline{\QQ}_\ell$ or $\overline\FF_\ell$. Assume that $G=G^*$ is unramified over $\QQ_p$. We can explicitly compute the spectral $S$-operators for elements in the center of $H_I$. Let $V$ be a finite dimensional representation of the Langlands group ${}^LG$. The central sheaf $Z_V\in \Shv_\fingen(\Iw\backslash LG/\Iw,\Lambda)$ corresponding to $V$ carries a natural Weil structure. On the other hand, it defines a function $\ch(V,\phi^{-1})\in Z^{\widehat\unip}_{{}^LG,\QQ_p}$ by sending an $L$-parameter $\varphi\in \Loc^{\widehat\unip}_{{}^LG,\QQ_p}$ to $\tr(\varphi(\phi^{-1})|V).$ 

\begin{prop}\label{prop-S=T-center}
    The action of $(-1)^{d(Z_V)}f_{Z_V}$ on $(i_1)_*\cInd_I^{G(\QQ_p)}\Lambda$ matches with the action of $\ch(V,\phi^{-1})$ on $\CohSpr^\unip_{{}^LG}$.
\end{prop}
\begin{proof}
    We need to compute the spectral $S$-operator 
    $$S_{Z_V}^{0,\spec}\colon \CohSpr_{{}^LG}^\unip\to \CohSpr^\unip_{{}^LG}$$
    associated to $\BB^{\unip}_{G}(Z_V)$. Under Bezrukavnikov equivalence, the object $Z_V$ corresponds to $\iota_*a^!V$ for
    $$\BB\hat{G} \xleftarrow{a} \hat{U}/\hat{B} \xrightarrow{\iota} S^\unip_{\hat{G}}.$$
    The functor 
    $$\iota_*a^!\colon \Rep(\hat{G})\to \Ind\Coh(S^\unip_{\hat{G}})$$
    is monoidal and compatible with $\phi$-actions. Taking Frobenius traces, it induces a morphism
    $$u\colon \tr(\Rep(\hat{G}),\phi_*)\to \tr(\Ind\Coh(S^\unip_{\hat{G}}),\phi_*).$$
    By \cite[Example 7.92]{Tame}, there is a natural isomorphism
    $$\tr(\Ind\Coh(S^\unip_{\hat{G}}),\phi_*)\simeq \mathrm{End}(\CohSpr^\unip_{{}^LG})^{\mathrm{op}}$$
    of $E_1$-algebras such that the diagram 
    $$\begin{tikzcd}
        \tr(\Rep(\hat{G}),\phi_*) \ar[d,"\simeq"]\ar[r,"u"] & \tr(\Ind\Coh(S^\unip_{\hat{G}}),\phi_*) \ar[d,"\simeq"] \\
        \cO(\hat{G}\phi/\hat{G}) \ar[r,"(\ev_\phi)^*"] & \mathrm{End}(\CohSpr^\unip_{{}^LG})^{\mathrm{op}}
    \end{tikzcd}$$
    commutes.
    Moreover, for $\cF\in \Coh(S^\unip_{\hat{G}})$ together with $\cF\xrightarrow{\sim}\phi_*\cF$, the Chern class $$\ch(\cF)\in\tr(\Ind\Coh(S^\unip_{\hat{G}}),\phi_*)$$ 
    of $\cF$ is corresponds to the $S$-operator $S_\cF^{0,\spec}$ by \cite[Proposition 7.94]{Tame}. We see that the $S$-operator $S_{Z_V}^{0,\spec}$ is induced by the pullback of the Chern class $\ch(V)\in \cO(\hat{G}\phi/\hat{G})$ along $\ev_\phi\colon \Loc^{\widehat\unip}_{{}^LG,\QQ_p}\to \hat{G}\phi/\hat{G}$. Over $\hat{G}\phi/\hat{G}$, the Chern class $\ch(V)$ is given by the composition
    $$\cO\xrightarrow{\mathrm{unit}} \cO\otimes V\otimes V^*\xrightarrow{\phi^{-1}} \cO\otimes \phi_*V\otimes V^*\xrightarrow{\can^{-1}}\cO\otimes V\otimes V^*\xrightarrow{\mathrm{counit}}\cO$$
    where $\phi^{-1}\colon V\to \phi_*V$ is the action of  $\phi^{-1}\in {}^LG$, and $\can\colon \cO\otimes V\simeq \cO\otimes\phi_*V$ is defined as in proof of Proposition \ref{prop-compatibility-Weil-group}.
    It is straightforward to check that the above composition sends $g\phi\in \hat{G}\phi/\hat{G}$ to $\tr(\phi^{-1}g^{-1}|V)$. Pulled back to $\Loc^{\widehat\unip}_{{}^LG,\QQ_p}$, it agrees with the function $\ch(V,\phi^{-1})$.
\end{proof}

Let $\Lambda=O_L/\varpi^n_L$ be torsion. Take $\cF\in D_\et^\ULA(\cHk_k,\Lambda)$ %\Xinwen{I changed $D_{\et,c}$ to $\Shv^*_c$ here. Is it correct?} together with $F\colon \cF\to\phi_*\cF$.
There is a natural morphism
$\Sht_{C,\mu}\to (\Sht_1^\loc)^\Diamond\times \cHk_{C,\mu}$
by restricting a $p$-adic shtuka $\phi^*\cE\dashrightarrow\cE$ with one leg at $S^\sharp$ to the formal neighborhood of $S$ and $S^\sharp$ respectively. This defines a commutative diagram
$$\begin{tikzcd}
    \cSht_{C,\mu}\ar[d] & \Hk_1(\cSht_{C,\mu}) \ar[d]\ar[l]\ar[r]\ar[ld,phantom,"\square",very near start]\ar[rd,phantom,"\square",very near start] & \cSht_{C,\mu} \ar[d] \\
    (\Sht^\loc_1)^\Diamond\times \cHk_{C,\mu} & \Hk_1(\Sht^\loc_1)^\Diamond\times\cHk_{C,\mu} \ar[l]\ar[r] & (\Sht^\loc_1)^\Diamond\times \cHk_{C,\mu}
\end{tikzcd}$$
with both sequare Cartesian. 

\begin{lemma}\label{lemma-coh-correspondence-generic-fiber}
    The cohomological correspondence
    $$\rS_{\cF,C}^\corr\colon (\cSht_{C,\mu},\delta^*\cS_\mu)\to (\cSht_{C,\mu},\delta^*\cS_\mu)$$
    agrees with the $*$-pullback of the cohomological correspondence 
    $$\rS^{\corr,0}_\cF\boxtimes\id\colon((\Sht^\loc_1)^\Diamond\times \cHk_{C,\mu},\Lambda\boxtimes\cS_\mu)\to((\Sht^\loc_1)^\Diamond\times \cHk_{C,\mu},\Lambda\boxtimes\cS_\mu)$$
    along the above diagram.
\end{lemma}
\begin{proof}
    This follows from definitions.
\end{proof}

By Proposition \ref{prop-p-isog-stack}, there is a commutative diagram
$$\begin{tikzcd}
    \scS^\diamond_{O_C} \ar[d,"\loc_p"] & \Hk_1^Y(\scS_{O_C}^\diamond) \ar[l,"c_l^\glob"swap]\ar[r,"c_r^\glob"]\ar[d]\ar[ld,phantom,"\square",very near start]\ar[rd,phantom,"\square",very near start] & \scS^\diamond_{O_C} \ar[d,"\loc_p"] \\
    \cSht_{O_C,\mu}  & \Hk_1^Y(\cSht_{O_C,\mu}) \ar[l,"c_l"swap]\ar[r,"c_r"] & \cSht_{O_C,\mu}
\end{tikzcd}$$
with both squares Cartesian. By Example \ref{eg-coh-correspondence} (1), we define the cohomological correspondence
$$\rS^{\glob,\corr}_{\cF}\colon (\scS^\diamond_{O_C},(\loc_p)^*\delta^*\cZ_\mu)\to (\scS^\diamond_{O_C},(\loc_p)^*\delta^*\cZ_\mu)$$
supported on $\Hk_1^Y(\scS^\diamond_{O_C})$ by pulling back $\rS^\corr_\cF$. Similarly, we define
$$\rS^{\glob,\corr,!}_\cF\colon (\Sh_\mu,(\loc_p)^!\delta^!Z_\mu)\to (\Sh_\mu,(\loc_p)^!\delta^!Z_\mu)$$
for $\cF\in\Shv_\fingen(\Iw\backslash LG/\Iw,\Lambda)$ over the special fiber. Note that if $\cF\in \Shv_\fingen(\Iw\backslash LG/\Iw,\Lambda)$ is associated to $\cF^*\in \Shv^*_c(\Iw\backslash LG/\Iw,\Lambda)$, under the natural functor
$$\Shv_c(\Sh_\mu,\Lambda)\simeq \Shv_c^*(\Sh_\mu,\Lambda)\xrightarrow{(j_{\Sh_\mu})^*(c_{\Sh_\mu})^*} D_\et(\scS^\diamond_k,\Lambda),$$
the object $(\loc_p)^!\delta^!Z_\mu$ is sent to $(\loc_p)^*\delta^* Z_\mu^*\simeq i^*(\loc_p)^*\delta^*\cZ_\mu$, and the correspondence $\rS^{\glob,\corr,!}_\cF$ is sent to $i^*\rS^{\glob,\corr}_{\cF^*}$ by Lemma \ref{lemma-S-operator-comparison}.

\begin{lemma}
    There is a canonical isomorphism
    $$(\loc_p)^*\delta^*\cZ_\mu\simeq j_*\Lambda(d/2)[d]$$
    in $D_\et(\scS^\diamond_{O_C},\Lambda)$, where $d=\dim \sSh_{K^pI}(\sG,\sX)$.
\end{lemma}
\begin{proof}
    By \cite[\S 4.9.1]{PR-p-adic-shtuka}, there is a $v$-sheaf theoretic local model diagram
    $$\scS_{O_C}^\diamond \xleftarrow{\pi^v} \widetilde{\scS}^\diamond_{O_C}\xrightarrow{q^v}\MM_{\cI,\mu}^\Diamond$$
    such that $\pi^v$ is a $\cI^\Diamond_{O_C}$-torsor and $q^v$ is $\cI^\Diamond_{O_C}$-equivariant and cohomologically smooth. Here $\MM_{\cI,\mu}$ is the schematic local model base changed to $O_C$.
    Moreover, it defines a commutative diagram
    $$\begin{tikzcd}
        \scS^\diamond_{O_C} \ar[r,"\loc_p"]\ar[d,"\bar{q}^v"swap] & \cSht_{O_C,\mu} \ar[d,"\delta"] \\
        (\cI_{O_C}\backslash \MM_{\cI,\mu})^\Diamond & L^+_{O_C}\cI\backslash (\MM_{\cI,\mu})^\Diamond \ar[l]
    \end{tikzcd}$$
    with $\bar{q}^v$ cohomologically smooth. Here, $L^+_{O_C}\cI$ is the base change of $L^+_{\cY_{[0,\infty)}}$ to $O_C$, and the the action of $L^+_{O_C}\cI$ on $\MM_{\cI,\mu}$ factors through $L^+_{O_C}\cI\twoheadrightarrow \cI_{O_C}^\Diamond$. The central sheaf $\cZ_\mu$ is by definition the pullback of $j_*\Lambda(d/2)[d]$ from $(\cI_{O_C}\backslash \MM_{\cI,\mu})^\Diamond$ to $L^+_{O_C}\cI\backslash (\MM_{\cI,\mu})^\Diamond$, using that $\cS_\mu\simeq\Lambda(d/2)[d]$. Therefore its pullback to $\scS^\diamond_{O_C}$ is also identified with $j_*\Lambda(d/2)[d]$ by smooth base change.
\end{proof}

Let $\pi^\diamond\colon \scS^\diamond_{O_C}\to \Spd O_C$ be the natural morphism.
\begin{prop}\label{prop-coh-diamond}
    Assume $\Lambda$ is torsion. The object $(\pi^\diamond)_!j_*\Lambda$ is a constant sheaf over $\Spd O_C$ associated to the $\Lambda$-module $R\Gamma_c(\sSh_{K^pI}(\sG,\sX)_C,\Lambda)$.
\end{prop}
\begin{proof}
    It suffices to compare the stalks over the generic point and the special point. Over $C$, this is proved in \cite[Lemma 5.10]{Wu-S-equal-T}. Over $k$, there are canonical isomorphisms 
    $$R\Gamma_c(\scS^\diamond_k,i^*j_*\Lambda)\simeq R\Gamma_c(\Sh_\mu,R\Psi\Lambda)\simeq R\Gamma_c(\sSh_{K^pI}(\sG,\sX)_C,\Lambda)$$
    by Lemma \ref{lemma-compare-cohom} and \cite[Corollary 4.6]{Lan-Stroh-nearby-cycle-II}.
\end{proof}
\begin{lemma}\label{lemma-compare-cohom}
    Assume $\Lambda$ is torsion. Let $X$ be a separated scheme of finite type over $k$.     Let $c_X\colon X^\Diamond_\et\to X_\et$ be the analytification functor between \'etale sites and $j_X\colon X^\diamond\hookrightarrow X^\Diamond$ be the natural open immersion as before. Let $\cF\in \Shv^*_c(X,\Lambda)$ be a constructible sheaf. Then there is a canonical isomorphism
    $$R\Gamma_c(X,\cF)\simeq R\Gamma_c(X^\diamond,(j_X)^*(c_X)^*\cF).$$
\end{lemma}
\begin{proof}
    The case $\cF=\Lambda$ is proved in \cite[Proposition A.3]{Wu-S-equal-T}. If $f\colon Y\to X$ is a proper morphism, then the claim holds for $f_*\Lambda\in \Shv_c^*(X,\Lambda)$ as the diagram 
    $$\begin{tikzcd}
        Y^\diamond \ar[r,"j_Y"]\ar[d,"f^\diamond"swap] & Y^\Diamond \ar[d,"f^\Diamond"] \\
        X^\diamond \ar[r,"j_X"] & X^\Diamond
    \end{tikzcd}$$
    is Cartesian, and $(f^\Diamond)_*(c_Y)^*\simeq (c_X)^*f_*$ (by \cite[Proposition 27.4]{Scholze-diamonds}). It suffices to show that objects of the form $f_*\Lambda$ for $f$ proper generates the category $\Shv_c^*(X,\Lambda)$. It suffices to show that $u_!\Lambda$ are generated by those $f_*\Lambda$ for any \'etale morphism $u\colon U\to X$. Let $\overline{U}$ be the relative normalization of $X$ in $U$. Then $\overline{u}\colon \overline{U}\to X$ is finite. Denote $\partial u\colon \overline{U}\backslash U\to X$. Then there is a fiber sequence
    $$u_!\Lambda\to (\overline{u})_*\Lambda\to (\partial u)_*\Lambda.$$
    Hence we finish the proof.
    %General constructible sheaves are generated by $(U\to X)_!\Lambda$ with $U\to X$ \'etale therefore follows from the constant sheaf case.
\end{proof}

Now we are ready to prove Theorem \ref{thm-local-global-compatibility}.

\begin{proof}[Proof of Theorem \ref{thm-local-global-compatibility}]
    Compatibility with $W_E$-actions is treated in Proposition \ref{prop-compatibility-Weil-group}. Thus it suffices to deal with $H_I$-actions. We only prove the statements for $\fI^\can$ as the statements for $\fI$ can be proved in the same way.

    Let $\cF\in \Shv_{\fingen}(\Iw\backslash LG/\Iw,\Lambda)$ be an object together with a Weil structure $F\colon \cF\to \phi_*\cF$. Recall the $S$-operator 
    $$\Nt_*(\rS_\cF^{\corr,!})\colon \Nt_*\delta^!Z_\mu\to \Nt_*\delta^!Z_\mu,$$
    defined by pushforward. By Lemma \ref{lemma-excursion-operator}, is corresponds to the spectral $S$-operator
    $$S_{\eta\cF}^\spec\colon \widetilde{V}_\mu\otimes \fA_b\to \widetilde{V}_\mu\otimes \fA_b$$
    under $\LL^\unip_G$. Moreover, $S_{\eta\cF}^\spec$ splits into $\id\otimes S^{\spec,0}_{\eta\cF}$. By Proposition \ref{prop-S=T-cllc}, the action of $S^{\spec,0}_\cF$ on $\CohSpr^\unip_{{}^LG}$ is identified with the action of $\varepsilon_\cF\cdot f_{\cF}$ on $\fA_b$. Thus we need to show that $S_{\eta\cF}$ induces the action of $\varepsilon_\cF\cdot f_{\cF}$ on $R\Gamma_c(\sSh_{K^pI}(\sG,\sX)_C,\Lambda)$. 

    By Lemma \ref{lemma-coh-correspondence}, the action of $S_{\eta\cF}$ on $R\Gamma_c(\sSh_{K^pI}(\sG,\sX)_C,\Lambda)(d/2)[d]$ is induced by the cohomological correspondence $\DD(S_\cF^{\glob,\corr,!})$ dual to $S_\cF^{\glob,\corr,!}$. If $\Lambda=\overline\FF_\ell$, we denote by $\cF^*\in D_\et^\ULA(\cHk_k,\Lambda)$ the object associated to $\cF$. Then the analytification of $S_\cF^{\glob,\corr,!}$ is equal to the restriction of $S_{\cF^*}^{\glob,\corr}$ to the special fiber. By Example \ref{eg-coh-correspondence} (3), the cohomological correspondence $\DD(\rS^{\glob,\corr}_{\cF^*})$ induces an endomorphism $S^\glob_{\cF^*}$ of $(\pi^\diamond)_*j_*\Lambda(d/2)[d]$ on $\Spd O_C$, where $\pi^\diamond\colon \scS^\diamond_{O_C}\to \Spd O_C$ is the natural projection. By Proposition \ref{prop-coh-diamond}, the object $(\pi^\diamond)_*j_*\Lambda(d/2)[d]$ is the constant sheaf associated to $R\Gamma_c(\sSh_{K^pI}(\sG,\sX)_C,\Lambda)(d/2)[d]$. By Lemma \ref{lemma-S-operator-comparison}, the natural analytification morphism 
    $$R\Gamma_c(\Sh_\mu,(\loc_p)^!\delta^!Z_\mu)\xrightarrow{\sim}R\Gamma_c(\scS^\diamond_k,(\loc_p)^*\delta^*Z_\mu^*)$$
    is compatible with the action of $S_{\eta\cF}$ on the left and the action of $S_{\cF^*,k}^\glob$ on the right. Therefore it suffices to compute the action of $S_{\cF^*}^\glob$ on $(\pi^\diamond)_*j_*\Lambda(d/2)[d]$.
    
    Over the generic fiber, there is a pullback diagram
    $$\begin{tikzcd}
        \scS^\diamond_C \ar[d] & \Hk_1(\scS^\diamond_C)\ar[r]\ar[l]\ar[d]\ar[ld,phantom,"\square",very near start]\ar[rd,phantom,"\square",very near start] & \scS^\diamond_C \ar[d] \\
        \BB I & I\backslash G(\QQ_p)/I \ar[r]\ar[l] & \BB I.
    \end{tikzcd}$$
    By Proposition \ref{prop-S=T-no-leg} and Lemma \ref{lemma-coh-correspondence-generic-fiber}, we know that $\DD(\rS_{\cF^*,C}^{\glob,\corr})$ is induced by the pullback of 
    $$\varepsilon_\cF\cdot \rT^\corr_{\sw (f_\cF)}\colon (\BB I,\Lambda(d/2)[d])\to (\BB I,\Lambda(d/2)[d])$$
    along the above diagram. Here $\sw(f_\cF)$ is the function defined by $\sw(f_\cF)(g)=f_\cF(g^{-1})$ for $g\in I\backslash G(\QQ_p)/I$. We see that over the generic fiber, $S_{\cF^*}^\glob$ is identified with the \emph{left} action of $\varepsilon_\cF\cdot f_\cF$ on $R\Gamma_c(\sSh_{K^pI}(\sG,\sX)_C,\Lambda)(d/2)[d]$. Hence over the special fiber, $S_{\cF^*}^\glob$ also agrees with the left action of $\varepsilon_\cF\cdot f_\cF$. Any element in $H_I$ can be written as $\varepsilon_\cF\cdot f_\cF$ for some $\cF$ and $F\colon \cF\to\phi_*\cF$. Hence we finish the proof. 

    If $\Lambda=\overline{\QQ}_\ell$, we can assume that $\cF\in \Shv_\fingen(\Iw\backslash LG/\Iw,\Lambda)$ is equals to $\cF_0\otimes_{O_L}\overline{\QQ}_\ell$ for some object $\cF_0\in \Shv_\fingen(\Iw\backslash LG/\Iw, O_L)$. Moreover, after modifying by a scalar, we can assume that $F\colon \cF\rightarrow\phi_*\cF$ comes from an endomorphism $F_0\colon \cF_0\rightarrow \phi_*\cF_0$. Then we can write $\cF_0$ as an inverse limit $\lim\limits_{\substack{\longleftarrow\\ n}}\cF_n$ with $\cF_n=\cF_0\otimes_{O_L}(O_L/\varpi_L^n)$. The action of $\rS_\cF^{\glob,\corr,!}$ on $R\Gamma_c(\sSh_{K^pI}(\sG,\sX)_C,\overline\QQ_\ell)(d/2)[d]$ is given by the base change of the action of $\rS_{\cF_0}^{\glob,\corr,!}$ on $R\Gamma_c(\sSh_{K^pI}(\sG,\sX)_C,O_L)(d/2)[d]$. Moreover, under the isomorphism  
    $$R\Gamma_c(\sSh_{K^pI}(\sG,\sX)_C,O_L)(d/2)[d]=\lim_{\substack{\longleftarrow\\ n}}R\Gamma_c(\sSh_{K^pI}(\sG,\sX)_C,O_L/\varpi_L^n)(d/2)[d],$$
    the action of $\rS_{\cF_0}^{\glob,\corr,!}$ is given by the limit of $\rS_{\cF_n}^{\glob,\corr,!}$. Thus we reduce to the torsion case. The above arguments show that the action of $\rS_{\cF_n}^{\glob,\corr,!}$ on $R\Gamma_c(\sSh_{K^pI}(\sG,\sX)_C,O_L/\varpi_L^n)(d/2)[d]$ is given by the left action of $\varepsilon_\cF\cdot f_{\cF_n}\in (O_L/\varpi_L^n)[I\backslash G(\QQ_p)/I]$. Passing to limits, and base change to $\overline\QQ_\ell$, we see that The action of $\rS_\cF^{\glob,\corr,!}$ on $R\Gamma_c(\sSh_{K^pI}(\sG,\sX)_C,\overline\QQ_\ell)(d/2)[d]$ is given by the action of $\varepsilon_\cF\cdot f_\cF\in H_I$.
\end{proof}

\subsection{Torsion vanishing for Shimura varieties of abelian type}

In this subsection we prove our main result Theorem \ref{thm-main-intro}. We shall first deal with Shimura varieties of Hodge type. Then we can deduce the results for Shimura varieties of abelian type using the description of connected components in \cite{Deligne-Shimura}. 

\subsubsection{Torsion vanishing for Shimura varieties of Hodge type}

Let notation be as in \S\ref{subsection-Igusa-stack-coh-formula}. In particular, we assume that the coxeter number of any simple factors of $G$ is invertible in $\Lambda$. We further assume that $G$ is unramified over $\QQ_p$. Recall that $\hat{A}=\hat{T}/(\phi-1)\hat{T}$ where $\phi\in W_E$ is a arithemtic Frobenius and $W_0=W^\phi$ is the relative Weyl group. We have the Bernstein isomorphism $Z(H_I)\simeq \cO(\hat{A}\git W_0)$ where $Z(H_I)$ is the center of $H_I$. For a closed point $\xi$ in $\hat{A}\git W_0$, let 
$$R\Gamma(\sSh_{K^pI}(\sG,\sX)_{\overline{E}},\Lambda)_\xi\quad \text{resp.}\quad R\Gamma_c(\sSh_{K^pI}(\sG,\sX)_{\overline{E}},\Lambda)_\xi$$
denote the localization of the (compactly supported) cohomology of $\sSh_{K^pI}(\sG,\sX)_{\overline{E}}$ with respect to the $Z(H_I)$-actions. The goal of this subsection is to prove that following theorem. 

\begin{thm}\label{thm-torsion-vanishing}
    Let $\xi\in (\hat{A}^\gen\git W_0)(\Lambda)$ be a generic element. Then $R\Gamma(\sSh_{K^pI}(\sG,\sX)_{\overline{E}},\Lambda)_\xi$ (resp. $R\Gamma_c(\sSh_{K^pI}(\sG,\sX)_{\overline{E}},\Lambda)_\xi$) is concentrated in degrees $[d,2d]$ (resp. $[0,d]$).
\end{thm}
\begin{proof}
    Let $\fI^\can_\xi$ denote the $\xi$-component of $\cP^{\widehat\unip}(\fI^\can)$ as in Proposition \ref{prop-adm-supp}. By Theorem \ref{thm-local-global-compatibility} and Lemma \ref{lem: decomposition of Igs sheaf}, we have
    $$\begin{aligned}
        R\Gamma_c(\sSh_{K^pI}(\sG,\sX)_{\overline{E}},\Lambda)_\xi[d]&\simeq \Hom(T_{V_\mu}\circ(i_1)_*\cInd_I^{G(\QQ_p)}\Lambda,(i_{\leq\mu^*})_*\fI^\can_\xi) \\
        &\simeq \Hom((i_1)_*\cInd_I^{G(\QQ_p)}\Lambda),T_{V_\mu^\vee}\circ(i_{\leq\mu^*})_*\fI^\can_\xi)\\
        &\simeq \Hom(\cInd_I^{G(\QQ_p)}\Lambda),(i_1)^!\circ T_{V_\mu^\vee}\circ(i_{\leq\mu^*})_*\fI^\can_\xi)
    \end{aligned}$$
    By Corollary \ref{cor-igusa-sheaf-exotic}, we know that $\fI^\can$ lies in $\Shv(\Isoc_{G,\leq\mu^*},\Lambda)_\xi^{e,\leq 0}$. The functor $\cP^{\widehat\unip}$ sends $\Shv(\Isoc_G,\Lambda)^{e,\leq0}$ to $\Shv^{\widehat\unip}(\Isoc_G,\Lambda)^{e,\leq 0}$ as $\cP^{\widehat\unip}\circ (i_b)^!\simeq (i_b)^!\circ \cP^{\widehat\unip}$ for any $b$ by Proposition \ref{prop-unip-local-Langlands-cat} (3), and $$\cP^{\widehat\unip}\colon \Rep(G_b(F),\Lambda)\to \Rep^{\widehat\unip}(G_b(F),\Lambda)$$
    is $t$-exact. Therefore $\cP^{\widehat\unip}(\fI^\can)$ lies in $\Shv^{\widehat\unip}(\Isoc_G,\Lambda)^{e,\leq 0}$. Hence $\fI^\can_\xi$ lies in $\Shv^{\widehat\unip}(\Isoc_G,\Lambda)^{e,\leq 0}_\xi$ as it is a direct summand of $\cP^{\widehat\unip}(\fI^\can)$. The functor $(i_{\leq\mu^*})_*$ is exotic right $t$-exact, hence $(i_{\leq\mu^*})_*\fI^\can_\xi$ lies in $\Shv(\Isoc_G,\Lambda)_\xi^{e,\leq 0}$. By Theorem \ref{thm-Hecke-t-exact}, we know that $T_{V_\mu^\vee}\circ(i_{\leq\mu^*})_*\fI^\can_\xi$ lies in $\Shv(\Isoc_G,\Lambda)_\xi^{e,\leq 0}$. The functor $(i_1)^!\simeq (i_1)^\sharp$ is exotic $t$-exact. It follows that the $G(\QQ_p)$-representation 
    $$M\coloneqq (i_1)^!\circ T_{V_\mu^\vee}\circ(i_{\leq\mu^*})_*\fI^\can_\xi$$ 
    lies in $\Rep^{\widehat\unip}(G(\QQ_p),\Lambda)_\xi^{\leq 0}$. 

    We claim that $\Hom(\cInd_I^{G(\QQ_p)}\Lambda,M)$ is connective. If $\Lambda$ has characteristic 0, this is clear as $\cInd_I^{G(\QQ_p)}\Lambda$ is projective. If $\Lambda$ has characteristic $\ell$, then $\Hom(\cInd_{I^{(\ell)}}^{G(\QQ_p)}\Lambda,M)$ is a perfect $\Lambda$-module and lies in $\Lambda\mbox{-}\mathrm{mod}^{\leq 0}.$ Therefore $\Hom(\cInd_{I^{(\ell)}}^{G(\QQ_p)}\Lambda,M)$ is a $I/I^{(\ell)}$-representation on perfect $\Lambda$-modules. Because
    $$\Hom(\cInd_{I}^{G(\QQ_p)}\Lambda,M)\simeq R\Gamma_c(\sSh_{K^pI}(\sG,\sX)_{\overline{E}},\Lambda)_\xi(d/2)[d]$$
    is bounded. By adjunction, we have
    $$\Hom(\cInd_{I}^{G(\QQ_p)}\Lambda,M)\simeq \Hom_{\Lambda[I/I^{(\ell)}]}(\Lambda,\Hom(\cInd_{I^{(\ell)}}^{G(\QQ_p)}\Lambda,M)).$$
    By Lemma \ref{lemma-fg-compact}, there is a canonical isomorphism
    $$\Hom(\cInd_{I}^{G(\QQ_p)}\Lambda,M)\simeq \Hom(\cInd_{I^{(\ell)}}^{G(\QQ_p)}\Lambda,M)\otimes_{\Lambda[I/I^{(\ell)}]}\Lambda.$$
    Therefore $\Hom(\cInd_{I}^{G(\QQ_p)}\Lambda,M)$ lies in $\Lambda\mbox{-}\mathrm{mod}^{\leq 0}$. Hence $R\Gamma_c(\sSh_{K^pI}(\sG,\sX)_{\overline{E}},\Lambda)_\xi$ lies in degrees $[0,d]$. By Poincar\'e duality, we see that $R\Gamma(\sSh_{K^pI}(\sG,\sX)_{\overline{E}},\Lambda)_\xi$ lies in degrees $[d,2d]$.    
\end{proof}

\begin{lemma}\label{lemma-fg-compact}
    Let $\Lambda=\overline{\FF}_\ell$. Let $H$ be a finite abelian $\ell$-group. Let $V\in \Rep_c(H)$ be an $H$-representation on perfect $\Lambda$-modules. If $V^H$ is bounded in the standard $t$-structure of $\Lambda\mbox{-}\mathrm{mod}$, then $V$ is compact in $\Rep(H)$. Moreover, the norm map $V_H\to V^H$ is an isomorphism.
\end{lemma}
\begin{proof}
    Let $R=\Lambda[H]$ be the group algebra. Then $R$ is of completely intersection as $H$ is a product of cyclic $\ell$-groups. We can identify $\Rep_c(H)\simeq \Coh(R)$ and $\Rep(H)^\omega\simeq \Perf(R)$. Note that the category $\Rep_c(H)$ is generated by $\Lambda$ under cones and retracts. Therefore the assumption implies that $\Hom_{H}(V,V)$ is bounded. Hence the singular support of $V$ as a coherent $R$-module is contained in the zero section. Hence $H$ is a perfect $R$-module by \cite[Corollary 9.64]{Tame}. If $V=\Lambda[H]$ is the regular representation, then the norm map is an isomorphism. In general, every compact object $V\in \Rep(H)^\omega$ can be obtained from $\Lambda[H]$ by taking finite cones and retracts. Therefore the norm map is an isomorphism for any compact object $V\in\Rep(H)^\omega$.
\end{proof}

It is straightforward to generalize the statement to more general parahoric levels as follows. Let $\sSh_{K^p}(\sG,\sX)$ denote the inverse limit $\lim\limits_{\substack{\longleftarrow\\K_p}}\sSh_{K^pK_p}(\sG,\sX)$, where $K_p$ runs through open compact subgroups of $G(\QQ_p)$.
Denote
$$R\Gamma(K^p,\Lambda)\coloneqq R\Gamma(\sSh_{K^p}(\sG,\sX)_{\overline{E}},\Lambda)$$
As $\sSh_{K^pK_p}(\sG,\sX)_{\overline{E}}$ are qcqs for any $K_p$, we have
$$R\Gamma(K^p,\Lambda)\simeq\lim_{\substack{\longrightarrow\\K_p}}R\Gamma(\sSh_{K^pK_p}(\sG,\sX)_{\overline{E}},\Lambda).$$
Therefore $R\Gamma(K^p,\Lambda)$ is an admissible $G(\QQ_p)$-representation.

\begin{lemma}\label{lemma-coh-Sh-as-group-homology}
    For $K_p\subseteq G(\QQ_p)$ an open compact subgroup, we have
    $$R\Gamma(\sSh_{K^pK_p}(\sG,\sX)_{\overline{E}},\Lambda)\simeq R\Gamma(K_p,R\Gamma(K^p,\Lambda)),$$
    where $R\Gamma(K_p,-)$ is the group cohomology.
\end{lemma}
\begin{proof}
    This is clear as $\sSh_{K^p}(\sG,\sX)\to \sSh_{K^pK_p}(\sG,\sX)$ is a pro-\'etale $K_p$-torsor.
\end{proof}

For $\xi\in(\hat{A}\git W_0)(\Lambda)$, we denote
$$R\Gamma(K^p,\Lambda)_\xi\coloneqq \cP^{\widehat\unip}(R\Gamma(K^p,\Lambda))_\xi.$$

\begin{cor}\label{cor-torsion-vanishing-infinite-level}
    Let assumptions be as in Theorem \ref{thm-torsion-vanishing}. Then $R\Gamma(K^p,\Lambda)_\xi$ is concentrated in degrees $[d,2d]$.
\end{cor}
\begin{proof}
    By Corollary \ref{cor-detect-t-str-supp} and the remark below it, we need the show that 
    $$\Hom(\cInd_{\widetilde{I}}^{G(\QQ_p)}\Lambda,R\Gamma(K^p,\Lambda)_\xi)\simeq R\Gamma(\widetilde{I},R\Gamma(K^p,\Lambda)_\xi)$$
    sits in degrees $[d,2d]$. By Theorem \ref{thm-torsion-vanishing}, we know that $R\Gamma(I,R\Gamma(K^p,\Lambda)_\xi)$ sits in degrees $[d,2d]$. Using that $\Lambda[I/\widetilde{I}]$ is a finite extension of $\Lambda$ as $I/\widetilde{I}$-representations, we see that
   $$R\Gamma(\widetilde{I},R\Gamma(K^p,\Lambda)_\xi)=\Hom_{I}(\Lambda[I/\widetilde{I}],R\Gamma(K^p,\Lambda)_\xi)$$
   is concentrated in degrees $[d,2d]$ as desired.
\end{proof}

We recall that a subgroup $\breve{K}_p\subseteq G(\breve{\QQ}_p)$ is called a quasi-parahoric subgroup if it fits into
$$G(\breve{\QQ}_p)^\circ\cap \mathrm{Stab}(\mathfrak{F})\subseteq \breve{K}_p\subseteq G(\breve{\QQ}_p)^1\cap \mathrm{Stab}(\mathfrak{F}),$$
where $\mathfrak{F}$ is a facet in the Bruhat--Tits building $\mathscr{B}(G,\breve{\QQ}_p)$, $G(\breve{\QQ}_p)^\circ=\ker(G(\breve\QQ_p)\to \pi_1(G)_{I_F})$, and $G(\breve{\QQ}_p)^1=\ker(G(\breve\QQ_p)\to \pi_1(G)_{I_F}\otimes\QQ)$. Any quasi-parahoric subgroup $\breve{K}_p$ contains a parahoric subgroup $\breve{K}_p^\circ=G(\breve{\QQ}_p)^\circ\cap \mathrm{Stab}(\mathfrak{F})$ of finite index. If $\breve{K}_p$ is stable under Frobenius action, then its Frobenius invariance defines a quasi-parahoric subgroup $K_p$ of $G(\QQ_p)$.

Let $K_p\subseteq G(\QQ_p)$ be a quasi-parahoric subgroup. The compact induction $\cInd_{K_p}^{G(F)}\Lambda$ is unipotent. Therefore there is a homomorphism $Z^{\widehat\unip}_{{}^LG,\QQ_p}\to Z(H_{K_p})$ where $Z(H_{K_p})$ is the center of the Hecke algebra $H_{K_p}\coloneqq \Lambda[K_p\backslash G(\QQ_p)/K_p]$. Therefore we can define the localizations
$R\Gamma(\sSh_{K^pK_p}(\sG,\sX)_{\overline{E}},\Lambda)_\xi$ and  $R\Gamma_c(\sSh_{K^pK_p}(\sG,\sX)_{\overline{E}},\Lambda)_\xi$
at $\xi\in (\hat{A}\git W_0)(\Lambda)$.

\begin{cor}\label{cor-torsion-vanishing-parahoric-level}
    Let assumptions $\xi$ be a generic element in $(\hat{A}\git W_0)(\Lambda)$. Then $R\Gamma(\sSh_{K^pK_p}(\sG,\sX)_{\overline{E}},\Lambda)_\xi$ (resp. $R\Gamma_c(\sSh_{K^pK_p}(\sG,\sX)_{\overline{E}},\Lambda)_\xi$) is concentrated in degrees $[d,2d]$ (resp. $[0,d]$).
\end{cor}
\begin{proof}
    The claim of $R\Gamma(\sSh_{K^pK_p}(\sG,\sX)_{\overline{E}},\Lambda)_\xi$ follows from Lemma \ref{lemma-coh-Sh-as-group-homology} and Corollary \ref{cor-torsion-vanishing-infinite-level}. The claim for $R\Gamma_c(\sSh_{K^pK_p}(\sG,\sX)_{\overline{E}},\Lambda)_\xi$ follows from Poincar\'e duality.
\end{proof}

\subsubsection{Passing to abelian type}
Now we deduce the torsion vanishing for Shimura varieties of abelian type from the results of Shimura varieties of Hodge type. Similar statements was proved in \cite[\S 5.2]{Hamann-Lee-vanishing}, in the case when the Shimura datum $(\sG,\sX)$ of abelian type has the same derived subgroup with some Shimura datum of Hodge type. 

Roughly speaking, our strategy is as follows. For an abelian type Shimura datum $(\sG,\sX)$, we find an axillary abelian type Shimura datum $(\sG'',\sX'')$ and a Hodge type Shimura datum $(\sG',\sX')$ such that there are maps
$$(\sG',\sX') \leftarrow (\sG'',\sX'')\to (\sG,\sX)$$
of Shimur data, such that the associated adjoint Shimura data are the same, and $\sG''_\der\simeq \sG'_\der$. We first deduce the torsion vanishing for $(\sG'',\sX'')$ from that for $(\sG',\sX')$, by an argument slightly different from the one in \cite{Hamann-Lee-vanishing}, and then deduce the torsion vanishing for $(\sG,\sX)$ from that for $(\sG'',\sX'')$ by writing $\Sh(\sG,\sX)$ as a quotient of the Shimura variety associated to $(\sG'',\sX'')$. To carry out the second step, we prove that generic unipotent representations remain generic unipotent under restrictions along central isogenies of $p$-adic groups.

We first set up some notations. Let $(\sG,\sX)$ be a Shimura datum. For an open compact subgroup $K\subseteq \sG(\AAA_f)$, denote by $\Sh_K(\sG,\sX)$ the underlying topological space of the $\CC$-points of $\sSh_K(\sG,\sX)_\CC$. Then we have
$$\Sh_K(\sG,\sX)=\sG(\QQ)\backslash (\sX\times\sG(\AAA_f)/K).$$ 
Denote 
$$\Sh(\sG,\sX)=\lim_{\substack{\longleftarrow\\K}}\Sh_K(\sG,\sX),$$
where $K$ runs through open compact subgroups of $\sG(\AAA_f)$. Fix a connected component $\sX^+\subseteq \sX$. Let $\Sh^+(\sG,\sX)$ denote the connected component of $\Sh(\sG,\sX)$ equals to the image of $\sX^+\times\{1\}$. Then $\Sh^+(\sG,\sX)$ depends only on the connected Shimura datum $(\sG_\der,\sX^+)$. Let $\sG_\ad(\RR)_+$ be the identity connected component of $\sG_\ad(\RR)$. Let $\sG(\RR)_+$ denote the preimage of $\sG_\ad(\RR)_+$ and denote $\sG(\QQ)_+=\sG(\QQ)\cap\sG(\RR)_+$.
We have
$$\pi_0(\Sh(\sG,\sX))=\frac{\sG(\AAA_f)}{\sG(\QQ)_+^-}$$
as pro-finite sets, where $\sG(\QQ)_+^-$ is the closure of $\sG(\QQ)_+$ in $\sG(\AAA_f)$. Let $\sZ$ be the center of $\sG$. Define the groups
$$\scA(\sG)\coloneqq\frac{\sG(\AAA_f)}{\sZ(\QQ)^-}*_{\sG(\QQ)_+/\sZ(\QQ)}\sG_\ad(\QQ)_1$$
and
$$\scA(\sG)^\circ\coloneqq\frac{\sG(\QQ)_+^-}{\sZ(\QQ)^-}*_{\sG(\QQ)_+/\sZ(\QQ)}\sG_\ad(\QQ)_1$$
following \cite{Deligne-Shimura}, where $\sZ(\QQ)^-$ is the closure of $\sZ(\QQ)$ in $\sG(\AAA_f)$. Then $\scA(\sG)$ (resp. $\scA(\sG)^\circ$) acts on $\Sh(\sG,\sX)$ (resp. $\Sh^+(\sG,\sX)$). Moreover, we have
$$\Sh^+(\sG,\sX)\times^{\scA(\sG)^\circ}\scA(\sG)=\Sh(\sG,\sX).$$
The group $\scA(\sG)^\circ$ depends only on $\sG_\der$. In fact, $\scA(\sG)^\circ$ is the completion of $\sG_\ad(\QQ)_+$ with topology generated by images of congruence subgroups of $\sG_\der(\QQ)_+$.

Denote $G\coloneqq\sG_{\QQ_p}$. Fix an open compact subgroup $K_p\subseteq G(\QQ_p)$. Denote
$$\Sh_{K_p}(\sG,\sX)=\lim_{\substack{\longleftarrow\\K^p}}\Sh_{K^pK_p}(\sG,\sX),$$
where $K^p$ runs through open compact subgroups of $\sG(\AAA_f^p)$. Let 
$R\Gamma(\Sh_{K_p}(\sG,\sX),\Lambda)$ denote the Betti cohomology. As $\Sh_{K^pK_p}(\sG,\sX)$ are qcqs for any $K^p$, we have
$$R\Gamma(\Sh_{K_p}(\sG,\sX),\Lambda)=\lim_{\substack{\longrightarrow\\K^p}}R\Gamma(\Sh_{K^pK_p}(\sG,\sX),\Lambda).$$
Hence $R\Gamma(\Sh_{K_p}(\sG,\sX),\Lambda)$ admits an action of the Hecke algebra $H_{K_p}=\Lambda[K_p\backslash G(\QQ_p)/K_p]$. 

Now we assume that $(\sG,\sX)$ is of abelian type. Assume that $\sG$ is unramified over $\QQ_p$ and the coxeter number of any simple factors of $\sG$ is invertible in $\Lambda$. 

\begin{lemma}
    Let $(\sG,\sX)$ be a Shimura datum of abelian type. There there exists a Shimura variety of Hodge type $(\sG',\sX')$ and a Shimura variety of abelian type $(\sG'',\sX'')$ together with morphisms
    $$(\sG',\sX')\leftarrow (\sG'',\sX'')\rightarrow (\sG,\sX)$$
    such that $\sG''_\der\xrightarrow{\sim} \sG'_\der$ is an isomorphism and $\sG''_\der\to \sG_\der$ is a central isogeny. Moreover, it induces isomorphisms $(\sG_\ad,\sX_\ad)\simeq (\sG''_\ad,\sX_\ad'')\simeq (\sG'_\ad,\sX'_\ad)$ of adjoint Shimura data.
\end{lemma}
\begin{proof}
    This is proved in  \cite{XZ-cycles}. As the revised version \cite{XZ-cycles} has not appeared, we give a proof here for completeness. Let $(\sG',\sX')$ be a Shimura variety of Hodge type with a central isogeny $\sG'_\der\to \sG_\der$. We take $\sG''=(\sG'\times_{\sG_\ad}\sG)^\circ$. Let $h'\in \sX'$ and $h\in \sX$. After conjugating $h'$ by some element $g\in \sG_\ad(\RR)$, we can assume that $h'=h\in \sX_\ad$. It induces a cocharacter $h''$ of $\sG''(\RR)$, and hence defines a Shimura datum $(\sG'',\sX'')$. After replacing the Shimura datum $(\sG',\sX')$ by $(\sG',g(\sX'))$, we obtain the desired diagram. The Shimura datum $(\sG',g(\sX'))$ is still of Hodge type by \cite[Proposition 2.3.2]{Deligne-Shimura}.
\end{proof}

By assumption, the results in Corollary \ref{cor-torsion-vanishing-parahoric-level} holds for $(\sG',\sX')$. By the above lemma, we need the following two steps to conclude torsion vanishing for $(\sG,\sX)$:
\begin{enumerate}
    \item Torsion vanishing for $(\sG',\sX')$ implies torsion vanishing for $(\sG'',\sX'')$.
    \item Torsion vanishing for $(\sG'',\sX'')$ implies torsion vanishing for $(\sG,\sX)$.
\end{enumerate}

We first deal with the step (1). Let $I'$ (resp. $I''$) be an Iwahori subgroup of $G'(\QQ_p)$ (resp. $G''(\QQ_p)$). We can assume that $I''$ is contained in the preimage of $I'$. Note that the commutative diagram
$$\begin{tikzcd}
    \sG'' \ar[r,"\nu"]\ar[d]\ar[rd,phantom,"\square", very near start] & \sG''_\ab \ar[d] \\
    \sG' \ar[r,"\nu"] & \sG'_\ab
\end{tikzcd}$$
is Cartesian. Fix an element $h''\in\sX''$. Let $h'\in\sX'$ be the image of $h''$. Let $(\sG_\ab',\{h'\})$ (resp. $(\sG_\ab'',\{h''\})$) denote the associated Shimura datum of $\sG_\ab'$ (resp. $\sG''_{\ab}$).

\begin{lemma}\label{lemma-Sh-same-derived-subgroup}
    The natural morphism
    $$\Sh(\sG'',\sX'')\to \Sh(\sG',\sX')\times_{\Sh(\sG'_\ab,\{h'\})}\Sh(\sG''_\ab,\{h''\})$$
    is a closed embedding and identifies $\Sh(\sG'',\sX'')$ with a union of connected components of the right hand side. Moreover, the above morphism is compatible with $\sG''(\AAA_f)$-actions. 
\end{lemma}
\begin{proof}
    The connected Shimura varieties $\Sh^+(\sG'',\sX'')$ and $\Sh^+(\sG',\sX')$ are identical as $\sG''_\der\simeq\sG'_\der$. It suffices to show that the above map induces a closed embedding on $\pi_0$'s. We have exact sequences of groups
    $$\begin{tikzcd}
        0 \ar[r] & \sG''(\QQ)_+^- \ar[r]\ar[d] & (\sG'\times\sG''_\ab)(\QQ)_+^- \ar[r]\ar[d] & \sG'_\ab(\QQ)_+^- \ar[d] \\
        0 \ar[r] & \sG''(\AAA_f) \ar[r] & (\sG'\times\sG''_\ab)(\AAA_f) \ar[r] & \sG'_\ab(\AAA_f).
    \end{tikzcd}$$
    A diagram chasing shows that there is  a closed embedding
    $$\frac{\sG''(\AAA_f)}{\sG''(\QQ)_+^-}\hookrightarrow \frac{\sG'(\AAA_f)}{\sG'(\QQ)_+^-}\times_{\frac{\sG'_\ab(\AAA_f)}{\sG'_\ab(\QQ)_+^-}} \frac{\sG''_\ab(\AAA_f)}{\sG''_\ab(\QQ)_+^-}.$$
    Hence we finish the proof.
\end{proof}

Let $I_\der'\coloneqq I'\cap G'_\der(\QQ_p)$ is an Iwahori subgroup of $G'_\der(\QQ_p)$. As $G''_\der=G'_\der$, we have embedding of Hecke algebras
$$H_{I'}\hookleftarrow H_{I'_\der}\hookrightarrow H_{I''}.$$
It induces embeddings
$$Z(H_{I'})\hookleftarrow Z(H_{I'_\der})\hookrightarrow Z(H_{I''})$$
of centers. Recall that $\Lambda$-points of $\Spec Z(H_{I'_\der})$ are in bijection with conjugacy classes of semisimple unramified $L$-parameters $W_{\QQ_p}\to {}^LG'_{\der}(\Lambda)$. Let $\xi'_\der$ be a $\Lambda$-point of $\Spec Z(H_{I'_\der})$. Denote by
$$R\Gamma(\Sh_{I'}(\sG',\sX'),\Lambda)_{\xi'_\der} \quad\text{resp.}\quad R\Gamma(\Sh_{I''}(\sG'',\sX''),\Lambda)_{\xi'_\der}$$
the localization of $R\Gamma(\Sh_{I'}(\sG',\sX'),\Lambda)$ (resp. $R\Gamma(\Sh_{I''}(\sG'',\sX''),\Lambda)$) at the maximal ideal of $Z(H_{I'_\der})$ defining $\xi'_\der$.

\begin{lemma}\label{lemma-G''-torsion-vanising}
    Assume that $\xi_\der'$ is generic. Then $R\Gamma(\Sh_{I''}(\sG'',\sX''),\Lambda)_{\xi_\der'}$ is concentrated in degrees $[d,2d]$. Here $d=\dim\Sh(\sG',\sX')=\dim \Sh(\sG'',\sX'')$. 
\end{lemma}
\begin{proof}
    Note that a semisimple unramified $L$-parameter $\varphi\colon W_{\QQ_p}\to {}^LG'(\Lambda)$ is generic if and only if the composition $W_{\QQ_p}\xrightarrow{\varphi}{}^LG'\to {}^LG'_\der$
    is generic. By Theorem \ref{thm-torsion-vanishing}, we know that $R\Gamma(\Sh_{I'}(\sG',\sX'),\Lambda)_{\xi_\der'}$ is concentrated in degree $[d,2d]$. Note that $I''\simeq I'\times_{\nu(I')}\nu(I'')$ inside $G''(\QQ_p)\simeq G'(\QQ_p)\times_{G'_\ab(\QQ_p)}G''_\ab(\QQ_p)$. Taking quotient of the closed embedding in Lemma \ref{lemma-Sh-same-derived-subgroup} by $I''$-actions, we see that
    $$\Sh_{I''}(\sG'',\sX'')\hookrightarrow\Sh_{I'}(\sG',\sX')\times_{\Sh_{\nu(I')}(\sG'_\ab,\{h'\})}\Sh_{\nu(I'')}(\sG''_\ab,\{h''\})$$
    is a closed embedding that identifies the left hand side as a union of connected components of the right hand side. Moreover, the embedding is compatible with $H_{I''}$-actions. Therefore $R\Gamma(\Sh_{I''}(\sG'',\sX''),\Lambda)_{\xi_\der'}$ is also concentrated in degrees $[d,2d]$. Note that $G'_\der(\QQ_p)$ acts trivially on $\Sh_{\nu(I')}(\sG_\ab',\{h'\})$ and $\Sh_{\nu(I'')}(\sG''_\ab,\{h''\})$.
\end{proof}

Now we deal with step (2).

\begin{lemma}\label{lemma-Sh-derived-not-equal}
    There is a $\sG(\AAA_f)$-equivariant isomorphism
    $$\Sh(\sG'',\sX'')\times^{\scA(\sG'')}\scA(\sG)\simeq \Sh(\sG,\sX).$$
\end{lemma}
\begin{proof}
    This is clear from definition. Note that $\scA(\sG)^\circ$ is a quotient of $\scA(\sG'')^\circ$. Therefore $\Sh^+(\sG,\sX)$ is a quotient of $\Sh^+(\sG'',\sX'')$ by the kernel of $\scA(\sG'')^\circ\to \scA(\sG)^\circ$.
\end{proof}

By Lemma \ref{lemma-Sh-derived-not-equal}, we have
$$R\Gamma(\Sh(\sG,\sX),\Lambda)\simeq R\Gamma(\scA(\sG''),\, R\Gamma(\Sh(\sG'',\sX''),\Lambda)\otimes_\Lambda C^\infty(\scA(\sG),\Lambda)),$$
where $C^\infty(\scA(\sG),\Lambda)$ is the space of smooth functions on $\scA(\sG)$, and $\scA(\sG'')$ acts on $C^\infty(\scA(\sG),\Lambda)$ via left multiplication. 

Recall that $\Rep^{\widehat\unip}(G(\QQ_p),\Lambda)$ is the subcategory of unipotent representations, and $\cP^{\widehat\unip}$ is the unipotent projector. The functor
$$\Res\colon \Rep(G(\QQ_p),\Lambda)\to \Rep(G''(\QQ_p),\Lambda)$$
defined by restriction along $G''(\QQ_p)\to G(\QQ_p)$ preserves the unipotent blocks. Therefore we have
$$\cP^{\widehat\unip}(R\Gamma(\Sh(\sG,\sX),\Lambda))\simeq R\Gamma(\scA(\sG''),\,\cP^{\widehat\unip}(R\Gamma(\Sh(\sG'',\sX''),\Lambda))\otimes_\Lambda \cP^{\widehat\unip}(C^\infty(\scA(\sG),\Lambda))).$$

As $\Rep^{\widehat\unip}(G(\QQ_p),\Lambda)$ is linear over $Z_{{}^LG,\QQ_p}^{\widehat\unip}$. Let $\xi$ be a generic $\Lambda$-point in $\hat{A}\git W_0$. Let
$$\cP^{\widehat\unip}(R\Gamma(\Sh(\sG,\sX),\Lambda))_\xi$$
denote the $\xi$-component as in Proposition \ref{prop-adm-supp}. It makes sense as $R\Gamma(\Sh(\sG,\sX),\Lambda)$ is a filtered colimit of admissible $G(\QQ_p)$-representations. We can similarly define $\cP^{\widehat\unip}(R\Gamma(\Sh(\sG,\sX),\Lambda))_{\xi''}$ for $\xi''\in (\hat{A}''\git W_0)(\Lambda)$. 

\begin{lemma}\label{lemma-restr}
    Let $\xi\in (\hat{A}\git W_0)(\Lambda)$ be a generic point. Then its image $\xi''\in (\hat{A}''\git W_0)(\Lambda)$ is generic and restriction along $G''(\QQ_p)\to G(\QQ_p)$ defines a functor
    $$\Rep^{\widehat\unip}(G(\QQ_p),\Lambda)_\xi\to \Rep^{\widehat\unip}(G''(\QQ_p),\Lambda)_{\xi''}.$$
\end{lemma}
\begin{proof}
    The first statement is clear, as the genericness of $\xi$ can be checked inside the adjoint group. Let $I$ (resp. $I''$) be an Iwahori subgroup of $G(\QQ_p)$ (resp. $G''(\QQ_p)$) such that $I''$ is contained in the preimage of $I$.
    By Proposition \ref{prop-generic-equiv}, the category $\Rep^{\widehat\unip}(G(\QQ_p),\Lambda)_\xi$ is generated under colimits by 
    $$(\cInd_I^{G(\QQ_p)}\Lambda)\otimes_{\cO(\hat{A}\git W_0)}\Lambda_\xi,$$
    where $\Lambda_\xi$ is the skyscraper module of $\cO(\hat{A}\git W_0)$ at $\xi$: The $\cO(\hat{A})$-module $D_\xi$ can be written as a colimit of $\cO(\hat{A})\otimes_{\cO(\hat{A}\git W_0)}\Lambda_\xi$. By Proposition \ref{prop-S=T-center}, we know that the action of $\cO(\hat{A}\git W_0)$ on $\cInd_I^{G(\QQ_p)}\Lambda$ agrees with the action of the center $Z(H_I)\cong \cO(\hat{A}\git W_0)$.
    By Lemma \ref{lemma-Iwahori-vs-Borel}, there is a canonical isomorphism
    $$\cInd_I^{G(\QQ_p)}\Lambda\simeq \nInd_{B(\QQ_p)}^{G(\QQ_p)} \chi_\univ,$$
    where $\chi_\univ \colon T(\QQ_p)\to \cO(\hat{A})^\times$ is the universal unramified character over $\cO(\hat{A})$. By \cite[Lemma 1.7.1, Lemma 2.3.1]{HKP-Iwahori}, the action of $Z(H_I)$ on the left agrees with the action of the subalgebra $\cO(\hat{A}\git W_0)\subseteq \cO(\hat{A})$ on the right. Note that in the \emph{loc. cit.} the authors only consider the case when $\Lambda$ is of characteristic $0$. However the proofs works for $\Lambda=\overline{\FF}_\ell$ as well.
    It follows that the $G(\QQ_p)$-representation $(\cInd_I^{G(\QQ_p)}\Lambda)\otimes_{\cO(\hat{A}\git W_0)}\Lambda_\xi$ is a finite successive extensions of the normalized Borel inductions
    $$\nInd_{B(\QQ_p)}^{G(\QQ_p)}\chi,$$
    where $\chi$ runs through preimages of $\xi$ in $\hat{A}(\Lambda)$. Therefore the collection $\nInd_{B(\QQ_p)}^{G(\QQ_p)}\chi$'s generate the category $\Rep^{\widehat\unip}(G(\QQ_p),\Lambda)_\xi$ under colimits. Note that the restriction of $\nInd_{B(\QQ_p)}^{G(\QQ_p)}\chi$ to $G''(\QQ_p)$ is identified with $\nInd_{B''(\QQ_p)}^{G''(\QQ_p)}\chi''$, where $\chi''\in \hat{A}''(\Lambda)$ is the image of $\chi$, which lies in the category $\Rep^{\widehat\unip}(G''(\QQ_p),\Lambda)_{\xi''}$. The lemma follows as restriction along $G''(\QQ_p)\to G(\QQ_p)$ preserves colimits.
\end{proof}

We used the following result, which is well-known for when the coefficient $\Lambda$ is a field of characteristic zero and is  implicitly in \cite[(1.4)]{Dat_2009} for general coefficient. We give an easy proof for general coefficient for completeness.

\begin{lemma}\label{lemma-Iwahori-vs-Borel}
    Let $G$ be an unramified reductive group over a non-archimedean local field $F$ of residue characteristic $p$. Let $B\subseteq G$ be a Borel subgroup and $T\subseteq B$ be a maximal torus. Let $I\subseteq G$ be an Iwahori subgroup that satisfies Iwahori decomposition relative to $(B,T)$. Then there is a canonical isomorphism
    $$\cInd_I^{G(F)}\Lambda\simeq \nInd_{B(F)}^{G(F)}\chi_\univ$$
    of $G(F)$-representations.
    Here $\chi_\univ\colon T(F)\to \cO(\hat{A})^\times$ is the universal unramified character defined over $\cO(\hat{A})$.
\end{lemma}
\begin{proof}
    Let $U^-$ be the unipotent radical of the opposite Borel subgroup $B^-$ relative to $B$. Denote $T_0=T\cap I$, $U_0=U\cap I$, and $U_0^-=U^-\cap I$. By Iwahori decomposition, we have a bijection
    $$U_0^- \cdot T_0\cdot U_0 \xrightarrow{\sim} I.$$
    Let $2\rho\colon \GG_m\to T$ be the sum of positive roots relative to $B$. Denote $U_n=\Ad_{\varpi^{-2n\rho}} U_0$ and $U_n^-=\Ad_{\varpi^{2n\rho}} U_0^-$ for $n\in \ZZ$. Then 
    $$ \cdots \subseteq U_{-1}\subseteq U_0\subseteq U_1\subseteq \cdots \subseteq U(F)\quad\text{resp.}\quad \cdots \subseteq U_{-1}^-\subseteq U_0^-\subseteq U_1^-\subseteq \cdots \subseteq U^-(F)$$
    defines an exhaustive and separated filtration of $U(F)$ (resp. $U^-(F)$) by open compact subgroups. Denote $I_n=U_{-n}^-T_0U_n$. Then
    $$I_n=\Ad_{\varpi^{-2n\rho}}I$$
    is also an Iwahori subgroup of $G(F)$.

    By definition, we have
    $$\nInd_{B(F)}^{G(F)}\chi_\univ\simeq \cInd_{T_0 U(F)}^{G(F)}\Lambda=C_c^\infty(T_0 N(F)\backslash G(F)).$$
    We can write $T_0U(F)$ as a union of compact open subgroups $T_0U_n,n\in \ZZ$. Hence we can write $C_c^\infty(T_0 N(F)\backslash G(F))$ as a sequential colimit
    $$C_c^\infty(T_0 N(F)\backslash G(F))=\lim_{\longrightarrow}(\cdots\to C_c^\infty(T_0 N_0\backslash G(F))\to C_c^\infty(T_0 N_1\backslash G(F))\to\cdots),$$
    where the connecting morphisms are given by averaging. For each $n\in\ZZ$, we have a commutative diagram
    $$\begin{tikzcd}
        C_c^\infty(T_0 N_{n}\backslash G(F)) \ar[rr,two heads] && C_c^\infty(T_0 N_{n+1}\backslash G(F)) \\
        C_c^\infty(I_n\backslash G(F)) \ar[u,hook]\ar[rr,"e_n"]\ar[rd,hook] && C_c^\infty(I_{n+1}\backslash G(F)) \ar[u,hook] \\
        & C_c^\infty(N_{-n-1}^-T_0N_{n}\backslash G(F)) \ar[ru,two heads]
    \end{tikzcd}$$
    Here the lower arrow $e_n$ is defined by first restriction along $I_n=N^-_{-n}T_0N_n\supseteq N^-_{-n-1}T_0N_n$, then averaging along $N^-_{-n-1}T_0N_n\subseteq N^-_{-n-1}T_0N_{n+1}=I_{n+1}$.

    We claim that $e_n\colon C_c^\infty(I_n\backslash G(F))\to C_c^\infty(I_{n+1}\backslash G(F))$ is an isomorphism. The left multiplication of $\varpi^{-2\rho}$ defines an isomorphism 
    $$\varpi^{-2\rho}\colon C_c^\infty(I_n\backslash G(F))\xrightarrow{\sim}C_c^\infty(I_{n+1}\backslash G(F))$$
    of $G(F)$-representations. Compositing with this isomorphism, the morphism $e_n$ becomes the action of $1_{I_n\varpi^{-2\rho}I_n}\in H_{I_n}$ on $C_c^\infty(I_n\backslash G(F))$, where $H_{I_n}=\Lambda[I_n\backslash G(F)/I_n]$ is the Iwahori--Hecke algebra. However, the element $1_{I_n\varpi^{-2\rho}I_n}$ is invertible in $H_{I_n}$. This proves the claim.

    We have a morphism of ind-systems
    $$\begin{tikzcd}
        C_c^\infty(T_0 N_{0}\backslash G(F)) \ar[r,two heads] & C_c^\infty(T_0 N_{1}\backslash G(F)) \ar[r,two heads] & \cdots \\
        C_c^\infty(I_0\backslash G(F)) \ar[u,hook]\ar[r,"e_0","\simeq"swap] & C_c^\infty(I_{1}\backslash G(F)) \ar[u,hook] \ar[r,"e_1",,"\simeq"swap] & \cdots
    \end{tikzcd}$$
    We claim that they have the same colimits. Let $f\in C_c^\infty(T_0N_n\backslash G(F))$ be an element for some $n\geq 0$. Because $\{N_{-m}^-T_0 N_n\}_{m\geq n}$ forms a basis of open neighborhoods of $T_0N_n$ in $G(F)$, we can assume that $f$ is left invariant function under $N_{-m}T_0N_n$ for some $m\gg 0$. Therefore the image of $f$ in $C_c^\infty(T_0N_m\backslash G(F))$ lies in the image of $C_c^\infty(I_m\backslash G(F))$. Hence the two ind-systems have the same colimits.

    Now we have an isomorphism
    $$C_c^\infty(I\backslash G(F))\simeq \lim_{\substack{\longrightarrow\\ n\geq 0}}(C_c^\infty(I_n\backslash G(F)) \xrightarrow{\sim}  \lim_{\substack{\longrightarrow\\ n\geq 0}} (C_c^\infty(T_0N_n\backslash G(F))\simeq C_c^\infty (T_0N(F)\backslash G(F))$$
    as desired. Note that the isomorphism sends the characteristic function on $I$ to the characteristic function on $T_0N(F)\cdot I$.
\end{proof}

\begin{lemma}\label{lemma-torsion-abelian-infinite-level}
    Let $\xi\in(\hat{A}^\gen\git W_0)(\Lambda)$ be a generic point. Then $\cP^{\widehat\unip}(R\Gamma(\Sh(\sG,\sX),\Lambda))_\xi$ is concentrated in degrees $[d,2d]$.
\end{lemma}
\begin{proof}
By Lemma \ref{lemma-G''-torsion-vanising}, if $\xi''\in(\hat{A}''\git W_0)(\Lambda)$ is generic, then $R\Gamma(\Sh_{I''}(\sG'',\sX''),\Lambda)_{\xi''}$ is concentrated in degrees $[d,2d]$.
The same argument of Corollary \ref{cor-torsion-vanishing-infinite-level} shows that if $\xi''$ is generic, then $\cP^{\widehat\unip}(R\Gamma(\Sh(\sG,\sX),\Lambda))_{\xi''}$ is concentrated in degrees $[d,2d]$. 

By Lemma \ref{lemma-restr}, the restriction functor sends $\Rep^{\widehat\unip}(G(\QQ_p),\Lambda))_\xi$ to $\Rep^{\widehat\unip}(G''(\QQ_p),\Lambda))_{\xi''}$, where $\xi''$ is the image of $\xi$ along $\hat{A}\git W_0\to \hat{A}''\git W_0$. Therefore $\cP^{\widehat\unip}(R\Gamma(\Sh(\sG,\sX),\Lambda))_\xi$ is equal to the $\xi$-component of
$$R\Gamma(\scA(\sG''),\cP^{\widehat\unip}(R\Gamma(\Sh(\sG'',\sX''),\Lambda))_{\xi''}\otimes_\Lambda \cP^{\widehat\unip}(C^\infty(\scA(\sG),\Lambda))).$$
Therefore $\cP^{\widehat\unip}(R\Gamma(\Sh(\sG,\sX),\Lambda))_\xi$ is concentrated in degree $\geq d$. The upper bound follows from the cohomology amplitude.
\end{proof}

Finally, we prove our main theorem. Let $K_p\subseteq G(\QQ_p)$ be a quasi-parahoric subgroup and $K^p\subseteq \sG(\AAA_f^p)$ be a neat open compact subgroup. As in Corollary \ref{cor-torsion-vanishing-parahoric-level}, we define the $\xi$-component
$$R\Gamma(\Sh_{K^pK_p}(\sG,\sX),\Lambda)_\xi\quad\text{resp.}\quad R\Gamma(\Sh_{K^pK_p}(\sG,\sX),\Lambda)_\xi$$
for $\xi\in(\hat{A}\git W_0)(\Lambda)$.

\begin{thm}\label{thm-torsion-vanising-abelian-type}
    Assume that $\xi$ is generic. Then $R\Gamma(\Sh_{K^pK_p}(\sG,\sX),\Lambda)_\xi$ (resp. $R\Gamma_c(\Sh_{K^pK_p}(\sG,\sX),\Lambda)_\xi$) is concentrated in degrees $[d,2d]$ (resp. $[0,d]$).
\end{thm}
\begin{proof}
    The claim for $R\Gamma(\Sh_{K^pK_p}(\sG,\sX),\Lambda)_\xi$ follows from Lemma \ref{lemma-torsion-abelian-infinite-level}, using that
    $$R\Gamma(\Sh_{K^pK_p}(\sG,\sX),\Lambda)_\xi\simeq R\Gamma(\overline{K^pK_p},\cP^{\widehat\unip}(R\Gamma(\Sh(\sG,\sX),\Lambda))_\xi),$$
    where $\overline{K^pK_p}=\frac{K^pK_p}{\sZ(\QQ)^-\cap K^pK_p}$.
    The claim for $R\Gamma_c(\Sh_{K^pK_p}(\sG,\sX),\Lambda)_\xi$ follows from Poincar\'e duality.
\end{proof}

\bibliographystyle{alpha}
\bibliography{bib}

\begin{thebibliography}{DvHKZ24}

\bibitem[AB09]{Arinkin-Bez-perverse-coh}
Dima Arinkin and Roman Bezrukavnikov.
\newblock Perverse coherent sheaves.
\newblock {\em Moscow Mathematical Journal}, 10(1):3--29, 2009.

\bibitem[Ach14]{Achar-exotic-coh}
Pramad~N. Achar.
\newblock Notes on exotic and perverse-coherent sheaves, 2014.
\newblock
  \href{https://arxiv.org/abs/1409.7346}{\color{blue}\texttt{arXiv:1409.7346}}.

\bibitem[AGLR22]{AGLR-local-model}
Johannes Ansch\"utz, Ian Gleason, João Louren\c{c}o, and Timo Richarz.
\newblock On the $p$-adic theory of local models, 2022.
\newblock
  \href{https://arxiv.org/abs/2201.01234}{\color{blue}\texttt{arXiv:2201.01234}}.

\bibitem[ALWY23]{ALWY-mixed-central}
Johannes Ansch\"utz, João Louren\c{c}o, Zhiyou Wu, and Jize Yu.
\newblock Gaitsgory’s central functor and the {Arkhipov–Bezrukavnikov}
  equivalence in mixed characteristic, 2023.
\newblock
  \href{https://arxiv.org/abs/2311.04043}{\color{blue}\texttt{arXiv:2311.04043}}.

\bibitem[AR24]{AR-central}
Pramod~N. Achar and Simon Riche.
\newblock Central sheaves on affine flag varieties, 2024.
\newblock
  \href{https://lmbp.uca.fr/~riche/central.pdf}{\color{blue}\texttt{https://lmbp.uca.fr/{~}riche/central.pdf}}.

\bibitem[Art69]{Artin-approximation}
Michael Artin.
\newblock Algebraic approximation of structures over complete local rings.
\newblock {\em Publications math\'ematiques de l’I.H.\'E'.S.}, 36:23--58,
  1969.

\bibitem[Bez06]{Bez-coh-tilting}
Roman Bezrukavnikov.
\newblock Cohomology of tilting modules over quantum groups and $t$-structures
  on derived categories of coherent sheaves.
\newblock {\em Invent. Math.}, 166:327--357, 2006.

\bibitem[Boy19]{Boyer-torsion}
Pascal Boyer.
\newblock {Sur la torsion dans la cohomologie des vari\'et\'es de Shimura de
  Kottwitz--Harris--Taylor}.
\newblock {\em J. Inst. Math. Jussieu}, 18(3):499--517, 2019.

\bibitem[BR24]{BR-modular-two-affine-Hecke}
Roman Bezrukavnikov and Simon Riche.
\newblock On two modular geometric realizations of an affine {H}ecke algebra,
  2024.
\newblock
  \href{https://arxiv.org/abs/2402.08281}{\color{blue}\texttt{arXiv:2402.08281}}.

\bibitem[CS17]{Caraiani-Scholze17}
Ana Caraiani and Peter Scholze.
\newblock On the generic part of the cohomology of compact unitary {S}himura
  varieties.
\newblock {\em Annals of Mathematics}, 186(3):649--766, 2017.

\bibitem[CS24]{Caraiani-Scholze24}
Ana Caraiani and Peter Scholze.
\newblock On the generic part of the cohomology of non-compact unitary
  {S}himura varieties.
\newblock {\em Annals of Mathematics}, 199(2):483--590, 2024.

\bibitem[CT23]{Caraiani-Tamiozzo}
Ana Caraiani and Matteo Tamiozzo.
\newblock On the \'etale cohomology of {H}ilbert modular varieties with torsion
  coefficients.
\newblock {\em Compositio Math.}, 159(11):2279--2325, 2023.

\bibitem[Dat09]{Dat_2009}
Jean-Francois Dat.
\newblock {Finitude pour les représentations lisses de groupes $p$-adiques}.
\newblock {\em Journal of the Institute of Mathematics of Jussieu},
  8(2):261--333, 2009.

\bibitem[Del79]{Deligne-Shimura}
Pierre Deligne.
\newblock Variètés de {S}himura: interprétation modulaire, et techniques de
  construction de modèles canoniques.
\newblock In {\em Automorphic forms, representations and L-functions (Corvallis
  1977)}, pages 247--289. Proc. Sympos. Pure Math XXXIII, 1979.

\bibitem[DvHKZ24]{DHKZ-igusa}
Patrick Daniels, Pol van Hoften, Dongryul Kim, and Mingjia Zhang.
\newblock Igusa stacks and the cohomology of {S}himura varieties, 2024.
\newblock
  \href{https://arxiv.org/abs/2408.01348}{\color{blue}\texttt{arXiv:2408.01348}}.

\bibitem[FS21]{FS}
Laurent Fargues and Peter Scholze.
\newblock Geometrization of the local {L}anglands correspondence, 2021.
\newblock
  \href{https://arxiv.org/abs/2102.13459}{\color{blue}\texttt{arXiv:2102.13459}}.

\bibitem[GI23]{GL-meromorphic}
Ian Gleason and Alexander Ivanov.
\newblock Meromorphic vector bundles on the {Fargues--Fontaine} curve, 2023.
\newblock
  \href{https://arxiv.org/abs/2307.00887}{\color{blue}\texttt{arXiv:2307.00887}}.

\bibitem[Ham22]{Hamann-Eisenstein}
Linus Hamann.
\newblock Geometric eisenstein series, intertwining operators, and {S}hin's
  averaging formula, 2022.
\newblock
  \href{https://arxiv.org/abs/2209.08175}{\color{blue}\texttt{arXiv:2209.08175}}.

\bibitem[Han23]{Hansen-Beijing-notes}
David Hansen.
\newblock Beijing notes on the categorical local {L}anglands conjecture, 2023.
\newblock
  \href{https://arxiv.org/abs/2310.04533}{\color{blue}\texttt{arXiv:2310.04533}}.

\bibitem[He14]{He-Kottwitz-Rapoport-conj}
Xuhua He.
\newblock {Kottwitz--Rapoport conjecture on unions of affine Deligne--Lusztig
  varieties}.
\newblock {\em Ann. Scient. \'Ec. Norm. Sup.}, 49:1125--1141, 2014.

\bibitem[HKP10]{HKP-Iwahori}
Thomas Haines, Robert Kottwitz, and Amritanshu Prasad.
\newblock Iwahori--hecke algebras.
\newblock {\em Journal of the Ramanujan Mathematical Society}, 25:113--145,
  2010.

\bibitem[HL23]{Hamann-Lee-vanishing}
Linus Hamann and Si~Ying Lee.
\newblock Torsion vanishing for some {S}himura varieties, 2023.
\newblock
  \href{https://arxiv.org/abs/2309.08705}{\color{blue}\texttt{arXiv:2309.08705}}.

\bibitem[HN14]{He-Nie-minimal-length}
Xuhua He and Sian Nie.
\newblock Minimal length elements of extended affine {W}eyl groups.
\newblock {\em Compositio Math.}, 150:1903--1927, 2014.

\bibitem[Hof23]{Hoff-parahoric-display}
Manuel Hoff.
\newblock On parahoric $(\mathcal{G},\mu)$-displays, 2023.
\newblock
  \href{https://arxiv.org/abs/2311.00127}{\color{blue}\texttt{arXiv:2311.00127}}.

\bibitem[Kim19]{Kim-central-leaf}
Wansu Kim.
\newblock On central leaves of hodge-type {S}himura varieties with parahoric
  level structure.
\newblock {\em Mathematische Zeitschrift}, 291(1-2):329--363, 2019.

\bibitem[Kos21]{Koshikawa-generic}
Teruhisa Koshikawa.
\newblock On the generic part of the cohomology of local and global {S}himura
  varieties, 2021.
\newblock
  \href{https://arxiv.org/abs/2106.10602}{\color{blue}\texttt{arXiv:2106.10602}}.

\bibitem[LS18a]{Lan-Stroh-compactification-subsch}
Kai-Wen Lan and Beno\^it Stroh.
\newblock Compactifications of subschemes of integral models of {S}himura
  varieties.
\newblock {\em Forum of Mathematics, Sigma}, 6(e18), 2018.

\bibitem[LS18b]{Lan-Stroh-nearby-cycle-II}
Kai-Wen Lan and Beno\^it Stroh.
\newblock Nearby cycles of automorphic sheaves, {II}.
\newblock In {\em Cohomology of Arithmetic Groups}, pages 83--106. Springer
  International Publishing, 2018.

\bibitem[Mao25]{Mao-Hodge-well-positioned}
Shengkai Mao.
\newblock Compactifications of closed subschemes of integral models of
  {H}odge-type {S}himura varieties with parahoric level structures, 2025.
\newblock
  \href{https://arxiv.org/abs/2504.08574}{\color{blue}\texttt{arXiv:2504.08574}}.

\bibitem[Mat90]{Mathieu-filtrations}
Oliver Mathieu.
\newblock Filtrations of {$G$}-modules.
\newblock {\em Ann. Scient. \'Ec. Norm. Sup.}, 23(4):625--644, 1990.

\bibitem[MR16]{MR-exotic-t-str}
Carl Mautner and Simon Riche.
\newblock On the exotic $t$-structure in positive characteristic.
\newblock {\em International Mathematics Research Notices},
  2016(18):5727--5774, 2016.

\bibitem[Pen25]{Peng-FS-comparison-orth-unitary}
Hao Peng.
\newblock Fargues--{S}cholze parameters and torsion vanishing for special
  orthogonal and unitary groups, 2025.
\newblock
  \href{https://arxiv.org/abs/2503.04623}{\color{blue}\texttt{arXiv:2503.04623}}.

\bibitem[Per19]{Pera-integral-compactification}
Keerthi~Madapusi Pera.
\newblock Toroidal compactifications of integral models of {S}himura varieties
  of hodge type.
\newblock {\em Annales scientifiques de l'\'Ecole normale sup\'erieure},
  1:393--514, 2019.

\bibitem[PR24]{PR-p-adic-shtuka}
Georgios Pappas and Michael Rapoport.
\newblock $p$-adic shtukas and the theory of global and local {S}himura
  varieties.
\newblock {\em Camb. J. Math.}, 12:1--164, 2024.

\bibitem[Sch15]{Scholze-on-torsion}
Peter Scholze.
\newblock On torsion in the cohomology of locally symmetric varieties.
\newblock {\em Annals of Mathematics}, 183(3):945--1066, 2015.

\bibitem[Sch17]{Scholze-diamonds}
Peter Scholze.
\newblock \'{E}tale cohomology of diamonds, 2017.
\newblock
  \href{https://arxiv.org/abs/1709.07343}{\color{blue}\texttt{arXiv:1709.07343}}.

\bibitem[Ste86]{Steinberg-endomorphism}
Robert Steinberg.
\newblock Endomorphisms of linear algebraic groups.
\newblock {\em Memoirs of the American Mathematical Society}, 80, 1986.

\bibitem[SW20]{SW-Berkeley-notes}
Peter Scholze and Jared Weinstein.
\newblock Berkeley lectures on $p$-adic geometry.
\newblock {\em Annals of Mathematics Studies}, AMS-207, 2020.

\bibitem[SYZ21]{SYZ-EKOR}
Xu~Shen, Chia-Fu Yu, and Chao Zhang.
\newblock {EKOR strata for Shimura varieties with parahoric level structure}.
\newblock {\em Duke Mathematical Journal}, 170(14):3111--3236, 2021.

\bibitem[Wu21]{Wu-S-equal-T}
Zhiyou Wu.
\newblock {$S=T$} for {S}himura varieties and moduli spaces of $p$-adic
  shtukas, 2021.
\newblock
  \href{https://arxiv.org/abs/2110.10350}{\color{blue}\texttt{arXiv:2110.10350}}.

\bibitem[XZ17]{XZ-cycles}
Liang Xiao and Xinwen Zhu.
\newblock Cycles on {Shimura} varieties via geometric {S}atake, 2017.
\newblock
  \href{https://arxiv.org/abs/1707.05700}{\color{blue}\texttt{arXiv:1707.05700}}.

\bibitem[XZ19]{XZ-vector}
Liang Xiao and Xinwen Zhu.
\newblock On vector-valued twisted conjugate invariant functions on a group.
\newblock In {\em Representation of Reductive Groups}, volume 101, pages
  361--425. Proceedings of Symposia in Pure Mathematics, 2019.

\bibitem[Zhu20]{Zhu-coherent-sheaf}
Xinwen Zhu.
\newblock Coherent sheaves on the stack of {L}anglands parameters, 2020.
\newblock
  \href{https://arxiv.org/abs/2008.02998}{\color{blue}\texttt{arXiv:2008.02998}}.

\bibitem[Zhu25]{Tame}
Xinwen Zhu.
\newblock Tame categorical local {L}anglands correspondence, 2025.
\newblock
  \href{https://arxiv.org/abs/2504.07482}{\color{blue}\texttt{arXiv:2504.07482}}.

\end{thebibliography}
\end{document}